%% file: Combinatorial_solution_revised_5.tex
\documentclass[11pt]{amsart}

\usepackage{amsmath, amssymb}
\usepackage{amsfonts}
\usepackage{mathrsfs}
\usepackage[arrow,matrix,curve,cmtip,ps]{xy}
\usepackage{color}

\usepackage{float}

\usepackage{etex}
\usepackage{tikz}
\usetikzlibrary{arrows,positioning,scopes,calc,intersections}

\usepackage{amsthm}
\usepackage{mathtools} 
\usepackage{array}
\usepackage{mathdots}

\usepackage[top= 2.25cm, bottom=2cm, left = 2cm, right= 2cm]{geometry}
\usepackage{hyperref}
\usepackage{setspace}

\allowdisplaybreaks

\numberwithin{equation}{section}

\input{Macros_reps}

\begin{document}
\title[A combinatorial solution to M{\oe}glin's parametrization of Arthur packets]
{A combinatorial solution to M{\oe}glin's parametrization of Arthur packets for p-adic quasisplit $Sp(N)$ and $O(N)$}

\author{Bin Xu}

\address{Yau Mathematical Sciences Center and Department of Mathematics \\ Tsinghua University, Beijing, China}
\email{binxu@tsinghua.edu.cn}

\subjclass[2010]{22E50 (primary); 11F70 (secondary)}
\keywords{symplectic and orthogonal group, Arthur packet, Jacquet functor}


\maketitle

\begin{abstract}
We develop a general procedure to study the combinatorial structure of Arthur packets for $p$-adic quasisplit $Sp(N)$ and $O(N)$ following the works of M{\oe}glin. This will allow us to answer many delicate questions concerning the Arthur packets of these groups, for example the size of the packets.
\end{abstract}

\section{Introduction}
\label{sec: introduction}

Let $F$ be a $p$-adic field and $G$ be a quasisplit symplectic or special orthogonal group, i.e., $G = Sp(2n), SO(2n+1)$ and $SO(2n, \eta)$. Here $\eta$ is a quadratic character associated with a quadratic extension $E/F$ by the local class field theory, and $SO(2n, \eta)$ is the outer form of the split $SO(2n)$ with respect to $E/F$ and an outer automorphism $\theta_{0}$ induced from the conjugate action of $O(2n)$. We let $\theta_{0} = id$ in other cases, and write $\Sigma_{0} = \langle\theta_{0}\rangle$, $G^{\Sigma_{0}} = G \rtimes \Sigma_{0}$. So for $G = SO(2n, \eta)$, $G^{\Sigma_{0}} \cong O(2n, \eta)$. For simplicity, we will denote $G(F)$ by $G$, which should not cause any confusion in the context. Let $\D{G}$ be the complex dual group of $G$, and $\L{G}$ be the Langlands dual group of $G$. Here we can simplify the Langlands dual groups as in the following table:
\begin{center}
\begin{spacing}{1.5}
\begin{tabular}{| c | c | }
     \hline
      $G$                        &      $\L{G}$              \\
     \hline
      $Sp(2n)$                &      $SO(2n+1, \mathbb{C})$  \\
     \hline
      $SO(2n+1)$           &      $Sp(2n, \mathbb{C})$ \\
     \hline
      $SO(2n, \eta)$       &      $SO(2n, \mathbb{C}) \rtimes \Gal{E/F}$     \\     
     \hline
\end{tabular}
\end{spacing}
\end{center}
In the last case, we will fix an isomorphism $SO(2n, \mathbb{C}) \rtimes \Gal{E/F} \cong O(2n, \mathbb{C})$. So in either of these cases, there is a natural embedding $\xi_{N}$ of $\L{G}$ into $GL(N, \mathbb{C})$ up to $GL(N, \mathbb{C})$-conjugacy, where $N = 2n+1$ if $G = Sp(2n)$ or $N = 2n$ otherwise. Let $W_{F}$ be the Weil group, the local Langlands group can be defined to be 
\[
L_{F} := W_{F} \times SL(2, \mathbb{C}).
\]
An Arthur parameter of $G$ is a $\D{G}$-conjugacy class of admissible homomorphisms
\[
\underline{\q}: L_{F} \times SL(2, \mathbb{C}) \longrightarrow \L{G},
\]
such that $\underline{\q}|_{W_{F}}$ is bounded. We denote the set of Arthur parameters of $G$ by $\Q{G}$. Let $\D{\theta}_{0}$ be the dual automorphism of $\theta_{0}$, then $\Sigma_{0}$ acts on $\Q{G}$ through $\D{\theta}_{0}$, and we denote the corresponding set of $\Sigma_{0}$-orbits by $\cQ{G}$. Let $\Pkt{}(G)$ be the set of equivalence classes of irreducible admissible representations of $G$, and we denote by $\cPkt{}(G)$ the set of $\Sigma_{0}$-orbits in $\Pkt{}(G)$. For $\q \in \cQ{G}$, Arthur \cite{Arthur:2013} shows there exists a finite ``multi-set" $\cPkt{\q}$ of elements in $\cPkt{}(G)$, which is related to certain twisted character on $GL(N)$ through the twisted endoscopic character identity (cf. \cite{Xu:Apacket}, Section 4). We call $\cPkt{\q}$ an Arthur packet of $G$. M{\oe}glin \cite{Moeglin:2011} constructs the elements in $\cPkt{\q}$, and shows it is in fact {\bf multiplicity free}. As a result, we can also define $\Pkt{\q}^{\Sigma_{0}}$ to be the set of irreducible representations of $G^{\Sigma_{0}}$, whose restriction to $G$ have irreducible constituents in $\cPkt{\q}$. To understand the structure of $\Pkt{\q}^{\Sigma_{0}}$, we need to introduce the set $Jord(\q)$ of Jordan blocks associated with $\q$.

For $\q \in \cQ{G}$, by composing with $\xi_{N}$ we get an equivalence class of $N$-dimensional self-dual representation of $L_{F} \times SL(2, \mathbb{C})$. So we can decompose $\q$ as follows
\begin{align}
\label{eq: Arthur parameter}
\q = \bigoplus_{i=1}^{r} l_{i} \q_{i} = \bigoplus_{i=1}^{r}l_{i}(\rho_{i} \otimes \nu_{a_{i}} \otimes \nu_{b_{i}}).
\end{align}
Here $\rho_{i}$ are equivalence classes of irreducible unitary representations of $W_{F}$, which can be identified with irreducible unitary supercuspidal representations of $GL(d_{\rho_{i}})$ under the local Langlands correspondence (cf. \cite{HarrisTaylor:2001}, \cite{Henniart:2000}, and \cite{Scholze:2013}). And $\nu_{a_{i}}$ (resp. $\nu_{b_{i}}$) are the $(a_{i}-1)$-th (resp. $(b_{i}-1)$-th) symmetric power representations of $SL(2, \mathbb{C})$. The irreducible constituent $\rho_{i} \otimes \nu_{a_{i}} \otimes \nu_{b_{i}}$ has dimension $n_{i} = n_{(\rho_{i}, a_{i}, b_{i})}$ and multiplicity $l_{i}$. We define the multi-set of Jordan blocks for $\q$ as follows,
\[
Jord(\q) := \{(\rho_{i}, a_{i}, b_{i}) \text{ with multiplicity } l_{i}: 1 \leqslant i \leqslant r \}.
\]
Moreover, for any  $\rho$ let us define
\[
Jord_{\rho}(\q) := \{(\rho', a', b') \in Jord(\q): \rho' = \rho\}.
\]
One can define the parity for self-dual irreducible unitary representations $\rho$ of $W_{F}$ as in (\cite{Xu:cusp}, Section 3). Then we say $(\rho_{i}, a_{i}, b_{i}) $ is of {\bf orthogonal type} if $a_{i} + b_{i}$ is even when $\rho_{i}$ is of orthogonal type, and $a_{i} + b_{i}$ is odd when $\rho_{i}$ is of symplectic type. Similarly we say $(\rho_{i}, a_{i}, b_{i})$ is of {\bf symplectic type} if $a_{i} + b_{i}$ is odd when $\rho_{i}$ is of orthogonal type, and $a_{i} + b_{i}$ is even when $\rho_{i}$ is of symplectic type. Let $\q_{p}$ be the parameter whose Jordan blocks consist of those in $Jord(\q)$ with the same parity as $\D{G}$, and 
let $\q_{np}$ be any parameter such that 
\[
\q = \q_{np} \+ \q_{p} \+ \q_{np}^{\vee},
\]
where $\q_{np}^{\vee}$ is the dual of $\q_{np}$. We also denote by $Jord(\q)_{p}$ the set of Jordan blocks in $Jord(\q_{p})$ without multiplicity. Then let us define
\[
\D{\S{\q^{>}}^{\Sigma_{0}}} = \{\e(\cdot) \in (\mathbb{Z}/2\mathbb{Z})^{Jord(\q_{p})} : \prod_{(\rho, a, b) \in Jord(\q_{p})} \e(\rho, a, b) = 1\}.
\]
and 
\[
\D{\S{\q}^{\Sigma_{0}}} = \{\e \in \D{\S{\q^{>}}^{\Sigma_{0}}} : \e(\rho, a, b) = \e(\rho', a', b') \text{ if } (\rho, a, b) = (\rho', a', b') \text{ in } Jord(\q)_{p}\}.
\]
If we choose a representative $\underline{\q}: L_{F} \times SL(2, \mathbb{C}) \rightarrow \L{G}$, then one can show $\D{\S{\q}^{\Sigma_{0}}}$ is canonically isomorphic to the group of characters of the component group of 
\[
\Cent(\Im \underline{\q}, \D{G} \rtimes \langle\D{\theta_{0}}\rangle) / Z(\D{G})^{\Gal{F}}.
\]
So we will also call elements in $\D{\S{\q}^{\Sigma_{0}}}$ (and also $\D{\S{\q^{>}}^{\Sigma_{0}}}$) characters. It follows from Arthur's theory that there is a canonical way to associate any irreducible representation in $\Pkt{\q}^{\Sigma_{0}}$ with an element $\e \in \D{\S{\q}^{\Sigma_{0}}}$ (cf. \cite{Xu:Apacket}, Section 8). Let us denote the direct sum of all irreducible representations associated with $\e \in \D{\S{\q}^{\Sigma_{0}}}$ by $\r^{\Sigma_{0}}_{W}(\q, \e)$, then
\begin{align}
\label{eq: W parametrization}
\Pkt{\q}^{\Sigma_{0}} = \bigoplus_{\e \in \D{\S{\q}^{\Sigma_{0}}}} \r^{\Sigma_{0}}_{W}(\q, \e),
\end{align}
where we identify $\Pkt{\q}^{\Sigma_{0}}$ with the direct sum of all its elements.

M{\oe}glin's construction of $\Pkt{\q}^{\Sigma_{0}}$ comes with a parametrization by $\D{\S{\q^{>}}^{\Sigma_{0}}}$. It also depends on some total order $>_{\q}$ on $Jord_{\rho}(\q_{p})$ for each $\rho$. To describe the condition on $>_{\q}$, we need to write the Jordan blocks differently. For $(\rho, a, b) \in Jord(\q_{p})$, let us write $A = (a+b)/2-1$, $B = |a-b|/2$, and set $\zeta = \zeta_{a, b} = \text{Sign}(a - b)$ if $a \neq b$ and arbitrary otherwise. Then we can denote $(\rho, a, b)$ also by $(\rho, A, B, \zeta)$. We say $>_{\q}$ is ``admissible" if it satisfies
\begin{align*}
(\mathcal{P}): \quad & \forall (\rho, A, B, \zeta), (\rho, A', B', \zeta') \in Jord(\q_{p}) \text{ with } A > A', B > B' \text{ and } \zeta = \zeta', \\
& \text{ then } (\rho, A, B, \zeta) >_{\q} (\rho, A', B', \zeta').
\end{align*}
Since the sign $\zeta$ is relevant in this condition, M{\oe}glin's parametrization will also depend on the choice of $\zeta_{a, b}$, when $a = b$. First, we have
\begin{align}
\label{eq: M parametrization}
\Pkt{\q}^{\Sigma_{0}} = \bigoplus_{\e \in \D{\S{\q^{>}}^{\Sigma_{0}}}} \r^{\Sigma_{0}}_{M, >_{\q}}(\q, \e).
\end{align}
The following theorem gives the connection between \eqref{eq: M parametrization} and \eqref{eq: W parametrization}.

\begin{theorem}[\cite{Xu:Apacket}, Theorem 8.9]
\label{thm: M/W}
There exists a character of $\e^{M/W}_{\q} \in \D{\S{\q^{>}}^{\Sigma_{0}}}$ such that 
\[
\r^{\Sigma_{0}}_{M, >_{\q}}(\q, \e) = \begin{cases}
                           \r^{\Sigma_{0}}_{W}(\q, \e \e^{M/W}_{\q}), & \text{ if $\e \e^{M/W}_{\q} \in \D{\S{\q}^{\Sigma_{0}}}$,} \\
                            0, & \text{ otherwise.}
                          \end{cases}
\]
\end{theorem}

In \cite{Xu:Apacket}, we define 
\begin{align}
\label{eq: M/W}
\e^{M/W}_{\q} := \e^{MW/W}_{\q}  \e^{M/MW}_{\q},
\end{align}
for $\e^{MW/W}_{\q}$ and $\e^{M/MW}_{\q}$ in $\D{\S{\q^{>}}^{\Sigma_{0}}}$. Here we will recall the definition of these characters. 

To define $\e^{MW/W}_{\q}$, we need to first define a set $\mathcal{Z}_{MW/W}(\q)$ of {\bf unordered pairs} of Jordan blocks from $Jord(\q_{p})$ as follows. We call a pair $\{(\rho, a, b), (\rho', a', b') \in Jord(\q_{p})\}$ is contained in $\mathcal{Z}_{MW/W}(\q)$ if and only if $\rho = \rho'$, and it is in one of the following situations.

\begin{enumerate}

\item Case: $a, b$ are even and $a', b'$ are odd. 

   \begin{enumerate}
   
   \item If $\zeta_{a, b} = -1$ and 
             \(
            \begin{cases}
                      & \zeta_{a', b'} = -1 \Rightarrow (\rho, a, b) >_{\q} (\rho, a', b'), a > a' \\
                      & \zeta_{a', b'} = +1 \Rightarrow a > a'          
             \end{cases}
             \)
   \item If $\zeta_{a, b} = \zeta_{a', b'} = +1$ and \(
            \begin{cases}
                      & (\rho, a, b) >_{\q} (\rho, a', b') \Rightarrow a' > a, b > b' \\
                      & (\rho, a, b) <_{\q} (\rho, a', b') \Rightarrow a > a', b > b'          
             \end{cases}
             \)
                
   \end{enumerate}

\item Case : $a$ is odd, $b$ is even and $a'$ is even, $b'$ is odd. 

   \begin{enumerate}
   
   \item If $\zeta_{a, b} = -1$ and 
             \(
            \begin{cases}
                      & \zeta_{a', b'} = -1 \Rightarrow (\rho, a, b) >_{\q} (\rho, a', b'), a < a' \\
                      & \zeta_{a', b'} = +1 \text{ and } \begin{cases}
                                                        (\rho, a, b) >_{\q} (\rho, a', b') \Rightarrow a < a' \\        
                                                        (\rho, a, b) <_{\q} (\rho, a', b') \Rightarrow a > a'
                                                      \end{cases}
             \end{cases}
             \)
   
   \item If $\zeta_{a, b} = \zeta_{a', b'} = +1$ and \(
            \begin{cases}
                      & (\rho, a, b) >_{\q} (\rho, a', b') \Rightarrow a < a', b > b' \\
                      & (\rho, a, b) <_{\q} (\rho, a', b') \Rightarrow a > a', b > b'          
             \end{cases}
             \)
                
   \end{enumerate}

\end{enumerate}
For $(\rho, a, b) \in Jord(\q_{p})$, let
\[
\mathcal{Z}_{MW/W}(\q)_{(\rho, a, b)} := \{(\rho', a', b') \in Jord(\q_{p}) : \text{ the pair of } (\rho, a, b) \text { and } (\rho', a', b') \text { lies in } \mathcal{Z}_{MW/W}(\q)\}.
\]
Then we can define
\[
\e^{MW/W}_{\q}(\rho, a, b) := (-1)^{|\mathcal{Z}_{MW/W}(\q)_{(\rho, a, b)}|}.
\]
Next, we define $\e^{M/MW}_{\q}$ according to the following rule. Let $(\rho, a, b) \in Jord(\q_{p})$.

\begin{enumerate}

\item If $a + b$ is odd, $\e^{M/MW}_{\q}(\rho, a, b) = 1$.

\item If $a + b$ is even, let 
\[
m = \sharp \{(\rho, a', b') \in Jord(\q): a', b' \text{ odd}, \zeta_{a', b'} = -1, (\rho, a', b') >_{\q} (\rho, a, b) \},
\] 
and 
\[
n = \sharp \{(\rho, a' , b')  \in Jord(\q): a', b' \text{ odd}, (\rho, a', b') <_{\q} (\rho, a, b) \}.
\] 
Then
\[
\e^{M/MW}_{\q}(\rho, a, b) = \begin{cases}
                                              1 &\text{ if $a, b$ even, } \\
                                              (-1)^{m} &\text{ if $a, b$ odd, $\zeta_{a, b} = +1$, } \\
                                              (-1)^{m + n} &\text{ if $a, b$ odd, $\zeta_{a, b} = -1$. }
                                              \end{cases}
\]

\end{enumerate}

M{\oe}glin further parametrizes the irreducible constituents in $\r^{\Sigma_{0}}_{M, >_{\q}}(\q, \e)$. To describe that, we need to briefly go through all the stages of M{\oe}glin's construction of $\Pkt{\q}^{\Sigma_{0}}$. Let us denote by $\q_{d}$ the composition of $\q$ with 
\[
\Delta: W_{F} \times SL(2, \mathbb{C}) \rightarrow W_{F} \times SL(2, \mathbb{C}) \times SL(2, \mathbb{C}),
\]
which is the diagonal embedding of $SL(2, \mathbb{C})$ into $SL(2, \mathbb{C}) \times SL(2, \mathbb{C})$ when restricted to $SL(2, \mathbb{C})$, and is identity on $W_{F}$. It is easy to see
\[
Jord(\q_{d}) = \cup_{(\rho, A, B. \zeta) \in Jord(\q)} \cup_{C \in [B, A]} \{ (\rho, C, C, +1) \}. 
\]
We call $\q$ has {\bf discrete diagonal restriction} if $\q = \q_{p}$ and $Jord(\q_{d})$ is multiplicity free. Note the second condition is equivalent to saying the intervals $[B, A], [B', A']$ do not intersect for any $(\rho, A, B, \zeta), (\rho, A', B', \zeta')$ in $Jord_{\rho}(\q)$. In this case, $Jord_{\rho}(\q)$ has a natural order $>_{\q}$, namely
\[
(\rho, A, B, \zeta) >_{\q} (\rho, A', B', \zeta') \text{ if and only if } A > A'.
\]
Among the parameters with discrete diagonal restriction, we call $\q$ is {\bf elementary} if $A = B$ for all $(\rho, A, B, \zeta) \in Jord(\q)$. For the elementary parameters, M{\oe}glin \cite{Moeglin:2006} shows $\r^{\Sigma_{0}}_{M, >_{\q}}(\q, \e)$ is irreducible. 

Suppose $\q$ has discrete diagonal restriction, M{\oe}glin shows the irreducible constitutes of $\r^{\Sigma_{0}}_{M, >_{\q}}(\q, \e)$ can be parametrized by pairs of integer-valued functions $(\ul, \ueta)$ over $Jord(\q)$, such that 
\begin{align}
\label{eq: refined parameter}
\ul(\rho, A, B, \zeta) \in [0, [(A-B+1)/2]] \text{ and } \ueta(\rho, A, B, \zeta) \in \{\pm 1\},
\end{align}
and
\begin{align}
\label{eq: endoscopic character formula}
\e(\rho, A, B, \zeta) = \e_{\ul, \ueta}(\rho, A, B, \zeta) := \ueta(\rho, A, B, \zeta)^{A - B + 1} (-1)^{[(A-B+1)/2] + \ul(\rho, A, B, \zeta)}.
\end{align}
Moreover,
\begin{align*}
\r_{M, >_{\q}}^{\Sigma_{0}}(\q, \ul, \ueta) & \hookrightarrow \times_{(\rho, A, B, \zeta) \in Jord(\q)} 
       \begin{pmatrix}
              \zeta B & \cdots & -\zeta A \\
              \vdots &  & \vdots \\
              \zeta (B + \ul(\rho, A, B, \zeta) - 1) & \cdots & - \zeta (A - \ul(\rho, A, B, \zeta) + 1)
       \end{pmatrix} \\
& \times \r_{M}^{\Sigma_{0}}\Big(\cup_{(\rho, A, B, \zeta) \in Jord(\q)} \cup_{C \in [B + \ul(\rho, A, B, \zeta), A - \ul(\rho, A, B, \zeta)]} (\rho, C, C, \ueta(\rho, A, B, \zeta)(-1)^{C - B - \ul(\rho, A, B, \zeta)}, \zeta)\Big)
\end{align*}
as the unique irreducible subrepresentation (see \eqref{eq: generalized Speh} and \eqref{eq: induced module}). There is an obvious equivalence relation to be made here on pairs $(\ul, \ueta)$, namely 
\[
(\ul, \ueta) \sim_{\Sigma_{0}} (\ul', \ueta')
\]
if and only if $\ul = \ul'$ and $(\ueta/\ueta') (\rho, A, B, \zeta)= 1$ unless $\ul(\rho, A, B, \zeta) = (A - B +1)/2$. Then 
\begin{align}
\label{eq: refined decomposition}
\r^{\Sigma_{0}}_{M, >_{\q}}(\q, \e) = \bigoplus_{\{(\ul, \ueta): \, \e = \e_{\ul, \ueta} \}/\sim_{\Sigma_{0}}} \r^{\Sigma_{0}}_{M, >_{\q}}(\q, \ul, \ueta).
\end{align}

To get to the more general case $\q = \q_{p}$, we need to choose an admissible order $>_{\q}$ on $Jord(\q)$. We can index $Jord_{\rho}(\q)$ such that 
\[
(\rho, A_{i}, B_{i}, \zeta_{i}) >_{\q} (\rho, A_{i-1}, B_{i-1}, \zeta_{i-1}).
\]
We say $\q_{\gg}$ dominates $\q$ with respect to $>_{\q}$ if $Jord_{\rho}(\q_{\gg})$ consists of $(\rho, A_{\gg, i}, B_{i} + T_{\gg, i}, \zeta_{\gg, i}) := (\rho, A_{i} + T_{i}, B_{i} + T_{i}, \zeta_{i})$ for $T_{i} \geqslant 0$, and inherits the same admissible order $>_{\q}$. We can further choose $\q_{\gg}$ to have discrete diagonal restriction with the natural order $>_{\q}$. After identifying $Jord(\q)$ with $Jord(\q_{\gg})$ in the natural way, we can define for any pair of functions $(\ul, \ueta)$ satisfying \eqref{eq: refined parameter} and \eqref{eq: endoscopic character formula},
\begin{align}
\label{eq: M construction general}
\r_{M, >_{\q}}^{\Sigma_{0}}(\q, \ul, \ueta) := \circ_{\{\rho: Jord_{\rho}(\q) \neq \emptyset \}} \circ_{(\rho, A_{i}, B_{i}, \zeta_{i}) \in Jord_{\rho}(\q)} \Jac_{(\rho, A_{i} + T_{i}, B_{i} + T_{i}, \zeta_{i}) \mapsto (\rho, A_{i}, B_{i}, \zeta_{i})} \r_{M, >_{\q}}^{\Sigma_{0}}(\q_{\gg}, \ul, \ueta),
\end{align}
where $i$ is decreasing in the composition of Jacquet functors (see \eqref{eq: shifting functor}). (This definition is different from that in (\cite{Xu:Apacket}, Section 8), for there we take a total order $>_{\q}$ on $Jord(\q)$. But it follows from Lemma~\ref{lemma: permuting Jacquet 1} that only the restriction of $>_{\q}$ to $Jord_{\rho}(\q)$ for each $\rho$ matters, and the two definition will give the same result.) Then M{\oe}glin shows the following facts (cf. \cite{Xu:Apacket}, Proposition 8.5 and Corollary 8.7):

\begin{enumerate}

\item $\r_{M, >_{\q}}^{\Sigma_{0}}(\q, \ul, \ueta)$ only depends on the choice of order $>_{\q}$, and it is either irreducible or zero.

\item If $\r_{M, >_{\q}}^{\Sigma_{0}}(\q, \ul, \ueta) \cong \r_{M, >_{\q}}^{\Sigma_{0}}(\q, \ul', \ueta') \neq 0$, then $(\ul, \ueta) \sim_{\Sigma_{0}}  (\ul', \ueta')$. 

\item The decomposition \eqref{eq: refined decomposition} still holds.

\end{enumerate}

Finally, for general $\q \in \cQ{G}$, we have
\[
\r^{\Sigma_{0}}_{M, >_{\q}}(\q, \e) = \Big(\times_{(\rho, a, b) \in Jord(\q_{np})}Sp(St(\rho, a), b) \Big) \rtimes \r^{\Sigma_{0}}_{M, >_{\q}}(\q_{p}, \e)
\]
(see \eqref{eq: Speh}). Moreover, M{\oe}glin shows 
\[
\r^{\Sigma_{0}}_{M, >_{\q}}(\q, \ul, \ueta) := \Big(\times_{(\rho, a, b) \in Jord(\q_{np})}Sp(St(\rho, a), b) \Big) \rtimes \r^{\Sigma_{0}}_{M, >_{\q}}(\q_{p}, \ul, \ueta)
\]
is irreducible (cf. \cite{Moeglin1:2006}, Theorem 6), when $\r^{\Sigma_{0}}_{M, >_{\q}}(\q_{p}, \ul, \ueta) \neq 0$.

To summarize, for $\q \in \cQ{G}$ we can refine the decomposition \eqref{eq: M parametrization} as follows
\begin{align}
\label{eq: M refined parametrization}
\Pkt{\q}^{\Sigma_{0}} = \bigoplus_{\{(\ul, \ueta): \,  \prod_{(\rho, a, b) \in Jord(\q_{p})} \e_{\ul, \ueta}(\rho,a ,b) = 1 \}/\sim_{\Sigma_{0}}} \r^{\Sigma_{0}}_{M, >_{\q}}(\q, \ul, \ueta).
\end{align}
where $\r^{\Sigma_{0}}_{M, >_{\q}}(\q, \ul, \ueta)$ is either irreducible or zero. So it is natural to ask the following question: 
\begin{question}
\label{que: nonvanishing}
When $\r^{\Sigma_{0}}_{M, >_{\q}}(\q, \ul, \ueta) \neq 0$?  
\end{question}

The main goal of this paper is to answer this question. An explicit answer to this question will certainly allow us to count the size of the Arthur packets (see Example~\ref{eg: three intervals}). In fact, it contains more information than that. For example, one can use it to determine the zeros of local normalized intertwining operators, which is important for describing the residue spectrum of automorphic forms (see \cite{Moeglin:2011}). Another example is to use it to characterize the image of the local theta correspondence of type I in many cases (see \cite{Moeglin2:2011}). We hope the results and techniques of this paper will open the door for investigating many other delicate questions concerning the Arthur packets.

Now we will describe our results. In the simplest case, one can consider $Jord(\q_{p})$ consisting of two Jordan blocks $(\rho, A_{2}, B_{2}, \zeta), (\rho, A_{1}, B_{1}, \zeta)$ satisfying $A_{2} \geqslant A_{1}, B_{2} \geqslant B_{1}$. Let the order be
\[
(\rho, A_{2}, B_{2}, \zeta) >_{\q} (\rho, A_{1}, B_{1}, \zeta).
\]
In Proposition~\ref{prop: basic}, we will prove $\r^{\Sigma_{0}}_{M, >_{\q}}(\q, \ul, \ueta) \neq 0$ if and only if
\begin{align*}
\begin{cases}
\eta_{2} = (-1)^{A_{1} - B_{1}}\eta_{1}       & \Rightarrow A_{2} - l_{2} \geqslant A_{1} - l_{1}, \quad B_{2} + l_{2} \geqslant B_{1} + l_{1},  \\
\eta_{2} \neq (-1)^{A_{1} - B_{1}}\eta_{1}  & \Rightarrow B_{2} + l_{2} > A_{1} - l_{1}.   
\end{cases}           
\end{align*}
Suppose we change the assumption such that $B_{2} \leqslant B_{1}$. Then in Lemma~\ref{lemma: sufficient condition}, we will prove $\r^{\Sigma_{0}}_{M, >_{\q}}(\q, \ul, \ueta) \neq 0$ if and only if 
\begin{align*}
\begin{cases}
\eta_{2} = (-1)^{A_{1} - B_{1}}\eta_{1}       & \Rightarrow 0 \leqslant l_{2} - l_{1} \leqslant (A_{2} - B_{2}) - (A_{1} - B_{1}),  \\
\eta_{2} \neq (-1)^{A_{1} - B_{1}}\eta_{1}  & \Rightarrow l_{2} + l_{1} > A_{1} - B_{1}.   
\end{cases} 
\end{align*}
The general case is much more complicated, because we also need to take account of pairs of Jordan blocks, which are not necessarily in adjacent positions under the chosen order. To do so, we will need to change the order, and the point is the parametrization will change as well. So in the end, we will develop a procedure (cf. Section~\ref{sec: procedure}), and it will give rise to some explicit combinatorial conditions on $(\ul, \ueta)$  like those in the case above. In Appendix~\ref{sec: example}, we give an example to demonstrate how it works. 

The key input in our procedure is an explicit formula describing how the parametrization changes with respect to the change of order $>_{\q}$ when $\q = \q_{p}$ (cf. Section~\ref{sec: change of order}). Since this result could have interests by itself, we would like to state it here. Let us consider any two adjacent Jordan blocks $(\rho, A_i, B_i, \zeta_{i})$ $(i = 1, 2)$ under an admissible order $>_{\psi}$ with
\[
(\rho, A_2, B_2, \zeta_{2}) >_{\psi} (\rho, A_1, B_1, \zeta_{1}).
\]
Suppose the new order $>'_{\psi}$ obtained by switching the two is still admissible. Then by definition, either $\zeta_{1} \neq \zeta_{2}$ or one of $\{[B_{i}, A_{i}]\}_{i = 1, 2}$ is included in the other. Let us define $\q_{-}$ by
\[
Jord(\q_{-}) = Jord(\q) \backslash \{(\rho, A_{2}, B_{2}, \zeta_{2}), (\rho, A_{1}, B_{1}, \zeta_{1})\}.
\]

\begin{theorem}[cf. Theorem~\ref{thm: change order equal sign} and Theorem~\ref{thm: change order unequal sign}]

Suppose 
\[
\r^{\Sigma_{0}}_{M, >_{\q}}(\q, \ul, \ueta) = \r^{\Sigma_{0}}_{M, >'_{\q}}(\q, \ul', \ueta') \neq 0,
\]
then the restrictions of $(\ul, \ueta)$ and $(\ul', \ueta')$ to $Jord(\q_{-})$ are equivalent ($\sim_{\Sigma_{0}}$) and the following conditions are satisfied.

\begin{enumerate}

\item If $\zeta_{1} = \zeta_{2}$, it suffices to consider the case $[B_2, A_2] \supseteq [B_1, A_1]$. Then we are in one of the following situations.

\begin{enumerate}

\item If $\eta_{2} \neq (-1)^{A_{1} - B_{1}} \eta_{1}$ and $\eta'_{1} = (-1)^{A_{2} - B_{2}} \eta'_{2}$, then
         \[
         \begin{cases}
         l_{1} = l'_{1} \\
          l_{2} - l'_{2} = (A_{1} - B_{1} - 2l_{1}) + 1 \\
         \eta'_{1} = (-1)^{A_{2} - B_{2}} \eta_{1} 
         \end{cases}
         \]
         
\item If $\eta_{2} = (-1)^{A_{1} - B_{1}} \eta_{1}$ and $\eta'_{1} \neq (-1)^{A_{2} - B_{2}}\eta'_{2}$, then 
         \[
         \begin{cases}
         l_{1} = l'_{1} \\
         l'_{2} - l_{2} = (A_{1} - B_{1} - 2l_{1}) + 1 \\
         \eta'_{1} = (-1)^{A_{2} - B_{2}} \eta_{1} 
         \end{cases}
         \]         
         
\item If $\eta_{2} = (-1)^{A_{1} - B_{1}} \eta_{1}$ and $\eta'_{1} = (-1)^{A_{2} - B_{2}}\eta'_{2}$, then 
         \[
         \begin{cases}
         l_{1} = l'_{1} \\
         (l'_{2} - l'_{1}) + (l_{2} - l_{1}) = (A_{2} - B_{2}) - (A_{1} - B_{1}) \\
         \eta'_{1} = (-1)^{A_{2} - B_{2}} \eta_{1} 
         \end{cases}
         \]          

\end{enumerate}

\item If $\zeta_{1} \neq \zeta_{2}$, then 
\[
\begin{cases}
l'_{2} = l_{2} \\
l'_{1} = l_{1} \\
\eta_{2} = (-1)^{A_{1} - B_{1} + 1} \eta'_{2} \\
\eta_{1} = (-1)^{A_{2} - B_{2} + 1} \eta'_{1}
\end{cases}
\]

\end{enumerate}
In both cases, we have denoted
\[
l_{i} = \ul(\rho, A_{i}, B_{i}, \zeta_{i}), \quad l'_{i} = \ul'(\rho, A_{i}, B_{i}, \zeta_{i}),
\] 
and 
\[
\eta_{i} = \ueta(\rho, A_{i}, B_{i}, \zeta_{i}), \quad \eta'_{i} = \ueta'(\rho, A_{i}, B_{i}, \zeta_{i}).
\] 
for $i = 1, 2$.

\end{theorem}

{\bf Acknowledgements}: This work was mainly done when I was a member of IAS during the academic year 2014-2015, and I was supported by the National Science Foundation number agreement No. DMS-1128155 and DMS-1252158. I also want to thank C.M{\oe}glin for useful comments.


\section{Conventions}
\label{sec: convention}

Now we want to set up some conventions for this paper. We follow the notations in the introduction. Since only $\q_{p}$ is relevant in answering Question~\ref{que: nonvanishing}, we will assume $\q = \q_{p}$ in the rest of the paper. We will also require $(\ul, \ueta)$ to always satisfy 
\begin{align}
\label{eq: quasisplit}
\prod_{(\rho, a, b) \in Jord(\q)} \e_{\ul, \ueta}(\rho,a ,b) = 1.
\end{align}
So we will not write down this condition later in the paper. 

In many arguments of the paper, we need to fix a self-dual irreducible unitary supercuspidal representation $\rho$ of $GL(d_{\rho})$. So if $\q_{\gg}$ is a dominating parameter of $\q$, we would like to define
\[
\Jac_{X^{c}}  := \circ_{\rho' \neq \rho} \circ_{(\rho', A', B', \zeta') \in Jord_{\rho'}(\q)} \Jac_{(\rho', A'_{\gg}, B'_{\gg}, \zeta') \mapsto (\rho', A', B', \zeta')}
\]
and 
\[
\mathcal{C}_{X^{c}} := \times_{\rho' \neq \rho} \times_{(\rho', A', B', \zeta') \in Jord_{\rho'}(\q)} \begin{pmatrix}
              \zeta' B'_{\gg} & \cdots & \zeta'(B' + 1) \\
              \vdots &  & \vdots \\
              \zeta' A'_{\gg} & \cdots & \zeta'(A' + 1)
       \end{pmatrix}.
\]       
Since we are taking $\rho' \neq \rho$ in $\Jac_{X^{c}}$ (resp. $\mathcal{C}_{X^{c}}$), it will ``commute" with all kinds of Jacquet functors (resp. induced modules) defined with respect to $\rho$ in our arguments (see Lemma~\ref{lemma: permuting Jacquet 1}, Corollary~\ref{cor: permuting general linear}). Later we will use this property freely without mentioning it.


Finally, for a fixed $\rho$, we often need to put apart some subset of $Jord_{\rho}(\q)$ in different ways. Here we want to quantify the corresponding notions. 

\begin{enumerate}


\item
Suppose $(\rho, A, B, \zeta) \in Jord_{\rho}(\q)$ and $r$ is a positive integer, we say $(\rho, A, B, \zeta)$ (or $[A, B]$) is in level $r$ ``far away", if 
\[
B > 2^{r} \cdot \sum_{(\rho, A', B', \zeta') \in Jord_{\rho}(\q)}(A' - B' +1),
\]
and we write 
\[
(\rho, A, B, \zeta) \gg_{r} 0 \, \text{ or } \, (\rho, A, B, \zeta) \gg 0 \text{ when $r = 1$.}
\]

\item
Suppose $(\rho, A, B, \zeta) \in Jord_{\rho}(\q)$ and $r$ is a positive integer, we say $(\rho, A, B, \zeta)$ (or $[A, B]$) is in level $r$ ``far away" from a subset $J$ of $Jord_{\rho}(\q)$, if
\[
B > 2^{r \, |J|} \cdot \Big(\sum_{(\rho, A', B', \zeta') \in J} A' + |J| \sum_{(\rho, A', B', \zeta') \in Jord_{\rho}(\q)}(A' - B' + 1) \Big),
\]
and we write 
\[
(\rho, A, B, \zeta) \gg_{r} J \, \text{ or } \, (\rho, A, B, \zeta) \gg J \text{ when $r = 1$.}
\]

\item
For a subset $J$ of $Jord_{\rho}(\q)$, we denote its complement in $Jord_{\rho}(\q)$ by $J^{c}$. We say $J$ is ``separated" from $J^{c}$, if the following conditions are satisfied.

   \begin{enumerate}
   
   \item For any $(\rho, A, B, \zeta) \in J$ and $(\rho, A', B', \zeta') \in J^{c}$, 
   \[
   \text{either $B' > A$ or $B > A'$.}
   \]
   
   \item For any admissible order $>_{J}$ on $J$, there exists a dominating set of Jordan blocks $J_{\gg}$ of $J$ with discrete diagonal restriction, such that for any $(\rho, A, B, \zeta) \in J$ and $(\rho, A', B', \zeta') \in J^{c}$, 
   \[
   \text{if $B' > A$ then $B' > A_{\gg}$.}
   \]
      
   \item There exists an admissible order $>_{J^{c}}$ on $J^{c}$, under which one can find a dominating set of Jordan blocks $J^{c}_{\gg}$ of $J^{c}$ with discrete diagonal restriction, such that for any $(\rho, A, B, \zeta) \in J$ and $(\rho, A', B', \zeta') \in J^{c}$, 
   \[
   \text{ if $B > A'$ then $B > A'_{\gg}$.}
   \]

   \end{enumerate}

\end{enumerate}
In application, what is important is only the fact that these notions (``far away", ``separated") can be quantified, but not the specific way that we quantify them. For example, once we can measure what it means for some Jordan blocks to be ``far away" from all the others, we can just take them as far as we want in practice.

\section{Parabolic induction and Jacquet module}
\label{sec: induction and restriction}

We will review the notations in \cite{Xu:Apacket}, \cite{Xu:cusp}. For $GL(n)$, let us take $B$ to be the group of upper-triangular matrices and $T$ to be the group of diagonal matrices, then the standard Levi subgroup $M$ can be identified with 
\[
GL(n_{1}) \times \cdots \times GL(n_{r})
\]
for any partition of $n = n_{1} + \cdots + n_{r}$ as follows
\[
\begin{pmatrix}
GL(n_{1})&& \\
&\ddots & \\
&& GL(n_{r})\\
\end{pmatrix}
\]
\begin{align*}
(g_{1}, \cdots, g_{r}) \longrightarrow \text{diag}\{g_{1}, \cdots, g_{r}\}.
\end{align*}
For $\r = \r_{1} \otimes \cdots \otimes \r_{r}$, where $\r_{i}$ is a finite-length admissible representation of $GL(n_{i})$ for $1 \leqslant i \leqslant r$, we denote the normalized parabolic induction $\Ind_{P}^{G} (\r)$ by 
\[
\r_{1} \times \cdots \times \r_{r}.
\] 
 
An irreducible supercuspidal representation of a general linear group can always be written in a unique way as $\rho ||^{x} : = \rho \otimes |\det(\cdot)|^{x}$ for an irreducible unitary supercuspidal representation $\rho$ and a real number $x$. For a finite length arithmetic progression of real numbers of common length $1$ or $-1$
\[
x, \cdots, y
\]
and an irreducible unitary supercuspidal representation $\rho$ of $GL(d_{\rho})$, it is a general fact that 
\[
\rho||^{x} \times \cdots \times \rho||^{y}
\]
has a unique irreducible subrepresentation, denoted by $\langle \rho; x, \cdots, y \rangle$ or $\langle x, \cdots, y \rangle$. If $x \geqslant y$, it is called a Steinberg representation; if $x < y$, it is called a Speh representation. Such sequence of ordered numbers is called a {\bf segment}, and we denote it by $[x, y]$ or $\{x, \cdots, y\}$. In particular, when $x = -y > 0$, we can let $a = 2x + 1 \in \mathbb{Z}$ and write 
\[
St(\rho, a) := \langle \frac{a-1}{2}, \cdots, -\frac{a-1}{2} \rangle.
\]
We also define a {\bf generalized segment} to be a matrix 
\begin{align}
\label{eq: generalized segment}
  \begin{bmatrix}
   x_{11} & \cdots & x_{1n} \\
   \vdots &  & \vdots \\
   x_{m1} & \cdots & x_{mn}
  \end{bmatrix}
\end{align}
such that each row is a decreasing (resp. increasing) segment and each column is an increasing (resp. decreasing) segment. The normalized induction
\[
\times_{i \in [1, m]} \langle \rho; x_{i1}, \cdots, x_{in} \rangle
\]
has a unique irreducible subrepresentation, and we denote it by $\langle \rho; \{x_{ij}\}_{m \times n} \rangle$. If there is no ambiguity with $\rho$, we will also write it as $\langle \{x_{ij}\}_{m \times n} \rangle$ or
\begin{align}
\label{eq: generalized Speh}
  \begin{pmatrix}
   x_{11} & \cdots & x_{1n} \\
   \vdots &  & \vdots \\
   x_{m1} & \cdots & x_{mn}
  \end{pmatrix}
\end{align}
Moreover, 
\[
\langle \rho; \{x_{ij}\}_{m \times n} \rangle \cong \langle \rho; \{x_{ij}\}^{T}_{m \times n} \rangle
\] 
where $\{x_{ij}\}^{T}_{m \times n}$ is the transpose of $\{x_{ij}\}_{m \times n}$. The dual of $\langle \rho; \{x_{ij}\}_{m \times n} \rangle$ is 
\[
\langle \rho; \{x_{ij}\}_{m \times n} \rangle^{\vee} \, \cong \, 
\begin{pmatrix}
   - x_{mn} & \cdots & - x_{m1} \\
   \vdots &  & \vdots \\
   - x_{1n} & \cdots & - x_{11}
  \end{pmatrix}.
\]
Let $a, b$ be positive integers, we define $Sp(St(\rho, a), b)$ to be the unique irreducible subrepresentation of 
\begin{align}
\label{eq: Speh}
St(\rho, a)||^{-(b-1)/2} \times St(\rho, a)||^{-(b-3)/2} \times \cdots \times St(\rho, a)||^{(b-1)/2}.
\end{align}
Then one can see $Sp(St(\rho, a), b)$ is given by the following generalized segment
\begin{align*}
  \begin{bmatrix}
   (a-b)/2 & \cdots & 1-(a+b)/2 \\
   \vdots &  & \vdots \\
   (a+b)/2-1 & \cdots & -(a-b)/2
  \end{bmatrix}.
\end{align*}

If $G = Sp(2n)$, let us define it with respect to 
\[
\begin{pmatrix} 
      0 & -J_{n} \\
      J_{n} &  0 
\end{pmatrix},
\]
where 
\[
J_{n} = 
\begin{pmatrix}
        &&1\\
        &\iddots&\\
        1&&&
\end{pmatrix}.
\]
Let us take $B$ to be subgroup of upper-triangular matrices in $G$ and $T$ to be subgroup of diagonal matrices in $G$,  then the standard Levi subgroup $M$ can be identified with 
\[
GL(n_{1}) \times \cdots \times GL(n_{r}) \times G_{-}
\] 
for any partition $n = n_{1} + \cdots + n_{r} + n_{-}$ and $G_{-} = Sp(2n_{-})$ as follows
\[
\begin{pmatrix}
GL(n_{1})&&&&&&0 \\
&\ddots &&&&& \\
&& GL(n_{r})&&&&\\
&&&G_{-} &&&\\
&&&&GL(n_{r})&& \\
&&&&&\ddots&\\
0&&&&&&GL(n_{1})
\end{pmatrix}
\]
\begin{align}
\label{eq: embedding}
(g_{1}, \cdots g_{r}, g) \longrightarrow \text{diag}\{g_{1}, \cdots, g_{r}, g, {}_tg^{-1}_{r}, \cdots, {}_tg^{-1}_{1}\},
\end{align}
where ${}_tg_{i} = J_{n_{i}}{}^tg_{i}J^{-1}_{n_{i}}$ for $1 \leqslant i \leqslant r$. Note $n_{-}$ can be $0$, in which case we simply write $Sp(0) = 1$. For $\r = \r_{1} \otimes \cdots \otimes \r_{r} \otimes \sigma$, where $\r_{i}$ is a finite-length admissible representation of $GL(n_{i})$ for $1 \leqslant i \leqslant r$ and $\sigma$ is a finite-length admissible representation of $G_{-}$, we denote the normalized parabolic induction $\Ind_{P}^{G}(\r)$ by 
\begin{align}
\label{eq: induced module}
\r_{1} \times \cdots \times \r_{r} \rtimes \sigma.
\end{align}
These notations can be easily extended to special orthogonal groups. If $G = SO(N)$ split, we define it with respect to $J_{N}$. When $N$ is odd, the situation is exactly the same as the symplectic case. When $N = 2n$, there are two distinctions. First, the standard Levi subgroups given through the embedding \eqref{eq: embedding} do not exhaust all standard Levi subgroups of $SO(2n)$. To get all of them, we need to take the $\theta_{0}$-conjugate of $M$ given in \eqref{eq: embedding}, where 
\[
\theta_{0} = 
\begin{pmatrix}
1 &&&&& \\
& \ddots &&&& \\
&&& 1 && \\
&& 1 &&& \\
&&&& \ddots & \\
&&&&& 1
\end{pmatrix}.
\]
Note $M^{\theta_{0}} \neq M$ only when $n_{-} = 0$ and $n_{r} > 1$. In this paper, we will only take those Levi subgroups $M$ given in \eqref{eq: embedding}. Second, if the partition $n = n_{1} + \cdots + n_{r} + n_{-}$ satisfies $n_{r} =1$ and $n_{-} = 0$, then we can rewrite it as $n = n_{1} + \cdots + n_{r - 1} + n'_{-}$ with $n'_{-} = 1$, and the corresponding Levi subgroup is the same. This is because $GL(1) \cong SO(2)$. If $G = SO(2n, \eta)$, the standard Levi subgroups of $SO(2n, \eta)$ will be the outer form of those $\theta_{0}$-stable standard Levi subgroups of $SO(2n)$. In particular, they can be identified with $GL(n_{1}) \times \cdots \times GL(n_{r}) \times SO(n_{-}, \eta)$ and $n_{-} \neq 0$.

Next we want to define parabolic induction and Jacquet module for the category $\Rep(G^{\Sigma_{0}})$ of finite-length representations of $G^{\Sigma_{0}}$. Let $P = MN$ be a standard parabolic subgroup of $G$. If $M$ is $\theta_{0}$-stable, we write $M^{\Sigma_{0}} := M \rtimes \Sigma_{0}$. Otherwise, we let $M^{\Sigma_{0}} = M$.
Suppose $\sigma^{\Sigma_{0}} \in \Rep(M^{\Sigma_{0}})$, $\r^{\Sigma_{0}} \in \Rep(G^{\Sigma_{0}})$. 

\begin{enumerate}

\item If $M^{\theta_{0}} = M$, we define the normalized parabolic induction $\Ind^{G^{\Sigma_{0}}}_{P^{\Sigma_{0}}} \sigma^{\Sigma_{0}}$ to be the extension of the representation $\Ind^{G}_{P}(\sigma^{\Sigma_{0}}|_{M})$ by an induced action of $\Sigma_{0}$, and we define the normalized Jacquet module $\Jac_{P^{\Sigma_{0}}} \r^{\Sigma_{0}}$ to be the extension of the representation $\Jac_{P}(\r^{\Sigma_{0}}|_{G})$ by an induced action of $\Sigma_{0}$.

\item If $M^{\theta_{0}} \neq M$, we define the normalized parabolic induction $\Ind^{G^{\Sigma_{0}}}_{P^{\Sigma_{0}}} \sigma^{\Sigma_{0}}$ to be $\Ind^{G^{\Sigma_{0}}}_{G} \Ind^{G}_{P}(\sigma^{\Sigma_{0}}|_{M})$, and we define the normalized Jacquet module $\Jac_{P^{\Sigma_{0}}} \r^{\Sigma_{0}}$ to be $\Jac_{P}(\r^{\Sigma_{0}}|_{G})$.

\end{enumerate}
Let $\rho$ be an irreducible unitary supercuspidal representation of $GL(d_{\rho})$, and $M = GL(d_{\rho}) \times G_{-}$ be the Levi component of a standard maximal parabolic subgroup $P$ of $G$. For $\r^{\Sigma_{0}} \in \Rep(G^{\Sigma_{0}})$, we can decompose the semisimplification of the Jacquet module
\[
s.s. \Jac_{P^{\Sigma_{0}}}(\r^{\Sigma_{0}}) = \bigoplus_{i} \tau_{i} \otimes \sigma_{i},
\] 
where $\tau_{i} \in \Rep(GL(d_{\rho}))$ and $\sigma_{i} \in \Rep(G_{-}^{\Sigma_{0}})$, both of which are irreducible. We define $\Jac_{x}\r^{\Sigma_{0}}$ for any real number $x$ to be 
\begin{align}
\label{eq: modified Jacquet}
\Jac_{x}(\r) = \bigoplus_{\tau_{i} = \rho||^{x}} \sigma_{i}.
\end{align}
If we have an ordered sequence of real numbers $\{x_{1}, \cdots, x_{s}\}$, we can define
\[
\Jac_{x_{1}, \cdots, x_{s}}\r^{\Sigma_{0}} = \Jac_{x_{s}} \circ \cdots \circ \Jac_{x_{1}} \r^{\Sigma_{0}}.
\]
For a generalized segment $X$ (cf. \eqref{eq: generalized segment}), we define $\Jac_{X}: = \circ_{x \in X} \Jac_{x}$, where $x$ ranges over $X$ from top to bottom and left to right. Similarly, we can define $\Jac^{op}_{x}$ analogous to $\Jac_{x}$, but with respect to $\rho^{\vee}$ and the standard Levi subgroup $GL(n_{-}) \times GL(d_{\rho^{\vee}})$.

For $\q \in \cQ{G}$, let $\q_{\gg}$ be a dominating parameter of $\q$ with respect to certain admissible order $>_{\q}$. For $(\rho, A, B, \zeta) \in Jord(\q)$, we define
\begin{align}
\label{eq: shifting functor}
\Jac_{(\rho, A_{\gg}, B_{\gg}, \zeta) \mapsto (\rho, A, B, \zeta)} : = \Jac_{X^{\gg}_{(\rho, A, B, \zeta)}}
\end{align}
where 
\begin{align*}
X^{\gg}_{(\rho, A, B, \zeta)} =
  \begin{bmatrix}
   \zeta B_{\gg} & \cdots & \zeta (B + 1) \\
   \vdots &  & \vdots \\
   \zeta A_{\gg} & \cdots & \zeta (A+1)
  \end{bmatrix}.
\end{align*}

The following lemmas are very useful when we want to permute the Jacquet functors defined in \eqref{eq: modified Jacquet}.

\begin{lemma}[\cite{Xu:cusp}, Lemma 5.6]
\label{lemma: permuting Jacquet}
If $\r^{\Sigma_{0}} \in \Rep(G^{\Sigma_{0}})$ and $|x - y| \neq 1$, then 
\[
\Jac_{x, y} \r^{\Sigma_{0}} = \Jac_{y, x} \r^{\Sigma_{0}}.
\]
\end{lemma}

\begin{lemma}
\label{lemma: permuting Jacquet 1}
Let $\rho, \rho'$ be two distinct unitary irreducible supercuspidal representations of general linear groups, and $x, y$ be any two real numbers. For $\r^{\Sigma_{0}} \in \Rep(G^{\Sigma_{0}})$, 
\[
\Jac'_{y} \circ \Jac_{x} \r^{\Sigma_{0}} = \Jac_{x} \circ \Jac'_{y} \r^{\Sigma_{0}},
\]
where $\Jac_{x}$ (resp. $\Jac'_{y}$) is defined with respect to $\rho$ (resp. $\rho'$).
\end{lemma} 

\begin{proof}
The proof is the same as Lemma~\ref{lemma: permuting Jacquet}.
\end{proof}

There are some explicit formulas for computing the Jacquet modules in the case of classical groups and general linear groups (cf. \cite{Xu:cusp}, Section 5). Since we will use them quite often, let us recall them here. We will fix a unitary irreducible supercuspidal representation $\rho$ of $GL(d_{\rho})$, and take $``\overset{s.s.}{=}"$ for equality after semisimplification.

For any decreasing segment $\{a, \cdots, b\}$ and $\zeta = \pm1$,
\[
\Jac_{x} \langle \rho'; \zeta a, \cdots, \zeta b \rangle = \begin{cases}
                                       \langle \rho'; \zeta (a-1), \cdots, \zeta b \rangle, & \text{ if $x = \zeta a$ and $\rho' \cong \rho$,}\\
                                       0,  & \text{ otherwise; }
                                       \end{cases}
\]
and
\[
\Jac^{op}_{x} \langle \rho'; \zeta a, \cdots, \zeta b \rangle = \begin{cases}
                                       \langle \rho'; \zeta a, \cdots, \zeta (b+1) \rangle, & \text{ if $x = \zeta b$ and $\rho' \cong \rho^{\vee}$,}\\
                                       0,  & \text{ otherwise. }
                                       \end{cases}
\]

Suppose $\r_{i} \in \Rep(GL(n_{i}))$ for $i = 1$ or $2$, we have
\[
\Jac_{x} (\r_{1} \times \r_{2}) \overset{s.s.}{=} (\Jac_{x} \r_{1}) \times \r_{2} \+ \r_{1} \times (\Jac_{x} \r_{2}),
\] 
and 
\[
\Jac^{op}_{x} (\r_{1} \times \r_{2}) \overset{s.s.}{=} (\Jac^{op}_{x} \r_{1}) \times \r_{2} \+ \r_{1} \times (\Jac^{op}_{x} \r_{2}).
\]

Suppose $\r^{\Sigma_{0}} \in \Rep(G)$ and $\tau \in \Rep(GL(d))$, we have

\begin{align*}
\Jac_{x} (\tau \rtimes \r^{\Sigma_{0}}) \overset{s.s.}{=} (\Jac_{x} \tau) \rtimes \r^{\Sigma_{0}} \+ (\Jac^{op}_{-x} \tau) \rtimes \r^{\Sigma_{0}} \+ \tau \rtimes \Jac_{x} \r^{\Sigma_{0}}.
\end{align*}

Finally, we want to recall the following vanishing result for Jacquet modules of elements in the Arthur packets. This will become very useful when we want to simplify the results of Jacquet modules after applying the above formulas.

\begin{proposition}[\cite{Xu:Apacket}, Proposition 8.3]

Suppose $\q \in \cQ{G}$ and $\r^{\Sigma_{0}} \in \Pkt{\q}^{\Sigma_{0}}$. Let $\rho$ be a unitary irreducible supercuspidal representation of $GL(d_{\rho})$.

\begin{enumerate}

\item For $\zeta \in \{\pm 1\}$ and segment $[x, y]$ with $0 \leqslant x \leqslant y$, 
\[
\Jac_{\zeta x, \cdots, \zeta y} \, \r^{\Sigma_{0}} = 0,
\]
unless there exists a sequence of Jordan blocks $\{(\rho, A_{i}, B_{i}, \zeta)\}_{i=1}^{n} \subseteq Jord_{\rho}(\q)$ such that $B_{1} = x, A_{n} > y$, and $B_{i} \leqslant B_{i+1} \leqslant A_{i} + 1$.

\item For $x \in \mathbb{R}$, let $m = \sharp \{ (\rho, A, B, \zeta) \in Jord(\q): \zeta B = x \}$. If $n > m$, then 
\[
\Jac_{\underbrace{x, \cdots, x}_{n}} \, \r^{\Sigma_{0}} = 0.
\]

\end{enumerate}

\end{proposition}

\section{Some irreducibility results}
\label{sec: irreducibility}

In this section, we want to recall some irreducibility results. We will start with general linear groups. For any two segments $[x, y]$ and $[x', y']$ such that $(x - y)(x' - y') \geqslant 0$, we say they are {\bf linked} if as sets $[x, y] \nsubseteq [x', y']$, $[x', y'] \nsubseteq [x, y]$, and $[x, y]  \cup [x', y']$ can form a segment after imposing the same order. The following theorem is fundamental in determining the reducibility of an induced representation of $GL(n)$.

\begin{theorem}[Zelevinsky \cite{Zelevinsky:1980}]
\label{thm: irreducibility for induced representation of GL(n)}
For unitary irreducible supercuspidal representations $\rho, \rho'$ of general linear groups, and segments $[x, y]$, $[x', y']$ such that $(x - y)(x' - y') \geqslant 0$, 
\[
\langle \rho; x, \cdots, y \rangle \times \langle \rho'; x', \cdots, y' \rangle
\]
is reducible if and only if $\rho \cong \rho'$ and $[x, y], [x', y']$ are linked. In case it is reducible, it consists of the unique irreducible subrepresentations of
\[
\langle \rho; x, \cdots, y \rangle \times \langle \rho; x', \cdots, y' \rangle \text{ and } \langle \rho; x', \cdots, y' \rangle \times \langle \rho; x, \cdots, y \rangle.
\]
\end{theorem}

To extend this theorem to generalized segments, we have to extend the notion of ``link" first. For any two generalized segments $\{x_{ij}\}_{m \times n}$ and $\{y_{ij}\}_{m' \times n'}$ with the same monotone properties for the rows and columns, we say they are {\bf linked} if $[x_{m1}, x_{1n}]$, $[y_{m'1}, y_{1n'}]$ are linked, and the four sides of the rectangle formed by $\{x_{ij}\}_{m \times n}$ do not have inclusive relations with the corresponding four sides of the rectangle formed by $\{y_{ij}\}_{m' \times n'}$ (e.g., $[x_{11}, x_{1n}] \nsubseteq [y_{11}, y_{1n'}] \text{ and } [x_{11}, x_{1n}] \nsupseteq [y_{11}, y_{1n'}]$, etc). It is easy to check that if $\{x_{ij}\}_{m \times n}$ and $\{y_{ij}\}_{m' \times n'}$ are linked, then $\{x_{ij}\}^{T}_{m \times n}$ and $\{y_{ij}\}^{T}_{m' \times n'}$ are also linked. So for generalized segments $\{x_{ij}\}_{m \times n}$ and $\{y_{ij}\}_{m' \times n'}$ with different monotone properties for the rows and columns, we say they are {\bf linked} if $\{x_{ij}\}^{T}_{m \times n}$ and $\{y_{ij}\}_{m' \times n'}$ are linked, or equivalently $\{x_{ij}\}_{m \times n}$ and $\{y_{ij}\}^{T}_{m' \times n'}$ are linked. One can check this notion of ``link" is equivalent to the one in \cite{MW:1989}. 

\begin{theorem}[M{\oe}glin-Waldspurger \cite{MW:1989}]
\label{thm: generalized irreducibility for induced representation of GL(n)}
For unitary irreducible supercuspidal representations $\rho, \rho'$ of general linear groups, and generalized segments $\{x_{ij}\}_{m \times n}$, $\{y_{ij}\}_{m' \times n'}$,
\[
\langle \rho; \{x_{ij}\}_{m \times n}\rangle \times \langle \rho'; \{y_{ij}\}_{m' \times n'} \rangle
\]
is irreducible unless $\rho \cong \rho'$ and $\{x_{ij}\}_{m \times n}, \{y_{ij}\}_{m' \times n'}$ are linked.
\end{theorem}

We will be mostly using the following corollary of this theorem.

\begin{corollary}
\label{cor: permuting general linear}
Let $\rho, \rho'$ be unitary irreducible supercuspidal representations of general linear groups, and $\{x_{ij}\}_{m \times n}$, $\{y_{ij}\}_{m' \times n'}$ be generalized segments. Suppose $\rho \ncong \rho'$, or $\{x_{ij}\}_{m \times n}, \{y_{ij}\}_{m' \times n'}$ are not linked, then
\[
\langle \rho; \{x_{ij}\}_{m \times n} \rangle \times \langle \rho'; \{y_{ij}\}_{m' \times n'} \rangle \quad \cong \quad \langle \rho'; \{y_{ij}\}_{m' \times n'}\rangle \times \langle \rho; \{x_{ij}\}_{m \times n} \rangle.
\]
\end{corollary}

\begin{proof}
One just needs to notice there a Weyl group action transform the inducing representation $\langle \rho; \{x_{ij}\}_{m \times n}\rangle \otimes \langle \rho'; \{y_{ij}\}_{m' \times n'} \rangle$ to $\langle \rho'; \{y_{ij}\}_{m' \times n'}\rangle \otimes \langle \rho; \{x_{ij}\}_{m \times n} \rangle$. Then the corollary follows from the fact that both induced representations are irreducible.
\end{proof}

Next, let us consider $G^{\Sigma_{0}}$.

\begin{lemma}[\cite{Moeglin:2011}, Lemma 8.2]
\label{lemma: basic irreducibility}

Let $\q \in \cQ{G}$ and $\r^{\Sigma_{0}} \in \Pkt{\q}^{\Sigma_{0}}$. For any self-dual irreducible unitary supercuspidal representation $\rho$ of $GL(d_{\rho})$ and real number $x$, 
\[
\rho||^{x} \rtimes \r^{\Sigma_{0}}
\]
is irreducible, provided for all $(\rho, A, B, \zeta) \in Jord_{\rho}(\q)$, we have either
\[
B > |x|  \text{ or } |x| > A+1.
\]
\end{lemma}

We will not give the proof of this lemma here, but we would like to discuss the idea behind the proof. Let $\tau$ be an irreducible representation of $GL(d)$, and $\r^{\Sigma_{0}}$ be an irreducible representation of $G^{\Sigma_{0}}$. To show $\tau \rtimes \r^{\Sigma_{0}}$ is irreducible, there is the following criterion. 

\begin{lemma}
\label{lemma: strategy for irreducibility}
Suppose there exits a unique irreducible subrepresentation
\[
\sigma \hookrightarrow \tau \rtimes \r^{\Sigma_{0}}
\]
such that $\sigma$ is multiplicity free in $s.s. (\tau \rtimes \r^{\Sigma_{0}})$, and
\[
\sigma \hookrightarrow \tau^{\vee} \rtimes \r^{\Sigma_{0}}.
\]
Then $\tau \rtimes \r^{\Sigma_{0}}$ is irreducible.
\end{lemma}

\begin{proof}
Since $\sigma \hookrightarrow \tau^{\vee} \rtimes \r^{\Sigma_{0}}$, we know $\tau \rtimes \r^{\Sigma_{0}}$ has a quotient isomorphic to $\sigma$. Then by the fact that
\(
\sigma \hookrightarrow \tau \rtimes \r^{\Sigma_{0}}
\)
and $\sigma$ is multiplicity free in $s.s. (\tau \rtimes \r^{\Sigma_{0}})$, we see $\sigma$ is a direct summand of $\tau \rtimes \r^{\Sigma_{0}}$. This means $\tau \rtimes \r^{\Sigma_{0}}$ necessarily has another irreducible subrepresentation. But this contradicts to the uniqueness of $\sigma$.
\end{proof}

By the same idea, we can generalize Lemma~\ref{lemma: basic irreducibility} to the following proposition. 

\begin{proposition}
\label{prop: general irreducibility}
Let $\q \in \cQ{G}$ and $\r^{\Sigma_{0}} \in \Pkt{\q}^{\Sigma_{0}}$. For any self-dual irreducible unitary supercuspidal representation $\rho$ of $GL(d_{\rho})$, and
\[
\tau = 
\begin{pmatrix}
   \zeta x & \cdots & \zeta x' \\
   \vdots &  & \vdots \\
   \zeta y & \cdots & \zeta y'
\end{pmatrix}
\]
such that $y \geqslant x \geqslant x' > 0$ and $\zeta = \pm 1$, if for all $(\rho, A, B, \zeta) \in Jord_{\rho}(\q)$, we have either
\[
B > y  \text{ or } x' > A+1.
\]
then
\(
\tau \rtimes \r^{\Sigma_{0}}
\)
is irreducible. Moreover, 
\begin{align}
\label{eq: dualizing classical}
\tau \rtimes \r^{\Sigma_{0}} \cong \tau^{\vee} \rtimes \r^{\Sigma_{0}}
\end{align}
in this case.
\end{proposition}

\begin{proof}
Taking conjugation by elements in $G^{\Sigma_{0}}$, one can transform the inducing representation $\tau \otimes \r^{\Sigma_{0}}$ to $\tau^{\vee} \otimes \r^{\Sigma_{0}}$. So $\tau \rtimes \r^{\Sigma_{0}} \cong \tau^{\vee} \rtimes \r^{\Sigma_{0}}$ if $\tau \rtimes \r^{\Sigma_{0}}$ is irreducible. To apply Lemma~\ref{lemma: strategy for irreducibility}, let us choose an irreducible subrepresentation
\(
\sigma \hookrightarrow \tau \rtimes \r^{\Sigma_{0}}.
\)
Let 
\[
X = 
\begin{bmatrix}
   \zeta x & \cdots & \zeta x' \\
   \vdots &  & \vdots \\
   \zeta y & \cdots & \zeta y'
\end{bmatrix},
\]
then by our assumption, $\Jac_{z} \r^{\Sigma_{0}} = 0$ for any $z \in X$. Also because $y \geqslant x \geqslant x' > 0$, we have
\[
\Jac_{X}(\tau \rtimes \r^{\Sigma_{0}}) \overset{s.s.}{=} (\Jac_{X} \tau) \rtimes \r^{\Sigma_{0}} = \r^{\Sigma_{0}}.
\]
This means $\sigma$ is the unique irreducible subrepresentation of $\tau \rtimes \r^{\Sigma_{0}}$, and it is multiplicity free in $s.s.(\tau \rtimes \r^{\Sigma_{0}})$. Then it suffices for us to show 
\(
\sigma \hookrightarrow \tau^{\vee} \rtimes \r^{\Sigma_{0}}.
\)
By Lemma~\ref{lemma: basic irreducibility}, we have
\begin{align*}
\sigma 
& \hookrightarrow 
\begin{pmatrix}
   \zeta x & \cdots & \zeta x' \\
   \vdots &  & \vdots \\
   \zeta (y-1) & \cdots & \zeta (y'-1)
\end{pmatrix}
\times \rho||^{\zeta y} \times \cdots \times \rho||^{\zeta y'} \rtimes \r^{\Sigma_{0}} \\
& \cong
\begin{pmatrix}
   \zeta x & \cdots & \zeta x' \\
   \vdots &  & \vdots \\
   \zeta (y-1) & \cdots & \zeta (y'-1)
\end{pmatrix}
\times \rho||^{\zeta y} \times \cdots \times \rho||^{\zeta (y' +1)} \times \rho||^{-\zeta y'} \rtimes \r^{\Sigma_{0}} \\
& \cong
\rho||^{-\zeta y'} \times 
\begin{pmatrix}
   \zeta x & \cdots & \zeta x' \\
   \vdots &  & \vdots \\
   \zeta (y-1) & \cdots & \zeta (y'-1)
\end{pmatrix}
\times \rho||^{\zeta y} \times \cdots \times \rho||^{\zeta (y' + 1)} \rtimes \r^{\Sigma_{0}} \\
& \cdots \quad \cdots \\
& \cong 
\rho||^{-\zeta y'} \times \cdots \times \rho||^{-\zeta y} \times
\begin{pmatrix}
   \zeta x & \cdots & \zeta x' \\
   \vdots &  & \vdots \\
   \zeta (y-1) & \cdots & \zeta (y'-1)
\end{pmatrix}
\rtimes \r^{\Sigma_{0}} 
\end{align*}
By induction on $y-x$, we can assume 
\[
\sigma' := \begin{pmatrix}
   \zeta x & \cdots & \zeta x' \\
   \vdots &  & \vdots \\
   \zeta (y-1) & \cdots & \zeta (y'-1)
\end{pmatrix}
\rtimes \r^{\Sigma_{0}} 
\]
is irreducible. Then 
\[
\sigma' \cong \begin{pmatrix}
  - \zeta (y'-1) & \cdots & - \zeta (y-1) \\
   \vdots &  & \vdots \\
  - \zeta x' & \cdots & - \zeta x
\end{pmatrix}
\rtimes \r^{\Sigma_{0}} 
\]
as we have seen in the beginning. Since
\(
\Jac_{z} \sigma' = 0
\)
for $z \in [-\zeta y', -\zeta y]$, then
\[
\Jac_{-\zeta y', \cdots, -\zeta y} (\rho||^{-\zeta y'} \times \cdots \times \rho||^{-\zeta y} \times \sigma') = \sigma'.
\]
Therefore, 
\[
\sigma \hookrightarrow \rho||^{-\zeta y'} \times \cdots \times \rho||^{-\zeta y} \times \sigma'
\]
as the unique irreducible subrepresentation. It follows 
\begin{align*}
\sigma  \hookrightarrow \langle -\zeta y', \cdots,  -\zeta y \rangle \rtimes \sigma' 
\cong \langle -\zeta y', \cdots,  -\zeta y \rangle \rtimes 
\begin{pmatrix}
  - \zeta (y'-1) & \cdots & - \zeta (y-1) \\
   \vdots &  & \vdots \\
  - \zeta x' & \cdots & - \zeta x
\end{pmatrix}
\rtimes \r^{\Sigma_{0}}
\end{align*}
as the unique irreducible subrepresentation. Hence
\[
\sigma  \hookrightarrow 
\begin{pmatrix}
  - \zeta y' & \cdots & - \zeta y \\
   \vdots &  & \vdots \\
  - \zeta x' & \cdots & - \zeta x
\end{pmatrix}
\rtimes \r^{\Sigma_{0}}
\]
This finishes the proof.

\end{proof}


\section{Basic case and generalization}
\label{sec: basic}

We describe the {\bf basic case} as follows. Let us fix a self-dual unitary irreducible supercuspidal representation $\rho$ of $GL(d_{\rho})$. There exists
\[
\{(\rho, A_{2}, B_{2}, \zeta_{2}), (\rho, A_{1}, B_{1}, \zeta_{1})\} \subseteq Jord(\q)
\]
such that
\(
A_{2} \geqslant A_{1}, B_{2} \geqslant B_{1},
\)
and $\zeta_{1} = \zeta_{2} = \zeta$. These two Jordan blocks are ``separated" from the other blocks in $Jord_{\rho}(\q)$. Moreover, let
\[
Jord(\q_{-}) = Jord(\q) \backslash \{(\rho, A_{2}, B_{2}, \zeta_{2}), (\rho, A_{1}, B_{1}, \zeta_{1})\},
\]
we require $\q_{-}$ has discrete diagonal restriction. We can extend the natural order on $Jord(\q_{-})$ to $Jord(\q)$ as follows 
\[
(\rho, A, B, \zeta) >_{\q} (\rho, A', B', \zeta') \text{ if and only if } A \geqslant A'.
\]
In particular, 
\[
(\rho, A_{2}, B_{2}, \zeta_{2}) >_{\q} (\rho, A_{1}, B_{1}, \zeta_{1}).
\]

For functions $\ul(\rho, A, B, \zeta) \in [0, [(A-B+1)/2]]$ and $\ueta(\rho, A, B, \zeta) \in \mathbb{Z}/2\mathbb{Z}$ on $Jord(\q)$, we denote 
\[
l_{1} = \ul(\rho, A_{1}, B_{1}, \zeta_{1}), \quad l_{2} = \ul(\rho, A_{2}, B_{2}, \zeta_{2}),
\] 
and 
\[
\eta_{1} = \ueta(\rho, A_{1}, B_{1}, \zeta_{1}), \quad \eta_{2} = \ueta(\rho, A_{2}, B_{2}, \zeta_{2}).
\]

\begin{lemma}[M{\oe}glin]
\label{lemma: fundamental case}
In the basic case, suppose 
\[
[A_{2}, B_{2}] = [A_{1}, B_{1}],
\]
then  $\r^{\Sigma_{0}}_{M, >_{\q}}(\q, \ul, \ueta) \neq 0$  if and only if  
\begin{align*}
\begin{cases}
\eta_{2} = (-1)^{A_{1} - B_{1}}\eta_{1}       & \Rightarrow l_{1} = l_{2},  \\
\eta_{2} \neq (-1)^{A_{1} - B_{1}}\eta_{1}  & \Rightarrow l_{1} = l_{2} = (A_{1} - B_{1} + 1)/2.
\end{cases}                                                                                                                                                                                                                                                                                                                                
\end{align*}

\end{lemma}

This lemma is in (\cite{Moeglin1:2006}, Lemma 3.4) and it is fundamental for all the results that we are going to derive in this paper. The lemma can also be generalized as follows.


\begin{proposition}
\label{prop: basic}
In the basic case, if $\r^{\Sigma_{0}}_{M, >_{\q}}(\q, \ul, \ueta) \neq 0$, then
\begin{align}
\label{eq: basic condition}
\begin{cases}
\eta_{2} = (-1)^{A_{1} - B_{1}}\eta_{1}       & \Rightarrow A_{2} - l_{2} \geqslant A_{1} - l_{1}, \quad B_{2} + l_{2} \geqslant B_{1} + l_{1},  \\
\eta_{2} \neq (-1)^{A_{1} - B_{1}}\eta_{1}  & \Rightarrow B_{2} + l_{2} > A_{1} - l_{1}.   
\end{cases}                                                                                                                                                                                                                                                                                                                                
\end{align}
Conversely, if \eqref{eq: basic condition} is satisfied, then $\r^{\Sigma_{0}}_{M, >_{\q}}(\q, \ul, \ueta) \neq 0$, moreover
\begin{align*}
\r^{\Sigma_{0}}_{M, >_{\q}}(\q, \ul, \ueta) & \hookrightarrow 
      \begin{pmatrix}
              \zeta B_{2} & \cdots & -\zeta A_{2} \\
              \vdots &  & \vdots \\
              \zeta (B_{2} + l_{2} - 1) & \cdots & - \zeta (A_{2} - l_{2} + 1)
       \end{pmatrix} 
\times       
       \begin{pmatrix}
              \zeta B_{1} & \cdots & -\zeta A_{1} \\
              \vdots &  & \vdots \\
              \zeta (B_{1} + l_{1} - 1) & \cdots & - \zeta (A_{1} - l_{1} + 1)
       \end{pmatrix} \\
& \times \r^{\Sigma_{0}}_{M, >_{\q}}\Big(\q_{-}, \ul_{-}, \ueta_{-}; (\rho, A_{2} - l_{2}, B_{2} + l_{2}, 0, \eta_{2}, \zeta), (\rho, A_{1} - l_{1}, B_{1} + l_{1}, 0, \eta_{1}, \zeta)\Big).       
\end{align*}

\end{proposition}

We will give the proof of Proposition~\ref{prop: basic} in Appendix~\ref{sec: basic}. Next we would like to generalize the basic case to the following situation. Suppose we can index $Jord_{\rho}(\q)$ for each $\rho$ such that 
$A_{i} \geqslant A_{i-1}$ and $B_{i} \geqslant B_{i-1}$. Moreover we can divide $Jord_{\rho}(\q)$ into chunks of 
\begin{align}
\label{eq: chunk}
\{(\rho, A_{i}, B_{i}, \zeta_{i}), (\rho, A_{i-1}, B_{i-1}, \zeta_{i-1})\} \text{ with $\zeta_{i} = \zeta_{i-1}$, or } \{(\rho, A_{j}, B_{j}, \zeta_{j})\},
\end{align}
such that each of them is ``separated" from the others in $Jord_{\rho}(\q)$. We call this the {\bf generalized basic case}. There is a natural order $>_{\q}$ on $Jord_{\rho}(\q)$, i.e.,
\[
(\rho, A_{i}, B_{i}, \zeta_{i}) >_{\q} (\rho, A_{i-1}, B_{i-1}, \zeta_{i-1}).
\]

\begin{proposition}
\label{prop: basic general}
In the generalized basic case, $\r^{\Sigma_{0}}_{M, >_{\q}}(\q, \ul, \ueta) \neq 0$ if and only if the condition \eqref{eq: basic condition} is satisfied for each chunk of pair $\{(\rho, A_{i}, B_{i}, \zeta_{i}), (\rho, A_{i-1}, B_{i-1}, \zeta_{i-1})\}$ in \eqref{eq: chunk} for all $\rho$.
\end{proposition}

\begin{proof}
We will first prove the sufficiency of the nonvanishing condition by induction on the number of intersected pairs in $Jord(\q)$. Let $\rho$ be fixed. For $Jord_{\rho}(\q)$, suppose $n$ is the biggest integer such that $[A_{n}, B_{n}]$ and $[A_{n-1}, B_{n-1}]$ intersects. Let
\[
Jord(\q_{-}) = Jord(\q) \backslash \{(\rho, A_{n}, B_{n}, \zeta_{n})\}
\]
By induction we can assume 
\[
\r^{\Sigma_{0}}_{M, >_{\q}}\Big(\q_{-}, \ul_{-}, \ueta_{-}; (\rho, A_{n} + T_{n}, B_{n} + T_{n}, l_{n}, \eta_{n}, \zeta_{n})\Big) \neq 0
\]
for the smallest $T_{n}$ such that $[A_{n} + T_{n}, B_{n} + T_{n}]$ does not intersect with $[A_{n-1}, B_{n-1}]$. For those intersected pairs $\{(\rho, A_{i}, B_{i}, \zeta_{i}), (\rho, A_{i-1}, B_{i-1}, \zeta_{i-1})\}$, we can put them apart by shifting $(\rho, A_{i}, B_{i}, \zeta_{i})$ to $(\rho, A_{i} + T_{i}, B_{i} + T_{i}, \zeta_{i})$ again for the smallest $T_{i}$. Let us write $T_{j} = 0$ for those $(\rho, A_{j}, B_{j}, \zeta_{j})$ remained in $Jord_{\rho}(\q)$. As a result we can get a parameter $\q_{\gg}$ dominating $\q$ with discrete diagonal restriction such that 
\[
(\rho, A_{\gg, i}, B_{\gg, i}, \zeta_{i}) = (\rho, A_{i} + T_{i}, B_{i} + T_{i}, \zeta_{i}).
\] 
Then
\begin{align*}
\r^{\Sigma_{0}}_{M, >_{\q}}(\q_{\gg}, \ul, \ueta) & \hookrightarrow \times_{i \neq n} 
       \begin{pmatrix}
              \zeta_{i} (B_{i} + T_{i}) & \cdots & \zeta_{i}(B_{i} + 1) \\
              \vdots &  & \vdots \\
              \zeta_{i} (A_{i} + T_{i}) & \cdots & \zeta_{i}(A_{i} + 1)
       \end{pmatrix} \times \mathcal{C}_{X^{c}} \\
& \rtimes  \r^{\Sigma_{0}}_{M, >_{\q}}\Big(\q_{-}, \ul_{-}, \ueta_{-}; (\rho, A_{n} + T_{n}, B_{n} + T_{n}, l_{n}, \eta_{n}, \zeta_{n})\Big),      
\end{align*}       
where $i$ is increasing. We would like to show 
\[
\Jac_{(\rho, A_{n} + T_{n}, B_{n} + T_{n}, \zeta_{n}) \mapsto (\rho, A_{n}, B_{n}, \zeta_{n})} \r^{\Sigma_{0}}_{M, >_{\q}}\Big(\q_{-}, \ul_{-}, \ueta_{-}; (\rho, A_{n} + T_{n}, B_{n} + T_{n}, l_{n}, \eta_{n}, \zeta_{n})\Big) \neq 0.
\]
Note by our assumption, 
\[
\Jac_{(\rho, A_{n} + T_{n}, B_{n} + T_{n}, \zeta_{n}) \mapsto (\rho, A_{n}, B_{n}, \zeta_{n})} \r^{\Sigma_{0}}_{M, >_{\q}}(\q_{\gg}, \ul, \ueta) \neq 0.
\]
So after we apply the same Jacquet functor to the full induced representation above, we should get something nonzero. To compute this Jacquet module, one notes $B_{n} + 1 > A_{i} + T_{i}$ for $T_{i} \neq 0$, so it can only be  
\begin{align*}
& \times_{i \neq n}   \begin{pmatrix}
              \zeta_{i} (B_{i} + T_{i}) & \cdots & \zeta_{i}(B_{i} + 1) \\
              \vdots &  & \vdots \\
              \zeta_{i} (A_{i} + T_{i}) & \cdots & \zeta_{i}(A_{i} + 1)
       \end{pmatrix} 
\times \mathcal{C}_{X^{c}} \rtimes \Jac_{(\rho, A_{n} + T_{n}, B_{n} + T_{n}, \zeta_{n}) \mapsto (\rho, A_{n}, B_{n}, \zeta_{n})} \\ & \r^{\Sigma_{0}}_{M, >_{\q}}\Big(\q_{-}, \ul_{-}, \ueta_{-}; (\rho, A_{n} + T_{n}, B_{n} + T_{n}, l_{n}, \eta_{n}, \zeta_{n})\Big) \neq 0.    
\end{align*}       
This gives what we want. 

Next for the necessity of the nonvanishing condition, we can assume $\r^{\Sigma_{0}}_{M, >_{\q}}(\q, \ul, \ueta) \neq 0$. We still fix $\rho$ and choose a dominating parameter $\q_{\gg}$ with discrete diagonal restriction in the way as above. Then by definition
\begin{align*}
\r^{\Sigma_{0}}_{M, >_{\q}}(\q_{\gg}, \ul, \ueta) & \hookrightarrow \times_{i} 
       \begin{pmatrix}
              \zeta_{i} (B_{i} + T_{i}) & \cdots & \zeta_{i}(B_{i} + 1) \\
              \vdots &  & \vdots \\
              \zeta_{i} (A_{i} + T_{i}) & \cdots & \zeta_{i}(A_{i} + 1)
       \end{pmatrix}  \times \mathcal{C}_{X^{c}} \rtimes  \r^{\Sigma_{0}}_{M, >_{\q}}(\q, \ul, \ueta).      
\end{align*}   
It is easy to see that those generalized segments in the induced representations are not linked. So we can change their orders in the induction. In particular, we can take any generalized segment to the front. As a result 
\[
\Jac_{(\rho, A_{i} + T_{i}, B_{i} + T_{i}, \zeta_{i}) \mapsto (\rho, A_{i}, B_{i}, \zeta_{i})} \r^{\Sigma_{0}}_{M, >_{\q}}(\q_{\gg}, \ul, \ueta) \neq 0.
\]
for any $i$. This gives the condition that we want with respect to $\rho$. By varying $\rho$, we prove the necessity of the condition.

\end{proof}

\begin{remark}
\label{rk: basic general}
Suppose a subset of Jordan blocks of $Jord_{\rho}(\q)$ satisfies the condition in the generalized basic case, then we say the Jordan blocks in this set have ``good shape".
\end{remark}

\subsection{Some necessary conditions on nonvanishing}
\label{subsec: necessary condition}

In this section, we want to use Proposition~\ref{prop: basic} to give some necessary conditions on the nonvanishing of $\r^{\Sigma_{0}}_{M, >_{\q}}(\q, \ul, \ueta)$ in general. Let us fix $\rho$ and index the Jordan blocks in $Jord_{\rho}(\q)$ such that 
\[
(\rho, A_{i}, B_{i}, \zeta_{i}) >_{\q} (\rho, A_{i-1}, B_{i-1}, \zeta_{i-1}).
\]
Let $(\rho, A_{k}, B_{k}, \zeta_{k}) >_{\q} (\rho, A_{k-1}, B_{k-1}, \zeta_{k-1})$ be two adjacent blocks under the order $>_{\q}$ and $\zeta_{k} = \zeta_{k-1}$.

\begin{lemma}
Suppose $A_{k} \geqslant A_{k-1}$ and $B_{k} \geqslant B_{k-1}$. If $\r^{\Sigma_{0}}_{M, >_{\q}}(\q, \ul, \ueta) \neq 0$, then $l_{k}, \eta_{k}, l_{k-1}, \eta_{k-1}$ satisfy the condition \eqref{eq: basic condition}.
\end{lemma}

\begin{proof}
Let $\q_{\gg}$ be a dominating parameter with discrete diagonal restriction. We also define $\q^{(k)}$ from $\q_{\gg}$ by shifting $(\rho, A_{i} + T_{i}, B_{i} + T_{i}, \zeta_{i})$ back to $(\rho, A_{i}, B_{i}, \zeta_{i})$ for $i \leqslant k$. Then
\begin{align*}
\r^{\Sigma_{0}}_{M, >_{\q}}(\q_{\gg}, \ul, \ueta) & \hookrightarrow \times_{i < k-1}
      \begin{pmatrix}
              \zeta_{i} (B_{i} + T_{i}) & \cdots & \zeta_{i}(B_{i} + 1) \\
              \vdots &  & \vdots \\
              \zeta_{i} (A_{i} + T_{i}) & \cdots & \zeta_{i}(A_{i} + 1)
       \end{pmatrix} \times
       \underbrace{\begin{pmatrix}
              \zeta_{k-1} (B_{k-1} + T_{k-1}) & \cdots & \zeta_{k-1}(B_{k-1} + 1) \\
              \vdots &  & \vdots \\
              \zeta_{k-1} (A_{k-1} + T_{k-1}) & \cdots & \zeta_{k-1}(A_{k-1} + 1)
       \end{pmatrix}}_{I} \\ 
& \times 
      \underbrace{\begin{pmatrix}
              \zeta_{k} (B_{k} + T_{k}) & \cdots & \zeta_{k}(B_{k} + 1) \\
              \vdots &  & \vdots \\
              \zeta_{k} (A_{k} + T_{k}) & \cdots & \zeta_{k}(A_{k} + 1)
       \end{pmatrix}}_{II}  \rtimes  \r^{\Sigma_{0}}_{M, >_{\q}}(\q^{(k)}, \ul, \ueta) \\
& \hookrightarrow  \times_{i < k-1}
      \begin{pmatrix}
              \zeta_{i} (B_{i} + T_{i}) & \cdots & \zeta_{i}(B_{i} + 1) \\
              \vdots &  & \vdots \\
              \zeta_{i} (A_{i} + T_{i}) & \cdots & \zeta_{i}(A_{i} + 1)
       \end{pmatrix} \times
     \underbrace{\begin{pmatrix}
              \zeta_{k-1} (B_{k-1} + T_{k-1}) & \cdots & \zeta_{k-1}(B_{k-1} + 1) \\
              \vdots &  & \vdots \\
              \zeta_{k-1} (A_{k-1} + T_{k-1}) & \cdots & \zeta_{k-1}(A_{k-1} + 1)
       \end{pmatrix}}_{I} \\
& \times 
      \underbrace{\begin{pmatrix}
              \zeta_{k} (B_{k} + T_{k}) & \cdots &  \zeta_{k}(B_{k} + T_{k-1} + 1) & \cdots & \zeta_{k}(B_{k} + 1) \\
              \vdots &  & \vdots  & & \vdots \\
              \zeta_{k} (A_{k-1} + T_{k}) & \cdots &  \zeta_{k}(A_{k-1} + T_{k-1} + 1) & \cdots & \zeta_{k}(A_{k-1} + 1)
       \end{pmatrix}}_{II_{1}} \\
& \times 
       \underbrace{\begin{pmatrix}
              \zeta_{k} (A_{k-1} + T_{k} + 1) & \cdots & \zeta_{k}(A_{k-1} + T_{k-1} + 2) \\
              \vdots &  & \vdots \\
              \zeta_{k} (A_{k} + T_{k}) & \cdots & \zeta_{k}(A_{k} + T_{k-1} + 1)
       \end{pmatrix}}_{II_{2}} \\
& \times
        \underbrace{\begin{pmatrix}
              \zeta_{k} (A_{k-1} + T_{k-1} + 1) & \cdots & \zeta_{k}(A_{k-1} + 2) \\
              \vdots &  & \vdots \\
              \zeta_{k} (A_{k} + T_{k-1}) & \cdots & \zeta_{k}(A_{k} + 1)
       \end{pmatrix}}_{II_{3}}  \rtimes  \r^{\Sigma_{0}}_{M, >_{\q}}(\q^{(k)}, \ul, \ueta),
\end{align*}  
where $i$ increases. Note $(I)$ is interchangeable with $(II_{1})$ and $(II_{2})$, and $B_{k} + T_{k-1} + 1 > A_{i} + T_{i} + 1$ for $i < k-1$. As a result, 
\[
\Jac_{(\rho, A_{k} + T_{k}, B_{k} + T_{k}, \zeta_{k}) \mapsto (\rho, A_{k} + T_{k-1}, B_{k} + T_{k-1}, \zeta_{k})} \r^{\Sigma_{0}}_{M, >_{\q}}(\q_{\gg}, \ul, \ueta) \neq 0.
\]
Then by Proposition~\ref{prop: basic}, $l_{k}, \eta_{k}, l_{k-1}, \eta_{k-1}$ satisfy the condition \eqref{eq: basic condition}.
\end{proof}

\begin{lemma}
\label{lemma: necessary condition sup}
Suppose $[A_{k}, B_{k}] \supseteq [A_{k-1}, B_{k-1}]$. If $\r^{\Sigma_{0}}_{M, >_{\q}}(\q, \ul, \ueta) \neq 0$, then $l_{k}, \eta_{k}, l_{k-1}, \eta_{k-1}$ satisfy the following condition: 
\begin{align}
\label{eq: necessary condition sup}
\begin{cases}
\eta_{k} = (-1)^{A_{k-1} - B_{k-1}}\eta_{k-1}       & \Rightarrow 0 \leqslant l_{k} - l_{k-1} \leqslant (A_{k} - B_{k}) - (A_{k-1} - B_{k-1}),  \\
\eta_{k} \neq (-1)^{A_{k-1} - B_{k-1}}\eta_{k-1}  & \Rightarrow l_{k} + l_{k-1} > A_{k-1} - B_{k-1}.   
\end{cases} 
\end{align}
\end{lemma}

\begin{proof}
Let $\q_{\gg}$ be a dominating parameter with discrete diagonal restriction. We also define $\q^{(k)}$ from $\q_{\gg}$ by shifting $(\rho, A_{i} + T_{i}, B_{i} + T_{i}, \zeta_{i})$ back to $(\rho, A_{i}, B_{i}, \zeta_{i})$ for $i \leqslant k$. Then
\begin{align*}
\r^{\Sigma_{0}}_{M, >_{\q}}(\q_{\gg}, \ul, \ueta) & \hookrightarrow \times_{i < k-1}
      \begin{pmatrix}
              \zeta_{i} (B_{i} + T_{i}) & \cdots & \zeta_{i}(B_{i} + 1) \\
              \vdots &  & \vdots \\
              \zeta_{i} (A_{i} + T_{i}) & \cdots & \zeta_{i}(A_{i} + 1)
       \end{pmatrix} \times
       \underbrace{\begin{pmatrix}
              \zeta_{k-1} (B_{k-1} + T_{k-1}) & \cdots & \zeta_{k-1}(B_{k-1} + 1) \\
              \vdots &  & \vdots \\
              \zeta_{k-1} (A_{k-1} + T_{k-1}) & \cdots & \zeta_{k-1}(A_{k-1} + 1)
       \end{pmatrix}}_{I} \\ 
& \times 
      \underbrace{\begin{pmatrix}
              \zeta_{k} (B_{k} + T_{k}) & \cdots & \zeta_{k}(B_{k} + 1) \\
              \vdots &  & \vdots \\
              \zeta_{k} (A_{k} + T_{k}) & \cdots & \zeta_{k}(A_{k} + 1)
       \end{pmatrix}}_{II}  \rtimes  \r^{\Sigma_{0}}_{M, >_{\q}}(\q^{(k)}, \ul, \ueta), 
\end{align*}  
where $i$ increases. Note $(I)$ and $(II)$ are interchangeable due to $[A_{k} + 1, B_{k} + 1] \supseteq [A_{k-1} + 1, B_{k-1} + 1]$. Since $B_{k} + T_{k-1} + 1 > A_{i} + T_{i} + 1$ for $i < k-1$, we have
\[
\Jac_{(\rho, A_{k} + T_{k}, B_{k} + T_{k}, \zeta_{k}) \mapsto (\rho, A_{k} + T_{k-1}, B_{k} + T_{k-1}, \zeta_{k})} \r^{\Sigma_{0}}_{M, >_{\q}}(\q_{\gg}, \ul, \ueta) \neq 0.
\]
In particular, 
\[
\Jac_{(\rho, A_{k} + T_{k}, B_{k} + T_{k}, \zeta_{k}) \mapsto (\rho, A_{k} + T_{k-1} + B_{k-1} - B_{k}, B_{k-1} + T_{k-1}, \zeta_{k})} \r^{\Sigma_{0}}_{M, >_{\q}}(\q_{\gg}, \ul, \ueta) \neq 0.
\]
By Proposition~\ref{prop: basic}, the condition \eqref{eq: basic condition} is satisfied for $(\rho, A_{k} + T_{k-1} + B_{k-1} - B_{k}, B_{k-1} + T_{k-1}, l_{k}, \eta_{k}, \zeta_{k})$ and $(\rho, A_{k-1} + T_{k-1}, B_{k-1} + T_{k-1}, l_{k-1}, \eta_{k-1}, \zeta_{k-1})$, i.e.,
\begin{itemize}

\item If $\eta_{k} = (-1)^{A_{k-1} - B_{k-1}}\eta_{k-1}$, then 
\[
\begin{cases}
& (A_{k} + T_{k-1} + B_{k-1} - B_{k}) - l_{k} \geqslant (A_{k-1} + T_{k-1}) - l_{k-1} \Rightarrow  l_{k} - l_{k-1} \leqslant (A_{k} - B_{k}) - (A_{k-1} - B_{k-1}), \\ \\

& (B_{k-1} + T_{k-1}) + l_{k} \geqslant (B_{k-1} + T_{k-1}) + l_{k-1} \Rightarrow  l_{k} - l_{k-1} \geqslant 0.
\end{cases}
\]

\item If $\eta_{k} \neq (-1)^{A_{k-1} - B_{k-1}}\eta_{k-1}$, then
\[
(B_{k-1} + T_{k-1}) + l_{k} > (A_{k-1} + T_{k-1}) - l_{k-1} \Rightarrow  l_{k} + l_{k-1} > A_{k-1} - B_{k-1}.
\]

\end{itemize}

This finishes the proof.

\end{proof}

\begin{lemma}
\label{lemma: necessary condition sub}
Suppose $[A_{k}, B_{k}] \subseteq [A_{k-1}, B_{k-1}]$. If $\r^{\Sigma_{0}}_{M, >_{\q}}(\q, \ul, \ueta) \neq 0$, then $l_{k}, l_{k-1}, \eta_{k}, \eta_{k-1}$ satisfy the following condition: 
\begin{align}
\label{eq: necessary condition sub}
\begin{cases}
\eta_{k} = (-1)^{A_{k-1} - B_{k-1}}\eta_{k-1}       & \Rightarrow 0 \leqslant l_{k-1} - l_{k} \leqslant (A_{k-1} - B_{k-1}) - (A_{k} - B_{k}),  \\
\eta_{k} \neq (-1)^{A_{k-1} - B_{k-1}}\eta_{k-1}  & \Rightarrow l_{k} + l_{k-1} > A_{k} - B_{k}.   
\end{cases} 
\end{align}
\end{lemma}

\begin{proof}

Let $\q_{\gg}$ be a dominating parameter with discrete diagonal restriction. We also define $\q^{(k)}$ from $\q_{\gg}$ by shifting $(\rho, A_{i} + T_{i}, B_{i} + T_{i}, \zeta_{i})$ back to $(\rho, A_{i}, B_{i}, \zeta_{i})$ for $i \leqslant k$. Then
\begin{align*}
\r^{\Sigma_{0}}_{M, >_{\q}}(\q_{\gg}, \ul, \ueta) & \hookrightarrow \times_{i < k-1}
      \begin{pmatrix}
              \zeta_{i} (B_{i} + T_{i}) & \cdots & \zeta_{i}(B_{i} + 1) \\
              \vdots &  & \vdots \\
              \zeta_{i} (A_{i} + T_{i}) & \cdots & \zeta_{i}(A_{i} + 1)
       \end{pmatrix} \times
       \underbrace{\begin{pmatrix}
              \zeta_{k-1} (B_{k-1} + T_{k-1}) & \cdots & \zeta_{k-1}(B_{k-1} + 1) \\
              \vdots &  & \vdots \\
              \zeta_{k-1} (A_{k-1} + T_{k-1}) & \cdots & \zeta_{k-1}(A_{k-1} + 1)
       \end{pmatrix}}_{I} \\ 
& \times 
      \underbrace{\begin{pmatrix}
              \zeta_{k} (B_{k} + T_{k}) & \cdots & \zeta_{k}(B_{k} + 1) \\
              \vdots &  & \vdots \\
              \zeta_{k} (A_{k} + T_{k}) & \cdots & \zeta_{k}(A_{k} + 1)
       \end{pmatrix}}_{II}  \rtimes  \r^{\Sigma_{0}}_{M, >_{\q}}(\q^{(k)}, \ul, \ueta), 
\end{align*}  
where $i$ increases. Note $(I)$ and $(II)$ are interchangeable due to $[A_{k} + 1, B_{k} + 1] \subseteq [A_{k-1} + 1, B_{k-1} + 1]$. Since $B_{k} + T_{k-1} + 1 > A_{i} + T_{i} + 1$ for $i < k-1$, we have
\[
\Jac_{(\rho, A_{k} + T_{k}, B_{k} + T_{k}, \zeta_{k}) \mapsto (\rho, A_{k} + T_{k-1}, B_{k} + T_{k-1}, \zeta_{k})} \r^{\Sigma_{0}}_{M, >_{\q}}(\q_{\gg}, \ul, \ueta) \neq 0.
\]
In particular, 
\[
\Jac_{(\rho, A_{k} + T_{k}, B_{k} + T_{k}, \zeta_{k}) \mapsto (\rho, A_{k-1} + T_{k-1}, B_{k} + T_{k-1} + A_{k-1} - A_{k}, \zeta_{k})} \r^{\Sigma_{0}}_{M, >_{\q}}(\q_{\gg}, \ul, \ueta) \neq 0.
\]
By Proposition~\ref{prop: basic}, the condition \eqref{eq: basic condition} is satisfied for $(\rho,  A_{k-1} + T_{k-1}, B_{k} + T_{k-1} + A_{k-1} - A_{k}, l_{k}, \eta_{k}, \zeta_{k})$ and $(\rho, A_{k-1} + T_{k-1}, B_{k-1} + T_{k-1}, l_{k-1}, \eta_{k-1}, \zeta_{k-1})$, i.e.,
\begin{itemize}

\item If $\eta_{k} = (-1)^{A_{k-1} - B_{k-1}}\eta_{k-1}$, then 
\[
\begin{cases}
& (A_{k-1} + T_{k-1}) - l_{k} \geqslant (A_{k-1} + T_{k-1}) - l_{k-1}  \Rightarrow  l_{k-1} - l_{k} \geqslant 0, \\ \\

& (B_{k} + T_{k-1} + A_{k-1} - A_{k}) + l_{k} \geqslant (B_{k-1} + T_{k-1}) + l_{k-1} \Rightarrow  l_{k-1} - l_{k} \leqslant (A_{k-1} - B_{k-1}) - (A_{k} - B_{k}).
\end{cases}
\]

\item If $\eta_{k} \neq (-1)^{A_{k-1} - B_{k-1}}\eta_{k-1}$, then
\[
(B_{k} + T_{k-1} + A_{k-1} - A_{k}) + l_{k} > (A_{k-1} + T_{k-1}) - l_{k-1} \Rightarrow  l_{k} + l_{k-1} > A_{k} - B_{k}.
\]

\end{itemize}

This finishes the proof.

\end{proof}

\section{Change of order formulas}
\label{sec: change of order}

For $\q = \q_{p} \in \cQ{G}$, we want to show how M{\oe}glin's parametrization of elements in $\Pkt{\q}^{\Sigma_{0}}$ changes as we change the order $>_{\q}$. So we will fix an admissible order $>_{\q}$ and we also fix a self-dual unitary irreducible supercuspidal representation $\rho$ of $GL(d_{\rho})$. We index the Jordan blocks in $Jord_{\rho}(\q)$ such that 
\[
(\rho, A_{i}, B_{i}, \zeta_{i}) >_{\q} (\rho, A_{i-1}, B_{i-1}, \zeta_{i-1}).
\]
Let $(\rho, A_{k}, B_{k}, \zeta_{k}) >_{\q} (\rho, A_{k-1}, B_{k-1}, \zeta_{k-1})$ be two adjacent blocks under the order $>_{\q}$. We denote by $>'_{\q}$ the order obtained from $>_{\q}$ by switching $(\rho, A_{k-1}, B_{k-1}, \zeta_{k-1})$ and $(\rho, A_{k}, B_{k}, \zeta_{k})$. And we assume $>'_{\q}$ is also admissible. Then we are in the following two cases.

\subsection{Case $\zeta_{k} = \zeta_{k-1}$}

In this case, we can assume without loss of generality that $[A_{k}, B_{k}] \supseteq [A_{k-1}, B_{k-1}]$. 
For functions $\ul(\rho, A, B, \zeta) \in [0, [(A-B+1)/2]]$ and $\ueta(\rho, A, B, \zeta) \in \mathbb{Z}/2\mathbb{Z}$ on $Jord(\q)$, we denote 
\[
l_{k} = \ul(\rho, A_{k}, B_{k}, \zeta_{k}), \quad l_{k-1} = \ul(\rho, A_{k-1}, B_{k-1}, \zeta_{k-1}),
\] 
and 
\[
\eta_{k} = \ueta(\rho, A_{k}, B_{k}, \zeta_{k}), \quad \eta_{k-1} = \ueta(\rho, A_{k-1}, B_{k-1}, \zeta_{k-1}).
\] 
From $(\ul, \ueta)$ satisfying \eqref{eq: necessary condition sup}, we want to construct another pair $(\ul', \ueta')$ such that 
\[
\ul'(\cdot) = \ul(\cdot) \text{ and } \ueta'(\cdot) = \ueta(\cdot)
\] 
over $Jord(\q) \backslash \{(\rho, A_{k}, B_{k}, \zeta_{k}), (\rho, A_{k-1}, B_{k-1}, \zeta_{k-1})\}$. Let us denote 
\[
l'_{k} = \ul'(\rho, A_{k}, B_{k}, \zeta_{k}), \quad l'_{k-1} = \ul'(\rho, A_{k-1}, B_{k-1}, \zeta_{k-1}),
\] 
and 
\[
\eta'_{k} = \ueta'(\rho, A_{k}, B_{k}, \zeta_{k}), \quad \eta'_{k-1} = \ueta'(\rho, A_{k-1}, B_{k-1}, \zeta_{k-1}).
\] 
Then we define $l'_{k}, l'_{k-1}, \eta'_{k}, \eta'_{k-1}$ according to the following formulas.

\begin{itemize}

\item If $\eta_{k} \neq (-1)^{A_{k-1} - B_{k-1}} \eta_{k-1}$, then $\eta'_{k-1} = (-1)^{A_{k} - B_{k}} \eta'_{k}$ and 
         \[
         \begin{cases}
         l_{k-1} = l'_{k-1} \\
          l_{k} - l'_{k} = (A_{k-1} - B_{k-1} - 2l_{k-1}) + 1 \\
         \eta'_{k-1} = (-1)^{A_{k} - B_{k}} \eta_{k-1} 
         \end{cases}
         \]
         
\item If $\eta_{k} = (-1)^{A_{k-1} - B_{k-1}} \eta_{k-1}$ and 
         \[
         l_{k} - l_{k-1} < (A_{k} - B_{k})/2 - (A_{k-1} - B_{k-1}) + l_{k-1},
         \] 
         then $\eta'_{k-1} \neq (-1)^{A_{k} - B_{k}}\eta'_{k}$ and 
         \[
         \begin{cases}
         l_{k-1} = l'_{k-1} \\
         l'_{k} - l_{k} = (A_{k-1} - B_{k-1} - 2l_{k-1}) + 1 \\
         \eta'_{k-1} = (-1)^{A_{k} - B_{k}} \eta_{k-1} 
         \end{cases}
         \]         
         
\item If $\eta_{k} = (-1)^{A_{k-1} - B_{k-1}} \eta_{k-1}$ and 
         \[
         l_{k} - l_{k-1} \geq (A_{k} - B_{k})/2 - (A_{k-1} - B_{k-1}) + l_{k-1},
         \] 
         then $\eta'_{k-1} = (-1)^{A_{k} - B_{k}}\eta'_{k}$ and 
         \[
         \begin{cases}
         l_{k-1} = l'_{k-1} \\
         (l'_{k} - l'_{k-1}) + (l_{k} - l_{k-1}) = (A_{k} - B_{k}) - (A_{k-1} - B_{k-1}) \\
         \eta'_{k-1} = (-1)^{A_{k} - B_{k}} \eta_{k-1} 
         \end{cases}
         \]         
               
\end{itemize}

         
         
         
One can check $(\ul', \ueta')$ satisfies \eqref{eq: necessary condition sub}. We denote this transformation by $S^{+}$. We can also define its ``inverse" $S^{-}$, namely we start with any $(\ul', \ueta')$ satisfying \eqref{eq: necessary condition sub}, and we define $l_{k}, l_{k-1}, \eta_{k}, \eta_{k-1}$ according to the following formulas.

\begin{itemize}

\item If $\eta'_{k-1} \neq (-1)^{A_{k} - B_{k}} \eta'_{k}$, then $\eta_{k} = (-1)^{A_{k-1} - B_{k-1}}\eta_{k-1}$ and 
         \[
         \begin{cases}
         l_{k-1} = l'_{k-1} \\
          l'_{k} - l_{k} = (A_{k-1} - B_{k-1} - 2l_{k-1}) + 1 \\
         \eta'_{k-1} = (-1)^{A_{k} - B_{k}} \eta_{k-1} 
         \end{cases}
        \]
         
\item If $\eta'_{k-1} = (-1)^{A_{k} - B_{k}} \eta'_{k}$ and 
         \[
         l'_{k} - l'_{k-1} < (A_{k} - B_{k})/2 - (A_{k-1} - B_{k-1}) + l'_{k-1},
         \] 
         then $\eta_{k} \neq (-1)^{A_{k-1} - B_{k-1}}\eta_{k-1}$ and 
         \[
         \begin{cases}
         l_{k-1} = l'_{k-1} \\
         l_{k} - l'_{k} = (A_{k-1} - B_{k-1} - 2l_{k-1}) + 1 \\
         \eta'_{k-1} = (-1)^{A_{k} - B_{k}} \eta_{k-1} 
         \end{cases}
        \]         
         
\item If $\eta'_{k-1} = (-1)^{A_{k} - B_{k}} \eta'_{k}$ and 
         \[
         l'_{k} - l'_{k-1} \geq (A_{k} - B_{k})/2 - (A_{k-1} - B_{k-1}) + l'_{k-1},
         \] 
         then $\eta_{k} = (-1)^{A_{k-1} - B_{k-1}}\eta_{k-1}$ and 
         \[
         \begin{cases}
         l_{k-1} = l'_{k-1} \\
         (l'_{k} - l'_{k-1}) + (l_{k} - l_{k-1}) = (A_{k} - B_{k}) - (A_{k-1} - B_{k-1}) \\
         \eta'_{k-1} = (-1)^{A_{k} - B_{k}} \eta_{k-1} 
         \end{cases}
         \]         
         
\end{itemize}


         
         

One can also check $(\ul, \ueta)$ satisfies \eqref{eq: necessary condition sup}. Moreover, we have 

\[
S^{-} \circ S^{+} (\ul, \ueta) \sim_{\Sigma_{0}} (\ul, \ueta),
\]
and
\[
S^{+} \circ S^{-} (\ul', \ueta') \sim_{\Sigma_{0}} (\ul', \ueta').
\]
So $S^{+}$ (resp. $S^{-}$) induces a bijection between $(\ul, \eta)$ satisfying \eqref{eq: necessary condition sup} and $(\ul', \ueta')$ satisfying \eqref{eq: necessary condition sub} modulo the equivalence relation $\sim_{\Sigma_{0}}$ on both sides.

\begin{theorem}
\label{thm: change order equal sign}
Suppose $(\ul', \ueta') = S^{+} (\ul, \ueta)$, then
\[
\r^{\Sigma_{0}}_{M, >_{\q}}(\q, \ul, \ueta) = \r^{\Sigma_{0}}_{M, >'_{\q}}(\q, \ul', \ueta').
\]
\end{theorem}

Let $\q_{\gg}$ be a dominating parameter of $\q$ such that $Jord_{\rho}(\q_{\gg}) = Jord_{\rho}(\q)$, and $Jord_{\rho'}(\q_{\gg})$ has discrete diagonal restriction for $\rho' \neq \rho$. Then
\[
\r^{\Sigma_{0}}_{M, >_{\q}}(\q, \ul, \ueta) = \Jac_{X^{c}} \r^{\Sigma_{0}}_{M, >_{\q}}(\q_{\gg}, \ul, \ueta), 
\]
and
\[
\r^{\Sigma_{0}}_{M, >'_{\q}}(\q, \ul', \ueta') = \Jac_{X^{c}} \r^{\Sigma_{0}}_{M, >'_{\q}}(\q_{\gg}, \ul', \ueta').
\]
So it suffices to prove the proposition for such $\q_{\gg}$. Therefore, in the following discussions of the proof of this proposition, we will always assume $Jord_{\rho'}(\q)$ has discrete diagonal restriction for $\rho' \neq \rho$, and if we choose some dominating $\q_{\gg}$ of $\q$, we will always assume $Jord_{\rho'}(\q_{\gg}) = Jord_{\rho'}(\q)$ for $\rho' \neq \rho$.

\subsubsection{Reduction}

Let $(\ul', \ueta') = S^{+} (\ul, \ueta)$. We want to reduce the proposition to the following case:
\begin{align}
\label{eq: reduction change order equal sign}
\text{ $(\rho, A_{i}, B_{i}, \zeta_{i}) \gg (\rho, A_{i-1}, B_{i-1}, \zeta_{i-1})$ for $i \neq k$, $(\rho, A_{k}, B_{k}, \zeta_{k}) \gg (\rho, A_{k-2}, B_{k-2}, \zeta_{k -2})$ and $0$. }
\end{align}
We will do this in two steps. First we will reduce it to the case:
\begin{align}
\label{eq: reduction change order equal sign 1}
\text{ $(\rho, A_{i}, B_{i}, \zeta_{i}) \gg (\rho, A_{i-1}, B_{i-1}, \zeta_{i-1})$ for $i > k$, $(\rho, A_{k}, B_{k}, \zeta_{k}) \gg \cup_{j=1}^{k-2}\{(\rho, A_{j}, B_{j}, \zeta_{j})\}$ and $0$. }
\end{align}
Let us choose a dominating parameter $\q_{\gg}$ with respect to $>_{\q}$ such that $T_{i} =  0$ for $i < k-1$,  
\[
\text{ $(\rho, A_{i} + T_{i}, B_{i} + T_{i}, \zeta_{i}) \gg (\rho, A_{i-1} + T_{i-1}, B_{i-1} + T_{i-1}, \zeta_{i-1})$ for $i \geqslant k$ }
\]
and
\[
(\rho, A_{k} + T_{k-1}, B_{k} + T_{k-1}, \zeta_{k}) \gg \cup_{j=1}^{k-2}\{(\rho, A_{j}, B_{j}, \zeta_{j}) \text{ and } 0.
\]
From $\q_{\gg}$, we can obtain a dominating parameter $\q'_{\gg}$ with respect to $>'_{\q}$ such that $T'_{i} = T_{i}$ for $i \neq k, k-1$, and  
\(
T'_{k} = T_{k-1}, T'_{k-1} = T_{k}.
\)
Let us also denote $T_{k-1}$ by $T$, and construct $\q^{T}_{\gg}$ from $\q_{\gg}$ by changing $T_{k}$ to $T$. Let $\q^{(k)}_{\gg}$ be obtained from $\q_{\gg}$ by changing $T_{k}, T_{k-1}$ to zero.

Suppose $\r^{\Sigma_{0}}_{M, >_{\q}}(\q, \ul, \ueta) \neq 0$, then 
\begin{align*}
\r^{\Sigma_{0}}_{M, >_{\q}}(\q_{\gg}, \ul, \ueta) & \hookrightarrow 
       \begin{pmatrix}
              \zeta_{k-1} (B_{k-1} + T_{k-1}) & \cdots & \zeta_{k-1}(B_{k-1} + 1) \\
              \vdots &  & \vdots \\
              \zeta_{k-1} (A_{k-1} + T_{k-1}) & \cdots & \zeta_{k-1}(A_{k-1} + 1)
       \end{pmatrix} \times
       \begin{pmatrix}
              \zeta_{k} (B_{k} + T_{k}) & \cdots & \zeta_{k}(B_{k} + 1) \\
              \vdots &  & \vdots \\
              \zeta_{k} (A_{k} + T_{k}) & \cdots & \zeta_{k}(A_{k} + 1)
       \end{pmatrix} \\       
& \rtimes  \r^{\Sigma_{0}}_{M, >_{\q}}(\q^{(k)}_{\gg}, \ul, \ueta),
\end{align*}  
where the two generalized segments are interchangeable. So $\r^{\Sigma_{0}}_{M, >_{\q}}(\q^{T}_{\gg}, \ul, \ueta) \neq 0$, and 
\begin{align*}
\r^{\Sigma_{0}}_{M, >_{\q}}(\q^{T}_{\gg}, \ul, \ueta) & \hookrightarrow 
      \begin{pmatrix}
              \zeta_{k} (B_{k} + T) & \cdots & \zeta_{k}(B_{k} + 1) \\
              \vdots &  & \vdots \\
              \zeta_{k} (A_{k} + T) & \cdots & \zeta_{k}(A_{k} + 1)
       \end{pmatrix} \\
  & \times 
       \begin{pmatrix}
              \zeta_{k-1} (B_{k-1} + T) & \cdots & \zeta_{k-1}(B_{k-1} + 1) \\
              \vdots &  & \vdots \\
              \zeta_{k-1} (A_{k-1} + T) & \cdots & \zeta_{k-1}(A_{k-1} + 1)
       \end{pmatrix}  \rtimes  \r^{\Sigma_{0}}_{M, >_{\q}}(\q^{(k)}_{\gg}, \ul, \ueta).
\end{align*}  
By \eqref{eq: reduction change order equal sign 1}, 
\[
\r^{\Sigma_{0}}_{M, >'_{\q}}(\q^{T}_{\gg}, \ul', \ueta') = \r^{\Sigma_{0}}_{M, >_{\q}}(\q^{T}_{\gg}, \ul, \ueta) \neq 0.
\] 
Then
\begin{align*}
\r^{\Sigma_{0}}_{M, >'_{\q}}(\q'_{\gg}, \ul', \ueta') & \hookrightarrow 
       \begin{pmatrix}
              \zeta_{k-1} (B_{k-1} + T'_{k-1}) & \cdots & \zeta_{k-1}(B_{k-1} + T +1 ) \\
              \vdots &  & \vdots \\
              \zeta_{k-1} (A_{k-1} + T'_{k-1}) & \cdots & \zeta_{k-1}(A_{k-1} + T + 1)
       \end{pmatrix}  \rtimes  \r^{\Sigma_{0}}_{M, >'_{\q}}(\q^{T}_{\gg}, \ul', \ueta') \\
& \hookrightarrow       
       \underbrace{\begin{pmatrix}
            \zeta_{k-1} (B_{k-1} + T'_{k-1}) & \cdots & \zeta_{k-1}(B_{k-1} + T +1 ) \\
              \vdots &  & \vdots \\
              \zeta_{k-1} (A_{k-1} + T'_{k-1}) & \cdots & \zeta_{k-1}(A_{k-1} + T + 1)
       \end{pmatrix}}_{I} \\
& \times 
      \underbrace{\begin{pmatrix}
              \zeta_{k} (B_{k} + T) & \cdots & \zeta_{k}(B_{k} + 1) \\
              \vdots &  & \vdots \\
              \zeta_{k} (A_{k} + T) & \cdots & \zeta_{k}(A_{k} + 1)
       \end{pmatrix}}_{II}
    \times
       \underbrace{\begin{pmatrix}
              \zeta_{k-1} (B_{k-1} + T) & \cdots & \zeta_{k-1}(B_{k-1} + 1) \\
              \vdots &  & \vdots \\
              \zeta_{k-1} (A_{k-1} + T) & \cdots & \zeta_{k-1}(A_{k-1} + 1)
       \end{pmatrix}}_{III} \\       
& \rtimes  \r^{\Sigma_{0}}_{M, >_{\q}}(\q^{(k)}_{\gg}, \ul, \ueta).
\end{align*}
We can interchange $(II)$ and $(III)$. If $B_{k-1} \neq B_{k}$, then $\Jac_{\zeta_{k-1} (B_{k-1} + T)}\r^{\Sigma_{0}}_{M, >'_{\q}}(\q'_{\gg}, \ul', \ueta') = 0$. So we can ``combine" $(I)$ and $(III)$, i.e.,
\begin{align*}
\r^{\Sigma_{0}}_{M, >'_{\q}}(\q'_{\gg}, \ul', \ueta') & \hookrightarrow       
       \underbrace{\begin{pmatrix}
            \zeta_{k-1} (B_{k-1} + T'_{k-1}) & \cdots & \zeta_{k-1}(B_{k-1} +1 ) \\
              \vdots &  & \vdots \\
              \zeta_{k-1} (A_{k-1} + T'_{k-1}) & \cdots & \zeta_{k-1}(A_{k-1} + 1)
       \end{pmatrix}}_{I + III} 
 \times 
      \underbrace{\begin{pmatrix}
              \zeta_{k} (B_{k} + T) & \cdots & \zeta_{k}(B_{k} + 1) \\
              \vdots &  & \vdots \\
              \zeta_{k} (A_{k} + T) & \cdots & \zeta_{k}(A_{k} + 1)
       \end{pmatrix}}_{II} \\     
& \rtimes  \r^{\Sigma_{0}}_{M, >_{\q}}(\q^{(k)}_{\gg}, \ul, \ueta).
\end{align*}
If $B_{k-1} = B_{k}$, let us write
\begin{align*}
\r^{\Sigma_{0}}_{M, >'_{\q}}(\q'_{\gg}, \ul', \ueta') & \hookrightarrow 
       \underbrace{\begin{pmatrix}
            \zeta_{k-1} (B_{k-1} + T'_{k-1}) & \cdots & \zeta_{k-1}(B_{k-1} + T + 1) \\
              \vdots &  & \vdots \\
              \zeta_{k-1} (A_{k-1} + T'_{k-1}) & \cdots & \zeta_{k-1}(A_{k-1} + T + 1)
       \end{pmatrix}}_{I}  \\     
& \times 
       \begin{pmatrix}
              \zeta_{k-1} (B_{k-1} + T) \\
              \vdots \\
              \zeta_{k-1} (A_{k-1} + T) 
       \end{pmatrix}     
   \times  
       \underbrace{\begin{pmatrix}
              \zeta_{k-1} (B_{k-1} + T - 1) & \cdots & \zeta_{k-1}(B_{k-1} + 1) \\
              \vdots &  & \vdots \\
              \zeta_{k-1} (A_{k-1} + T -1) & \cdots & \zeta_{k-1}(A_{k-1} + 1)
       \end{pmatrix}}_{III_{-}}      \\
&\times
       \underbrace{\begin{pmatrix}
              \zeta_{k} (B_{k} + T) & \cdots & \zeta_{k}(B_{k} + 1) \\
              \vdots &  & \vdots \\
              \zeta_{k} (A_{k} + T) & \cdots & \zeta_{k}(A_{k} + 1)
       \end{pmatrix}}_{II} \rtimes  \r^{\Sigma_{0}}_{M, >_{\q}}(\q^{(k)}_{\gg}, \ul, \ueta).
\end{align*}
There exists an irreducible constituent $\sigma$ of 
\begin{align*}
       \underbrace{\begin{pmatrix}
            \zeta_{k-1} (B_{k-1} + T'_{k-1}) & \cdots & \zeta_{k-1}(B_{k-1} + T + 1) \\
              \vdots &  & \vdots \\
              \zeta_{k-1} (A_{k-1} + T'_{k-1}) & \cdots & \zeta_{k-1}(A_{k-1} + T + 1)
       \end{pmatrix}}_{I}  
\times 
       \begin{pmatrix}
              \zeta_{k-1} (B_{k-1} + T) \\
              \vdots \\
              \zeta_{k-1} (A_{k-1} + T) 
       \end{pmatrix}     
\end{align*}
such that 
\begin{align*}
\r^{\Sigma_{0}}_{M, >'_{\q}}(\q'_{\gg}, \ul', \ueta') \hookrightarrow 
       \sigma   
  & \times  
       \underbrace{\begin{pmatrix}
              \zeta_{k-1} (B_{k-1} + T - 1) & \cdots & \zeta_{k-1}(B_{k-1} + 1) \\
              \vdots &  & \vdots \\
              \zeta_{k-1} (A_{k-1} + T -1) & \cdots & \zeta_{k-1}(A_{k-1} + 1)
       \end{pmatrix}}_{III_{-}}      \\
&\times
       \underbrace{\begin{pmatrix}
              \zeta_{k} (B_{k} + T) & \cdots & \zeta_{k}(B_{k} + 1) \\
              \vdots &  & \vdots \\
              \zeta_{k} (A_{k} + T) & \cdots & \zeta_{k}(A_{k} + 1)
       \end{pmatrix}}_{II} \rtimes  \r^{\Sigma_{0}}_{M, >_{\q}}(\q^{(k)}_{\gg}, \ul, \ueta).
\end{align*}
Suppose $\Jac_{\zeta_{k-1}(B_{k-1} + T)} \sigma \neq 0$, then $\Jac_{\zeta_{k-1}(B_{k-1} + T)} \sigma$ is contained in 
\begin{align*}
       \underbrace{\begin{pmatrix}
            \zeta_{k-1} (B_{k-1} + T'_{k-1}) & \cdots & \zeta_{k-1}(B_{k-1} + T + 1) \\
              \vdots &  & \vdots \\
              \zeta_{k-1} (A_{k-1} + T'_{k-1}) & \cdots & \zeta_{k-1}(A_{k-1} + T + 1)
       \end{pmatrix}}_{I}  
\times 
       \begin{pmatrix}
              \zeta_{k-1} (B_{k-1} + T + 1) \\
              \vdots \\
              \zeta_{k-1} (A_{k-1} + T) 
       \end{pmatrix}     
\end{align*}
which is irreducible. So
\begin{align*}
\sigma \hookrightarrow \rho||^{\zeta_{k-1}(B_{k-1} + T)}  \times 
        \underbrace{\begin{pmatrix}
            \zeta_{k-1} (B_{k-1} + T'_{k-1}) & \cdots & \zeta_{k-1}(B_{k-1} + T + 1) \\
              \vdots &  & \vdots \\
              \zeta_{k-1} (A_{k-1} + T'_{k-1}) & \cdots & \zeta_{k-1}(A_{k-1} + T + 1)
       \end{pmatrix}}_{I}  
\times 
       \begin{pmatrix}
              \zeta_{k-1} (B_{k-1} + T + 1) \\
              \vdots \\
              \zeta_{k-1} (A_{k-1} + T) 
       \end{pmatrix}.
\end{align*}
Hence
\begin{align*}
\r^{\Sigma_{0}}_{M, >'_{\q}}(\q'_{\gg}, \ul', \ueta') & \hookrightarrow \rho||^{\zeta_{k-1}(B_{k-1} + T)}  \times 
       \underbrace{\begin{pmatrix}
            \zeta_{k-1} (B_{k-1} + T'_{k-1}) & \cdots & \zeta_{k-1}(B_{k-1} + T + 1) \\
              \vdots &  & \vdots \\
              \zeta_{k-1} (A_{k-1} + T'_{k-1}) & \cdots & \zeta_{k-1}(A_{k-1} + T + 1)
       \end{pmatrix}}_{I}  \\     
& \times 
       \begin{pmatrix}
              \zeta_{k-1} (B_{k-1} + T + 1) \\
              \vdots \\
              \zeta_{k-1} (A_{k-1} + T) 
       \end{pmatrix}     
   \times  
       \underbrace{\begin{pmatrix}
              \zeta_{k} (B_{k} + T) & \cdots & \zeta_{k}(B_{k} + 1) \\
              \vdots &  & \vdots \\
              \zeta_{k} (A_{k} + T) & \cdots & \zeta_{k}(A_{k} + 1)
       \end{pmatrix}}_{II}\\
&\times
       \underbrace{\begin{pmatrix}
              \zeta_{k-1} (B_{k-1} + T - 1) & \cdots & \zeta_{k-1}(B_{k-1} + 1) \\
              \vdots &  & \vdots \\
              \zeta_{k-1} (A_{k-1} + T -1) & \cdots & \zeta_{k-1}(A_{k-1} + 1)
       \end{pmatrix}}_{III_{-}}   \rtimes  \r^{\Sigma_{0}}_{M, >_{\q}}(\q^{(k)}_{\gg}, \ul, \ueta).
\end{align*}
If $A_{k} = A_{k-1}$, then $[A_{k}, B_{k}] = [A_{k-1}, B_{k-1}]$, and there is nothing to prove. So we can assume $A_{k} > A_{k-1}$. Then
\begin{align*}
\r^{\Sigma_{0}}_{M, >'_{\q}}(\q'_{\gg}, \ul', \ueta') & \hookrightarrow \rho||^{\zeta_{k-1}(B_{k-1} + T)}  
\times 
       \begin{pmatrix}
              \zeta_{k} (B_{k} + T)  \\
              \vdots \\
              \zeta_{k} (A_{k} + T) 
       \end{pmatrix}  
 \times   
       \underbrace{\begin{pmatrix}
            \zeta_{k-1} (B_{k-1} + T'_{k-1}) & \cdots & \zeta_{k-1}(B_{k-1} + T + 1) \\
              \vdots &  & \vdots \\
              \zeta_{k-1} (A_{k-1} + T'_{k-1}) & \cdots & \zeta_{k-1}(A_{k-1} + T + 1)
       \end{pmatrix}}_{I}  \\     
& \times 
       \begin{pmatrix}
              \zeta_{k-1} (B_{k-1} + T + 1) \\
              \vdots \\
              \zeta_{k-1} (A_{k-1} + T) 
       \end{pmatrix}     
   \times  
       \underbrace{\begin{pmatrix}
              \zeta_{k} (B_{k} + T - 1) & \cdots & \zeta_{k}(B_{k} + 1) \\
              \vdots &  & \vdots \\
              \zeta_{k} (A_{k} + T -1) & \cdots & \zeta_{k}(A_{k} + 1)
       \end{pmatrix}}_{II_{-}}\\
&\times
       \underbrace{\begin{pmatrix}
              \zeta_{k-1} (B_{k-1} + T - 1) & \cdots & \zeta_{k-1}(B_{k-1} + 1) \\
              \vdots &  & \vdots \\
              \zeta_{k-1} (A_{k-1} + T -1) & \cdots & \zeta_{k-1}(A_{k-1} + 1)
       \end{pmatrix}}_{III_{-}}   \rtimes  \r^{\Sigma_{0}}_{M, >_{\q}}(\q^{(k)}_{\gg}, \ul, \ueta).
\end{align*}
As a result, we have 
\[
\Jac_{\zeta_{k-1} (B_{k-1} + T), \zeta_{k}(B_{k} + T)}\r^{\Sigma_{0}}_{M, >'_{\q}}(\q'_{\gg}, \ul', \ueta') \neq 0,
\]
which is impossible. Therefore, we must have $\Jac_{\zeta_{k-1}(B_{k-1} + T)} \sigma = 0$, and hence
\begin{align*}
\sigma =       
        \underbrace{\begin{pmatrix}
            \zeta_{k-1} (B_{k-1} + T'_{k-1}) & \cdots & \zeta_{k-1}(B_{k-1} + T) \\
              \vdots &  & \vdots \\
              \zeta_{k-1} (A_{k-1} + T'_{k-1}) & \cdots & \zeta_{k-1}(A_{k-1} + T)
       \end{pmatrix}}_{I_{+}}    
\end{align*}
In this case, 
\[
\Jac_{\zeta_{k-1} (B_{k-1} + T - 1)}\r^{\Sigma_{0}}_{M, >'_{\q}}(\q'_{\gg}, \ul', \ueta') = 0. 
\]
So we again have
\begin{align*}
\r^{\Sigma_{0}}_{M, >'_{\q}}(\q'_{\gg}, \ul', \ueta') & \hookrightarrow       
       \underbrace{\begin{pmatrix}
            \zeta_{k-1} (B_{k-1} + T'_{k-1}) & \cdots & \zeta_{k-1}(B_{k-1} +1 ) \\
              \vdots &  & \vdots \\
              \zeta_{k-1} (A_{k-1} + T'_{k-1}) & \cdots & \zeta_{k-1}(A_{k-1} + 1)
       \end{pmatrix}}_{(I + III)} 
 \times 
      \underbrace{\begin{pmatrix}
              \zeta_{k} (B_{k} + T) & \cdots & \zeta_{k}(B_{k} + 1) \\
              \vdots &  & \vdots \\
              \zeta_{k} (A_{k} + T) & \cdots & \zeta_{k}(A_{k} + 1)
       \end{pmatrix}}_{II} \\     
& \rtimes  \r^{\Sigma_{0}}_{M, >_{\q}}(\q^{(k)}_{\gg}, \ul, \ueta).
\end{align*}
Since $[\zeta_{k-1} (A_{k-1} + T'_{k-1}), \zeta_{k-1}(A_{k-1} + 1)] \supseteq [\zeta_{k} (A_{k} + T), \zeta_{k}(A_{k} + 1)]$, we can interchange $(I+III)$ and $(II)$. Therefore,
\[
\r^{\Sigma_{0}}_{M, >'_{\q}}(\q^{(k)}_{\gg}, \ul', \ueta') = \r^{\Sigma_{0}}_{M, >_{\q}}(\q^{(k)}_{\gg}, \ul, \ueta).
\]
After applying $\circ_{i>k} \Jac_{(\rho, A_{i} + T_{i}, B_{i} + T_{i}, \zeta_{i}) \mapsto (\rho, A_{i}, B_{i}, \zeta_{i})}$ to both sides, we get
\[
\r^{\Sigma_{0}}_{M, >_{\q}}(\q, \ul, \ueta) = \r^{\Sigma_{0}}_{M, >'_{\q}}(\q, \ul', \ueta').
\]

Secondly, we want to further reduce it to \eqref{eq: reduction change order equal sign}. So let us assume we are in case \eqref{eq: reduction change order equal sign 1}. We can choose a dominating parameter $\q_{\gg}$ with discrete diagonal restriction so that $T_{i} = 0$ for $i >k$ and $i = k-1$. We also require 
\[
\text{ $(\rho, A_{i} + T_{i}, B_{i} + T_{i}, \zeta_{i}) \gg (\rho, A_{i-1} + T_{i-1}, B_{i-1} + T_{i-1}, \zeta_{i-1})$ for $i < k$}, 
\]
and 
\[
(\rho, A_{k}, B_{k}, \zeta_{k}) \gg (\rho, A_{k-2} + T_{k-2}, B_{k-2} + T_{k-2}, \zeta_{k -2}) \text{ and } 0.
\]
Suppose $\r^{\Sigma_{0}}_{M, >_{\q}}(\q, \ul, \ueta) \neq 0$, then
\begin{align*}
\r^{\Sigma_{0}}_{M, >_{\q}}(\q_{\gg}, \ul, \ueta) & \hookrightarrow \times_{i<k-1}
       \underbrace{\begin{pmatrix}
              \zeta_{i} (B_{i} + T_{i}) & \cdots & \zeta_{i}(B_{i} + 1) \\
              \vdots &  & \vdots \\
              \zeta_{i} (A_{i} + T_{i}) & \cdots & \zeta_{i}(A_{i} + 1)
       \end{pmatrix}}_{I_{i}} 
  \times 
      \underbrace{\begin{pmatrix}
              \zeta_{k} (B_{k} + T_{k}) & \cdots & \zeta_{k}(B_{k} + 1) \\
              \vdots &  & \vdots \\
              \zeta_{k} (A_{k} + T_{k}) & \cdots & \zeta_{k}(A_{k} + 1)
       \end{pmatrix}}_{II} \\
& \rtimes  \r^{\Sigma_{0}}_{M, >_{\q}}(\q, \ul, \ueta),
\end{align*}   
where $i$ increases. Since $B_{k} + 1 > A_{i} + T_{i} + 1$ for $i < k-1$, we can interchange $(II)$ with $(I_{i})$. Let $\q^{(k)}_{\gg}$ be obtained from $\q_{\gg}$ by changing $T_{k}$ to zero.
Then
\begin{align*}
\r^{\Sigma_{0}}_{M, >_{\q}}(\q^{(k)}_{\gg}, \ul, \ueta) & \hookrightarrow \times_{i<k-1}
       \begin{pmatrix}
              \zeta_{i} (B_{i} + T_{i}) & \cdots & \zeta_{i}(B_{i} + 1) \\
              \vdots &  & \vdots \\
              \zeta_{i} (A_{i} + T_{i}) & \cdots & \zeta_{i}(A_{i} + 1)
       \end{pmatrix} \rtimes  \r^{\Sigma_{0}}_{M, >_{\q}}(\q, \ul, \ueta),
\end{align*}   
By \eqref{eq: reduction change order equal sign}, 
\[
\r^{\Sigma_{0}}_{M, >_{\q}}(\q^{(k)}_{\gg}, \ul, \ueta) = \r^{\Sigma_{0}}_{M, >'_{\q}}(\q^{(k)}_{\gg}, \ul', \ueta') \neq 0.
\]
Since 
\[
\Jac_{(\rho, A_{k} + T_{k}, B_{k} + T_{k}, \zeta_{k}) \mapsto (\rho, A_{k}, B_{k}, \zeta_{k})}
\]
commutes with 
\[
\circ_{i < k-1}\Jac_{(\rho, A_{i} + T_{i}, B_{i} + T_{i}, \zeta_{k}) \mapsto (\rho, A_{i}, B_{i}, \zeta_{i})},
\]
we have 
\[
\circ_{i < k-1}\Jac_{(\rho, A_{i} + T_{i}, B_{i} + T_{i}, \zeta_{k}) \mapsto (\rho, A_{i}, B_{i}, \zeta_{i})} \r^{\Sigma_{0}}_{M, >_{\q}}(\q^{(k)}_{\gg}, \ul, \ueta) = \r^{\Sigma_{0}}_{M, >_{\q}}(\q, \ul, \ueta)
\]
Similarly, 
\[
\circ_{i < k-1}\Jac_{(\rho, A_{i} + T_{i}, B_{i} + T_{i}, \zeta_{k}) \mapsto (\rho, A_{i}, B_{i}, \zeta_{i})} \r^{\Sigma_{0}}_{M, >'_{\q}}(\q^{(k)}_{\gg}, \ul', \ueta') = \r^{\Sigma_{0}}_{M, >'_{\q}}(\q, \ul', \ueta')
\]
So 
\[
\r^{\Sigma_{0}}_{M, >_{\q}}(\q, \ul, \ueta) = \r^{\Sigma_{0}}_{M, >'_{\q}}(\q, \ul', \ueta').
\] 
This finishes our reduction step.

\subsubsection{Critical case}

From the previous reduction, we can now assume \eqref{eq: reduction change order equal sign}:
\begin{align*}
\text{ $(\rho, A_{i}, B_{i}, \zeta_{i}) \gg (\rho, A_{i-1}, B_{i-1}, \zeta_{i-1})$ for $i \neq k$, $(\rho, A_{k}, B_{k}, \zeta_{k}) \gg (\rho, A_{k-2}, B_{k-2}, \zeta_{k -2})$ and $0$. }
\end{align*}
In this critical case, we can actually get the nonvanishing condition.

\begin{lemma}
\label{lemma: sufficient condition}
Suppose we are in case \eqref{eq: reduction change order equal sign}.
\begin{enumerate}

\item 
$\r^{\Sigma_{0}}_{M, >_{\q}}(\q, \ul, \ueta) \neq 0$ if and only if 
\begin{align*}
\begin{cases}
\eta_{k} = (-1)^{A_{k-1} - B_{k-1}}\eta_{k-1}       & \Rightarrow 0 \leqslant l_{k} - l_{k-1} \leqslant (A_{k} - B_{k}) - (A_{k-1} - B_{k-1}),  \\
\eta_{k} \neq (-1)^{A_{k-1} - B_{k-1}}\eta_{k-1}  & \Rightarrow l_{k} + l_{k-1} > A_{k-1} - B_{k-1}.   
\end{cases} 
\end{align*}

\item 
$\r^{\Sigma_{0}}_{M, >'_{\q}}(\q, \ul', \ueta') \neq 0$ if and only if 
\begin{align*}
\begin{cases}
\eta'_{k-1} = (-1)^{A_{k} - B_{k}}\eta'_{k}       & \Rightarrow 0 \leqslant l'_{k} - l'_{k-1} \leqslant (A_{k} - B_{k}) - (A_{k-1} - B_{k-1}),  \\
\eta'_{k-1} \neq (-1)^{A_{k} - B_{k}}\eta'_{k}  & \Rightarrow l'_{k} + l'_{k-1} > A_{k-1} - B_{k-1}.   
\end{cases} 
\end{align*}
\end{enumerate}
\end{lemma}

\begin{proof}
We will only show $(1)$, and $(2)$ is similar. One first notes the necessity of the nonvanishing condition has been shown in Lemma~\ref{lemma: necessary condition sup}, so we get an upper bound for the size of the packet $|\Pkt{\q}^{\Sigma_{0}}|$. In fact we can also get a lower bound for it. Let us define $\q^{*}$ by changing $(\rho, A_{k}, B_{k}, \zeta)$ to $(\rho, A_{k-1}, B_{k} - A_{k} + A_{k-1}, \zeta)$. Then the functor $\Jac_{(\rho, A_{k}, B_{k}, \zeta) \mapsto (\rho, A_{k-1}, B_{k} - A_{k} + A_{k-1}, \zeta)}$ induces a surjection from $\Pkt{\q}^{\Sigma_{0}}$ to $\Pkt{\q^{*}}^{\Sigma_{0}}$:
\[
\r^{\Sigma_{0}}_{M, >_{\q^{*}}}(\q, \ul, \ueta) = \Jac_{(\rho, A_{k}, B_{k}, \zeta) \mapsto (\rho, A_{k-1}, B_{k} - A_{k} + A_{k-1}, \zeta)} \r^{\Sigma_{0}}_{M, >_{\q}}(\q, \ul, \ueta).
\]
So $|\Pkt{\q^{*}}^{\Sigma_{0}}| < |\Pkt{\q}^{\Sigma_{0}}|$. By Proposition~\ref{prop: basic}, we have $\r^{\Sigma_{0}}_{M, >'_{\q}}(\q^{*}, \ul', \ueta') \neq 0$ if and only if
\begin{align*}
\begin{cases}
\eta'_{k-1} = (-1)^{A_{k} - B_{k}}\eta'_{k}       & \Rightarrow 0 \leqslant l'_{k} - l'_{k-1} \leqslant (A_{k-1} - (B_{k} - A_{k} + A_{k-1})) - (A_{k-1} - B_{k-1}),  \\
\eta'_{k-1} \neq (-1)^{A_{k} - B_{k}}\eta'_{k}  & \Rightarrow l'_{k} + l'_{k-1} > A_{k-1} - B_{k-1}.   
\end{cases} 
\end{align*}
Comparing this condition with the necessary condition for $\r^{\Sigma_{0}}_{M, >_{\q}}(\q, \ul, \ueta) \neq 0$, one can easily see that $|\Pkt{\q^{*}}^{\Sigma_{0}}|$ is equal to the upper bound for $|\Pkt{\q}^{\Sigma_{0}}|$. Therefore, $|\Pkt{\q}^{\Sigma_{0}}|$ must be equal to its upper bound, i.e., the necessary condition for $\r^{\Sigma_{0}}_{M, >_{\q}}(\q, \ul, \ueta) \neq 0$ is also sufficient.
\end{proof}

Now we begin to prove the change of order formula in this case. Let us define $\q_{-}$ by
\[
Jord(\q_{-}) = Jord(\q) \backslash \{(\rho, A_{k}, B_{k}, \zeta_{k}), (\rho, A_{k-1}, B_{k-1}, \zeta_{k-1})\},
\]
then $\q_{-}$ has discrete diagonal restriction. Let $\zeta = \zeta_{k} = \zeta_{k-1}$. 
We are going to break the proof into four steps.

{\bf Step One:} We want to show if $\r^{\Sigma_{0}}_{M, >_{\q}}(\q, \ul, \ueta) = \r^{\Sigma_{0}}_{M, >'_{\q}}(\q, \ul', \ueta') \neq 0$, then we can choose $(\ul', \ueta')$ within its ($\sim_{\Sigma_{0}}$) equivalence class such that
\[
\ul'(\cdot) = \ul(\cdot) \text{ and } \ueta'(\cdot) = \ueta(\cdot)
\] 
over $Jord(\q_{-})$. 

\begin{itemize}

\item Suppose $\ul (\cdot) = 0$ over $Jord(\q_{-})$. We can define $\q_{e_{-}}$ by
\[
Jord(\q_{e_{-}}) := \cup_{(\rho^{\sharp}, A^{\sharp}, B^{\sharp}, \zeta^{\sharp}) \in Jord(\q_{-})} \cup_{C^{\sharp} \in [A^{\sharp}, B^{\sharp}]} \{(\rho^{\sharp}, C^{\sharp}, C^{\sharp}, \zeta^{\sharp})\} 
\]
And we define $\q_{e}$ by adding $(\rho, A_{k}, B_{k}, \zeta_{k}), (\rho, A_{k-1}, B_{k-1}, \zeta_{k-1})$. From $(\ul, \ueta)$, we obtain $(\ul_{e}, \ueta_{e})$ such that
\[
\r^{\Sigma_{0}}_{M, >_{\q}}(\q, \ul, \ueta) = \r^{\Sigma_{0}}_{M, >_{\q}}(\q_{e}, \ul_{e}, \ueta_{e}).
\]
Suppose  
\[
\r^{\Sigma_{0}}_{M, >'_{\q}}(\q_{e}, \ul'_{e}, \ueta'_{e}) = \r^{\Sigma_{0}}_{M, >_{\q}}(\q_{e}, \ul_{e}, \ueta_{e}).
\] 
By computing $\e^{M/W}_{\q_{e}}$ with respect to $>_{\q}$ and $>'_{\q}$ (cf. \eqref{eq: M/W}), one finds 
\[
\ueta'_{e}(\cdot) = \ueta_{e}(\cdot)
\] 
over $Jord(\q_{e_{-}})$. Therefore, if we let
\[
\ul'(\cdot) = \ul(\cdot) = 0 \text{ and } \ueta'(\cdot) = \ueta(\cdot)
\] 
over $Jord(\q_{-})$, and 
\[
\ul'(\cdot) = \ul'_{e}(\cdot) \text{ and } \ueta'(\cdot) = \ueta'_{e}(\cdot)
\]
over $(\rho, A_{k}, B_{k}, \zeta_{k}), (\rho, A_{k-1}, B_{k-1}, \zeta_{k-1})$, then
\[
\r^{\Sigma_{0}}_{M, >'_{\q}}(\q, \ul', \ueta')  = \r^{\Sigma_{0}}_{M, >'_{\q}}(\q_{e}, \ul'_{e}, \ueta'_{e}) = \r^{\Sigma_{0}}_{M, >_{\q}}(\q, \ul, \ueta).
\]

\item Let $(\q_{0}, \ul_{0})$ be obtained from $(\q, \ul)$ by changing $(\rho^{\sharp}, A^{\sharp}, B^{\sharp}, \zeta^{\sharp})$ to $(\rho^{\sharp}, A^{\sharp} - l^{\sharp}, B^{\sharp} + l^{\sharp}, \zeta^{\sharp})$ and letting $\ul_{0}(\rho^{\sharp}, A^{\sharp} - l^{\sharp}, B^{\sharp} + l^{\sharp}, \zeta^{\sharp}) = 0$ for all $(\rho^{\sharp}, A^{\sharp}, B^{\sharp}, \zeta^{\sharp}) \in Jord(\q_{-})$, where $l^{\sharp} = \ul(\rho^{\sharp}, A^{\sharp}, B^{\sharp}, \zeta^{\sharp})$. Then
\[
\r^{\Sigma_{0}}_{M, >_{\q}}(\q, \ul, \ueta) \hookrightarrow \times_{(\rho^{\sharp}, A^{\sharp}, B^{\sharp}, \zeta^{\sharp}) \in Jord(\q_{-})} \tau(\rho^{\sharp}, A^{\sharp} - l^{\sharp}, B^{\sharp} + l^{\sharp}, \zeta^{\sharp}) \rtimes \r^{\Sigma_{0}}_{M, >_{\q}}(\q_{0}, \ul_{0}, \ueta),
\]
as the unique irreducible subrepresentation, where 
\[
\tau(\rho^{\sharp}, A^{\sharp} - l^{\sharp}, B^{\sharp} + l^{\sharp}, \zeta^{\sharp}) := \begin{pmatrix}
                \zeta^{\sharp} B^{\sharp} & \cdots & -\zeta^{\sharp} A^{\sharp} \\
                 \vdots &  & \vdots \\
                \zeta^{\sharp} (B^{\sharp} + l^{\sharp} - 1) & \cdots & - \zeta^{\sharp} (A^{\sharp} - l^{\sharp} + 1)
                \end{pmatrix}.
\]
Suppose $\r^{\Sigma_{0}}_{M, >_{\q}}(\q_{0}, \ul_{0}, \ueta) = \r^{\Sigma_{0}}_{M, >'_{\q}}(\q_{0}, \ul'_{0}, \ueta')$. We know from the previous discussion that
\[
\ul'_{0}(\cdot) = \ul_{0}(\cdot) = 0 \text{ and } \ueta'(\cdot) = \ueta(\cdot) 
\]
over $Jord(\q_{0}) \backslash \{(\rho, A_{k}, B_{k}, \zeta_{k}), (\rho, A_{k-1}, B_{k-1}, \zeta_{k-1})\}$. From $\ul'_{0}$ we can obtain $\ul'$ such that 
\[
\ul'(\rho^{\sharp}, A^{\sharp}, B^{\sharp}, \zeta^{\sharp}) = l^{\sharp} = \ul(\rho^{\sharp}, A^{\sharp}, B^{\sharp}, \zeta^{\sharp})
\]
for $(\rho^{\sharp}, A^{\sharp}, B^{\sharp}, \zeta^{\sharp}) \in Jord(\q_{-})$.
Then
\[
\r^{\Sigma_{0}}_{M, >'_{\q}}(\q, \ul', \ueta') \hookrightarrow \times_{(\rho^{\sharp}, A^{\sharp}, B^{\sharp}, \zeta^{\sharp}) \in Jord(\q_{-})} \tau(\rho^{\sharp}, A^{\sharp} - l^{\sharp}, B^{\sharp} + l^{\sharp}, \zeta^{\sharp})  \rtimes \r^{\Sigma_{0}}_{M, >'_{\q}}(\q_{0}, \ul'_{0}, \ueta').
\]
as the unique irreducible subrepresentation. Therefore, $\r^{\Sigma_{0}}_{M, >_{\q}}(\q, \ul, \ueta) = \r^{\Sigma_{0}}_{M, >'_{\q}}(\q, \ul', \ueta')$. This finishes the first step.

\end{itemize}

{\bf Step Two:} We want to give some restrictions on $(\ul', \ueta')$ in terms of $(\ul, \ueta)$, when $\r^{\Sigma_{0}}_{M, >_{\q}}(\q, \ul, \ueta) = \r^{\Sigma_{0}}_{M, >'_{\q}}(\q, \ul', \ueta') \neq 0$. From the previous step, we can asssume 
\[
\ul'(\cdot) = \ul(\cdot) \text{ and } \ueta'(\cdot) = \ueta(\cdot)
\] 
over $Jord(\q_{-})$. Next we will consider the partial cuspidal support of $\r^{\Sigma_{0}}_{M, >_{\q}}(\q, \ul, \ueta)$, which is defined as follows. Recall the cuspidal support of an irreducible admissible representation $\pi^{\Sigma_{0}}$ of $G^{\Sigma_{0}}$ is a $G^{\Sigma_{0}}$-conjugacy class of pairs $(M^{\Sigma_{0}}, \sigma^{\Sigma_{0}})$, where 
\[
M^{\Sigma_{0}} \cong \prod_{i} GL(n_{i}) \times G^{\Sigma_{0}}_{-}
\] 
is a Levi subgroup of $G^{\Sigma_{0}}$, and 
\[
\sigma^{\Sigma_{0}} \cong (\otimes_{i} \tau_{i}) \otimes \pi^{\Sigma_{0}}_{-}
\]
is an irreducible supercuspidal representation of $M^{\Sigma_{0}}$. We call $(G^{\Sigma_{0}}_{-}, \pi^{\Sigma_{0}}_{-})$ the partial cuspidal support of $\pi^{\Sigma_{0}}$. The partial cuspidal support of $\r^{\Sigma_{0}}_{M, >_{\q}}(\q, \ul, \ueta)$ can be computed as follows. Let $\psi_{\gg}$ be a dominating parameter of $\psi$ with discrete diagonal restriction, obtained by shifting $(\rho, A_{k}, B_{k}, \zeta_{k})$ to $(\rho, A_{k} + T_{k}, B_{k} + T_{k}, \zeta_{k})$. Then $\r^{\Sigma_{0}}_{M, >_{\q}}(\q, \ul, \ueta)$ has the same partial cuspidal support as $\r^{\Sigma_{0}}_{M, >_{\q}}(\q_{\gg}, \ul, \ueta)$. Combined with the generalized Aubert involution (cf. \cite[Definition 6.15]{Xu:Apacket}), it is also equal to that of 
\begin{align}
\label{eq: partial cuspidal support}
\r^{\Sigma_{0}}_{M, >_{\q}}\Big(\q^{0}_{\gg}, 0, \ueta\Big),
\end{align}
where $Jord(\q^{0}_{\gg})$ is obtained from $Jord(\q_{\gg})$ by changing $(\rho^{\sharp}, A^{\sharp}, B^{\sharp}, \zeta^{\sharp})$ to $(\rho^{\sharp}, A^{\sharp} - l^{\sharp}, B^{\sharp} + l^{\sharp}, +)$ for $l^{\sharp} = \ul(\rho^{\sharp}, A^{\sharp}, B^{\sharp}, \zeta^{\sharp})$.  This is a discrete series representation and its cuspidal support can be computed as in \cite{Xu:cusp}. To handle the combinatorics involved, we define an operation $(\ast)$ on the equivalence classes of $(\rho, A, B, 0, \eta, +) \sim (\rho, A + T, B + T, 0, \eta, +)$ as follows:
\[
(\rho, A, B, 0, \eta, +) \ast (\rho, A', B', 0, \eta', +) \sim (\rho, A^{*}, B^{*}, 0, \eta^{*}, +).
\]
\begin{enumerate}

\item If $\eta \neq (-1)^{A' - B'}\eta'$, then 
\[
\begin{cases}
A^{*} - B^{*} = (A - B) + (A' - B') + 1\\
\eta^{*} = \eta'
\end{cases}                 
\]                  

\item If $\eta = (-1)^{A' - B'}\eta'$ and 
\begin{enumerate}

\item $A - B > A' - B' \Rightarrow \begin{cases}
A^{*} - B^{*} = (A - B) - (A' - B') - 1 \\
\eta^{*} \neq \eta'
\end{cases}$

\item $A - B < A' - B' \Rightarrow \begin{cases}
A^{*} - B^{*} = (A' - B') - (A - B) - 1 \\
\eta^{*} = \eta'
\end{cases}$

\end{enumerate}                           
                     
\end{enumerate}
If $\eta = (-1)^{A' - B'}\eta'$ and $A - B = A' - B'$, we define the product to be $\emptyset$. We can force $\emptyset$ to be the identity element under this operation, then it is easy to check that $\{\emptyset\} \sqcup \{(\rho, A, B, 0, \eta, +)\}/_{\sim}$ forms a group. The partial cuspidal support of \eqref{eq: partial cuspidal support} is equal to that of 
\begin{align*}
\r^{\Sigma_{0}}_{M, >_{\q}}\Big(\q^{\sharp}, 0, \ueta^{\sharp}; \, \ast_{i}(\rho, A_{i} - l_{i}, B_{i} + l_{i}, 0, \eta_{i}, +)\Big),
\end{align*}
where $Jord(\q^{\sharp}) := Jord(\psi^{0}_{\gg}) \backslash Jord_{\rho}(\psi^{0}_{\gg}) $, $\ueta^{\sharp}$ is obtained by restriction, and the product $\ast$ is taken in the decreasing order with respect to $>_{\q}$. In the same way, one can show the partial cuspidal support of $\r^{\Sigma_{0}}_{M, >'_{\q}}(\q, \ul', \ueta')$ is equal to that of
\[
\r^{\Sigma_{0}}_{M, >'_{\q}}\Big(\q^{\sharp}, 0, \ueta^{\sharp}; \, \ast_{i}(\rho, A_{i} - l'_{i}, B_{i} + l'_{i}, 0, \eta'_{i}, +)\Big),
\]
where the product $\ast$ is taken in the decreasing order with respect to $>'_{\q}$. As a result, we must have 
\[
\ast_{i}(\rho, A_{i} - l_{i}, B_{i} + l_{i}, 0, \eta_{i}, +) \sim \ast_{i}(\rho, A_{i} - l'_{i}, B_{i} + l'_{i}, 0, \eta'_{i}, +).
\]
It follows
\[
(\rho, A_{k} - l_{k}, B_{k} + l_{k}, 0, \eta_{k}, +) \ast (\rho, A_{k-1} - l_{k-1}, B_{k-1} + l_{k-1}, 0, \eta_{k-1}, +)  
\]
is equivalent to 
\[
(\rho, A_{k-1} - l'_{k-1}, B_{k-1} + l'_{k-1}, 0, \eta'_{k-1}, +) \ast (\rho, A_{k} - l'_{k}, B_{k} + l'_{k}, 0, \eta'_{k}, +).
\]
So we are necessarily in one of the following situations.

\begin{enumerate}

\item If $\eta_{k} \neq (-1)^{A_{k-1} -B_{k-1}}\eta_{k-1}$ and $\eta'_{k-1} = (-1)^{A_{k} - B_{k}} \eta'_{k}$, then one of the following cases is satisfied.
        \\
        \begin{enumerate}
        
        \item $\eta_{k-1} = (-1)^{A_{k} - B_{k}} \eta'_{k-1}$ and \\
        \[
        (A_{k-1} - B_{k-1} - 2l_{k-1}) + (A_{k} - B_{k} - 2l_{k}) + 2 = (A_{k} - B_{k} - 2l'_{k}) - (A_{k-1} - B_{k-1} - 2l'_{k-1})
        \]
        i.e.,
        \[
        (l_{k} + l_{k-1}) - (l'_{k} - l'_{k-1})= A_{k-1} - B_{k-1} + 1.
        \]
        \\
        \item $\eta_{k-1} \neq (-1)^{A_{k} - B_{k}} \eta'_{k-1}$ and \\
        \[
        (A_{k-1} - B_{k-1} - 2l_{k-1}) + (A_{k} - B_{k} - 2l_{k}) + 2 = (A_{k-1} - B_{k-1} - 2l'_{k-1}) - (A_{k} - B_{k} - 2l'_{k}) 
        \]
         i.e.,
        \[
        (l_{k} + l_{k-1}) + (l'_{k} - l'_{k-1})= A_{k} - B_{k} + 1.
        \]        
         \\              
        \end{enumerate}

\item If $\eta_{k} = (-1)^{A_{k-1} -B_{k-1}}\eta_{k-1}$ and $\eta'_{k-1} \neq (-1)^{A_{k} - B_{k}} \eta'_{k}$, then one of the following cases is satisfied.
        \\
        \begin{enumerate}
        
        \item $\eta_{k-1} = (-1)^{A_{k} - B_{k}} \eta'_{k-1}$ and \\
        \[
        (A_{k} - B_{k} - 2l_{k}) - (A_{k-1} - B_{k-1} - 2l_{k-1}) = (A_{k-1} - B_{k-1} - 2l'_{k-1}) + (A_{k} - B_{k} - 2l'_{k}) + 2
        \]
        i.e.,
        \[
        (l'_{k} + l'_{k-1}) - (l_{k} - l_{k-1})= A_{k-1} - B_{k-1} + 1.
        \]        
        \\
        
        \item $\eta_{k-1} \neq (-1)^{A_{k} - B_{k}} \eta'_{k-1}$ and \\
        \[
        (A_{k-1} - B_{k-1} - 2l_{k-1}) - (A_{k} - B_{k} - 2l_{k}) = (A_{k-1} - B_{k-1} - 2l'_{k-1}) + (A_{k} - B_{k} - 2l'_{k}) + 2
        \]
        i.e.,
        \[
        (l'_{k} + l'_{k-1}) + (l_{k} - l_{k-1})= A_{k} - B_{k} + 1.
        \]        
        \\
                       
        \end{enumerate}

\item If $\eta_{k} = (-1)^{A_{k-1} -B_{k-1}}\eta_{k-1}$ and $\eta'_{k-1} = (-1)^{A_{k} - B_{k}} \eta'_{k}$, then one of the following cases is satisfied.
        \\
        \begin{enumerate}
        
        \item $\eta_{k-1} = (-1)^{A_{k} - B_{k}} \eta'_{k-1}$ and \\
        \[
        (A_{k-1} - B_{k-1} - 2l_{k-1}) - (A_{k} - B_{k} - 2l_{k}) = (A_{k} - B_{k} - 2l'_{k}) - (A_{k-1} - B_{k-1} - 2l'_{k-1})
        \]
        i.e.,
        \[
        (l_{k} - l_{k-1}) + (l'_{k} - l'_{k-1}) = (A_{k} - B_{k}) - (A_{k-1} - B_{k-1}).
        \]
        \\
                
        
        \item $\eta_{k-1} \neq (-1)^{A_{k} - B_{k}} \eta'_{k-1}$ and \\
        \[
        (A_{k-1} - B_{k-1} - 2l_{k-1}) - (A_{k} - B_{k} - 2l_{k}) = (A_{k-1} - B_{k-1} - 2l'_{k-1}) - (A_{k} - B_{k} - 2l'_{k})
        \]
        i.e.,
        \[
        l_{k} - l_{k-1} = l'_{k} - l'_{k-1}.
        \]
        \\       
        
                              
        \end{enumerate}

\item If $\eta_{k} \neq (-1)^{A_{k-1} -B_{k-1}}\eta_{k-1}$ and $\eta'_{k-1} \neq (-1)^{A_{k} - B_{k}} \eta'_{k}$, then the following case is satisfied.
        \\
        \begin{enumerate}
        
        \item $\eta_{k-1} \neq (-1)^{A_{k} - B_{k}} \eta'_{k-1}$ and \\
        \[
        (A_{k-1} - B_{k-1} - 2l_{k-1}) + (A_{k} - B_{k} - 2l_{k}) = (A_{k-1} - B_{k-1} - 2l'_{k-1}) + (A_{k} - B_{k} - 2l'_{k})
        \]
        i.e.,
        \[
        l_{k} + l_{k-1} = l'_{k} + l'_{k-1}.
        \]         
                                             
        \end{enumerate}

\end{enumerate}
Since in our change of order formulas, we always have 
\[
\eta_{k-1} = (-1)^{A_{k} - B_{k}} \eta'_{k-1},
\]
it is enough to eliminate those cases in which this is not satisfied. This is not easy in general, but at least we can do this when $l_{k-1} = 0$.

{\bf Step Three:} We would like to derive the change of order formula when $l_{k-1} = 0$. Let us define $\q_{e}$ by
\[
Jord(\q_{e}) := \cup_{C_{k-1} \in [A_{k-1}, B_{k-1}]} \{(\rho, C_{k-1}, C_{k-1}, \zeta_{k-1})\} \cup Jord(\q) \backslash \{(\rho, A_{k-1}, B_{k-1}, \zeta_{k-1})\}.
\]
Then we can assume $\r^{\Sigma_{0}}_{M, >_{\q}}(\q, \ul, \ueta) = \r^{\Sigma_{0}}_{M, >_{\q}}(\q_{e}, \ul_{e}, \ueta_{e}) \neq 0$. Suppose 
\[
\r^{\Sigma_{0}}_{M, >_{\q}}(\q_{e}, \ul_{e}, \ueta_{e}) = \r^{\Sigma_{0}}_{M, >'_{\q}}(\q_{e}, \ul'_{e}, \ueta'_{e}).
\]
One can show as in {\bf Step one} that 
\[
\ul'_{e}(\cdot) = \ul_{e}(\cdot) \text{ and } \ueta'_{e}(\cdot) = \ueta_{e}(\cdot)
\]
over $Jord(\q_{-})$. Moreover, by computing $\e^{M/W}_{\q_{e}}$ with respect to $>_{\q}$ and $>'_{\q}$, one finds $\ueta'_{e}$ is alternating over $\{\cup_{C_{k-1} \in [A_{k-1}, B_{k-1}]} (\rho, C_{k-1}, C_{k-1}, \zeta_{k-1})\}$. So from $(\ul'_{e}, \ueta'_{e})$, we can obtain $(\ul', \ueta')$ by letting $l'_{k-1} = 0$ and $\eta'_{k-1} = \ueta'_{e}(\rho, B_{k-1}, B_{k-1}, \zeta)$. Then 
\[
\r^{\Sigma_{0}}_{M, >'_{\q}}(\q, \ul', \ueta') = \r^{\Sigma_{0}}_{M, >'_{\q}}(\q_{e}, \ul'_{e}, \ueta'_{e}).
\]
It follows from {\bf Step two} that we have several restrictions on $(\ul', \ueta')$. To eliminate the case that $\eta_{k-1} \neq (-1)^{A_{k} - B_{k}} \eta'_{k-1}$, we would like to compute the difference between $\eta_{k-1}$ and $\eta'_{k-1}$ explicitly. The idea is again to compute $\e^{M/W}_{\q_{e}}$ with respect to $>_{\q}$ and $>'_{\q}$ (cf. \eqref{eq: M/W}). To distinguish this two orders, we write $\e^{M/W}_{\q_{e}}$ for $>_{\q}$ and $\e'^{M/W}_{\q_{e}}$ for $>'_{\q}$. Then
\[
\eta_{k-1} \e^{M/W}_{\q_{e}}(\rho, B_{k-1}, B_{k-1}, \zeta) = \eta'_{k-1} \e'^{M/W}_{\q_{e}}(\rho, B_{k-1}, B_{k-1}, \zeta).
\]
To apply the formula for $\e^{M/W}_{\q_{e}}$ (resp. $\e'^{M/W}_{\q_{e}}$), we need to write $(\rho, A_{k}, B_{k}, \zeta) = (\rho, a_{k}, b_{k})$.

\begin{itemize}

\item Suppose $\zeta = +1$.
        
        \begin{enumerate}
        
        \item $A_{k} \in \mathbb{Z}$, then 
        \(
        \begin{cases}
        a_{k}, b_{k} \text{ even } & \Rightarrow \eta_{k-1} = -\eta'_{k-1} \\
        a_{k}, b_{k} \text{ odd }  & \Rightarrow \eta_{k-1} = \eta'_{k-1}. 
        \end{cases}
        \)
        
        \item $A_{k} \notin \mathbb{Z}$, then 
        \(
        \begin{cases}
        a_{k} \text{ odd, } b_{k} \text{ even } & \Rightarrow \eta_{k-1} = -\eta'_{k-1} \\
        a_{k} \text{ even, } b_{k} \text{ odd }  & \Rightarrow \eta_{k-1} = \eta'_{k-1}. 
        \end{cases}
        \)        
        
        \end{enumerate}

\item Suppose $\zeta = -1$.

        \begin{enumerate}
        
        \item $A_{k} \in \mathbb{Z}$, then 
        \(
        \begin{cases}
        a_{k}, b_{k} \text{ even } & \Rightarrow \eta_{k-1} = -\eta'_{k-1} \\
        a_{k}, b_{k} \text{ odd }  & \Rightarrow \eta_{k-1} = \eta'_{k-1}. 
        \end{cases}
        \)
        
        \item $A_{k} \notin \mathbb{Z}$, then 
        \(
        \begin{cases}
        a_{k} \text{ even, } b_{k} \text{ odd } & \Rightarrow \eta_{k-1} = -\eta'_{k-1} \\
        a_{k} \text{ odd, } b_{k} \text{ even }  & \Rightarrow \eta_{k-1} = \eta'_{k-1}. 
        \end{cases}
        \)        
        
        \end{enumerate}

\end{itemize}
It follows from the computations here that 
\[
\eta_{k-1} = (-1)^{inf(a_{k}, b_{k})-1} \eta'_{k-1}.
\]
Recall $A_{k} - B_{k} + 1 = inf (a_{k}, b_{k})$, so this is exactly what we want. Adding this condition, the remaining cases in {\bf Step two} are as follows.
\begin{enumerate}

\item If $\eta_{k} \neq (-1)^{A_{k-1} -B_{k-1}}\eta_{k-1}$ and $\eta'_{k-1} = (-1)^{A_{k} - B_{k}} \eta'_{k}$, then $\eta_{k-1} = (-1)^{A_{k} - B_{k}} \eta'_{k-1}$ and 
        \[
        l_{k} - l'_{k} = A_{k-1} - B_{k-1} + 1.
        \]

\item If $\eta_{k} = (-1)^{A_{k-1} -B_{k-1}}\eta_{k-1}$ and $\eta'_{k-1} \neq (-1)^{A_{k} - B_{k}} \eta'_{k}$, then $\eta_{k-1} = (-1)^{A_{k} - B_{k}} \eta'_{k-1}$ and 
        \[
        l'_{k} - l_{k} = A_{k-1} - B_{k-1} + 1.
        \]

\item If $\eta_{k} = (-1)^{A_{k-1} -B_{k-1}}\eta_{k-1}$ and $\eta'_{k-1} = (-1)^{A_{k} - B_{k}} \eta'_{k}$, then $\eta_{k-1} = (-1)^{A_{k} - B_{k}} \eta'_{k-1}$ and 
        \[
        l_{k} + l'_{k} = (A_{k} - B_{k}) - (A_{k-1} - B_{k-1}) .
        \]

\end{enumerate}
So this finishes the proof of the change of order formula in the case $l_{k-1} = 0$.

{\bf Step Four:} In this last step, we want to show that if $\r^{\Sigma_{0}}_{M, >_{\q}}(\q, \ul, \ueta) = \r^{\Sigma_{0}}_{M, >'_{\q}}(\q, \ul', \ueta') \neq 0$ with $\l_{k-1} \neq 0$, then $(\ul', \ueta') \sim_{\Sigma_{0}} S^{+}(\ul, \ueta)$. Note when $[A_{k}, B_{k}] = [A_{k-1}, B_{k-1}]$, this is obvious. So from now on, we will assume 
\[
[A_{k}, B_{k}] \neq [A_{k-1}, B_{k-1}].
\] 
By the nonvanishing condition in Lemma~\ref{lemma: sufficient condition}, we have $l_{k} \geqslant l_{k-1}$. Since we have assumed $l_{k-1} \neq 0$, then $l_{k} \neq 0$.

First we would like to reduce it to the case $B_{k} = B_{k-1}$. Suppose $B_{k-1} > B_{k}$, let us define $\q^{*}$ from $\q$ by shifting $(\rho, A_{k-1}, B_{k-1}, \zeta)$ to $(\rho, A_{k-1} - B_{k-1} + B_{k}, B_{k}, \zeta)$. Then we have 
\[
\r^{\Sigma_{0}}_{M, >'_{\q}}(\q^{*}, \ul', \ueta') = \Jac_{(\rho, A_{k-1}, B_{k-1}, \zeta) \mapsto (\rho, A_{k-1} - B_{k-1} + B_{k}, B_{k}, \zeta)} \r^{\Sigma_{0}}_{M, >'_{\q}}(\q, \ul', \ueta').
\]
So $\Jac_{(\rho, A_{k-1}, B_{k-1}, \zeta) \mapsto (\rho, A_{k-1} - B_{k-1} + B_{k}, B_{k}, \zeta)}$ induces a bijection from $\Pkt{\q}^{\Sigma_{0}}$ to $\Pkt{\q^{*}}^{\Sigma_{0}}$ by Lemma~\ref{lemma: sufficient condition}. On the other side, we claim 
\[
\r^{\Sigma_{0}}_{M, >_{\q}}(\q^{*}, \ul, \ueta) = \Jac_{(\rho, A_{k-1}, B_{k-1}, \zeta) \mapsto (\rho, A_{k-1} - B_{k-1} + B_{k}, B_{k}, \zeta)} \r^{\Sigma_{0}}_{M, >_{\q}}(\q, \ul, \ueta).
\]
To see this, we let $\q_{\gg}$ be a dominating parameter with respect to $>_{\q}$, obtained from $\q$ by shifting $(\rho, A_{k}, B_{k}, \zeta)$ to $(\rho, A_{k} + T, B_{k} + T, \zeta)$, and $\q_{\gg}$ has discrete diagonal restriction. Then
\begin{align*}
\r^{\Sigma_{0}}_{M, >_{\q}}(\q_{\gg}, \ul, \ueta) & \hookrightarrow       
       \begin{pmatrix}
              \zeta (B_{k} + T) & \cdots & \zeta (B_{k} + 1)  \\
              \vdots &  & \vdots \\
              \zeta (A_{k} + T) & \cdots &  \zeta (A_{k} + 1)
       \end{pmatrix} 
 \rtimes \r^{\Sigma_{0}}_{M, >_{\q}}(\q, \ul, \ueta) \\
 & \hookrightarrow       
       \underbrace{\begin{pmatrix}
              \zeta (B_{k} + T) & \cdots & \zeta (B_{k} + 1)  \\
              \vdots &  & \vdots \\
              \zeta (A_{k} + T) & \cdots &  \zeta (A_{k} + 1)
       \end{pmatrix}}_{*-1} 
  \times 
      \underbrace{\begin{pmatrix}
              \zeta B_{k-1} & \cdots & \zeta (B_{k} + 1)  \\
              \vdots &  & \vdots \\
              \zeta A_{k-1} & \cdots &  \zeta (A_{k-1} - B_{k-1} + B_{k} + 1)
       \end{pmatrix}}_{*-2} 
   \rtimes \sigma    
\end{align*}
where 
\[
\sigma := \Jac_{(\rho, A_{k-1}, B_{k-1}, \zeta) \mapsto (\rho, A_{k-1} - B_{k-1} + B_{k}, B_{k}, \zeta)} \r^{\Sigma_{0}}_{M, >_{\q}}(\q, \ul, \ueta).
\]
Since we can interchange $(*-1)$ and $(*-2)$, it is easy to see $\r^{\Sigma_{0}}_{M, >_{\q}}(\q^{*}, \ul, \ueta) = \sigma$. This shows our claim. So if $\r^{\Sigma_{0}}_{M, >_{\q}}(\q, \ul, \ueta) = \r^{\Sigma_{0}}_{M, >'_{\q}}(\q, \ul', \ueta') \neq 0$, then $\r^{\Sigma_{0}}_{M, >_{\q}}(\q^{*}, \ul, \ueta) = \r^{\Sigma_{0}}_{M, >'_{\q}}(\q^{*}, \ul', \ueta')$. And suppose we know $(\ul, \ueta)$ is related to $(\ul', \ueta')$ according to our formula with respect to $\q^{*}$ modulo the equivalence relation $\sim_{\Sigma_{0}}$, then it is easy to see they are related in the same way with respect to $\q$. Hence $(\ul', \ueta') \sim_{\Sigma_{0}} S^{+}(\ul, \ueta)$.

Now we will only consider the case $B_{k-1} = B_{k}$, and by our previous assumption we have $A_{k} > A_{k-1}$. Let $\q^{**}$ be defined from $\q$ by changing $(\rho, A_{k}, B_{k}, \zeta)$ and $(\rho, A_{k-1}, B_{k-1}, \zeta)$ to $(\rho, A_{k} - 1, B_{k} + 1, \zeta)$ and $(\rho, A_{k-1} - 1, B_{k-1} + 1, \zeta)$ respectively. Then we claim for $l_{k-1} \neq 0$,
\begin{align*}
& \Jac_{\zeta B_{k-1}, \cdots, -\zeta A_{k-1}} \circ \Jac_{\zeta B_{k}, \cdots, -\zeta A_{k}} \r^{\Sigma_{0}}_{M, >_{\q}}(\q, \ul, \ueta) = \r^{\Sigma_{0}}_{M, >_{\q}}\Big(\q_{-}, \ul_{-}, \ueta_{-}; \\
& (\rho, A_{k} - 1, B_{k} + 1, l_{k} - 1, \eta_{k}, \zeta), (\rho, A_{k-1} - 1, B_{k-1} + 1, l_{k-1} - 1, \eta_{k-1}, \zeta)\Big).
\end{align*}
In particular, this means we get a bijection from $\Pkt{\q}^{\Sigma_{0}} \backslash \{\r^{\Sigma_{0}}_{M, >_{\q}}(\q, \ul, \ueta) : l_{k-1} = 0\}$ to $\Pkt{\q^{**}}^{\Sigma_{0}}$. To prove the claim, we will first show for $l_{k-1} \neq 0$,
\begin{align*}
& \Jac_{\zeta B_{k}, \cdots, -\zeta A_{k}} \r^{\Sigma_{0}}_{M, >_{\q}}(\q, \ul, \ueta) = \r^{\Sigma_{0}}_{M, >_{\q}}\Big(\q_{-}, \ul_{-}, \ueta_{-}; \\
& (\rho, A_{k} - 1, B_{k} + 1, l_{k} - 1, \eta_{k}, \zeta), (\rho, A_{k-1}, B_{k-1}, l_{k-1}, \eta_{k-1}, \zeta)\Big).
\end{align*}
Again let $\q_{\gg}$ be a dominating parameter with respect to $>_{\q}$, obtained from $\q$ by shifting $(\rho, A_{k}, B_{k}, \zeta)$ to $(\rho, A_{k} + T, B_{k} + T, \zeta)$, and $\q_{\gg}$ has discrete diagonal restriction. Then
\begin{align*}
\r^{\Sigma_{0}}_{M, >_{\q}}(\q_{\gg}, \ul, \ueta) & \hookrightarrow 
      \langle \zeta (B_{k} + T), \cdots, -\zeta (A_{k} + T) \rangle
 \times       
       \begin{pmatrix}
              \zeta (B_{k} + 1 + T) & \cdots & \zeta (B_{k} + 2)  \\
              \vdots &  & \vdots \\
              \zeta (A_{k} - 1 + T) & \cdots &  \zeta A_{k}
       \end{pmatrix} \\
& \rtimes \r^{\Sigma_{0}}_{M, >_{\q}}\Big(\q_{-}, \ul_{-}, \ueta_{-}; 
 (\rho, A_{k} - 1, B_{k} + 1, l_{k} - 1, \eta_{k}, \zeta), (\rho, A_{k-1}, B_{k-1}, l_{k-1}, \eta_{k-1}, \zeta)\Big) \\
& \hookrightarrow 
      \langle \zeta (B_{k} + T), \cdots, -\zeta A_{k} \rangle
 \times       
       \begin{pmatrix}
              \zeta (B_{k} + 1 + T) & \cdots & \zeta (B_{k} + 2)  \\
              \vdots &  & \vdots \\
              \zeta (A_{k} - 1 + T) & \cdots &  \zeta A_{k}
       \end{pmatrix} \\
& \times \underbrace{\langle -\zeta (A_{k} + 1), \cdots, -\zeta (A_{k} + T) \rangle}_{**-1}
   \rtimes \r^{\Sigma_{0}}_{M, >_{\q}}\Big(\q_{-}, \ul_{-}, \ueta_{-}; \\
& (\rho, A_{k} - 1, B_{k} + 1, l_{k} - 1, \eta_{k}, \zeta), (\rho, A_{k-1}, B_{k-1}, l_{k-1}, \eta_{k-1}, \zeta)\Big) 
\end{align*}
Since $A_{k} > A_{k-1}$, we can take the dual of $(**-1)$ by \eqref{eq: dualizing classical}. Therefore,  
\begin{align*}
\r^{\Sigma_{0}}_{M, >_{\q}}(\q, \ul, \ueta) & \hookrightarrow 
      \langle \zeta B_{k}, \cdots, -\zeta A_{k} \rangle
 \rtimes \r^{\Sigma_{0}}_{M, >_{\q}}\Big(\q_{-}, \ul_{-}, \ueta_{-}; \\
& (\rho, A_{k} - 1, B_{k} + 1, l_{k} - 1, \eta_{k}, \zeta), (\rho, A_{k-1}, B_{k-1}, l_{k-1}, \eta_{k-1}, \zeta)\Big) 
\end{align*}
By applying $\Jac_{\zeta B_{k}, \cdots, -\zeta A_{k}}$ to the full induced representation above, we get $\r^{\Sigma_{0}}_{M, >_{\q}}\Big(\q_{-}, \ul_{-}, \ueta_{-}; (\rho, A_{k} - 1, B_{k} + 1, l_{k} - 1, \eta_{k}, \zeta), (\rho, A_{k-1}, B_{k-1}, l_{k-1}, \eta_{k-1}, \zeta)\Big)$. So
\begin{align*}
& \Jac_{\zeta B_{k}, \cdots, -\zeta A_{k}} \r^{\Sigma_{0}}_{M, >_{\q}}(\q, \ul, \ueta) = \r^{\Sigma_{0}}_{M, >_{\q}}\Big(\q_{-}, \ul_{-}, \ueta_{-}; \\
& (\rho, A_{k} - 1, B_{k} + 1, l_{k} - 1, \eta_{k}, \zeta), (\rho, A_{k-1}, B_{k-1}, l_{k-1}, \eta_{k-1}, \zeta)\Big).
\end{align*}
Next for the same $T$, 
\begin{align*}
& \r^{\Sigma_{0}}_{M, >_{\q}}\Big(\q_{-}, \ul_{-}, \ueta_{-}; (\rho, A_{k} - 1 + T, B_{k} + 1 + T, l_{k} - 1, \eta_{k}, \zeta), (\rho, A_{k-1}, B_{k-1}, l_{k-1}, \eta_{k-1}, \zeta)\Big) \\
& \hookrightarrow \underbrace{\langle \zeta B_{k-1}, \cdots, -\zeta A_{k-1} \rangle}_{**-2} 
\times \underbrace{\begin{pmatrix}
              \zeta (B_{k} + 1 + T) & \cdots & \zeta (B_{k} + 2)  \\
              \vdots &  & \vdots \\
              \zeta (A_{k} - 1 + T) & \cdots &  \zeta A_{k}
           \end{pmatrix}}_{**-3} \\
& \rtimes \r^{\Sigma_{0}}_{M, >_{\q}}\Big(\q_{-}, \ul_{-}, \ueta_{-};  (\rho, A_{k} - 1, B_{k} + 1, l_{k} - 1, \eta_{k}, \zeta), (\rho, A_{k-1} - 1, B_{k-1} + 1, l_{k-1} - 1, \eta_{k-1}, \zeta)\Big).
\end{align*}
Since $B_{k} = B_{k-1}$, we can interchange $(**-2)$ and $(**-3)$. So
\begin{align*}
& \r^{\Sigma_{0}}_{M, >_{\q}}\Big(\q_{-}, \ul_{-}, \ueta_{-}; (\rho, A_{k} - 1, B_{k} + 1, l_{k} - 1, \eta_{k}, \zeta), (\rho, A_{k-1}, B_{k-1}, l_{k-1}, \eta_{k-1}, \zeta)\Big) \\
& \hookrightarrow \underbrace{\langle \zeta B_{k-1}, \cdots, -\zeta A_{k-1}\rangle}_{**-2} \rtimes \r^{\Sigma_{0}}_{M, >_{\q}}\Big(\q_{-}, \ul_{-}, \ueta_{-}; (\rho, A_{k} - 1, B_{k} + 1, l_{k} - 1, \eta_{k}, \zeta), \\
& (\rho, A_{k-1} - 1, B_{k-1} + 1, l_{k-1} - 1, \eta_{k-1}, \zeta)\Big).
\end{align*}
After applying $\Jac_{\zeta B_{k-1}, \cdots, -\zeta A_{k-1}}$ to the full induced representation above, we get $\r^{\Sigma_{0}}_{M, >_{\q}}\Big(\q_{-}, \ul_{-}, \ueta_{-}; (\rho, A_{k} - 1, B_{k} + 1, l_{k} - 1, \eta_{k}, \zeta), (\rho, A_{k-1} - 1, B_{k-1} + 1, l_{k-1} - 1, \eta_{k-1}, \zeta)\Big)$. So
\begin{align*}
& \Jac_{\zeta B_{k-1}, \cdots, -\zeta A_{k-1}} \r^{\Sigma_{0}}_{M, >_{\q}}\Big(\q_{-}, \ul_{-}, \ueta_{-}; (\rho, A_{k} - 1, B_{k} + 1, l_{k} - 1, \eta_{k}, \zeta), (\rho, A_{k-1}, B_{k-1}, l_{k-1}, \eta_{k-1}, \zeta)\Big) \\
& = \r^{\Sigma_{0}}_{M, >_{\q}}\Big(\q_{-}, \ul_{-}, \ueta_{-}; (\rho, A_{k} - 1, B_{k} + 1, l_{k} - 1, \eta_{k}, \zeta), (\rho, A_{k-1} - 1, B_{k-1} + 1, l_{k-1} - 1, \eta_{k-1}, \zeta)\Big).
\end{align*}
This finishes the proof of our claim. At last, we want to compute 
\[
\sigma^{**} := \Jac_{\zeta B_{k-1}, \cdots, -\zeta A_{k-1}} \circ \Jac_{\zeta B_{k}, \cdots, -\zeta A_{k}}  \r^{\Sigma_{0}}_{M, >'_{\q}}(\q, \ul', \ueta')
\]
for $l'_{k-1} \neq 0$. Let $\q'_{\gg}$ be a dominating parameter with respect to $>'_{\q}$, obtained from $\q$ by shifting $(\rho, A_{k-1}, B_{k-1}, \zeta)$ to $(\rho, A_{k-1} + T', B_{k-1} + T', \zeta)$, and $\q'_{\gg}$ has discrete diagonal restriction. Then
\begin{align*}
\r^{\Sigma_{0}}_{M, >'_{\q}}(\q'_{\gg}, \ul', \ueta') & \hookrightarrow 
      \langle \zeta (B_{k-1} + T'), \cdots, -\zeta (A_{k-1} + T') \rangle \times \langle \zeta B_{k}, \cdots, -\zeta A_{k} \rangle \\
& \times       
       \begin{pmatrix}
              \zeta (B_{k-1} + 1 + T') & \cdots & \zeta (B_{k-1} + 2)  \\
              \vdots &  & \vdots \\
              \zeta (A_{k-1} - 1 + T') & \cdots &  \zeta A_{k-1}
       \end{pmatrix} 
 \rtimes \r^{\Sigma_{0}}_{M, >'_{\q}}\Big(\q_{-}, \ul_{-}, \ueta_{-}; \\
& (\rho, A_{k} - 1, B_{k} + 1, l'_{k} - 1, \eta'_{k}, \zeta), (\rho, A_{k-1} - 1, B_{k-1} + 1, l'_{k-1} - 1, \eta'_{k-1}, \zeta)\Big) \\
& \hookrightarrow 
      \underbrace{\langle \zeta (B_{k-1} + T'), \cdots, \zeta(B_{k-1} + 1) \rangle}_{I} \times \underbrace{\langle \zeta B_{k-1}, \cdots, -\zeta (A_{k-1} + T')\rangle}_{II} \\
& \times \underbrace{\langle \zeta B_{k}, \cdots, -\zeta A_{k} \rangle}_{III} 
   \times \underbrace{\begin{pmatrix}
              \zeta (B_{k-1} + 1 + T') & \cdots & \zeta (B_{k-1} + 2)  \\
              \vdots &  & \vdots \\
              \zeta (A_{k-1} - 1 + T') & \cdots &  \zeta A_{k-1}
       \end{pmatrix}}_{IV} \\
&\rtimes \r^{\Sigma_{0}}_{M, >'_{\q}}\Big(\q_{-}, \ul_{-}, \ueta_{-}; (\rho, A_{k} - 1, B_{k} + 1, l'_{k} - 1, \eta'_{k}, \zeta), (\rho, A_{k-1} - 1, B_{k-1} + 1, l'_{k-1} - 1, \eta'_{k-1}, \zeta)\Big).       
\end{align*}
We can interchange $(IV)$ with $(III)$ and $(II)$. Also $(II)$ and $(III)$ are interchangeable. So
\begin{align*}
\r^{\Sigma_{0}}_{M, >'_{\q}}(\q'_{\gg}, \ul', \ueta') 
& \hookrightarrow 
      \underbrace{\langle \zeta (B_{k-1} + T'), \cdots, \zeta(B_{k-1} + 1)\rangle }_{I} 
   \times \underbrace{\begin{pmatrix}
              \zeta (B_{k-1} + 1 + T') & \cdots & \zeta (B_{k-1} + 2)  \\
              \vdots &  & \vdots \\
              \zeta (A_{k-1} - 1 + T') & \cdots &  \zeta A_{k-1}
       \end{pmatrix}}_{IV} \\
& \times \underbrace{\langle \zeta B_{k}, \cdots, -\zeta A_{k} \rangle}_{III}      
   \times \underbrace{\langle \zeta B_{k-1}, \cdots, -\zeta (A_{k-1} + T') \rangle}_{II} \\            
&\rtimes \r^{\Sigma_{0}}_{M, >'_{\q}}\Big(\q_{-}, \ul_{-}, \ueta_{-}; (\rho, A_{k} - 1, B_{k} + 1, l'_{k} - 1, \eta'_{k}, \zeta), (\rho, A_{k-1} - 1, B_{k-1} + 1, l'_{k-1} - 1, \eta'_{k-1}, \zeta)\Big) \\
& \hookrightarrow 
      \underbrace{\langle \zeta (B_{k-1} + T'), \cdots, \zeta(B_{k-1} + 1) \rangle}_{I} 
   \times \underbrace{\begin{pmatrix}
              \zeta (B_{k-1} + 1 + T') & \cdots & \zeta (B_{k-1} + 2)  \\
              \vdots &  & \vdots \\
              \zeta (A_{k-1} - 1 + T') & \cdots &  \zeta A_{k-1}
       \end{pmatrix}}_{IV} \\
& \times \underbrace{\langle \zeta B_{k}, \cdots, -\zeta A_{k-1} \rangle}_{III_{1}} \times \underbrace{\langle -\zeta (A_{k-1} + 1), \cdots - \zeta A_{k} \rangle}_{III_{2}}     
   \times \underbrace{\langle \zeta B_{k-1}, \cdots, -\zeta (A_{k-1} + T') \rangle}_{II} \\ 
& \times \underbrace{\begin{pmatrix}
              \zeta (B_{k} + 1) & \cdots & \zeta (B_{k} + A_{k-1} - A_{k}  + 2)  \\
              \vdots &  & \vdots \\
              \zeta (A_{k} - 1) & \cdots &  \zeta A_{k-1}
             \end{pmatrix}}_{V}          
   \rtimes \r^{\Sigma_{0}}_{M, >'_{\q}}\Big(\q_{-}, \ul_{-}, \ueta_{-}; \\
& (\rho, A_{k-1} - 1, B_{k} + A_{k-1} - A_{k} + 1, l'_{k} - 1, \eta'_{k}, \zeta), (\rho, A_{k-1} - 1, B_{k-1} + 1, l'_{k-1} - 1, \eta'_{k-1}, \zeta)\Big) 
\end{align*}
Since $\Jac_{\zeta (B_{k-1} + 1 + T')} \r^{\Sigma_{0}}_{M, >'_{\q}}(\q'_{\gg}, \ul', \ueta') = 0$, we can ``combine" $(I)$ and $(IV)$. We can also interchange $(III_{2})$ with $(II)$ and $(V)$, and then take dual of $(III_{2})$. As a result,
\begin{align*}
\r^{\Sigma_{0}}_{M, >'_{\q}}(\q'_{\gg}, \ul', \ueta') 
& \hookrightarrow 
      \underbrace{\begin{pmatrix}
              \zeta (B_{k-1} + T') & \cdots & \zeta (B_{k-1} + 1)  \\
              \vdots &  & \vdots \\
              \zeta (A_{k-1} - 1 + T') & \cdots &  \zeta A_{k-1}
       \end{pmatrix}}_{IV_{+}} 
   \times \underbrace{\langle \zeta B_{k}, \cdots, -\zeta A_{k-1} \rangle}_{III_{1}} \\
& \times \underbrace{\langle \zeta B_{k-1}, \cdots, -\zeta (A_{k-1} + T') \rangle}_{II} 
    \times \underbrace{\begin{pmatrix}
              \zeta (B_{k} + 1) & \cdots & \zeta (B_{k} + A_{k-1} - A_{k}  + 2)  \\
              \vdots &  & \vdots \\
              \zeta (A_{k} - 1) & \cdots &  \zeta A_{k-1}
             \end{pmatrix}}_{V} \\   
& \times \underbrace{\langle \zeta A_{k}, \cdots, \zeta (A_{k-1} + 1) \rangle}_{(III_{2})^{\vee}}          
   \rtimes \r^{\Sigma_{0}}_{M, >'_{\q}}\Big(\q_{-}, \ul_{-}, \ueta_{-}; \\
& (\rho, A_{k-1} - 1, B_{k} + A_{k-1} - A_{k} + 1, l'_{k} - 1, \eta'_{k}, \zeta), (\rho, A_{k-1} - 1, B_{k-1} + 1, l'_{k-1} - 1, \eta'_{k-1}, \zeta)\Big) 
\end{align*}
Since $A_{k} > A_{k-1} > B_{k-1}$ and $\Jac_{\zeta A_{k}} \r^{\Sigma_{0}}_{M, >'_{\q}}(\q'_{\gg}, \ul', \ueta') = 0$, we can further ``combine" $(III_{2})^{\vee}$ and $(V)$.
\begin{align*}
\r^{\Sigma_{0}}_{M, >'_{\q}}(\q'_{\gg}, \ul', \ueta') 
& \hookrightarrow 
      \underbrace{\begin{pmatrix}
              \zeta (B_{k-1} + T') & \cdots & \zeta (B_{k-1} + 1)  \\
              \vdots &  & \vdots \\
              \zeta (A_{k-1} - 1 + T') & \cdots &  \zeta A_{k-1}
       \end{pmatrix}}_{IV_{+}} 
   \times \underbrace{\langle \zeta B_{k}, \cdots, -\zeta A_{k-1} \rangle}_{III_{1}} \\
& \times \underbrace{\langle \zeta B_{k-1}, \cdots, -\zeta (A_{k-1} + T') \rangle}_{II} 
    \times \underbrace{\begin{pmatrix}
              \zeta (B_{k} + 1) & \cdots & \zeta (B_{k} + A_{k-1} - A_{k}  + 2)  \\
              \vdots &  & \vdots \\
              \zeta A_{k} & \cdots &  \zeta (A_{k-1} + 1)
             \end{pmatrix}}_{V_{+}} \\   
& \rtimes \r^{\Sigma_{0}}_{M, >'_{\q}}\Big(\q_{-}, \ul_{-}, \ueta_{-}; (\rho, A_{k-1} - 1, B_{k} + A_{k-1} - A_{k} + 1, l'_{k} - 1, \eta'_{k}, \zeta), \\
& (\rho, A_{k-1} - 1, B_{k-1} + 1, l'_{k-1} - 1, \eta'_{k-1}, \zeta)\Big) \\
& \hookrightarrow 
      \underbrace{\begin{pmatrix}
              \zeta (B_{k-1} + T') & \cdots & \zeta (B_{k-1} + 1)  \\
              \vdots &  & \vdots \\
              \zeta (A_{k-1} - 1 + T') & \cdots &  \zeta A_{k-1}
       \end{pmatrix}}_{IV_{+}} 
   \times \underbrace{\langle \zeta B_{k}, \cdots, -\zeta A_{k-1} \rangle}_{III_{1}} \\
& \times \underbrace{\langle \zeta B_{k-1}, \cdots, -\zeta A_{k-1} \rangle}_{II_{1}} \times \underbrace{\langle -\zeta (A_{k-1} + 1), \cdots -\zeta (A_{k-1} + T') \rangle}_{II_{2}} \\
& \times \underbrace{\begin{pmatrix}
              \zeta (B_{k} + 1) & \cdots & \zeta (B_{k} + A_{k-1} - A_{k}  + 2)  \\
              \vdots &  & \vdots \\
              \zeta A_{k} & \cdots &  \zeta (A_{k-1} + 1)
             \end{pmatrix}}_{V_{+}}  
  \rtimes \r^{\Sigma_{0}}_{M, >'_{\q}}\Big(\q_{-}, \ul_{-}, \ueta_{-}; \\
& (\rho, A_{k-1} - 1, B_{k} + A_{k-1} - A_{k} + 1, l'_{k} - 1, \eta'_{k}, \zeta), (\rho, A_{k-1} - 1, B_{k-1} + 1, l'_{k-1} - 1, \eta'_{k-1}, \zeta)\Big) 
\end{align*}
So we can interchange $(II_{2})$ with $(V_{+})$, and take dual of $(II_{2})$. Note $(II_{2})^{\vee}$ is interchangeable with $(V_{+})$. Since $A_{k-1} > B_{k-1}$, $(II_{2})^{\vee}$ is also interchangeable with $(II_{1})$ and $(III_{1})$. Therefore,
\begin{align*}
\r^{\Sigma_{0}}_{M, >'_{\q}}(\q'_{\gg}, \ul', \ueta') 
& \hookrightarrow 
      \underbrace{\begin{pmatrix}
              \zeta (B_{k-1} + T') & \cdots & \zeta (B_{k-1} + 1)  \\
              \vdots &  & \vdots \\
              \zeta (A_{k-1} - 1 + T') & \cdots &  \zeta A_{k-1}
       \end{pmatrix}}_{IV_{+}} 
  \times \underbrace{\langle \zeta (A_{k-1} + T'), \cdots \zeta (A_{k-1} + 1) \rangle}_{(II_{2})^{\vee}}  \\              
& \times \underbrace{\langle \zeta B_{k}, \cdots, -\zeta A_{k-1} \rangle}_{III_{1}} 
   \times \underbrace{\langle \zeta B_{k-1}, \cdots, -\zeta A_{k-1} \rangle}_{II_{1}} \\
& \times \underbrace{\begin{pmatrix}
              \zeta (B_{k} + 1) & \cdots & \zeta (B_{k} + A_{k-1} - A_{k}  + 2)  \\
              \vdots &  & \vdots \\
              \zeta A_{k} & \cdots &  \zeta (A_{k-1} + 1)
             \end{pmatrix}}_{V_{+}}  
  \rtimes \r^{\Sigma_{0}}_{M, >'_{\q}}\Big(\q_{-}, \ul_{-}, \ueta_{-}; \\
& (\rho, A_{k-1} - 1, B_{k} + A_{k-1} - A_{k} + 1, l'_{k} - 1, \eta'_{k}, \zeta), (\rho, A_{k-1} - 1, B_{k-1} + 1, l'_{k-1} - 1, \eta'_{k-1}, \zeta)\Big). 
\end{align*}
Consequently,
\begin{align*}
\r^{\Sigma_{0}}_{M, >'_{\q}}(\q, \ul', \ueta') 
& \hookrightarrow 
             \underbrace{\langle \zeta B_{k}, \cdots, -\zeta A_{k-1} \rangle}_{III_{1}} 
   \times \underbrace{\langle \zeta B_{k-1}, \cdots, -\zeta A_{k-1} \rangle}_{II_{1}} \\
& \times \underbrace{\begin{pmatrix}
              \zeta (B_{k} + 1) & \cdots & \zeta (B_{k} + A_{k-1} - A_{k}  + 2)  \\
              \vdots &  & \vdots \\
              \zeta A_{k} & \cdots &  \zeta (A_{k-1} + 1)
             \end{pmatrix}}_{V_{+}}  
  \rtimes \r^{\Sigma_{0}}_{M, >'_{\q}}\Big(\q_{-}, \ul_{-}, \ueta_{-}; \\
& (\rho, A_{k-1} - 1, B_{k} + A_{k-1} - A_{k} + 1, l'_{k} - 1, \eta'_{k}, \zeta), (\rho, A_{k-1} - 1, B_{k-1} + 1, l'_{k-1} - 1, \eta'_{k-1}, \zeta)\Big)  \\
& \hookrightarrow 
             \underbrace{\langle \zeta B_{k}, \cdots, -\zeta A_{k-1} \rangle}_{III_{1}} 
   \times \underbrace{\langle \zeta B_{k-1}, \cdots, -\zeta A_{k-1} \rangle}_{II_{1}} \\
& \times \underbrace{\begin{pmatrix}
              \zeta (B_{k} + 1) & \cdots & \zeta (B_{k} + A_{k-1} - A_{k}  + 2)  \\
              \vdots &  & \vdots \\
              \zeta (A_{k}-1) & \cdots &  \zeta A_{k-1}
             \end{pmatrix}}_{V}  
   \times \underbrace{\langle \zeta A_{k}, \cdots, \zeta (A_{k-1} + 1) \rangle}_{(III_{2})^{\vee}} \\
& \rtimes \r^{\Sigma_{0}}_{M, >'_{\q}}\Big(\q_{-}, \ul_{-}, \ueta_{-}; (\rho, A_{k-1} - 1, B_{k} + A_{k-1} - A_{k} + 1, l'_{k} - 1, \eta'_{k}, \zeta), \\
& (\rho, A_{k-1} - 1, B_{k-1} + 1, l'_{k-1} - 1, \eta'_{k-1}, \zeta)\Big). 
\end{align*}
Then we take dual of $(III_{2})^{\vee}$, and interchange $(III_{2})$ with $(V)$. 
\begin{align*}
\r^{\Sigma_{0}}_{M, >'_{\q}}(\q, \ul', \ueta') 
& \hookrightarrow 
             \underbrace{\langle \zeta B_{k}, \cdots, -\zeta A_{k-1} \rangle}_{III_{1}} 
   \times \underbrace{\langle \zeta B_{k-1}, \cdots, -\zeta A_{k-1} \rangle}_{II_{1}} \\
& \underbrace{\langle -\zeta (A_{k-1} + 1), \cdots, -\zeta A_{k} \rangle}_{III_{2}}
    \times \underbrace{\begin{pmatrix}
              \zeta (B_{k} + 1) & \cdots & \zeta (B_{k} + A_{k-1} - A_{k}  + 2)  \\
              \vdots &  & \vdots \\
              \zeta (A_{k}-1) & \cdots &  \zeta A_{k-1}
             \end{pmatrix}}_{V}  \\
& \rtimes \r^{\Sigma_{0}}_{M, >'_{\q}}\Big(\q_{-}, \ul_{-}, \ueta_{-}; (\rho, A_{k-1} - 1, B_{k} + A_{k-1} - A_{k} + 1, l'_{k} - 1, \eta'_{k}, \zeta), \\
& (\rho, A_{k-1} - 1, B_{k-1} + 1, l'_{k-1} - 1, \eta'_{k-1}, \zeta)\Big). 
\end{align*}
Suppose
\begin{align*}
\r^{\Sigma_{0}}_{M, >'_{\q}}(\q, \ul', \ueta') 
& \hookrightarrow 
             \underbrace{\langle \zeta B_{k}, \cdots, -\zeta A_{k-1} \rangle}_{III_{1}} 
   \times \underbrace{\langle \zeta B_{k-1}, \cdots, -\zeta A_{k} \rangle}_{II_{1}+III_{2}} \\
& \times \underbrace{\begin{pmatrix}
              \zeta (B_{k} + 1) & \cdots & \zeta (B_{k} + A_{k-1} - A_{k}  + 2)  \\
              \vdots &  & \vdots \\
              \zeta (A_{k}-1) & \cdots &  \zeta A_{k-1}
             \end{pmatrix}}_{V} 
   \rtimes \r^{\Sigma_{0}}_{M, >'_{\q}}\Big(\q_{-}, \ul_{-}, \ueta_{-}; \\
& (\rho, A_{k-1} - 1, B_{k} + A_{k-1} - A_{k} + 1, l'_{k} - 1, \eta'_{k}, \zeta), (\rho, A_{k-1} - 1, B_{k-1} + 1, l'_{k-1} - 1, \eta'_{k-1}, \zeta)\Big). 
\end{align*}
Since we can interchange $(III_{1})$ with $(II_{1}+III_{2})$, and $B_{k} = B_{k-1}$, we have 
\begin{align*}
\sigma^{**} 
& \hookrightarrow 
            \underbrace{\begin{pmatrix}
              \zeta (B_{k} + 1) & \cdots & \zeta (B_{k} + A_{k-1} - A_{k}  + 2)  \\
              \vdots &  & \vdots \\
              \zeta (A_{k}-1) & \cdots &  \zeta A_{k-1}
             \end{pmatrix}}_{V} 
   \rtimes \r^{\Sigma_{0}}_{M, >'_{\q}}\Big(\q_{-}, \ul_{-}, \ueta_{-}; \\
& (\rho, A_{k-1} - 1, B_{k} + A_{k-1} - A_{k} + 1, l'_{k} - 1, \eta'_{k}, \zeta), (\rho, A_{k-1} - 1, B_{k-1} + 1, l'_{k-1} - 1, \eta'_{k-1}, \zeta)\Big). 
\end{align*}
Otherwise, we would have
\begin{align*}
\r^{\Sigma_{0}}_{M, >'_{\q}}(\q, \ul', \ueta') 
& \hookrightarrow 
             \underbrace{\langle \zeta B_{k}, \cdots, -\zeta A_{k-1} \rangle}_{III_{1}} 
   \times \underbrace{\langle -\zeta (A_{k-1} + 1), \cdots, -\zeta A_{k} \rangle}_{III_{2}} \\
& \underbrace{\langle \zeta B_{k-1}, \cdots, -\zeta A_{k-1} \rangle}_{II_{1}}
    \times \underbrace{\begin{pmatrix}
              \zeta (B_{k} + 1) & \cdots & \zeta (B_{k} + A_{k-1} - A_{k}  + 2)  \\
              \vdots &  & \vdots \\
              \zeta (A_{k}-1) & \cdots &  \zeta A_{k-1}
             \end{pmatrix}}_{V}  \\
& \rtimes \r^{\Sigma_{0}}_{M, >'_{\q}}\Big(\q_{-}, \ul_{-}, \ueta_{-}; (\rho, A_{k-1} - 1, B_{k} + A_{k-1} - A_{k} + 1, l'_{k} - 1, \eta'_{k}, \zeta), \\
& (\rho, A_{k-1} - 1, B_{k-1} + 1, l'_{k-1} - 1, \eta'_{k-1}, \zeta)\Big). 
\end{align*}
Then we again have 
\begin{align*}
\sigma^{**} 
& \hookrightarrow 
            \underbrace{\begin{pmatrix}
              \zeta (B_{k} + 1) & \cdots & \zeta (B_{k} + A_{k-1} - A_{k}  + 2)  \\
              \vdots &  & \vdots \\
              \zeta (A_{k}-1) & \cdots &  \zeta A_{k-1}
             \end{pmatrix}}_{V} 
   \rtimes \r^{\Sigma_{0}}_{M, >'_{\q}}\Big(\q_{-}, \ul_{-}, \ueta_{-}; \\
& (\rho, A_{k-1} - 1, B_{k} + A_{k-1} - A_{k} + 1, l'_{k} - 1, \eta'_{k}, \zeta), (\rho, A_{k-1} - 1, B_{k-1} + 1, l'_{k-1} - 1, \eta'_{k-1}, \zeta)\Big). 
\end{align*}
Note the full induced representation above has a unique irreducible subrepresentation: 
\[
\r^{\Sigma_{0}}_{M, >'_{\q}}\Big(\q_{-}, \ul_{-}, \ueta_{-}; (\rho, A_{k} - 1, B_{k} + 1, l'_{k} - 1, \eta'_{k}, \zeta), (\rho, A_{k-1} - 1, B_{k-1} + 1, l'_{k-1} - 1, \eta'_{k-1}, \zeta)\Big).
\]
So it must be equal to $\sigma^{**}$. To summarize, if $\r^{\Sigma_{0}}_{M, >_{\q}}(\q, \ul, \ueta) = \r^{\Sigma_{0}}_{M, >'_{\q}}(\q, \ul', \ueta') \neq 0$ for $l_{k-1} \neq 0$, then we have $l'_{k-1} \neq 0$ by the previous step. After applying 
\[
\Jac_{\zeta B_{k-1}, \cdots, -\zeta A_{k-1}} \circ \Jac_{\zeta B_{k}, \cdots, -\zeta A_{k}}
\]
to both sides we get 
\begin{align*}
& \r^{\Sigma_{0}}_{M, >_{\q}}\Big(\q_{-}, \ul_{-}, \ueta_{-}; (\rho, A_{k} - 1, B_{k} + 1, l_{k} - 1, \eta_{k}, \zeta), (\rho, A_{k-1} - 1, B_{k-1} + 1, l_{k-1} - 1, \eta_{k-1}, \zeta)\Big) \\
= \, & \r^{\Sigma_{0}}_{M, >'_{\q}}\Big(\q_{-}, \ul_{-}, \ueta_{-}; (\rho, A_{k} - 1, B_{k} + 1, l'_{k} - 1, \eta'_{k}, \zeta), (\rho, A_{k-1} - 1, B_{k-1} + 1, l'_{k-1} - 1, \eta'_{k-1}, \zeta)\Big). 
\end{align*}
By induction on $l_{k-1}$, we can assume $(l_{k} - 1, \eta_{k}; l_{k-1} - 1, \eta_{k-1})$ is related to $(l'_{k} - 1, \eta'_{k}; l'_{k-1} - 1, \eta'_{k-1})$ according to our formula with respect to $\q^{**}$. Then it is easy to deduce that $(l_{k}, \eta_{k}; l_{k-1}, \eta_{k-1})$ and $(l'_{k}, \eta'_{k}; l'_{k-1}, \eta'_{k-1})$ are also related according to our formula with respect to $\q$. Hence $(\ul', \ueta') \sim_{\Sigma_{0}} S^{+}(\ul, \ueta)$.

\subsection{Case $\zeta_{k} \neq \zeta_{k-1}$}

In this case, there is no extra conditions on $[A_{k}, B_{k}], [A_{k-1}, B_{k-1}]$. For functions $\ul(\rho, A, B, \zeta) \in [0, [(A-B+1)/2]]$ and $\ueta(\rho, A, B, \zeta) \in \mathbb{Z}/2\mathbb{Z}$ on $Jord(\q)$, we denote 
\[
l_{k} = \ul(\rho, A_{k}, B_{k}, \zeta_{k}), \quad l_{k-1} = \ul(\rho, A_{k-1}, B_{k-1}, \zeta_{k-1}),
\] 
and 
\[
\eta_{k} = \ueta(\rho, A_{k}, B_{k}, \zeta_{k}), \quad \eta_{k-1} = \ueta(\rho, A_{k-1}, B_{k-1}, \zeta_{k-1}).
\] 
From $(\ul, \ueta)$, we want to construct another pair $(\ul', \ueta')$ such that 
\[
\ul'(\cdot) = \ul(\cdot) \text{ and } \ueta'(\cdot) = \ueta(\cdot)
\] 
over $Jord(\q) \backslash \{(\rho, A_{k}, B_{k}, \zeta_{k}), (\rho, A_{k-1}, B_{k-1}, \zeta_{k-1})\}$. Let us denote 
\[
l'_{k} = \ul'(\rho, A_{k}, B_{k}, \zeta_{k}), \quad l'_{k-1} = \ul(\rho, A_{k-1}, B_{k-1}, \zeta_{k-1}),
\] 
and 
\[
\eta'_{k} = \ueta'(\rho, A_{k}, B_{k}, \zeta_{k}), \quad \eta'_{k-1} = \ueta'(\rho, A_{k-1}, B_{k-1}, \zeta_{k-1}).
\] 
Then we define $l'_{k}, l'_{k-1}, \eta'_{k}, \eta'_{k-1}$ according to the following formulas.
\[
\begin{cases}
l'_{k} = l_{k} \\
l'_{k-1} = l_{k-1} \\
\eta_{k} = (-1)^{A_{k-1} - B_{k-1} + 1} \eta'_{k} \\
\eta_{k-1} = (-1)^{A_{k} - B_{k} + 1} \eta'_{k-1}
\end{cases}
\]
We denote this transformation by $U$. Since the situation is symmetric here, we have $U \circ U = id$.

\begin{theorem}
\label{thm: change order unequal sign}
Suppose $(\ul', \ueta') = U (\ul, \ueta)$, then
\[
\r^{\Sigma_{0}}_{M, >_{\q}}(\q, \ul, \ueta) = \r^{\Sigma_{0}}_{M, >'_{\q}}(\q, \ul', \ueta').
\]

\end{theorem}

Let $\q_{\gg}$ be a dominating parameter of $\q$ such that $Jord_{\rho}(\q_{\gg}) = Jord_{\rho}(\q)$, and $Jord_{\rho'}(\q_{\gg})$ has discrete diagonal restriction for $\rho' \neq \rho$. Then
\[
\r^{\Sigma_{0}}_{M, >_{\q}}(\q, \ul, \ueta) = \Jac_{X^{c}} \r^{\Sigma_{0}}_{M, >_{\q}}(\q_{\gg}, \ul, \ueta), 
\]
and
\[
\r^{\Sigma_{0}}_{M, >'_{\q}}(\q, \ul', \ueta') = \Jac_{X^{c}} \r^{\Sigma_{0}}_{M, >'_{\q}}(\q_{\gg}, \ul', \ueta').
\]
So it suffices to prove the proposition for such $\q_{\gg}$. Therefore, in the following discussions of the proof of this proposition, we will always assume $Jord_{\rho'}(\q)$ has discrete diagonal restriction for $\rho' \neq \rho$, and if we choose some dominating $\q_{\gg}$ of $\q$, we will always assume $Jord_{\rho'}(\q_{\gg}) = Jord_{\rho'}(\q)$ for $\rho' \neq \rho$.

\subsubsection{First reduction}

Let $(\ul', \ueta') = U (\ul, \ueta)$. We want to reduce the proposition to the following cases:
\begin{align}
\label{eq: first reduction change order unequal sign}
\text{ $(\rho, A_{i}, B_{i}, \zeta_{i}) \gg_{r} (\rho, A_{i-1}, B_{i-1}, \zeta_{i-1})$ for all $i$, and $(\rho, A_{k-1}, B_{k-1}, \zeta_{k-1}) \gg_{r} 0$. }
\end{align}
We denote the case with respect to $r$ by \eqref{eq: first reduction change order unequal sign}$_{r}$. We will do this in two steps. First we will reduce it to the cases:
\begin{align}
\label{eq: first reduction change order unequal sign 1}
\begin{cases}
\text{ $(\rho, A_{i}, B_{i}, \zeta_{i}) \gg_{r} (\rho, A_{i-1}, B_{i-1}, \zeta_{i-1})$ for $i > k - 1$, } \\ \\
\text{ $(\rho, A_{k-1}, B_{k-1}, \zeta_{k-1}) \gg_{r} \cup_{j=1}^{k-2}\{(\rho, A_{j}, B_{j}, \zeta_{j})\}$ and $0$. }
\end{cases}
\end{align}
We denote the case with respect to $r$ by \eqref{eq: first reduction change order unequal sign 1}$_{r}$. Let us choose a dominating parameter $\q_{\gg}$ with respect to $>_{\q}$ such that $T_{i} =  0$ for $i < k-1$,  
\[
\text{ $(\rho, A_{i} + T_{i}, B_{i} + T_{i}, \zeta_{i}) \gg_{r} (\rho, A_{i-1} + T_{i-1}, B_{i-1} + T_{i-1}, \zeta_{i-1})$ for $i \geqslant k$. }
\]
We further require the existence of $T$ such that $0 \leqslant T < T_{k}$,
\[
(\rho, A_{k-1} + T_{k-1}, B_{k-1} + T_{k-1}, \zeta_{k-1}) \gg_{r} (\rho, A_{k} + T_{k} - T, B_{k} + T_{k} - T, \zeta_{k}) \gg_{r} \cup_{j=1}^{k-2}\{(\rho, A_{j}, B_{j}, \zeta_{j})\} \text{ and } 0.
\]
Let $\q^{(k)}_{\gg}$ be obtained from $\q_{\gg}$ by changing $T_{k}, T_{k-1}$ to zero.
Suppose $\r^{\Sigma_{0}}_{M, >_{\q}}(\q, \ul, \ueta) \neq 0$, then
\begin{align*}
\r^{\Sigma_{0}}_{M, >_{\q}}(\q_{\gg}, \ul, \ueta) & \hookrightarrow 
       \begin{pmatrix}
              \zeta_{k-1} (B_{k-1} + T_{k-1}) & \cdots & \zeta_{k-1}(B_{k-1} + 1) \\
              \vdots &  & \vdots \\
              \zeta_{k-1} (A_{k-1} + T_{k-1}) & \cdots & \zeta_{k-1}(A_{k-1} + 1)
       \end{pmatrix} \times
       \begin{pmatrix}
              \zeta_{k} (B_{k} + T_{k}) & \cdots & \zeta_{k}(B_{k} + 1) \\
              \vdots &  & \vdots \\
              \zeta_{k} (A_{k} + T_{k}) & \cdots & \zeta_{k}(A_{k} + 1)
       \end{pmatrix} \\       
& \rtimes  \r^{\Sigma_{0}}_{M, >_{\q}}(\q^{(k)}_{\gg}, \ul, \ueta),
\end{align*}   
where the two generalized segments are interchangeable for $\zeta_{k} \neq \zeta_{k-1}$. Let $\q^{T}_{\gg}$ be obtained from $\q_{\gg}$ by changing $(\rho, A_{k} + T_{k}, B_{k} + T_{k}, \zeta_{k})$ to $(\rho, A_{k} + T_{k} - T, B_{k} + T_{k} - T, \zeta_{k})$. Then
\begin{align*}
\r^{\Sigma_{0}}_{M, >_{\q}}(\q^{T}_{\gg}, \ul, \ueta) & \hookrightarrow 
       \begin{pmatrix}
              \zeta_{k} (B_{k} + T_{k} -T) & \cdots & \zeta_{k}(B_{k} + 1) \\
              \vdots &  & \vdots \\
              \zeta_{k} (A_{k} + T_{k} - T) & \cdots & \zeta_{k}(A_{k} + 1)
       \end{pmatrix} \times
       \begin{pmatrix}
              \zeta_{k-1} (B_{k-1} + T_{k-1}) & \cdots & \zeta_{k-1}(B_{k-1} + 1) \\
              \vdots &  & \vdots \\
              \zeta_{k-1} (A_{k-1} + T_{k-1}) & \cdots & \zeta_{k-1}(A_{k-1} + 1)
       \end{pmatrix} \\
& \rtimes  \r^{\Sigma_{0}}_{M, >_{\q}}(\q^{(k)}_{\gg}, \ul, \ueta),
\end{align*}   
By \eqref{eq: first reduction change order unequal sign 1}$_{r}$, we have
\[
\r^{\Sigma_{0}}_{M, >_{\q}}(\q^{T}_{\gg}, \ul, \ueta) = \r^{\Sigma_{0}}_{M, >'_{\q}}(\q^{T}_{\gg}, \ul', \ueta').
\]
Since
\begin{align*}
\r^{\Sigma_{0}}_{M, >'_{\q}}(\q^{(k)}_{\gg}, \ul', \ueta')  = & \Jac_{(\rho, A_{k-1} + T_{k-1}, B_{k-1} + T_{k-1}, \zeta_{k-1}) \mapsto (\rho, A_{k-1}, B_{k-1}, \zeta_{k-1})} \circ \\
& \Jac_{(\rho, A_{k} + T_{k} - T, B_{k} + T_{k} - T, \zeta_{k}) \mapsto (\rho, A_{k}, B_{k}, \zeta_{k})} \r^{\Sigma_{0}}_{M, >'_{\q}}(\q^{T}_{\gg}, \ul', \ueta') 
\end{align*}
and $\r^{\Sigma_{0}}_{M, >_{\q}}(\q^{(k)}_{\gg}, \ul, \ueta)$ is contained in
\begin{align*}
& \Jac_{(\rho, A_{k-1} + T_{k-1}, B_{k-1} + T_{k-1}, \zeta_{k-1}) \mapsto (\rho, A_{k-1}, B_{k-1}, \zeta_{k-1})} \circ \\
& \Jac_{(\rho, A_{k} + T_{k} - T, B_{k} + T_{k} - T, \zeta_{k}) \mapsto (\rho, A_{k}, B_{k}, \zeta_{k})} \r^{\Sigma_{0}}_{M, >_{\q}}(\q^{T}_{\gg}, \ul, \ueta), 
\end{align*}
then
\[
\r^{\Sigma_{0}}_{M, >_{\q}}(\q^{(k)}_{\gg}, \ul, \ueta) = \r^{\Sigma_{0}}_{M, >'_{\q}}(\q^{(k)}_{\gg}, \ul', \ueta').
\]
After applying $\circ_{i>k} \Jac_{(\rho, A_{i} + T_{i}, B_{i} + T_{i}, \zeta_{i}) \mapsto (\rho, A_{i}, B_{i}, \zeta_{i})}$ to both sides, we get $\r^{\Sigma_{0}}_{M, >_{\q}}(\q, \ul, \ueta) = \r^{\Sigma_{0}}_{M, >'_{\q}}(\q, \ul', \ueta')$. 

Secondly we want to further reduce it to \eqref{eq: first reduction change order unequal sign}$_{r}$. Let us assume we are in case \eqref{eq: first reduction change order unequal sign 1}$_{r'}$ for $r'$ sufficiently large with respect to $r$. We can choose a dominating parameter $\q_{\gg}$ with respect to $>_{\q}$ such that $T_{i} =  0$ for $i > k$, and
\[
(\rho, A_{k+1}, B_{k+1}, \zeta_{k+1}) \gg_{r} (\rho, A_{i} + T_{i}, B_{i} + T_{i}, \zeta_{i}) \gg_{r} (\rho, A_{i-1} + T_{i-1}, B_{i-1} + T_{i-1}, \zeta_{i-1}) \text{ for $i \leqslant k$. }
\]
Suppose $\r^{\Sigma_{0}}_{M, >_{\q}}(\q, \ul, \ueta) \neq 0$, then 
\begin{align*}
\r^{\Sigma_{0}}_{M, >_{\q}}(\q_{\gg}, \ul, \ueta) & \hookrightarrow \times_{i<k-1}
       \begin{pmatrix}
              \zeta_{i} (B_{i} + T_{i}) & \cdots & \zeta_{i}(B_{i} + 1) \\
              \vdots &  & \vdots \\
              \zeta_{i} (A_{i} + T_{i}) & \cdots & \zeta_{i}(A_{i} + 1)
       \end{pmatrix} \times 
       \begin{pmatrix}
              \zeta_{k-1} (B_{k-1} + T_{k-1}) & \cdots & \zeta_{k-1}(B_{k-1} + 1) \\
              \vdots &  & \vdots \\
              \zeta_{k-1} (A_{k-1} + T_{k-1}) & \cdots & \zeta_{k-1}(A_{k-1} + 1)
       \end{pmatrix} \\
&\times
       \begin{pmatrix}
              \zeta_{k} (B_{k} + T_{k}) & \cdots & \zeta_{k}(B_{k} + 1) \\
              \vdots &  & \vdots \\
              \zeta_{k} (A_{k} + T_{k}) & \cdots & \zeta_{k}(A_{k} + 1)
       \end{pmatrix} 
 \rtimes  \r^{\Sigma_{0}}_{M, >_{\q}}(\q, \ul, \ueta),
\end{align*}   
where $i$ increases. We can also assume $B_{k-1} + 1 > A_{k-2} + T_{k-2} + 1$. Then we can change the order of the generalized segments as follows,
\begin{align*}
\r^{\Sigma_{0}}_{M, >_{\q}}(\q_{\gg}, \ul, \ueta) & \hookrightarrow        
       \begin{pmatrix}
              \zeta_{k-1} (B_{k-1} + T_{k-1}) & \cdots & \zeta_{k-1}(B_{k-1} + 1) \\
              \vdots &  & \vdots \\
              \zeta_{k-1} (A_{k-1} + T_{k-1}) & \cdots & \zeta_{k-1}(A_{k-1} + 1)
       \end{pmatrix} \times
       \begin{pmatrix}
              \zeta_{k} (B_{k} + T_{k}) & \cdots & \zeta_{k}(B_{k} + 1) \\
              \vdots &  & \vdots \\
              \zeta_{k} (A_{k} + T_{k}) & \cdots & \zeta_{k}(A_{k} + 1)
       \end{pmatrix} \\
&\times_{i<k-1}
       \begin{pmatrix}
              \zeta_{i} (B_{i} + T_{i}) & \cdots & \zeta_{i}(B_{i} + 1) \\
              \vdots &  & \vdots \\
              \zeta_{i} (A_{i} + T_{i}) & \cdots & \zeta_{i}(A_{i} + 1)
       \end{pmatrix} 
\rtimes  \r^{\Sigma_{0}}_{M, >_{\q}}(\q, \ul, \ueta) \\
& \cong         \begin{pmatrix}
              \zeta_{k} (B_{k} + T_{k}) & \cdots & \zeta_{k}(B_{k} + 1) \\
              \vdots &  & \vdots \\
              \zeta_{k} (A_{k} + T_{k}) & \cdots & \zeta_{k}(A_{k} + 1)
       \end{pmatrix} \times
       \begin{pmatrix}
              \zeta_{k-1} (B_{k-1} + T_{k-1}) & \cdots & \zeta_{k-1}(B_{k-1} + 1) \\
              \vdots &  & \vdots \\
              \zeta_{k-1} (A_{k-1} + T_{k-1}) & \cdots & \zeta_{k-1}(A_{k-1} + 1)
       \end{pmatrix} \\
&\times_{i<k-1}
       \begin{pmatrix}
              \zeta_{i} (B_{i} + T_{i}) & \cdots & \zeta_{i}(B_{i} + 1) \\
              \vdots &  & \vdots \\
              \zeta_{i} (A_{i} + T_{i}) & \cdots & \zeta_{i}(A_{i} + 1)
       \end{pmatrix} 
\rtimes  \r^{\Sigma_{0}}_{M, >_{\q}}(\q, \ul, \ueta).
\end{align*}  
We can choose $0 \leqslant T < T_{k}$ such that
\[
(\rho, A_{k-1} + T_{k-1}, B_{k-1} + T_{k-1}, \zeta_{k-1}) \gg_{r} (\rho, A_{k} + T_{k} - T, B_{k} + T_{k} - T, \zeta_{k}) \gg_{r}  (\rho, A_{k-2} + T_{k-2} , B_{k-2} + T_{k-2}, \zeta_{k-2}) \text{ and } 0.
\]
Let $\q^{T}_{\gg}$ be obtained from $\q_{\gg}$ by changing $(\rho, A_{k} + T_{k}, B_{k} + T_{k}, \zeta_{k})$ to $(\rho, A_{k} + T_{k} - T, B_{k} + T_{k} - T, \zeta_{k})$.
Then 
\begin{align*}
\r^{\Sigma_{0}}_{M, >_{\q}}(\q^{T}_{\gg}, \ul, \ueta) & \hookrightarrow 
       \begin{pmatrix}
              \zeta_{k} (B_{k} + T_{k} -T) & \cdots & \zeta_{k}(B_{k} + 1) \\
              \vdots &  & \vdots \\
              \zeta_{k} (A_{k} + T_{k} - T) & \cdots & \zeta_{k}(A_{k} + 1)
       \end{pmatrix} \times
       \begin{pmatrix}
              \zeta_{k-1} (B_{k-1} + T_{k-1}) & \cdots & \zeta_{k-1}(B_{k-1} + 1) \\
              \vdots &  & \vdots \\
              \zeta_{k-1} (A_{k-1} + T_{k-1}) & \cdots & \zeta_{k-1}(A_{k-1} + 1)
       \end{pmatrix} \\
&\times_{i<k-1}
       \begin{pmatrix}
              \zeta_{i} (B_{i} + T_{i}) & \cdots & \zeta_{i}(B_{i} + 1) \\
              \vdots &  & \vdots \\
              \zeta_{i} (A_{i} + T_{i}) & \cdots & \zeta_{i}(A_{i} + 1)
       \end{pmatrix}  
 \rtimes  \r^{\Sigma_{0}}_{M, >_{\q}}(\q, \ul, \ueta) \\
& \cong \times_{i<k-1}
       \begin{pmatrix}
              \zeta_{i} (B_{i} + T_{i}) & \cdots & \zeta_{i}(B_{i} + 1) \\
              \vdots &  & \vdots \\
              \zeta_{i} (A_{i} + T_{i}) & \cdots & \zeta_{i}(A_{i} + 1)
       \end{pmatrix} \times 
       \begin{pmatrix}
              \zeta_{k} (B_{k} + T_{k} -T) & \cdots & \zeta_{k}(B_{k} + 1) \\
              \vdots &  & \vdots \\
              \zeta_{k} (A_{k} + T_{k} - T) & \cdots & \zeta_{k}(A_{k} + 1)
       \end{pmatrix} \\
& \times \begin{pmatrix}
              \zeta_{k-1} (B_{k-1} + T_{k-1}) & \cdots & \zeta_{k-1}(B_{k-1} + 1) \\
              \vdots &  & \vdots \\
              \zeta_{k-1} (A_{k-1} + T_{k-1}) & \cdots & \zeta_{k-1}(A_{k-1} + 1)
       \end{pmatrix} 
 \rtimes  \r^{\Sigma_{0}}_{M, >_{\q}}(\q, \ul, \ueta).
\end{align*}   
By \eqref{eq: first reduction change order unequal sign}$_{r}$, we have
\[
\r^{\Sigma_{0}}_{M, >_{\q}}(\q^{T}_{\gg}, \ul, \ueta) = \r^{\Sigma_{0}}_{M, >'_{\q}}(\q^{T}_{\gg}, \ul', \ueta').
\]
Since
\begin{align*}
\r^{\Sigma_{0}}_{M, >'_{\q}}(\q, \ul', \ueta')  = & \Jac_{(\rho, A_{k-1} + T_{k-1}, B_{k-1} + T_{k-1}, \zeta_{k-1}) \mapsto (\rho, A_{k-1}, B_{k-1}, \zeta_{k-1})} \circ \\
& \Jac_{(\rho, A_{k} + T_{k} - T, B_{k} + T_{k} - T, \zeta_{k}) \mapsto (\rho, A_{k}, B_{k}, \zeta_{k})} \circ \\
& \circ_{i<k-1}\Jac_{(\rho, A_{i} + T_{i}, B_{i} + T_{i}, \zeta_{i}) \mapsto (\rho, A_{i}, B_{i}, \zeta_{i})} \r^{\Sigma_{0}}_{M, >'_{\q}}(\q^{T}_{\gg}, \ul', \ueta')
\end{align*}
and $\r^{\Sigma_{0}}_{M, >_{\q}}(\q, \ul, \ueta)$ is contained in
\begin{align*}
& \Jac_{(\rho, A_{k-1} + T_{k-1}, B_{k-1} + T_{k-1}, \zeta_{k-1}) \mapsto (\rho, A_{k-1}, B_{k-1}, \zeta_{k-1})} \circ \\
& \Jac_{(\rho, A_{k} + T_{k} - T, B_{k} + T_{k} - T, \zeta_{k}) \mapsto (\rho, A_{k}, B_{k}, \zeta_{k})} \circ \\
& \circ_{i<k-1}\Jac_{(\rho, A_{i} + T_{i}, B_{i} + T_{i}, \zeta_{i}) \mapsto (\rho, A_{i}, B_{i}, \zeta_{i})} \r^{\Sigma_{0}}_{M, >_{\q}}(\q^{T}_{\gg}, \ul, \ueta), 
\end{align*}
then
\[
\r^{\Sigma_{0}}_{M, >_{\q}}(\q, \ul, \ueta) = \r^{\Sigma_{0}}_{M, >'_{\q}}(\q, \ul', \ueta').
\]
This finishes the first reduction.

\subsubsection{Second reduction}

We want to reduce the proposition further to the cases:
\begin{align}
\label{eq: second reduction change order unequal sign}
\text{ $(\rho, A_{i}, B_{i}, \zeta_{i}) \gg_{r} (\rho, A_{i-1}, B_{i-1}, \zeta_{i-1})$, and $l_{i} = 0$ for all $i$. } 
\end{align}
Let us denote the case with respect to $r$ by \eqref{eq: second reduction change order unequal sign}$_{r}$. Suppose we are in case \eqref{eq: first reduction change order unequal sign}$_{r'}$ for $r'$ sufficiently large with respect to $r$. Let $\q^{T}$ be obtained by changing $(\rho, A_{k-1}, B_{k-1}, \zeta_{k-1})$ to $(\rho, A_{k-1} + T, B_{k-1} + T, \zeta_{k-1})$ such that 
\[
B_{k-1} + T > A_{k} \text { and } B_{k+1} > A_{k-1} + T.
\]
Let
\[
Jord(\q_{-}) = \{(\rho, A_{i} - l_{i}, B_{i} + l_{i}, \zeta_{i}): i \neq k, k-1\}.
\]
We define $(\ul_{-}, \ueta_{-})$ such that $\ul_{-}(\rho, A_{i} - l_{i}, B_{i} + l_{i}, \zeta_{i}) = 0$ and $\ueta_{-}(\rho, A_{i} - l_{i}, B_{i} + l_{i}, \zeta_{i}) = \eta_{i}$. Then
\begin{align*}
\r^{\Sigma_{0}}_{M, >'_{\q}}(\q^{T}, \ul', \ueta') & \hookrightarrow 
       \begin{pmatrix}
              \zeta_{k-1} (B_{k-1} + T) & \cdots & -\zeta_{k-1}(A_{k-1} + T) \\
              \vdots &  & \vdots \\
              \zeta_{k-1} (B_{k-1} + l_{k-1} - 1 + T) & \cdots & -\zeta_{k-1}(A_{k-1} - l_{k-1} + 1 + T)
       \end{pmatrix} \\
& \times_{i \neq k-1}
       \begin{pmatrix}
              \zeta_{i} B_{i} & \cdots & -\zeta_{i}A_{i} \\
              \vdots &  & \vdots \\
              \zeta_{i} (B_{i} + l_{i} - 1) & \cdots & - \zeta_{i}(A_{i} - l_{i} + 1)
       \end{pmatrix} \\
& \rtimes  \r^{\Sigma_{0}}_{M, >'_{\q}}\Big(\q_{-}, \ul_{-}, \eta_{-}; (\rho, A_{k-1} - l_{k-1} + T, B_{k-1} + l_{k-1} + T, 0, \eta'_{k-1}, \zeta_{k-1}), \\
& (\rho, A_{k} - l_{k} , B_{k} + l_{k}, 0, \eta'_{k}, \zeta_{k})\Big).
\end{align*}   
We choose $t$ such that 
\[
(\rho, A_{k-1} - l_{k-1}, B_{k-1} + l_{k-1}, \zeta_{k-1}) \gg_{r} (\rho, A_{k} - l_{k} -t , B_{k} + l_{k} - t, \zeta_{k}) \gg_{r} (\rho, A_{k-2} - l_{k-2}, B_{k-2} + l_{k-2}, \zeta_{k-2}).
\]
Then
\begin{align*}
\r^{\Sigma_{0}}_{M, >'_{\q}}(\q^{T}, \ul', \ueta') & \hookrightarrow 
       \begin{pmatrix}
              \zeta_{k-1} (B_{k-1} + T) & \cdots  & -\zeta_{k-1}(A_{k-1} + T) \\
              \vdots &  & \vdots \\
              \zeta_{k-1} (B_{k-1} + l_{k-1} - 1 + T) & \cdots & -\zeta_{k-1}(A_{k-1} - l_{k-1} + 1 + T)
       \end{pmatrix} \\
& \times_{i \neq k-1}
      \underbrace{ \begin{pmatrix}
              \zeta_{i} B_{i} & \cdots & -\zeta_{i}A_{i} \\
              \vdots &  & \vdots \\
              \zeta_{i} (B_{i} + l_{i} - 1) & \cdots & - \zeta_{i}(A_{i} - l_{i} + 1)
       \end{pmatrix} }_{I_{i}}
 \times 
      \underbrace{\begin{pmatrix}
              \zeta_{k} (B_{k} + l_{k}) & \cdots & \zeta_{k}(B_{k} + l_{k} - t + 1) \\
              \vdots &  & \vdots \\
              \zeta_{k} (A_{k} - l_{k}) & \cdots & \zeta_{k}(A_{k} - l_{k} - t + 1)
       \end{pmatrix}}_{II}\\
&\times
       \underbrace{\begin{pmatrix}
              \zeta_{k-1} (B_{k-1} + l_{k-1} + T) & \cdots & \zeta_{k-1}(B_{k-1} + l_{k-1} + 1) \\
              \vdots &  & \vdots \\
              \zeta_{k-1} (A_{k-1} - l_{k-1} + T) & \cdots & \zeta_{k-1}(A_{k-1} - l_{k-1} + 1)
       \end{pmatrix}}_{III} \\
& \rtimes  \r^{\Sigma_{0}}_{M, >'_{\q}}\Big(\q_{-}, \ul_{-}, \eta_{-}; (\rho, A_{k-1} - l_{k-1}, B_{k-1} + l_{k-1}, 0, \eta'_{k-1}, \zeta_{k-1}), \\
& (\rho, A_{k} - l_{k} -t , B_{k} + l_{k} - t, 0, \eta'_{k}, \zeta_{k})\Big).
\end{align*} 
It is clear that the generalized segments $(III)$ and $(II)$ are interchangeable. We would like to show $(III)$ and $(I_{i})$ are also interchangeable for $i \neq k-1$. It suffices to make the following observations:

\begin{enumerate}

\item If $\zeta_{i} = \zeta_{k-1}$ 

        \begin{enumerate}
        \item $i > k$, one observes $B_{i} > B_{k-1} + l_{k-1} + T$ 
        \item $i < k-1$, one observes $B_{k-1} + l_{k-1} > B_{i} + l_{i}$
        \end{enumerate} 

\item If $\zeta_{i} \neq \zeta_{k-1}$

        \begin{enumerate}
        \item $i > k$, one observes $A_{i} > A_{k-1} - l_{k-1} + T$ 
        \item $i < k-1$, one observes $B_{k-1} + l_{k-1} > A_{i}$
        \item $i = k$, one observes $[A_{k}, A_{k} - l_{k} + 1] \subseteq [A_{k-1} + T - l_{k-1}, A_{k-1} - l_{k-1} + 1]$
        \end{enumerate}

\end{enumerate}
Therefore,
\begin{align*}
\r^{\Sigma_{0}}_{M, >'_{\q}}(\q^{T}, \ul', \ueta') & \hookrightarrow 
       \begin{pmatrix}
              \zeta_{k-1} (B_{k-1} + T) & \cdots & -\zeta_{k-1}(A_{k-1} + T) \\
              \vdots &  & \vdots \\
              \zeta_{k-1} (B_{k-1} + l_{k-1} - 1 + T) & \cdots & -\zeta_{k-1}(A_{k-1} - l_{k-1} + 1 + T)
       \end{pmatrix} \\
&\times
       \underbrace{\begin{pmatrix}
              \zeta_{k-1} (B_{k-1} + l_{k-1} + T) & \cdots & \zeta_{k-1}(B_{k-1} + l_{k-1} + 1) \\
              \vdots &  & \vdots \\
              \zeta_{k-1} (A_{k-1} - l_{k-1} + T) & \cdots & \zeta_{k-1}(A_{k-1} - l_{k-1} + 1)
       \end{pmatrix}}_{III} \\      
& \times_{i \neq k-1}
      \underbrace{ \begin{pmatrix}
              \zeta_{i} B_{i} & \cdots & -\zeta_{i}A_{i} \\
              \vdots &  & \vdots \\
              \zeta_{i} (B_{i} + l_{i} - 1) & \cdots & - \zeta_{i}(A_{i} - l_{i} + 1)
       \end{pmatrix} }_{I_{i}}
 \times 
      \underbrace{\begin{pmatrix}
              \zeta_{k} (B_{k} + l_{k}) & \cdots & \zeta_{k}(B_{k} + l_{k} - t + 1) \\
              \vdots &  & \vdots \\
              \zeta_{k} (A_{k} - l_{k}) & \cdots & \zeta_{k}(A_{k} - l_{k} - t + 1)
       \end{pmatrix}}_{II}\\
& \rtimes  \r^{\Sigma_{0}}_{M, >'_{\q}}\Big(\q_{-}, \ul_{-}, \eta_{-}; (\rho, A_{k-1} - l_{k-1}, B_{k-1} + l_{k-1}, 0, \eta'_{k-1}, \zeta_{k-1}), \\
& (\rho, A_{k} - l_{k} -t , B_{k} + l_{k} - t, 0, \eta'_{k}, \zeta_{k})\Big).
\end{align*} 
Next we want to take dual of
\begin{align*}
\underbrace{\begin{pmatrix}
               -\zeta_{k-1}(A_{k-1} + 1) & \cdots & -\zeta_{k-1}(A_{k-1} + T) \\
              \vdots &  & \vdots \\
              -\zeta_{k-1} (A_{k-1} - l_{k-1} +2) & \cdots & -\zeta_{k-1}(A_{k-1} - l_{k-1} + 1 + T)
       \end{pmatrix}}_{IV}
\end{align*}
from
\[
\begin{pmatrix}
              \zeta_{k-1} (B_{k-1} + T) & \cdots & -\zeta_{k-1}(A_{k-1} + T) \\
              \vdots &  & \vdots \\
              \zeta_{k-1} (B_{k-1} + l_{k-1} - 1 + T) & \cdots & -\zeta_{k-1}(A_{k-1} - l_{k-1} + 1 + T)
       \end{pmatrix}.
\]
It is clear that $(IV)$ and $(III)$ are interchangeable. To see $(IV)$ and $(II)$ are interchangeable, one notes $\zeta_{k} \neq \zeta_{k-1}$ and $[A_{k-1} + T, A_{k-1} + 1] \supseteq [A_{k} - l_{k}, B_{k} + l_{k}]$. To see $(IV)$ and $(I_{i})$ are also interchangeable, it suffices to make the following observations:
\begin{enumerate}

\item If $\zeta_{i} = \zeta_{k-1}$ 

        \begin{enumerate}
        \item $i > k$, one observes $A_{i} > A_{k-1} + T$ 
        \item $i < k-1$, one observes $A_{k-1} + l_{k-1} > A_{i}$
        \end{enumerate} 

\item If $\zeta_{i} \neq \zeta_{k-1}$

        \begin{enumerate}
        \item $i \geqslant k$, one observes $B_{i} > A_{k-1} + 1$ 
        \item $i < k-1$, one observes $A_{k-1} - l_{k-1} > B_{i} + l_{i}$
        \end{enumerate}

\end{enumerate}
As a result, 
\begin{align*}
\r^{\Sigma_{0}}_{M, >'_{\q}}(\q^{T}, \ul', \ueta') & \hookrightarrow 
       \begin{pmatrix}
              \zeta_{k-1} (B_{k-1} + T) & \cdots & -\zeta_{k-1}A_{k-1} \\
              \vdots &  & \vdots \\
              \zeta_{k-1} (B_{k-1} + l_{k-1} - 1 + T) & \cdots & -\zeta_{k-1}(A_{k-1} - l_{k-1} + 1)
       \end{pmatrix} \\
&\times
       \underbrace{\begin{pmatrix}
              \zeta_{k-1} (B_{k-1} + l_{k-1} + T) & \cdots & \zeta_{k-1}(B_{k-1} + l_{k-1} + 1) \\
              \vdots &  & \vdots \\
              \zeta_{k-1} (A_{k-1} - l_{k-1} + T) & \cdots & \zeta_{k-1}(A_{k-1} - l_{k-1} + 1)
       \end{pmatrix}}_{III} \\      
& \times_{i \neq k-1}
      \underbrace{ \begin{pmatrix}
              \zeta_{i} B_{i} & \cdots & -\zeta_{i}A_{i} \\
              \vdots &  & \vdots \\
              \zeta_{i} (B_{i} + l_{i} - 1) & \cdots & - \zeta_{i}(A_{i} - l_{i} + 1)
       \end{pmatrix} }_{I_{i}}
 \times 
      \underbrace{\begin{pmatrix}
              \zeta_{k} (B_{k} + l_{k}) & \cdots & \zeta_{k}(B_{k} + l_{k} - t + 1) \\
              \vdots &  & \vdots \\
              \zeta_{k} (A_{k} - l_{k}) & \cdots & \zeta_{k}(A_{k} - l_{k} - t + 1)
       \end{pmatrix}}_{II}\\
& \times 
\underbrace{\begin{pmatrix}
               -\zeta_{k-1}(A_{k-1} + 1) & \cdots & -\zeta_{k-1}(A_{k-1} + T) \\
              \vdots &  & \vdots \\
              -\zeta_{k-1} (A_{k-1} - l_{k-1} +2) & \cdots & -\zeta_{k-1}(A_{k-1} - l_{k-1} + 1 + T)
       \end{pmatrix}}_{IV}\\
& \rtimes  \r^{\Sigma_{0}}_{M, >'_{\q}}\Big(\q_{-}, \ul_{-}, \eta_{-}; (\rho, A_{k-1} - l_{k-1}, B_{k-1} + l_{k-1}, 0, \eta'_{k-1}, \zeta_{k-1}), \\
& (\rho, A_{k} - l_{k} -t , B_{k} + l_{k} - t, 0, \eta'_{k}, \zeta_{k})\Big).
\end{align*} 
By \eqref{eq: dualizing classical}, we can take the dual of $(IV)$. Therefore
\begin{align*}
\r^{\Sigma_{0}}_{M, >'_{\q}}(\q^{T}, \ul', \ueta') & \hookrightarrow 
       \begin{pmatrix}
              \zeta_{k-1} (B_{k-1} + T) & \cdots & -\zeta_{k-1}A_{k-1} \\
              \vdots &  & \vdots \\
              \zeta_{k-1} (B_{k-1} + l_{k-1} - 1 + T) & \cdots & -\zeta_{k-1}(A_{k-1} - l_{k-1} + 1)
       \end{pmatrix} \\
&\times
       \underbrace{\begin{pmatrix}
              \zeta_{k-1} (B_{k-1} + l_{k-1} + T) & \cdots & \zeta_{k-1}(B_{k-1} + l_{k-1} + 1) \\
              \vdots &  & \vdots \\
              \zeta_{k-1} (A_{k-1} - l_{k-1} + T) & \cdots & \zeta_{k-1}(A_{k-1} - l_{k-1} + 1)
       \end{pmatrix}}_{III} \\      
& \times_{i \neq k-1}
      \underbrace{ \begin{pmatrix}
              \zeta_{i} B_{i} & \cdots & -\zeta_{i}A_{i} \\
              \vdots &  & \vdots \\
              \zeta_{i} (B_{i} + l_{i} - 1) & \cdots & - \zeta_{i}(A_{i} - l_{i} + 1)
       \end{pmatrix} }_{I_{i}}
 \times 
      \underbrace{\begin{pmatrix}
              \zeta_{k} (B_{k} + l_{k}) & \cdots & \zeta_{k}(B_{k} + l_{k} - t + 1) \\
              \vdots &  & \vdots \\
              \zeta_{k} (A_{k} - l_{k}) & \cdots & \zeta_{k}(A_{k} - l_{k} - t + 1)
       \end{pmatrix}}_{II}\\
& \times 
\underbrace{\begin{pmatrix}
              \zeta_{k-1}(A_{k-1} - l_{k-1} + 1 + T) & \cdots & \zeta_{k-1} (A_{k-1} - l_{k-1} +2) \\
              \vdots &  & \vdots \\
              \zeta_{k-1}(A_{k-1} + T) & \cdots & \zeta_{k-1}(A_{k-1} + 1)
       \end{pmatrix}}_{(IV)^{\vee}}\\
& \rtimes  \r^{\Sigma_{0}}_{M, >'_{\q}}\Big(\q_{-}, \ul_{-}, \eta_{-}; (\rho, A_{k-1} - l_{k-1}, B_{k-1} + l_{k-1}, 0, \eta'_{k-1}, \zeta_{k-1}), \\
& (\rho, A_{k} - l_{k} -t , B_{k} + l_{k} - t, 0, \eta'_{k}, \zeta_{k})\Big).
\end{align*} 
As before, one can show $(IV)^{\vee}$ are interchangeable with $(II)$ and $(I_{i})$. Then
\begin{align*}
\r^{\Sigma_{0}}_{M, >'_{\q}}(\q^{T}, \ul', \ueta') & \hookrightarrow 
       \begin{pmatrix}
              \zeta_{k-1} (B_{k-1} + T) & \cdots & -\zeta_{k-1}A_{k-1} \\
              \vdots &  & \vdots \\
              \zeta_{k-1} (B_{k-1} + l_{k-1} - 1 + T) & \cdots & -\zeta_{k-1}(A_{k-1} - l_{k-1} + 1)
       \end{pmatrix} \\
&\times
       \underbrace{\begin{pmatrix}
              \zeta_{k-1} (B_{k-1} + l_{k-1} + T) & \cdots & \zeta_{k-1}(B_{k-1} + l_{k-1} + 1) \\
              \vdots &  & \vdots \\
              \zeta_{k-1} (A_{k-1} - l_{k-1} + T) & \cdots & \zeta_{k-1}(A_{k-1} - l_{k-1} + 1)
       \end{pmatrix}}_{III} \\     
& \times 
\underbrace{\begin{pmatrix}
              \zeta_{k-1}(A_{k-1} - l_{k-1} + 1 + T) & \cdots & \zeta_{k-1} (A_{k-1} - l_{k-1} +2) \\
              \vdots &  & \vdots \\
              \zeta_{k-1}(A_{k-1} + T) & \cdots & \zeta_{k-1}(A_{k-1} + 1)
       \end{pmatrix}}_{(IV)^{\vee}}\\       
& \times_{i \neq k-1}
      \underbrace{ \begin{pmatrix}
              \zeta_{i} B_{i} & \cdots & -\zeta_{i}A_{i} \\
              \vdots &  & \vdots \\
              \zeta_{i} (B_{i} + l_{i} - 1) & \cdots & - \zeta_{i}(A_{i} - l_{i} + 1)
       \end{pmatrix} }_{I_{i}}
 \times 
      \underbrace{\begin{pmatrix}
              \zeta_{k} (B_{k} + l_{k}) & \cdots & \zeta_{k}(B_{k} + l_{k} - t + 1) \\
              \vdots &  & \vdots \\
              \zeta_{k} (A_{k} - l_{k}) & \cdots & \zeta_{k}(A_{k} - l_{k} - t + 1)
       \end{pmatrix}}_{II}\\
& \rtimes  \r^{\Sigma_{0}}_{M, >'_{\q}}\Big(\q_{-}, \ul_{-}, \eta_{-}; (\rho, A_{k-1} - l_{k-1}, B_{k-1} + l_{k-1}, 0, \eta'_{k-1}, \zeta_{k-1}), \\
& (\rho, A_{k} - l_{k} -t , B_{k} + l_{k} - t, 0, \eta'_{k}, \zeta_{k})\Big).
\end{align*} 
This implies 
\begin{align*}
\r^{\Sigma_{0}}_{M, >'_{\q}}(\q, \ul', \ueta') & \hookrightarrow 
       \begin{pmatrix}
              \zeta_{k-1} B_{k-1} & \cdots & -\zeta_{k-1}A_{k-1} \\
              \vdots &  & \vdots \\
              \zeta_{k-1} (B_{k-1} + l_{k-1} - 1) & \cdots & -\zeta_{k-1}(A_{k-1} - l_{k-1} + 1)
       \end{pmatrix} \\      
& \times_{i \neq k-1}
      \underbrace{ \begin{pmatrix}
              \zeta_{i} B_{i} & \cdots & -\zeta_{i}A_{i} \\
              \vdots &  & \vdots \\
              \zeta_{i} (B_{i} + l_{i} - 1) & \cdots & - \zeta_{i}(A_{i} - l_{i} + 1)
       \end{pmatrix} }_{I_{i}}
 \times 
      \underbrace{\begin{pmatrix}
              \zeta_{k} (B_{k} + l_{k}) & \cdots & \zeta_{k}(B_{k} + l_{k} - t + 1) \\
              \vdots &  & \vdots \\
              \zeta_{k} (A_{k} - l_{k}) & \cdots & \zeta_{k}(A_{k} - l_{k} - t + 1)
       \end{pmatrix}}_{II}\\
& \rtimes  \r^{\Sigma_{0}}_{M, >'_{\q}}\Big(\q_{-}, \ul_{-}, \eta_{-}; (\rho, A_{k-1} - l_{k-1}, B_{k-1} + l_{k-1}, 0, \eta'_{k-1}, \zeta_{k-1}), \\
& (\rho, A_{k} - l_{k} -t , B_{k} + l_{k} - t, 0, \eta'_{k}, \zeta_{k})\Big).
\end{align*} 
One can further show $\r^{\Sigma_{0}}_{M, >'_{\q}}(\q, \ul', \ueta')$ is the unique irreducible subrepresentation. On the other hand,
\begin{align*}
\r^{\Sigma_{0}}_{M, >_{\q}}(\q, \ul, \ueta) & \hookrightarrow 
            \times_{i}
      \underbrace{\begin{pmatrix}
              \zeta_{i} B_{i} & \cdots & -\zeta_{i}A_{i} \\
              \vdots &  & \vdots \\
              \zeta_{i} (B_{i} + l_{i} - 1) & \cdots & - \zeta_{i}(A_{i} - l_{i} + 1)
       \end{pmatrix}}_{I_{i}} 
  \times 
      \underbrace{\begin{pmatrix}
              \zeta_{k} (B_{k} + l_{k}) & \cdots & \zeta_{k}(B_{k} + l_{k} - t + 1) \\
              \vdots &  & \vdots \\
              \zeta_{k} (A_{k} - l_{k}) & \cdots & \zeta_{k}(A_{k} - l_{k} - t + 1)
       \end{pmatrix}}_{II} \\
& \rtimes  \r^{\Sigma_{0}}_{M, >_{\q}}\Big(\q_{-}, \ul_{-}, \eta_{-}; (\rho, A_{k-1} - l_{k-1}, B_{k-1} + l_{k-1}, 0, \eta_{k-1}, \zeta_{k-1}), \\
& (\rho, A_{k} - l_{k} -t , B_{k} + l_{k} - t, 0, \eta_{k}, \zeta_{k})\Big).
\end{align*} 
By \eqref{eq: second reduction change order unequal sign}$_{r}$, 
\begin{align*}
&\r^{\Sigma_{0}}_{M, >'_{\q}}\Big(\q_{-}, \ul_{-}, \eta_{-}; (\rho, A_{k-1} - l_{k-1}, B_{k-1} + l_{k-1}, 0, \eta'_{k-1}, \zeta_{k-1}),(\rho, A_{k} - l_{k} -t , B_{k} + l_{k} - t, 0, \eta'_{k}, \zeta_{k})\Big) \\
= \, & \r^{\Sigma_{0}}_{M, >_{\q}}\Big(\q_{-}, \ul_{-}, \eta_{-}; (\rho, A_{k-1} - l_{k-1}, B_{k-1} + l_{k-1}, 0, \eta_{k-1}, \zeta_{k-1}), (\rho, A_{k} - l_{k} -t , B_{k} + l_{k} - t, 0, \eta_{k}, \zeta_{k})\Big).
\end{align*}
Hence $\r^{\Sigma_{0}}_{M, >_{\q}}(\q, \ul, \ueta) = \r^{\Sigma_{0}}_{M, >'_{\q}}(\q, \ul', \ueta')$. This finishes the second reduction.

\subsubsection{Final resolution}

Now we want to resolve the case \eqref{eq: second reduction change order unequal sign}$_{r}$. Since $l_{i} = 0$, we can also view $\q$ as an elementary parameter, denoted by $\q_{e}$. The function $\ueta$ over $Jord(\q)$ determines a function $\e_{e}$ over $Jord(\q_{e})$, i.e., for $C_{i} \in [A_{i}, B_{i}]$,
\[
\e_{e}(\rho, C_{i}, C_{i}, \zeta_{i}) = \eta_{i} (-1)^{C_{i} - B_{i}}.
\]
Similarly, we can define $\e'_{e}$. It is obvious that 
\[
\r^{\Sigma_{0}}_{M, >_{\q}}(\q, 0, \ueta) = \r^{\Sigma_{0}}_{M, >_{\q_{e}}}(\q_{e}, \e_{e}).
\]
Let $\q^{T}$ be obtained by changing $(\rho, A_{k-1}, B_{k-1}, \zeta_{k-1})$ to $(\rho, A_{k-1} + T, B_{k-1} + T, \zeta_{k-1})$ such that 
\[
B_{k-1} + T > A_{k} \text { and } B_{k+1} > A_{k-1} + T.
\]
Then 
\[
\r^{\Sigma_{0}}_{M, >'_{\q}}(\q, 0, \ueta') = \Jac_{(\rho, A_{k-1} + T, B_{k-1} + T, \zeta_{k-1}) \mapsto (\rho, A_{k-1}, B_{k-1}, \zeta_{k-1})} \r^{\Sigma_{0}}_{M, >'_{\q}}(\q^{T}, 0, \ueta')
\]
The order $>'_{\q}$ induces an order $>'_{\q_{e}}$ on $Jord(\q_{e})$, and we define
\[
\r^{\Sigma_{0}}_{M, >'_{\q_{e}}}(\q_{e}, \e'_{e}) := \circ_{C_{k-1} \in [B_{k -1}, A_{k-1}]} \Jac_{(\rho, C_{k-1} + T, C_{k-1} + T, \zeta_{k-1}) \mapsto (\rho, C_{k-1}, C_{k-1}, \zeta_{k-1})} \r^{\Sigma_{0}}_{M, >'_{\q_{e}}}(\q^{T}_{e}, \e'_{e})
\]
Since 
\[
\r^{\Sigma_{0}}_{M, >'_{\q_{e}}}(\q^{T}_{e}, \e'_{e}) = \r^{\Sigma_{0}}_{M, >'_{\q}}(\q^{T}, 0, \ueta'),
\] 
and 
\[
\Jac_{(\rho, A_{k-1} + T, B_{k-1} + T, \zeta_{k-1}) \mapsto (\rho, A_{k-1}, B_{k-1}, \zeta_{k-1})} = \circ_{C_{k-1} \in [B_{k -1}, A_{k-1}]} \Jac_{(\rho, C_{k-1} + T, C_{k-1} + T, \zeta_{k-1}) \mapsto (\rho, C_{k-1}, C_{k-1}, \zeta_{k-1})}
\]
we get
\[
\r^{\Sigma_{0}}_{M, >'_{\q}}(\q, 0, \ueta') = \r^{\Sigma_{0}}_{M, >'_{\q_{e}}}(\q_{e}, \e'_{e}).
\]
So it is enough to show 
\(
\r^{\Sigma_{0}}_{M, >_{\q_{e}}}(\q_{e}, \e_{e}) = \r^{\Sigma_{0}}_{M, >'_{\q_{e}}}(\q_{e}, \e'_{e}).
\) 
Note 
\begin{align*}
\r^{\Sigma_{0}}_{M, >_{\q_{e}}}(\q_{e}, \e_{e}) & = \r^{\Sigma_{0}}_{W}(\q_{e}, \e_{e} \e^{M/W}_{\q_{e}}) \\
\r^{\Sigma_{0}}_{M, >'_{\q_{e}}}(\q_{e}, \e'_{e}) & = \r^{\Sigma_{0}}_{W}(\q_{e}, \e'_{e} \e'^{M/W}_{\q_{e}}),
\end{align*}
where $\e^{M/W}_{\q_{e}}$ (resp. $\e'^{M/W}_{\q_{e}}$) is defined with respect to the order $>_{\q_{e}}$ (resp. $>'_{\q_{e}}$) (cf. \eqref{eq: M/W}). Then we just need to verify 
\[
\e_{e} \e^{M/W}_{\q_{e}} =  \e'_{e} \e'^{M/W}_{\q_{e}},
\]
or equivalently,
\begin{align}
\label{eq: final resolution change order unequal sign}
\frac{\e_{e}}{\e'_{e}} = \frac{\e^{M/W}_{\q_{e}}}{\e'^{M/W}_{\q_{e}}} = \frac{\e^{M/MW}_{\q_{e}}}{\e'^{M/MW}_{\q_{e}}} \cdot \frac{\e^{MW/W}_{\q_{e}}}{\e'^{MW/W}_{\q_{e}}}.
\end{align}
We divide it into two cases:

\begin{enumerate}

\item If $A_{i} \in \mathbb{Z}$, 
\[
\e^{MW/W}_{\q_{e}}(\rho, C_{i}, C_{i}, \zeta_{i}) = \e'^{MW/W}_{\q_{e}}(\rho, C_{i}, C_{i}, \zeta_{i}) = 1.
\]
And
\[
\e^{M/MW}_{\q_{e}}(\rho, C_{i}, C_{i}, \zeta_{i}) = \begin{cases}
                                              (-1)^{m} &\text{ if $\zeta_{i} = +1$, } \\
                                              (-1)^{m + n} &\text{ if $\zeta_{i} = -1$,}
                                              \end{cases}
\]
where
\[
m = \sharp \{C_{j} \in [A_{j}, B_{j}] \text{ for all } j:  \zeta_{j} = -1, (\rho, C_{j}, C_{j}, \zeta_{j}) >_{\q_{e}} (\rho, C_{i}, C_{i}, \zeta_{i}) \},
\] 

\[
n = \sharp \{C_{j} \in [A_{j}, B_{j}] \text{ for all } j:   (\rho, C_{i}, C_{i}, \zeta_{i}) >_{\q_{e}} (\rho, C_{j}, C_{j}, \zeta_{j})\}.
\] 
And
\[
\e'^{M/MW}_{\q_{e}}(\rho, C_{i}, C_{i}, \zeta_{i}) = \begin{cases}
                                              (-1)^{m'} &\text{ if $\zeta_{i} = +1$, } \\
                                              (-1)^{m' + n'} &\text{ if $\zeta_{i} = -1$,}
                                              \end{cases}
\]
where
\[
m' = \sharp \{C_{j} \in [A_{j}, B_{j}] \text{ for all } j:  \zeta_{j} = -1, (\rho, C_{j}, C_{j}, \zeta_{j}) >'_{\q_{e}} (\rho, C_{i}, C_{i}, \zeta_{i}) \},
\] 

\[
n' = \sharp \{C_{j} \in [A_{j}, B_{j}] \text{ for all } j:  (\rho, C_{i}, C_{i}, \zeta_{i}) >'_{\q_{e}} (\rho, C_{j}, C_{j}, \zeta_{j}) \}.
\]

  \begin{enumerate}
  
  \item $i \neq k, k-1$ \\
        
         \begin{itemize}
         
         \item $\e_{e}(\rho, C_{i}, C_{i}, \zeta_{i}) / \e'_{e}(\rho, C_{i}, C_{i}, \zeta_{i})= 1$ \\
         
         \item $\e^{M/MW}_{\q_{e}}(\rho, C_{i}, C_{i}, \zeta_{i})/\e'^{M/MW}_{\q_{e}}(\rho, C_{i}, C_{i}, \zeta_{i}) = 1$ \\
         
         \end{itemize}
         
  \item $i = k$ \\
        
         \begin{itemize}
         
         \item $\e_{e}(\rho, C_{k}, C_{k}, \zeta_{k}) / \e'_{e}(\rho, C_{k}, C_{k}, \zeta_{k})= (-1)^{A_{k-1} - B_{k-1} + 1}$ \\
         
         \item $\e^{M/MW}_{\q_{e}}(\rho, C_{k}, C_{k}, \zeta_{k}) / \e'^{M/MW}_{\q_{e}}(\rho, C_{k}, C_{k}, \zeta_{k}) = (-1)^{A_{k-1} - B_{k-1} + 1}$ \\
         
         \end{itemize}    
               
  \item $i = k-1$ \\
        
         \begin{itemize}
         
         \item $\e_{e}(\rho, C_{k-1}, C_{k-1}, \zeta_{k-1}) / \e'_{e}(\rho, C_{k-1}, C_{k-1}, \zeta_{k-1})= (-1)^{A_{k} - B_{k} + 1}$ \\
         
         \item $\e^{M/MW}_{\q_{e}}(\rho, C_{k-1}, C_{k-1}, \zeta_{k-1}) / \e'^{M/MW}_{\q_{e}}(\rho, C_{k-1}, C_{k-1}, \zeta_{k-1}) = (-1)^{A_{k} - B_{k} + 1}$ \\
         
         \end{itemize}       
    
    \end{enumerate}

\item $A_{i} \notin \mathbb{Z}$

\[
\e^{M/MW}_{\q_{e}}(\rho, C_{i}, C_{i}, \zeta_{i}) = \e'^{M/MW}_{\q_{e}}(\rho, C_{i}, C_{i}, \zeta_{i}) = 1.
\]
And
\[
\e^{MW/W}_{\q_{e}}(\rho, C_{i}, C_{i}, \zeta_{i}) = \begin{cases}
                                              (-1)^{m} &\text{ if $\zeta_{i} = +1$, } \\
                                              (-1)^{n} &\text{ if $\zeta_{i} = -1$,}
                                              \end{cases}
\]
where
\[
m = \sharp \{C_{j} \in [A_{j}, B_{j}] \text{ for all } j:  \zeta_{j} = -1, (\rho, C_{j}, C_{j}, \zeta_{j}) >_{\q_{e}} (\rho, C_{i}, C_{i}, \zeta_{i}) \},
\] 

\[
n = \sharp \{C_{j} \in [A_{j}, B_{j}] \text{ for all } j:   \zeta_{j} = +1, (\rho, C_{i}, C_{i}, \zeta_{i}) >_{\q_{e}} (\rho, C_{j}, C_{j}, \zeta_{j})\}.
\] 
And
\[
\e'^{MW/W}_{\q_{e}}(\rho, C_{i}, C_{i}, \zeta_{i}) = \begin{cases}
                                              (-1)^{m'} &\text{ if $\zeta_{i} = +1$, } \\
                                              (-1)^{n'} &\text{ if $\zeta_{i} = -1$,}
                                              \end{cases}
\]
where
\[
m' = \sharp \{C_{j} \in [A_{j}, B_{j}] \text{ for all } j:  \zeta_{j} = -1, (\rho, C_{j}, C_{j}, \zeta_{j}) >'_{\q_{e}} (\rho, C_{i}, C_{i}, \zeta_{i}) \},
\] 

\[
n' = \sharp \{C_{j} \in [A_{j}, B_{j}] \text{ for all } j:  \zeta_{j} = +1, (\rho, C_{i}, C_{i}, \zeta_{i}) >'_{\q_{e}} (\rho, C_{j}, C_{j}, \zeta_{j}) \}.
\]

  \begin{enumerate}
  
  \item $i \neq k, k-1$ \\
        
         \begin{itemize}
         
         \item $\e_{e}(\rho, C_{i}, C_{i}, \zeta_{i}) / \e'_{e}(\rho, C_{i}, C_{i}, \zeta_{i})= 1$ \\
         
         \item $\e^{MW/W}_{\q_{e}}(\rho, C_{i}, C_{i}, \zeta_{i})/\e'^{MW/W}_{\q_{e}}(\rho, C_{i}, C_{i}, \zeta_{i}) = 1$ \\
         
         \end{itemize}
         
  \item $i = k$ \\
        
         \begin{itemize}
         
         \item $\e_{e}(\rho, C_{k}, C_{k}, \zeta_{k}) / \e'_{e}(\rho, C_{k}, C_{k}, \zeta_{k})= (-1)^{A_{k-1} - B_{k-1} + 1}$ \\
         
         \item $\e^{MW/W}_{\q_{e}}(\rho, C_{k}, C_{k}, \zeta_{k}) / \e'^{MW/W}_{\q_{e}}(\rho, C_{k}, C_{k}, \zeta_{k}) = (-1)^{A_{k-1} - B_{k-1} + 1}$ \\
         
         \end{itemize}    
               
  \item $i = k-1$ \\
        
         \begin{itemize}
         
         \item $\e_{e}(\rho, C_{k-1}, C_{k-1}, \zeta_{k-1}) / \e'_{e}(\rho, C_{k-1}, C_{k-1}, \zeta_{k-1})= (-1)^{A_{k} - B_{k} + 1}$ \\
         
         \item $\e^{MW/W}_{\q_{e}}(\rho, C_{k-1}, C_{k-1}, \zeta_{k-1}) / \e'^{MW/W}_{\q_{e}}(\rho, C_{k-1}, C_{k-1}, \zeta_{k-1}) = (-1)^{A_{k} - B_{k} + 1}$ \\
         
         \end{itemize}

  \end{enumerate}

\end{enumerate}
It follows from the calculations above that \eqref{eq: final resolution change order unequal sign} holds, and this ends the proof of Theorem~\ref{thm: change order unequal sign}.

\section{Reduction operations}
\label{sec: reduction}

In this section, we want to give three operations, which will be used in our general procedure to reduce the problem of finding nonvanishing conditions for $\r^{\Sigma_{0}}_{M, >_{\q}}(\q, \ul, \ueta)$.

\subsection{Pull}

\subsubsection{Case of unequal length}

We choose an admissible order $>_{\q}$, and we also fix a self-dual unitary irreducible supercuspidal representation $\rho$ of $GL(d_{\rho})$. We index the Jordan blocks in $Jord_{\rho}(\q)$ such that 
\[
(\rho, A_{i+1}, B_{i+1}, \zeta_{i+1}) >_{\q} (\rho, A_{i}, B_{i}, \zeta_{i}).
\]
Suppose there exists $n$ such that for $i > n$, 
\[
(\rho, A_{i}, B_{i}, \zeta_{i}) \gg \cup_{j=1}^{n}\{(\rho, A_{j}, B_{j}, \zeta_{j})\}.
\]
Moreover 
\[
[A_{n}, B_{n}] \supsetneq [A_{n-1}, B_{n-1}] \text{ and } \zeta_{n} = \zeta_{n-1}.
\] 
We denote by $>'_{\q}$ the order obtained from $>_{\q}$ by switching $(\rho, A_{n}, B_{n}, \zeta_{n})$ and $(\rho, A_{n-1}, B_{n-1}, \zeta_{n-1})$. It is still admissible. Let $S^{+}_{n}$ be the corresponding transformation on $(\ul, \ueta)$. We define $\q_{-}$ by
\[
Jord(\q_{-}) = Jord(\q) \backslash \{(\rho, A_{n}, B_{n}, \zeta_{n}), (\rho, A_{n-1}, B_{n-1}, \zeta_{n-1})\}.
\]
We denote the restriction of $(\ul, \ueta)$ to $Jord(\q_{-})$ by $(\ul_{-}, \ueta_{-})$.

\begin{proposition}
\label{prop: pull}
For any $(\ul, \ueta)$, $\r^{\Sigma_{0}}_{M, >_{\q}}(\q, \ul, \ueta) \neq 0$ if the following three conditions are satisfied:

\begin{enumerate}

\item 
\[
\r^{\Sigma_{0}}_{M, >_{\q}}\Big(\q_{-}, \ul_{-}, \ueta_{-}; (\rho, A_{n} + T_{n}, B_{n} + T_{n}, l_{n}, \eta_{n}, \zeta_{n}), (\rho, A_{n-1} + T_{n-1}, B_{n-1} + T_{n-1}, l_{n-1}, \eta_{n-1}, \zeta_{n-1})\Big) \neq 0
\]
for {\bf some} $T_{n}, T_{n-1}$, such that 
\[
[A_{n} + T_{n}, B_{n} + T_{n}] \supsetneq [A_{n-1} + T_{n-1}, B_{n-1} + T_{n-1}]
\] 
and
\(
(\rho, A_{i}, B_{i}, \zeta_{i}) \gg (\rho, A_{n} + T_{n}, B_{n} + T_{n}, \zeta_{n})
\)
for $i > n$.

\item 
\[
\r^{\Sigma_{0}}_{M, >_{\q}}\Big(\q_{-}, \ul_{-}, \ueta_{-}; (\rho, A_{n} + T, B_{n} + T, l_{n}, \eta_{n}, \zeta_{n}), (\rho, A_{n-1}, B_{n-1}, l_{n-1}, \eta_{n-1}, \zeta_{n-1})\Big) \neq 0
\]
for {\bf some} $T$ such that $B_{i} > A_{n} + T$ for $i > n$.

\item 
\[
\r^{\Sigma_{0}}_{M, >'_{\q}}\Big(\q_{-}, \ul'_{-}, \ueta'_{-}; (\rho, A_{n}, B_{n}, l'_{n}, \eta'_{n}, \zeta_{n}), (\rho, A_{n-1} + T, B_{n-1} +T , l'_{n-1}, \eta'_{n-1}, \zeta_{n-1})\Big) \neq 0
\]
for {\bf some} $T$ such that $B_{i} > A_{n-1} + T$ for $i > n$, and $(\ul', \ueta') = S^{+}_{n}(\ul, \ueta)$.

\end{enumerate}
Conversely, if $\r^{\Sigma_{0}}_{M, >_{\q}}(\q, \ul, \ueta) \neq 0$, then (1), (2), (3) still hold after we replace ``{\bf some}" by ``{\bf all}".

\end{proposition}

\begin{proof}
The converse is obvious. So we mainly need to show the sufficiency of the above three conditions. Let $\zeta = \zeta_{n} = \zeta_{n-1}$, and $(\ul', \ueta') = S^{+}_{n}(\ul, \ueta)$ as in the proposition. Since $[A_{n}, B_{n}] \supsetneq [A_{n-1}, B_{n-1}]$, we necessarily have $[A_{n} + 1, B_{n} + 1] \supsetneq [A_{n-1}, B_{n-1}]$ or $[A_{n}, B_{n}] \supsetneq [A_{n-1} + 1, B_{n-1} + 1]$. So we divide it into two cases.

Suppose $[A_{n} + 1, B_{n} + 1] \supsetneq [A_{n-1}, B_{n-1}]$, we claim $\r^{\Sigma_{0}}_{M, >_{\q}}(\q, \ul, \ueta) \neq 0$ if the following conditions are satisfied:

\begin{itemize}

\item 
\[
\r^{\Sigma_{0}}_{M, >_{\q}}\Big(\q_{-}, \ul_{-}, \ueta_{-}; (\rho, A_{n} + 1, B_{n} + 1, l_{n}, \eta_{n}, \zeta_{n}), (\rho, A_{n-1}, B_{n-1}, l_{n-1}, \eta_{n-1}, \zeta_{n-1})\Big) \neq 0;
\]

\item 
\[
\r^{\Sigma_{0}}_{M, >'_{\q}}\Big(\q_{-}, \ul'_{-}, \ueta'_{-}; (\rho, A_{n}, B_{n}, l'_{n}, \eta'_{n}, \zeta_{n}), (\rho, A_{n-1} + T, B_{n-1} +T , l'_{n-1}, \eta'_{n-1}, \zeta_{n-1})\Big) \neq 0
\]
for {\bf some} $T$ such that $B_{i} > A_{n-1} + T$ for $i > n$.

\end{itemize}

It suffices to take $T$ sufficiently large so that $B_{i} > A_{n-1} + T$ for $i > n$, and
\[
(\rho, A_{n-1} + T, B_{n-1} + T, \zeta_{n-1}) \gg \cup_{j=1}^{n-2}\{(\rho, A_{j}, B_{j}, \zeta_{j})\} \cup \{(\rho, A_{n}+1, B_{n}+1, \zeta_{n})\} .
\]
By Theorem~\ref{thm: change order equal sign}, we have
\begin{align*}
&\r^{\Sigma_{0}}_{M, >'_{\q}}\Big(\q_{-}, \ul'_{-}, \ueta'_{-}; (\rho, A_{n} + 1, B_{n} + 1, l'_{n}, \eta'_{n}, \zeta_{n}), (\rho, A_{n-1}, B_{n-1}, l'_{n-1}, \eta'_{n-1}, \zeta_{n-1})\Big) = \\
&\r^{\Sigma_{0}}_{M, >_{\q}}\Big(\q_{-}, \ul_{-}, \ueta_{-}; (\rho, A_{n} + 1, B_{n} + 1, l_{n}, \eta_{n}, \zeta_{n}), (\rho, A_{n-1}, B_{n-1}, l_{n-1}, \eta_{n-1}, \zeta_{n-1})\Big) \neq 0.
\end{align*}
So
\[
\r^{\Sigma_{0}}_{\gg} := \r^{\Sigma_{0}}_{M, >'_{\q}}\Big(\q_{-}, \ul'_{-}, \ueta'_{-}; (\rho, A_{n} + 1, B_{n} + 1, l'_{n}, \eta'_{n}, \zeta_{n}), (\rho, A_{n-1} + T, B_{n-1} +T , l'_{n-1}, \eta'_{n-1}, \zeta_{n-1})\Big) \neq 0
\]
and 
\begin{align*}
\r^{\Sigma_{0}}_{\gg} & \hookrightarrow 
       \begin{pmatrix}
              \zeta (B_{n-1} + T) & \cdots & \zeta(B_{n-1} + 1) \\
              \vdots &  & \vdots \\
              \zeta (A_{n-1} + T) & \cdots & \zeta(A_{n-1} + 1)
       \end{pmatrix} \\
& \rtimes \r^{\Sigma_{0}}_{M, >'_{\q}}\Big(\q_{-}, \ul'_{-}, \ueta'_{-}; (\rho, A_{n} + 1, B_{n} + 1, l'_{n}, \eta'_{n}, \zeta_{n}), (\rho, A_{n-1}, B_{n-1}, l'_{n-1}, \eta'_{n-1}, \zeta_{n-1})\Big) \\
& = \begin{pmatrix}
              \zeta (B_{n-1} + T) & \cdots & \zeta(B_{n-1} + 1) \\
              \vdots &  & \vdots \\
              \zeta (A_{n-1} + T) & \cdots & \zeta(A_{n-1} + 1)
       \end{pmatrix} \\
& \rtimes \r^{\Sigma_{0}}_{M, >_{\q}}\Big(\q_{-}, \ul_{-}, \ueta_{-}; (\rho, A_{n} + 1, B_{n} + 1, l_{n}, \eta_{n}, \zeta_{n}), (\rho, A_{n-1}, B_{n-1}, l_{n-1}, \eta_{n-1}, \zeta_{n-1})\Big).
\end{align*}   
Note 
\begin{align*}
&\Jac_{(\rho, A_{n} + 1, B_{n} + 1, \zeta) \mapsto (\rho, A_{n}, B_{n}, \zeta)} \r^{\Sigma_{0}}_{\gg} = \\
& \r^{\Sigma_{0}}_{M, >'_{\q}}\Big(\q_{-}, \ul'_{-}, \ueta'_{-}; (\rho, A_{n}, B_{n}, l'_{n}, \eta'_{n}, \zeta_{n}), (\rho, A_{n-1} + T, B_{n-1} +T , l'_{n-1}, \eta'_{n-1}, \zeta_{n-1})\Big) \neq 0.
\end{align*}
So after we apply the same Jacquet functor to the full induced representation above, we should get something nonzero. To compute this Jacquet module, one notes
\[
\zeta(B_{n-1} + T), -\zeta(A_{n-1} + 1) \notin \{\zeta(B_{n} + 1), \cdots, \zeta(A_{n} + 1)\},
\]
so it can only be  
\begin{align*}
 & \begin{pmatrix}
              \zeta (B_{n-1} + T) & \cdots & \zeta(B_{n-1} + 1) \\
              \vdots &  & \vdots \\
              \zeta (A_{n-1} + T) & \cdots & \zeta(A_{n-1} + 1)
       \end{pmatrix}
\rtimes \Jac_{(\rho, A_{n} + 1, B_{n} + 1, \zeta) \mapsto (\rho, A_{n}, B_{n}, \zeta)} \\
& \r^{\Sigma_{0}}_{M, >_{\q}}\Big(\q_{-}, \ul_{-}, \ueta_{-}; (\rho, A_{n} + 1, B_{n} + 1, l_{n}, \eta_{n}, \zeta_{n}), (\rho, A_{n-1}, B_{n-1}, l_{n-1}, \eta_{n-1}, \zeta_{n-1})\Big) \\
= & \begin{pmatrix}
              \zeta (B_{n-1} + T) & \cdots & \zeta(B_{n-1} + 1) \\
              \vdots &  & \vdots \\
              \zeta (A_{n-1} + T) & \cdots & \zeta(A_{n-1} + 1)
       \end{pmatrix}
\rtimes \r^{\Sigma_{0}}_{M, >_{\q}}(\q, \ul, \ueta) \neq 0.
\end{align*}       
Hence $\r^{\Sigma_{0}}_{M, >_{\q}}(\q, \ul, \ueta) \neq 0$. This shows our claim in the first case. 
 
Suppose $[A_{n}, B_{n}] \supsetneq [A_{n-1} + 1, B_{n-1} + 1]$, we claim $\r^{\Sigma_{0}}_{M, >_{\q}}(\q, \ul, \ueta) \neq 0$ if the following conditions are satisfied:

\begin{itemize}

\item 
\[
\r^{\Sigma_{0}}_{M, >'_{\q}}\Big(\q_{-}, \ul'_{-}, \ueta'_{-}; (\rho, A_{n}, B_{n}, l'_{n}, \eta'_{n}, \zeta_{n}), (\rho, A_{n-1} + 1, B_{n-1} + 1, l'_{n-1}, \eta'_{n-1}, \zeta_{n-1})\Big) \neq 0;
\]

\item 
\[
\r^{\Sigma_{0}}_{M, >_{\q}}\Big(\q_{-}, \ul_{-}, \ueta_{-}; (\rho, A_{n} + T, B_{n} + T, l_{n}, \eta_{n}, \zeta_{n}), (\rho, A_{n-1}, B_{n-1}, l_{n-1}, \eta_{n-1}, \zeta_{n-1})\Big) \neq 0
\]
for {\bf some} $T$ such that $B_{i} > A_{n} + T$ for $i > n$.

\end{itemize}
The argument of this case is essentially the same as the previous one. Again it suffices to take $T$ sufficiently large so that $B_{i} > A_{n} + T$ for $i > n$, and
\[
(\rho, A_{n} + T, B_{n} + T, \zeta_{n}) \gg \cup_{j=1}^{n-2}\{(\rho, A_{j}, B_{j}, \zeta_{j})\} \cup \{(\rho, A_{n-1}+1, B_{n-1}+1, \zeta_{n-1})\} .
\]
By Theorem~\ref{thm: change order equal sign}, we have
\begin{align*}
&\r^{\Sigma_{0}}_{M, >_{\q}}\Big(\q_{-}, \ul_{-}, \ueta_{-}; (\rho, A_{n}, B_{n}, l_{n}, \eta_{n}, \zeta_{n}), (\rho, A_{n-1} + 1, B_{n-1} + 1, l_{n-1}, \eta_{n-1}, \zeta_{n-1})\Big) = \\
&\r^{\Sigma_{0}}_{M, >'_{\q}}\Big(\q_{-}, \ul'_{-}, \ueta'_{-}; (\rho, A_{n}, B_{n}, l'_{n}, \eta'_{n}, \zeta_{n}), (\rho, A_{n-1} + 1, B_{n-1} + 1, l'_{n-1}, \eta'_{n-1}, \zeta_{n-1})\Big) \neq 0.
\end{align*}
So
\[
\r^{\Sigma_{0}}_{\gg} := \r^{\Sigma_{0}}_{M, >_{\q}}\Big(\q_{-}, \ul_{-}, \ueta_{-}; (\rho, A_{n} + T, B_{n} + T, l_{n}, \eta_{n}, \zeta_{n}), (\rho, A_{n-1} + 1, B_{n-1} + 1, l_{n-1}, \eta_{n-1}, \zeta_{n-1})\Big) \neq 0
\]
and
\begin{align*}
\r^{\Sigma_{0}}_{\gg} & \hookrightarrow 
       \begin{pmatrix}
              \zeta (B_{n} + T) & \cdots & \zeta(B_{n} + 1) \\
              \vdots &  & \vdots \\
              \zeta (A_{n} + T) & \cdots & \zeta(A_{n} + 1)
       \end{pmatrix} \\
& \rtimes \r^{\Sigma_{0}}_{M, >_{\q}}\Big(\q_{-}, \ul_{-}, \ueta_{-}; (\rho, A_{n}, B_{n}, l_{n}, \eta_{n}, \zeta_{n}), (\rho, A_{n-1} + 1, B_{n-1} + 1, l_{n-1}, \eta_{n-1}, \zeta_{n-1})\Big) \\
& = \begin{pmatrix}
              \zeta (B_{n} + T) & \cdots & \zeta(B_{n} + 1) \\
              \vdots &  & \vdots \\
              \zeta (A_{n} + T) & \cdots & \zeta(A_{n} + 1)
       \end{pmatrix} \\
& \rtimes \r^{\Sigma_{0}}_{M, >'_{\q}}\Big(\q_{-}, \ul'_{-}, \ueta'_{-}; (\rho, A_{n}, B_{n}, l'_{n}, \eta'_{n}, \zeta_{n}), (\rho, A_{n-1} + 1, B_{n-1} + 1, l'_{n-1}, \eta'_{n-1}, \zeta_{n-1})\Big). 
\end{align*}   
Note 
\begin{align*}
&\Jac_{(\rho, A_{n-1} + 1, B_{n-1} + 1, \zeta) \mapsto (\rho, A_{n-1}, B_{n-1}, \zeta)} \r^{\Sigma_{0}}_{\gg} = \\
& \r^{\Sigma_{0}}_{M, >_{\q}}\Big(\q_{-}, \ul_{-}, \ueta_{-}; (\rho, A_{n} + T, B_{n} + T, l_{n}, \eta_{n}, \zeta_{n}), (\rho, A_{n-1}, B_{n-1}, l_{n-1}, \eta_{n-1}, \zeta_{n-1})\Big) \neq 0.
\end{align*}
So after we apply the same Jacquet functor to the full induced representation above, we should get something nonzero. To compute this Jacquet module, one notes
\[
\zeta(B_{n} + T), -\zeta(A_{n} + 1) \notin \{\zeta(B_{n-1} + 1), \cdots, \zeta(A_{n-1} + 1)\},
\]
so it can only be  
\begin{align*}
& \begin{pmatrix}
              \zeta (B_{n} + T) & \cdots & \zeta(B_{n} + 1) \\
              \vdots &  & \vdots \\
              \zeta (A_{n} + T) & \cdots & \zeta(A_{n} + 1)
       \end{pmatrix}
\rtimes \Jac_{(\rho, A_{n-1} + 1, B_{n-1} + 1, \zeta) \mapsto (\rho, A_{n-1}, B_{n-1}, \zeta)} \\
& \r^{\Sigma_{0}}_{M, >'_{\q}}\Big(\q_{-}, \ul'_{-}, \ueta'_{-}; (\rho, A_{n}, B_{n}, l'_{n}, \eta'_{n}, \zeta_{n}), (\rho, A_{n-1} + 1, B_{n-1} +1, l'_{n-1}, \eta'_{n-1}, \zeta_{n-1})\Big) \\
= & \begin{pmatrix}
              \zeta (B_{n} + T) & \cdots & \zeta(B_{n} + 1) \\
              \vdots &  & \vdots \\
              \zeta (A_{n} + T) & \cdots & \zeta(A_{n} + 1)
       \end{pmatrix}
\rtimes \r^{\Sigma_{0}}_{M, >'_{\q}}(\q, \ul', \ueta') \neq 0.
\end{align*}       
Hence $\r^{\Sigma_{0}}_{M, >_{\q}}(\q, \ul, \ueta) = \r^{\Sigma_{0}}_{M, >'_{\q}}(\q, \ul', \ueta') \neq 0$. This shows our claim in the second case.

By combining our claims in both cases in some alternating way, we can shift both $[A_{n}, B_{n}], [A_{n-1}, B_{n-1}]$ to $[A_{n} + T_{n}, B_{n} + T_{n}], [A_{n-1} + T_{n-1}, B_{n-1} + T_{n-1}]$ for any $T_{n}, T_{n-1}$ such that 
\[
[A_{n} + T_{n}, B_{n} + T_{n}] \supsetneq [A_{n-1} + T_{n-1}, B_{n-1} + T_{n-1}]
\] 
and
\(
(\rho, A_{i}, B_{i}, \zeta_{i}) \gg (\rho, A_{n} + T_{n}, B_{n} + T_{n}, \zeta_{n})
\)
for $i > n$. Then the statement of this proposition is clear.
\end{proof}

\begin{remark}
The way we will use this proposition is to take all $T_{n}, T_{n-1}$ and $T$ to be large.
\end{remark}

\subsubsection{Case of equal length}
\label{equal length}

We choose an admissible order $>_{\q}$, and we also fix a self-dual unitary irreducible supercuspidal representation $\rho$ of $GL(d_{\rho})$. We index the Jordan blocks in $Jord_{\rho}(\q)$ such that 
\[
(\rho, A_{i+1}, B_{i+1}, \zeta_{i+1}) >_{\q} (\rho, A_{i}, B_{i}, \zeta_{i}).
\]
Suppose there exists $n$ such that for $i > n$, 
\[
(\rho, A_{i}, B_{i}, \zeta_{i}) \gg \cup_{j=1}^{n}\{(\rho, A_{j}, B_{j}, \zeta_{j})\}.
\]
Moreover 
\[
[A_{n}, B_{n}] = [A_{n-1}, B_{n-1}] \text{ and } \zeta_{n} = \zeta_{n-1}.
\] 
Note
\begin{align}
\label{eq: equal length fact}
\text{there exists no $i < n-1$ satisfying $\zeta_{i} = \zeta_n$, $A_i > A_n$ and $B_{i} > B_n$.}
\end{align}
We define $\q_{-}$ by
\[
Jord(\q_{-}) = Jord(\q) \backslash \{(\rho, A_{n}, B_{n}, \zeta_{n}), (\rho, A_{n-1}, B_{n-1}, \zeta_{n-1})\}.
\]
We denote the restriction of $(\ul, \ueta)$ to $Jord(\q_{-})$ by $(\ul_{-}, \ueta_{-})$.

\begin{proposition}
\label{prop: pull equal length}
For any $(\ul, \ueta)$, $\r^{\Sigma_{0}}_{M, >_{\q}}(\q, \ul, \ueta) \neq 0$ if the following two conditions are satisfied:

\begin{enumerate}

\item 
\[
\r^{\Sigma_{0}}_{M, >_{\q}}\Big(\q_{-}, \ul_{-}, \ueta_{-}; (\rho, A_{n} + T_{n}, B_{n} + T_{n}, l_{n}, \eta_{n}, \zeta_{n}), (\rho, A_{n-1} + T_{n-1}, B_{n-1} + T_{n-1}, l_{n-1}, \eta_{n-1}, \zeta_{n-1})\Big) \neq 0
\]
for {\bf some} $T_{n} = T_{n-1}$ such that
\(
(\rho, A_{i}, B_{i}, \zeta_{i}) \gg (\rho, A_{n} + T_{n}, B_{n} + T_{n}, \zeta_{n})
\)
for $i > n$.

\item 
\[
\r^{\Sigma_{0}}_{M, >_{\q}}\Big(\q_{-}, \ul_{-}, \ueta_{-}; (\rho, A_{n} + T, B_{n} + T, l_{n}, \eta_{n}, \zeta_{n}), (\rho, A_{n-1}, B_{n-1}, l_{n-1}, \eta_{n-1}, \zeta_{n-1})\Big) \neq 0
\]
for {\bf some} $T$ such that $B_{i} > A_{n} + T$ for $i > n$.

\end{enumerate}
Conversely, if $\r^{\Sigma_{0}}_{M, >_{\q}}(\q, \ul, \ueta) \neq 0$, then (1), (2) still hold after we replace ``{\bf some}" by ``{\bf all}".

\end{proposition}

\begin{proof}
The converse is obvious. So we mainly need to show the sufficiency of the above two conditions. Let $[A_{n}, B_{n}] = [A, B]$ and $\zeta_{n} = \zeta$. It is enough to prove the proposition by taking $T_{n} = T_{n-1} = 1$ in the first condition. So let us suppose
\[
\r^{\Sigma_{0}}_{M, >_{\q}}\Big(\q_{-}, \ul_{-}, \ueta_{-}; (\rho, A_{n} + 1, B_{n} + 1, l_{n}, \eta_{n}, \zeta_{n}), (\rho, A_{n-1} + 1, B_{n-1} + 1, l_{n-1}, \eta_{n-1}, \zeta_{n-1})\Big) \neq 0.
\]
We can take $T$ sufficiently large such that $B_{i} > A_{n} + T$ for $i > n$, and
\[
(\rho, A_{n} + T, B_{n} + T, \zeta_{n}) \gg \cup_{j=1}^{n-2}\{(\rho, A_{j}, B_{j}, \zeta_{j})\} \cup \{(\rho, A_{n-1}+1, B_{n-1}+1, \zeta_{n-1})\}.
\]
Let
\[
\r^{\Sigma_{0}}_{\gg} := \r^{\Sigma_{0}}_{M, >_{\q}}\Big(\q_{-}, \ul_{-}, \ueta_{-}; (\rho, A_{n} + T, B_{n} + T, l_{n}, \eta_{n}, \zeta_{n}), (\rho, A_{n-1} + 1, B_{n-1} + 1, l_{n-1}, \eta_{n-1}, \zeta_{n-1})\Big). 
\]
Then
\begin{align*}
\r^{\Sigma_{0}}_{\gg} & \hookrightarrow 
       \begin{pmatrix}
              \zeta (B + T) & \cdots & \zeta(B + 2) \\
              \vdots &  & \vdots \\
              \zeta (A + T) & \cdots & \zeta(A + 2)
       \end{pmatrix} \\
& \rtimes \r^{\Sigma_{0}}_{M, >_{\q}}\Big(\q_{-}, \ul_{-}, \ueta_{-}; (\rho, A_{n} + 1, B_{n} + 1, l_{n}, \eta_{n}, \zeta_{n}), (\rho, A_{n-1} + 1, B_{n-1} + 1, l_{n-1}, \eta_{n-1}, \zeta_{n-1})\Big) 
\end{align*}   
Since 
\[
\r^{\Sigma_{0}}_{M, >_{\q}}\Big(\q_{-}, \ul_{-}, \ueta_{-}; (\rho, A_{n} + T, B_{n} + T, l_{n}, \eta_{n}, \zeta_{n}), (\rho, A_{n-1}, B_{n-1}, l_{n-1}, \eta_{n-1}, \zeta_{n-1})\Big) \neq 0,
\]
then 
\begin{align}
\label{eq: pull equal length 0}
\Jac_{\zeta(B+1), \cdots, \zeta(A+1)} \r^{\Sigma_{0}}_{\gg} \neq 0.
\end{align}
Note $\zeta(B+T) \notin [\zeta(B+1), \zeta(A+1)]$. So this implies 
\begin{align*}
& \Jac_{\zeta(B+1), \cdots, \zeta(A+1)}  \r^{\Sigma_{0}}_{M, >_{\q}}\Big(\q_{-}, \ul_{-}, \ueta_{-}; (\rho, A_{n} + 1, B_{n} + 1, l_{n}, \eta_{n}, \zeta_{n}), \\
& (\rho, A_{n-1} + 1, B_{n-1} + 1, l_{n-1}, \eta_{n-1}, \zeta_{n-1})\Big) \neq 0
\end{align*}
Then there exists $C \in [A+1, B+1]$ and an irreducible representation $\sigma$ such that 
\begin{align*}
& \r^{\Sigma_{0}}_{M, >_{\q}}\Big(\q_{-}, \ul_{-}, \ueta_{-}; (\rho, A_{n} + 1, B_{n} + 1, l_{n}, \eta_{n}, \zeta_{n}), (\rho, A_{n-1} + 1, B_{n-1} + 1, l_{n-1}, \eta_{n-1}, \zeta_{n-1})\Big) \\
& \hookrightarrow \langle \zeta C, \cdots, \zeta (A+1) \rangle \rtimes \sigma.
\end{align*}   
By \eqref{eq: equal length fact}, we must have $C = B+1$. Therefore,
\begin{align*}
\r^{\Sigma_{0}}_{\gg} & \hookrightarrow 
       \begin{pmatrix}
              \zeta (B + T) & \cdots & \zeta(B + 2) \\
              \vdots &  & \vdots \\
              \zeta (A + T) & \cdots & \zeta(A + 2)
       \end{pmatrix} 
   \times \begin{pmatrix}
              \zeta(B + 1) \\
              \vdots \\
              \zeta(A + 1)
       \end{pmatrix} 
 \rtimes \sigma
\end{align*}
Let us denote the full induced representation above by $(*-1)$. By Frobenius reciprocity, $\sigma$ is an irreducible constituent of
\[
\Jac_{\zeta(B+1), \cdots, \zeta(A+1)}  \r^{\Sigma_{0}}_{M, >_{\q}}\Big(\q_{-}, \ul_{-}, \ueta_{-}; (\rho, A_{n} + 1, B_{n} + 1, l_{n}, \eta_{n}, \zeta_{n}), \\
(\rho, A_{n-1} + 1, B_{n-1} + 1, l_{n-1}, \eta_{n-1}, \zeta_{n-1})\Big).
\]
In fact it is not hard to show the Jacquet module above consists of representations in 
\begin{align*}
& \Pkt{}^{\Sigma_{0}}\Big(\q_{-}, (\rho, A_{n} + 1, B_{n} + 1, \zeta_{n}), (\rho, A_{n-1}, B_{n-1}, \zeta_{n-1})\Big) \bigcup \\
& \Pkt{}^{\Sigma_{0}}\Big(\q_{-}, (\rho, A_{n} + 1, B_{n-1}, \zeta_{n}), (\rho, A_{n-1}, B_{n} + 1, \zeta_{n-1})\Big)
\end{align*}
So in particular, $\sigma$ is an element in the above packets. We claim 
\begin{align}
\label{eq: pull equal length}
\Jac_{\zeta(B+1), \cdots, \zeta(A+1)} \sigma \neq 0.
\end{align}
Otherwise, one finds 
\[
\Jac_{\zeta C', \cdots, \zeta C''} \sigma = 0
\]
for any $C' \in [B+1, A+T]$ and $C'' \in [A+1, A+T]$. This implies
\[
\Jac_{(\rho, A+T, B+T, \zeta) \mapsto (\rho, A, B, \zeta)} (*-1) = \sigma.
\]
So $(*-1)$ has a unique irreducible subrepresentation, and hence 
\begin{align*}
\r^{\Sigma_{0}}_{\gg} & \hookrightarrow 
       \begin{pmatrix}
              \zeta (B + T) & \cdots & \zeta(B + 1) \\
              \vdots &  & \vdots \\
              \zeta (A + T) & \cdots & \zeta(A + 1)
       \end{pmatrix} 
 \rtimes \sigma
\end{align*}
Then $\Jac_{\zeta(B+1), \cdots, \zeta(A+1)} \sigma = 0$ implies $\Jac_{\zeta(B+1), \cdots, \zeta(A+1)} \r^{\Sigma_{0}}_{\gg} = 0$, but this contradicts to \eqref{eq: pull equal length 0}. As a consequence, $\sigma$ can only be in 
\[
\Pkt{}^{\Sigma_{0}}\Big(\q_{-}, (\rho, A_{n} + 1, B_{n} + 1, \zeta_{n}), (\rho, A_{n-1}, B_{n-1}, \zeta_{n-1})\Big).
\]
Now by \eqref{eq: pull equal length},
we have 
\begin{align*}
\sigma & \hookrightarrow 
       \begin{pmatrix}
              \zeta C \\
              \vdots \\
              \zeta (A + 1)
       \end{pmatrix} 
 \rtimes \sigma'
\end{align*}
for some $C \in [B+1, A+1]$ and some irreducible representation $\sigma'$. For the same reason as before, we must have $C = B+1$. This also implies $\sigma' \in \Pkt{\q}^{\Sigma_{0}}$. Therefore
\begin{align*}
\r^{\Sigma_{0}}_{\gg} & \hookrightarrow 
       \underbrace{\begin{pmatrix}
              \zeta (B + T) & \cdots & \zeta(B + 2) \\
              \vdots &  & \vdots \\
              \zeta (A + T) & \cdots & \zeta(A + 2)
       \end{pmatrix} 
   \times \begin{pmatrix}
              \zeta(B + 1) \\
              \vdots \\
              \zeta(A + 1)
       \end{pmatrix}
   \times \begin{pmatrix}
              \zeta(B + 1) \\
              \vdots \\
              \zeta(A + 1)
       \end{pmatrix}}_{(*-2)} 
 \rtimes \sigma'
\end{align*}
There exists an irreducible constituent $\tau$ of $(*-2)$ such that 
\begin{align*}
\r^{\Sigma_{0}}_{\gg} & \hookrightarrow 
       \tau \rtimes \sigma'.
\end{align*}
By \eqref{eq: equal length fact}, 
\(
\Jac_{\zeta(C), \cdots, \zeta(A+1)} \sigma' = 0
\)
for all $C \in [B + 1, A +1]$. Then we can conclude from \eqref{eq: pull equal length 0} that
\[
\Jac_{\zeta(B+1), \cdots, \zeta(A+1)} \tau \neq 0.
\]
So there exists $C \in [B+1, A+1]$ and an irreducible representation $\tau'$ such that 
\begin{align*}
\tau & \hookrightarrow 
       \begin{pmatrix}
              \zeta C \\
              \vdots \\
              \zeta(A + 1)
       \end{pmatrix}
 \rtimes \tau'
\end{align*}
From $(*-2)$, we see $C$ can only be $B+1$. Hence, $\tau'$ is an irreducible constituent of
\begin{align*}
\Jac_{\zeta(B+1), \cdots, \zeta(A+1)} (*-2) = 2 \cdot 
     \begin{pmatrix}
              \zeta (B + T) & \cdots & \zeta(B + 2) \\
              \vdots &  & \vdots \\
              \zeta (A + T) & \cdots & \zeta(A + 2)
       \end{pmatrix} 
   \times \begin{pmatrix}
              \zeta(B + 1) \\
              \vdots \\
              \zeta(A + 1)
       \end{pmatrix}
\end{align*}
If $\Jac_{\zeta (B+1)} \tau' \neq 0$, then it is necessary that
\begin{align*}
\Jac_{\zeta (B+1)} \tau' =
     \begin{pmatrix}
              \zeta (B + T) & \cdots & \zeta(B + 2) \\
              \vdots &  & \vdots \\
              \zeta (A + T) & \cdots & \zeta(A + 2)
       \end{pmatrix} 
   \times \begin{pmatrix}
              \zeta(B + 2) \\
              \vdots \\
              \zeta(A + 1)
       \end{pmatrix}
\cong 
 \begin{pmatrix}
              \zeta(B + 2) \\
              \vdots \\
              \zeta(A + 1)
       \end{pmatrix}
 \times
 \begin{pmatrix}
              \zeta (B + T) & \cdots & \zeta(B + 2) \\
              \vdots &  & \vdots \\
              \zeta (A + T) & \cdots & \zeta(A + 2)
       \end{pmatrix}, 
\end{align*}
which is irreducible. Consequently,
\[
\Jac_{\zeta (B+1), \cdots \zeta (A+1)} \circ \Jac_{\zeta (B+1), \cdots \zeta (A+1)} \r^{\Sigma_{0}}_{\gg} \neq 0.
\]
But this is impossible by \eqref{eq: equal length fact}.
Therefore, we must have 
\begin{align*}
\tau' = \begin{pmatrix}
              \zeta (B + T) & \cdots & \zeta(B + 1) \\
              \vdots &  & \vdots \\
              \zeta (A + T) & \cdots & \zeta(A + 1)
       \end{pmatrix}.
\end{align*}
To summarize, we get
\begin{align*}
\r^{\Sigma_{0}}_{\gg} & \hookrightarrow 
         \begin{pmatrix}
              \zeta(B + 1) \\
              \vdots \\
              \zeta(A + 1)
           \end{pmatrix}
\times
         \begin{pmatrix}
              \zeta (B + T) & \cdots & \zeta(B + 1) \\
              \vdots &  & \vdots \\
              \zeta (A + T) & \cdots & \zeta(A + 1)
          \end{pmatrix} 
    \rtimes \sigma'. 
\end{align*}
Hence, 
\[
\r^{\Sigma_{0}}_{M, >_{\q}}(\q, \ul, \ueta) = \Jac_{(\rho, A+T, B+T, \zeta) \mapsto (\rho, A, B, \zeta)} \circ \Jac_{\zeta (B+1), \cdots, \zeta (A+1)} \r^{\Sigma_{0}}_{\gg} \neq 0.
\]
This finishes the proof.

\end{proof}

\subsection{Expand}

We choose an admissible order $>_{\q}$, and we also fix a self-dual unitary irreducible supercuspidal representation $\rho$ of $GL(d_{\rho})$. We index the Jordan blocks in $Jord_{\rho}(\q)$ such that 
\[
(\rho, A_{i+1}, B_{i+1}, \zeta_{i+1}) >_{\q} (\rho, A_{i}, B_{i}, \zeta_{i}).
\]
Suppose there exists $n$ such that for $i > n$, 
\[
(\rho, A_{i}, B_{i}, \zeta_{i}) \gg_{2} \cup_{j=1}^{n}\{(\rho, A_{j}, B_{j}, \zeta_{j})\}.
\]
Moreover, for $i < n$, 
\[
\text{ $A_{n} \geqslant A_{i}$ and there exists no $[A_{i}, B_{i}] \subseteq [A_{n}, B_{n}]$ with $\zeta_{i} = \zeta_{n}$.}
\] 
Let $t_{n}$ be the smallest integer such that $B_{n} - t_{n} = B_{i}$ for some $i < n$ and $\zeta_{i} = \zeta_{n}$. If such $t_{n}$ does not exist, we let $t_{n} := [B_{n}]$. We define $\q_{-}$ by
\[
Jord(\q_{-}) = Jord(\q) \backslash \{(\rho, A_{n}, B_{n}, \zeta_{n})\}.
\]
We denote the restriction of $(\ul, \ueta)$ to $Jord(\q_{-})$ by $(\ul_{-}, \ueta_{-})$.
\begin{proposition}
\label{prop: expand}
We fix a positive integer $t \leqslant t_{n}$. Then for any $(\ul, \ueta)$, $\r^{\Sigma_{0}}_{M, >_{\q}}(\q, \ul, \ueta) \neq 0$ if and only if
\[
\r^{\Sigma_{0}}_{M, >_{\q}}\Big(\q_{-}, \ul_{-}, \ueta_{-}; (\rho, A_{n} + t, B_{n} -t, l_{n} + t, \eta_{n}, \zeta_{n})\Big) \neq 0.
\]
Moreover,
\begin{align*}
& \r^{\Sigma_{0}}_{M, >_{\q}}\Big(\q_{-}, \ul_{-}, \ueta_{-}; (\rho, A_{n} + t, B_{n} -t, l_{n} + t, \eta_{n}, \zeta_{n})\Big) \hookrightarrow  \\
&       \begin{pmatrix}
              \zeta_{n} (B_{n} - t) & \cdots & -\zeta_{n}(A_{n} + t) \\
              \vdots &  & \vdots \\
              \zeta_{n} (B_{n} - 1)  & \cdots & -\zeta_{n}(A_{n} + 1)
       \end{pmatrix} \rtimes  \r^{\Sigma_{0}}_{M, >_{\q}}(\q, \ul, \ueta).      
\end{align*}   
as the unique irreducible subrepresentation, and
\begin{align*}
\r^{\Sigma_{0}}_{M, >_{\q}}(\q, \ul, \ueta) = \circ_{i \in [1, t]} \Jac_{\zeta_{n} (B_{n} -i), \cdots, - \zeta_{n}(A_{n} + i)} \r^{\Sigma_{0}}_{M, >_{\q}}\Big(\q_{-}, \ul_{-}, \ueta_{-}; (\rho, A_{n} + t, B_{n} -t, l_{n} + t, \eta_{n}, \zeta_{n})\Big).
\end{align*}
\end{proposition}

\begin{proof}
We will first consider the case $t = 1$. Let $\q_{\gg}$ dominates $\q$ with discrete diagonal restriction such that 
\[
(\rho, A_{n+1}, B_{n+1}, \zeta_{n+1}) \gg (\rho, A_{n} + T_{n} + 1, B_{n} + T_{n} -1, \zeta_{n}) \gg (\rho, A_{n-1} + T_{n-1}, B_{n} + T_{n-1}, \zeta_{n-1}).
\]
Let $\q_{\gg, -}$ be obtained from $\q_{\gg}$ by removing $(\rho, A_{n} + T_{n}, B_{n} + T_{n}, \zeta_{n})$. Then
\begin{align*}
& \r^{\Sigma_{0}}_{M, >_{\q}}\Big(\q_{\gg, -}, \ul_{-}, \ueta_{-}; (\rho, A_{n} + T_{n} + 1, B_{n} + T_{n} - 1, l_{n} + 1, \eta_{n}, \zeta_{n})\Big) \hookrightarrow \\
& \langle \zeta_{n} (B_{n} + T_{n} -1), \cdots, -\zeta_{n}(A_{n} + T_{n} + 1) \rangle \rtimes  \r^{\Sigma_{0}}_{M, >_{\q}}\Big(\q_{\gg, -}, \ul_{-}, \ueta_{-}; (\rho, A_{n} + T_{n}, B_{n} + T_{n}, l_{n}, \eta_{n}, \zeta_{n})\Big).      
\end{align*}   
Suppose 
\[
\r^{\Sigma_{0}}_{M, >_{\q}}\Big(\q_{-}, \ul_{-}, \ueta_{-}; (\rho, A_{n} + 1, B_{n} - 1, l_{n} + 1, \eta_{n}, \zeta_{n})\Big) \neq 0.
\]
Let
\begin{align*}
\Jac_{X_{>n}} & := \circ_{i > n} \Jac_{(\rho, A_{i} + T_{i}, B_{i} + T_{i}, \zeta_{i}) \mapsto (\rho, A_{i}, B_{i}, \zeta_{i})} \\
\Jac_{X'_{n}} & := \Jac_{(\rho, A_{n} + T_{n} + 1, B_{n} + T_{n} - 1, \zeta_{n}) \mapsto (\rho, A_{n} +1, B_{n} -1, \zeta_{n})} \\
\Jac_{X_{<n}} & := \circ_{i < n} \Jac_{(\rho, A_{i} + T_{i}, B_{i} + T_{i}, \zeta_{i}) \mapsto (\rho, A_{i}, B_{i}, \zeta_{i})} 
\end{align*}
where $i$ decreases. Then after we apply 
\[
\Jac_{X_{>n}} \circ \Jac_{X'_{n}} \circ \Jac_{X_{<n}}
\]
and $\Jac_{X^{c}}$ to the full induced representation
\begin{align}
\label{eq: expand}
\langle \zeta_{n} (B_{n} + T_{n} -1), \cdots, -\zeta_{n}(A_{n} + T_{n} + 1) \rangle \rtimes  \r^{\Sigma_{0}}_{M, >_{\q}}\Big(\q_{\gg, -}, \ul_{-}, \ueta_{-}; (\rho, A_{n} + T_{n}, B_{n} + T_{n}, l_{n}, \eta_{n}, \zeta_{n})\Big),    
\end{align}
we should get something nonzero. 

For $i < n$, one notes
\[
\zeta_{n} (B_{n} + T_{n} - 1), \zeta_{n}(A_{n} + T_{n} + 1) \notin [\zeta_{i} (A_{i} + T_{i}), \zeta_{i} (B_{i} + 1)].
\] 
So $\Jac_{X_{<n}} \eqref{eq: expand}$ becomes
\begin{align*}
& \langle \zeta_{n} (B_{n} + T_{n} -1), \cdots, -\zeta_{n}(A_{n} + T_{n} + 1) \rangle  \rtimes \Jac_{X_{<n}} \r^{\Sigma_{0}}_{M, >_{\q}}\Big(\q_{\gg, -}, \ul_{-}, \ueta_{-}; (\rho, A_{n} + T, B_{n} + T, l_{n}, \eta_{n}, \zeta_{n})\Big) \\
& = \langle \zeta_{n} (B_{n} + T_{n} -1), \cdots, -\zeta_{n}(A_{n} + T_{n} + 1) \rangle \rtimes \r^{\Sigma_{0}}_{M, >_{\q}}\Big(\q^{(n - 1)}_{\gg, -}, \ul_{-}, \ueta_{-}; (\rho, A_{n} + T, B_{n} + T, l_{n}, \eta_{n}, \zeta_{n})\Big).
\end{align*}
where $\q^{(n - 1)}_{\gg, -}$ is obtained from $\q_{\gg, -}$ by changing $(\rho, A_{i} + T_{i}, B_{i} + T_{i}, \zeta_{i})$ to $(\rho, A_{i}, B_{i}, \zeta_{i})$ for $i < n$.

For $i=n$, we can further write $\Jac_{X'_{n}}$ as
\[
\Jac_{\zeta_{n}(A_{n} + T_{n} +1), \cdots, \zeta_{n}(A_{n} + 2)} \circ \Jac_{(\rho, A_{n} + T_{n}, B_{n} + T_{n}, \zeta_{n}) \mapsto (\rho, A_{n}, B_{n}, \zeta_{n})} \circ \Jac_{\zeta_{n}(B_{n} + T_{n} - 1), \cdots, \zeta_{n}B_{n}}
\]
First, we claim $\Jac_{\zeta_{n}(B_{n} + T_{n} - 1), \cdots, \zeta_{n}B_{n}}$ can only apply to 
\[
\langle \zeta_{n} (B_{n} + T_{n} - 1), \cdots, -\zeta_{n}(A_{n} + T_{n} + 1) \rangle.
\] 
Otherwise, there exists $x \in [\zeta_{n}(B_{n} + T_{n} - 1), \zeta_{n}B_{n}]$ such that 
\[
\Jac_{x} \r^{\Sigma_{0}}_{M, >_{\q}}\Big(\q^{(n - 1)}_{\gg, -}, \ul_{-}, \ueta_{-}; (\rho, A_{n} + T, B_{n} + T, l_{n}, \eta_{n}, \zeta_{n})\Big) \neq 0. 
\]
This can only happen when there exists $i < n$ such that $B_{i} \geqslant B_{n}$ and $\zeta_{i} = \zeta_{n}$, but that contradicts to our assumption. As a result, 
\begin{align*}
& \Jac_{\zeta_{n}(B_{n} + T_{n} - 1), \cdots, \zeta_{n}B_{n}} \circ \Jac_{X_{<n}} \eqref{eq: expand} = \langle \zeta_{n} (B_{n} - 1), \cdots, -\zeta_{n}(A_{n} + T_{n} + 1) \rangle \\
& \rtimes \r^{\Sigma_{0}}_{M, >_{\q}}\Big(\q^{(n - 1)}_{\gg, -}, \ul_{-}, \ueta_{-}; (\rho, A_{n} + T, B_{n} + T, l_{n}, \eta_{n}, \zeta_{n})\Big).
\end{align*}
Secondly, we claim $\Jac_{(\rho, A_{n} + T_{n}, B_{n} + T_{n}, \zeta_{n}) \mapsto (\rho, A_{n}, B_{n}, \zeta_{n})}$ can only apply to 
\[
\r^{\Sigma_{0}}_{M, >_{\q}}\Big(\q^{(n - 1)}_{\gg, -}, \ul_{-}, \ueta_{-}; (\rho, A_{n} + T_{n}, B_{n} + T_{n}, l_{n}, \eta_{n}, \zeta_{n})\Big).
\] 
This is because
\[
\zeta_{n} (B_{n} - 1), \zeta_{n}(A_{n} + T_{n} + 1) \notin [\zeta_{n} (A_{n} + T_{n}), \zeta_{n}(B_{n} + 1)].
\]
So 
\begin{align*}
&\Jac_{(\rho, A_{n} + T_{n}, B_{n} + T_{n}, \zeta_{n}) \mapsto (\rho, A_{n}, B_{n}, \zeta_{n})} \circ \Jac_{\zeta_{n}(B_{n} + T_{n} - 1), \cdots, \zeta_{n}(B_{n})} \circ \Jac_{X_{<n}} \eqref{eq: expand} = \\
& \langle \zeta_{n} (B_{n} - 1), \cdots, -\zeta_{n}(A_{n} + T_{n} + 1) \rangle \rtimes \r^{\Sigma_{0}}_{M, >_{\q}}\Big(\q^{(n - 1)}_{\gg, -}, \ul_{-}, \ueta_{-}; (\rho, A_{n}, B_{n}, l_{n}, \eta_{n}, \zeta_{n})\Big).
\end{align*}
Thirdly, $\Jac_{\zeta_{n}(A_{n} + T_{n} +1), \cdots, \zeta_{n}(A_{n} + 2)}$ can only apply to 
\[
\langle \zeta_{n} (B_{n} - 1), \cdots, -\zeta_{n}(A_{n} + T_{n} + 1) \rangle
\] 
for the same reason as before, so
\begin{align*}
\Jac_{X'_{n}} \circ \Jac_{X_{<n}} \eqref{eq: expand} = \langle \zeta_{n} (B_{n} - 1), \cdots, -\zeta_{n}(A_{n} + 1) \rangle \rtimes \r^{\Sigma_{0}}_{M, >_{\q}}\Big(\q^{(n - 1)}_{\gg, -}, \ul_{-}, \ueta_{-}; (\rho, A_{n}, B_{n}, l_{n}, \eta_{n}, \zeta_{n})\Big).
\end{align*}

For $i > n$, $\Jac_{X_{>n}}$ can only apply to $\r^{\Sigma_{0}}_{M, >_{\q}}\Big(\q^{n - 1}_{\gg, -}, \ul_{-}, \ueta_{-}; (\rho, A_{n}, B_{n}, l_{n}, \eta_{n}, \zeta_{n})\Big)$ as $B_{i} > A_{n} + 1$. Therefore, 
\[
\Jac_{X^{c}} \circ \Jac_{X_{>n}} \circ \Jac_{X'_{n}} \circ \Jac_{X_{<n}} \eqref{eq: expand} = \langle \zeta_{n} (B_{n} - 1), \cdots, -\zeta_{n}(A_{n} + 1) \rangle \rtimes \r^{\Sigma_{0}}_{M, >_{\q}}(\q, \ul, \ueta) \neq 0.
\]
Hence $\r^{\Sigma_{0}}_{M, >_{\q}}(\q, \ul, \ueta) \neq 0$.

Next, we suppose $\r^{\Sigma_{0}}_{M, >_{\q}}(\q, \ul, \ueta) \neq 0$. Let
\begin{align*}
\mathcal{C}_{X_{>n}} & := \times_{i > n} \begin{pmatrix}
              \zeta_{i} (B_{i} + T_{i}) & \cdots & \zeta_{i}(B_{i} + 1) \\
              \vdots &  & \vdots \\
              \zeta_{i} (A_{i} + T_{i}) & \cdots & \zeta_{i}(A_{i} + 1)
       \end{pmatrix},
\end{align*}

\begin{align*}
\mathcal{C}_{X_{<n}} & := \times_{i < n} \begin{pmatrix}
              \zeta_{i} (B_{i} + T_{i}) & \cdots & \zeta_{i}(B_{i} + 1) \\
              \vdots &  & \vdots \\
              \zeta_{i} (A_{i} + T_{i}) & \cdots & \zeta_{i}(A_{i} + 1)
       \end{pmatrix},
\end{align*}
where $i$ increases, and
\begin{align*}
\mathcal{C}_{X_{n}} & := \begin{pmatrix}
              \zeta_{n} (B_{n} + T_{n}) & \cdots & \zeta_{n}(B_{n} + 1) \\
              \vdots &  & \vdots \\
              \zeta_{n} (A_{n} + T_{n}) & \cdots & \zeta_{n}(A_{n} + 1)
       \end{pmatrix}.
\end{align*}
Then 
\begin{align*}
& \r^{\Sigma_{0}}_{M, >_{\q}}\Big(\q_{\gg, -}, \ul_{-}, \ueta_{-}; (\rho, A_{n} + T_{n} + 1, B_{n} + T_{n} - 1, l_{n} + 1, \eta_{n}, \zeta_{n})\Big) \hookrightarrow \\ 
& \langle \zeta_{n} (B_{n} + T_{n} -1), \cdots, -\zeta_{n}(A_{n} + T_{n} + 1) \rangle \times  \mathcal{C}_{X_{<n}} \times \mathcal{C}_{X_{n}} \times \mathcal{C}_{X_{>n}} \times \mathcal{C}_{X^{c}} \rtimes \r^{\Sigma_{0}}_{M, >_{\q}}(\q, \ul, \ueta) \cong \\ 
& \mathcal{C}_{X^{c}} \times \mathcal{C}_{X_{<n}} \times  \langle \zeta_{n} (B_{n} + T_{n} -1), \cdots, \zeta_{n}B_{n} \rangle \times \mathcal{C}_{X_{n}} \times \mathcal{C}_{X_{>n}} \times \langle \zeta_{n} (B_{n} -1), \cdots, - \zeta_{n}(A_{n} + 1) \rangle \times \\ 
& \langle -\zeta_{n} (A_{n} + 2), \cdots, -\zeta_{n}(A_{n} + T_{n} + 1) \rangle \rtimes \r^{\Sigma_{0}}_{M, >_{\q}}(\q, \ul, \ueta).
\end{align*}  
By \eqref{eq: dualizing classical}, we can take the dual of $\langle -\zeta_{n} (A_{n} + 2), \cdots, -\zeta_{n}(A_{n} + T_{n} + 1) \rangle$. Hence
\begin{align*}
& \r^{\Sigma_{0}}_{M, >_{\q}}\Big(\q_{\gg, -}, \ul_{-}, \ueta_{-}; (\rho, A_{n} + T_{n} + 1, B_{n} + T_{n} - 1, l_{n} + 1, \eta_{n}, \zeta_{n})\Big) \hookrightarrow \\ 
& \mathcal{C}_{X^{c}} \times \mathcal{C}_{X_{<n}} \times \langle \zeta_{n} (B_{n} + T_{n} -1), \cdots, \zeta_{n}B_{n} \rangle \times \mathcal{C}_{X_{n}} \times 
 \mathcal{C}_{X_{>n}} \times \langle \zeta_{n} (B_{n} -1), \cdots, - \zeta_{n}(A_{n} + 1) \rangle  \times \\ 
 & \langle \zeta_{n}(A_{n} + T_{n} + 1), \cdots, \zeta_{n} (A_{n} + 2) \rangle  \rtimes \r^{\Sigma_{0}}_{M, >_{\q}}(\q, \ul, \ueta) \cong \\ 
& \mathcal{C}_{X^{c}} \times \mathcal{C}_{X_{<n}} \times  \langle \zeta_{n} (B_{n} + T_{n} -1), \cdots, \zeta_{n}B_{n} \rangle \times \mathcal{C}_{X_{n}} \times \langle \zeta_{n}(A_{n} + T_{n} + 1), \cdots, \zeta_{n} (A_{n} + 2) \rangle \times \\ 
& \mathcal{C}_{X_{>n}} \times \langle \zeta_{n} (B_{n} -1), \cdots, - \zeta_{n}(A_{n} + 1) \rangle  \rtimes \r^{\Sigma_{0}}_{M, >_{\q}}(\q, \ul, \ueta)
\end{align*} 
Therefore,
\begin{align*}
& \r^{\Sigma_{0}}_{M, >_{\q}}\Big(\q_{-}, \ul_{-}, \ueta_{-}; (\rho, A_{n} + 1, B_{n} -1, l_{n} + 1, \eta_{n}, \zeta_{n})\Big) \hookrightarrow  \\
& \langle \zeta_{n} (B_{n} -1), \cdots, - \zeta_{n}(A_{n} + 1)\rangle \rtimes \r^{\Sigma_{0}}_{M, >_{\q}}(\q, \ul, \ueta).    
\end{align*}  
This proves the proposition in case $t = 1$, except for the uniqueness and the statement about Jacquet modules.

In fact, the first part of the proposition follows easily from that of case $t = 1$. Moreover, we have 
\begin{align*}
& \r^{\Sigma_{0}}_{M, >_{\q}}\Big(\q_{-}, \ul_{-}, \ueta_{-}; (\rho, A_{n} + t, B_{n} -t, l_{n} + t, \eta_{n}, \zeta_{n})\Big) \hookrightarrow  \\
& \langle \zeta_{n} (B_{n} -t), \cdots, - \zeta_{n}(A_{n} + t) \rangle  \times \cdots \times \langle \zeta_{n} (B_{n} -1), \cdots, - \zeta_{n}(A_{n} + 1) \rangle  \rtimes \r^{\Sigma_{0}}_{M, >_{\q}}(\q, \ul, \ueta).    
\end{align*}  
Then there exists an irreducible constituent $\tau$ of  
\[
\langle \zeta_{n} (B_{n} -t), \cdots, - \zeta_{n}(A_{n} + t) \rangle  \times \cdots \times \langle \zeta_{n} (B_{n} -1), \cdots, - \zeta_{n}(A_{n} + 1) \rangle
\] 
such that 
\[
\r^{\Sigma_{0}}_{M, >_{\q}}\Big(\q_{-}, \ul_{-}, \ueta_{-}; (\rho, A_{n} + t, B_{n} -t, l_{n} + t, \eta_{n}, \zeta_{n})\Big) \hookrightarrow \tau \rtimes \r^{\Sigma_{0}}_{M, >_{\q}}(\q, \ul, \ueta).
\]
We claim
\[
\tau = \begin{pmatrix}
              \zeta_{n} (B_{n} - t) & \cdots & -\zeta_{n}(A_{n} + t) \\
              \vdots &  & \vdots \\
              \zeta_{n} (B_{n} - 1)  & \cdots & -\zeta_{n}(A_{n} + 1)
       \end{pmatrix}.
\]
Otherwise, $\Jac_{x} \tau \neq 0$ for some $x$ in $]\zeta_{n}(B_{n} - t), \zeta_{n} (B_{n} - 1)]$, and hence  
\[
\Jac_{x} \r^{\Sigma_{0}}_{M, >_{\q}}\Big(\q_{-}, \ul_{-}, \ueta_{-}; (\rho, A_{n} + t, B_{n} -t, l_{n} + t, \eta_{n}, \zeta_{n})\Big) \neq 0.
\]
This means there exists $i < n$ such that $B_{i} > B_{n} - t$ and $\zeta_{i} = \zeta_{n}$, which contradicts to our assumption. 

Finally, since $A_{n} \geqslant A_{i}$ for $i < n$, after we apply 
\[
\Jac_{\zeta_{n} (B_{n} -1), \cdots, - \zeta_{n}(A_{n} + 1)} \circ \cdots \circ \Jac_{\zeta_{n} (B_{n} -t), \cdots, - \zeta_{n}(A_{n} + t)}  
\]
to the full induced representation $\tau \rtimes \r^{\Sigma_{0}}_{M, >_{\q}}(\q, \ul, \ueta)$, we will get $\r^{\Sigma_{0}}_{M, >_{\q}}(\q, \ul, \ueta)$. So 
\[
\r^{\Sigma_{0}}_{M, >_{\q}}\Big(\q_{-}, \ul_{-}, \ueta_{-}; (\rho, A_{n} + t, B_{n} -t, l_{n} + 1, \eta_{n}, \zeta_{n})\Big) \hookrightarrow \tau \rtimes \r^{\Sigma_{0}}_{M, >_{\q}}(\q, \ul, \ueta)      
\] 
as the unique irreducible subrepresentation, and
\begin{align*}
\r^{\Sigma_{0}}_{M, >_{\q}}(\q, \ul, \ueta) = \circ_{i \in [1, t]} \Jac_{\zeta_{n} (B_{n} -i), \cdots, - \zeta_{n}(A_{n} + i)} \r^{\Sigma_{0}}_{M, >_{\q}}\Big(\q_{-}, \ul_{-}, \ueta_{-}; (\rho, A_{n} + t, B_{n} -t, l_{n} + t, \eta_{n}, \zeta_{n})\Big).
\end{align*}
So we have finished the proof.

\end{proof}

\subsection{Change sign}

We choose an admissible order $>_{\q}$, and we also fix a self-dual unitary irreducible supercuspidal representation $\rho$ of $GL(d_{\rho})$. We index the Jordan blocks in $Jord_{\rho}(\q)$ such that 
\[
(\rho, A_{i+1}, B_{i+1}, \zeta_{i+1}) >_{\q} (\rho, A_{i}, B_{i}, \zeta_{i}).
\]
Suppose there exists $n$ such that for $i > n$, 
\[
(\rho, A_{i}, B_{i}, \zeta_{i}) \gg \cup_{j=1}^{n}\{(\rho, A_{j}, B_{j}, \zeta_{j})\}.
\]
Moreover 
\[
\text{ for $1 < i \leqslant n$, $A_{1} \geqslant A_{i}$, $B_{1} = 1/2$ or $0$, and $\zeta_{i} \neq \zeta_{1}$.}
\]
We define $\q_{-}$ by
\[
Jord(\q_{-}) = Jord(\q) \backslash \{(\rho, A_{1}, B_{1}, \zeta_{1})\}.
\]
We denote the restriction of $(\ul, \ueta)$ to $Jord(\q_{-})$ by $(\ul_{-}, \ueta_{-})$.

\subsubsection{$B_{1} = 0$}

\begin{proposition}
\label{prop: change sign integral}
For any $(\ul, \ueta)$, $\r^{\Sigma_{0}}_{M, >_{\q}}(\q, \ul, \ueta) \neq 0$ if and only if
\[
\r^{\Sigma_{0}}_{M, >_{\q}}\Big(\q_{-}, \ul_{-}, \ueta_{-}; (\rho, A_{1}, 0, l_{1}, \eta_{1}, -\zeta_{1})\Big) \neq 0.
\]
Moreover,
\begin{align}
\label{eq: change sign integral}
\r^{\Sigma_{0}}_{M, >_{\q}}(\q, \ul, \ueta) = \r^{\Sigma_{0}}_{M, >_{\q}}\Big(\q_{-}, \ul_{-}, \ueta_{-}; (\rho, A_{1}, 0, l_{1}, \eta_{1}, -\zeta_{1})\Big).
\end{align}

\end{proposition}

\begin{proof}
Let $\q_{\gg}$ dominates $\q$ with discrete diagonal restriction such that $T_{1} = 0$. It suffices to prove \eqref{eq: change sign integral} for $\q_{\gg}$. When $l_{1} = 0$, \eqref{eq: change sign integral} is clear. So we can further assume $l_{1} \neq 0$. Let $\q_{\gg, -}$ be obtained from $\q_{\gg}$ by removing $(\rho, A_{1}, B_{1}, \zeta_{1})$. Then
\begin{align*}
& \r^{\Sigma_{0}}_{M, >_{\q}}\Big(\q_{\gg, -}, \ul_{-}, \ueta_{-}; (\rho, A_{1}, 0, l_{1}, \eta_{1}, -\zeta_{1})\Big) \hookrightarrow       
       \langle 0, \cdots, \zeta_{1} A_{1} \rangle \times \langle - \zeta_{1} 1, \cdots, - \zeta_{1} (A_{1}-1) \rangle \\ 
& \rtimes \r^{\Sigma_{0}}_{M, >_{\q}}\Big(\q_{\gg, -}, \ul_{-}, \ueta_{-}; (\rho, A_{1} - 2, 0, l_{1} - 1, \eta_{1}, - \zeta_{1})\Big).  
\end{align*}
as the unique irreducible subrepresentation. On the other hand,
\begin{align*}
\r^{\Sigma_{0}}_{M, >_{\q}}(\q_{\gg}, \ul, \ueta) & \hookrightarrow       
       \langle 0, \cdots, - \zeta_{1} A_{1}\rangle \times \langle \zeta_{1} 1, \cdots, \zeta_{1} (A_{1}-1) \rangle \\
& \rtimes \r^{\Sigma_{0}}_{M, >_{\q}}\Big(\q_{\gg, -}, \ul_{-}, \ueta_{-}; (\rho, A_{1} - 2, 0, l_{1} - 1, \eta_{1}, \zeta_{1})\Big) \\ 
& \hookrightarrow \rho \times \langle -\zeta_{1} 1 \cdots, - \zeta_{1} A_{1} \rangle \times \langle \zeta_{1} 1, \cdots, \zeta_{1} (A_{1}-1) \rangle \\
& \rtimes \r^{\Sigma_{0}}_{M, >_{\q}}\Big(\q_{\gg, -}, \ul_{-}, \ueta_{-}; (\rho, A_{1} - 2, 0, l_{1} - 1, \eta_{1}, \zeta_{1})\Big) \\ 
& \cong \rho \times \langle \zeta_{1} 1, \cdots, \zeta_{1} (A_{1}-1) \rangle \times \langle  -\zeta_{1} 1 \cdots, - \zeta_{1} A_{1} \rangle \\
& \rtimes \r^{\Sigma_{0}}_{M, >_{\q}}\Big(\q_{\gg, -}, \ul_{-}, \ueta_{-}; (\rho, A_{1} - 2, 0, l_{1} - 1, \eta_{1}, \zeta_{1})\Big) \\ 
& \hookrightarrow \rho \times \langle \zeta_{1} 1, \cdots, \zeta_{1} (A_{1}-1) \rangle \times \langle -\zeta_{1} 1 \cdots, - \zeta_{1} (A_{1} - 1)\rangle \times \rho||^{- \zeta_{1} A_{1}} \\
& \rtimes \r^{\Sigma_{0}}_{M, >_{\q}}\Big(\q_{\gg, -}, \ul_{-}, \ueta_{-}; (\rho, A_{1} - 2, 0, l_{1} - 1, \eta_{1}, \zeta_{1})\Big).
\end{align*}
Since $\rho||^{- \zeta_{1} A_{1}} \rtimes \r^{\Sigma_{0}}_{M, >_{\q}}\Big(\q_{\gg, -}, \ul_{-}, \ueta_{-}; (\rho, A_{1} - 2, 0, l_{1} - 1, \eta_{1}, \zeta_{1})\Big)$ is irreducible, we have
\begin{align*}
\r^{\Sigma_{0}}_{M, >_{\q}}(\q_{\gg}, \ul, \ueta) 
& \hookrightarrow \rho \times \langle \zeta_{1} 1, \cdots, \zeta_{1} (A_{1}-1)\rangle \times \langle -\zeta_{1} 1 \cdots, - \zeta_{1} (A_{1} - 1)\rangle \times \rho||^{\zeta_{1} A_{1}} \\
& \rtimes \r^{\Sigma_{0}}_{M, >_{\q}}\Big(\q_{\gg, -}, \ul_{-}, \ueta_{-}; (\rho, A_{1} - 2, 0, l_{1} - 1, \eta_{1}, \zeta_{1})\Big) \\
& \cong \rho \times \langle \zeta_{1} 1, \cdots, \zeta_{1} (A_{1}-1) \rangle \times \rho||^{\zeta_{1} A_{1}} \times \langle -\zeta_{1} 1 \cdots, - \zeta_{1} (A_{1} - 1) \rangle  \\
& \rtimes \r^{\Sigma_{0}}_{M, >_{\q}}\Big(\q_{\gg, -}, \ul_{-}, \ueta_{-}; (\rho, A_{1} - 2, 0, l_{1} - 1, \eta_{1}, \zeta_{1})\Big).
\end{align*}
Since $\Jac_{x}\r^{\Sigma_{0}}_{M, >_{\q}}(\q_{\gg}, \ul, \ueta) = 0$ for $x = \zeta_{1} 1, \zeta_{1} A_{1}$, then
\begin{align*}
\r^{\Sigma_{0}}_{M, >_{\q}}(\q_{\gg}, \ul, \ueta) 
& \hookrightarrow \langle 0, \cdots, \zeta_{1} A_{1} \rangle \times \langle -\zeta_{1} 1 \cdots, - \zeta_{1} (A_{1} - 1) \rangle  \\
& \rtimes \r^{\Sigma_{0}}_{M, >_{\q}}\Big(\q_{\gg, -}, \ul_{-}, \ueta_{-}; (\rho, A_{1} - 2, 0, l_{1} - 1, \eta_{1}, \zeta_{1})\Big).
\end{align*}
By induction on $l_{1}$, we can assume 
\[
\r^{\Sigma_{0}}_{M, >_{\q}}\Big(\q_{\gg, -}, \ul_{-}, \ueta_{-}; (\rho, A_{1} - 2, 0, l_{1} - 1, \eta_{1}, \zeta_{1})\Big) = \r^{\Sigma_{0}}_{M, >_{\q}}\Big(\q_{\gg, -}, \ul_{-}, \ueta_{-}; (\rho, A_{1} - 2, 0, l_{1} - 1, \eta_{1}, - \zeta_{1})\Big).
\]
Then we necessarily have
\[
\r^{\Sigma_{0}}_{M, >_{\q}}(\q_{\gg}, \ul, \ueta) = \r^{\Sigma_{0}}_{M, >_{\q}}\Big(\q_{\gg, -}, \ul_{-}, \ueta_{-}; (\rho, A_{1}, 0, l_{1}, \eta_{1}, -\zeta_{1})\Big).
\]
This finishes the proof.

\end{proof}

\subsubsection{$B_{1} = 1/2$}

\begin{proposition}
\label{prop: change sign half-integral}
For any $(\ul, \ueta)$, one can construct 
\begin{align*}
\r^{\Sigma_{0}}_{M, >_{\q}}(\q^{*}, \ul^{*}, \ueta^{*}) := 
                              \begin{cases}
                               \r^{\Sigma_{0}}_{M, >_{\q}}\Big(\q_{-}, \ul_{-}, \ueta_{-}; (\rho, A_{1} + 1, 1/2, l_{1} + 1, - \eta_{1}, -\zeta_{1})\Big) & \text{ if } \eta_{1} = +1, \\
                               \r^{\Sigma_{0}}_{M, >_{\q}}\Big(\q_{-}, \ul_{-}, \ueta_{-}; (\rho, A_{1} + 1, 1/2, l_{1}, - \eta_{1}, -\zeta_{1})\Big) & \text{ if } \eta_{1} = -1.
                              \end{cases}
\end{align*}
In case $l_{1} = (A_{1} + \frac{1}{2})/2$, we fix $\eta_{1} = -1$. Then 
\[
\r^{\Sigma_{0}}_{M, >_{\q}}(\q, \ul, \ueta) \neq 0 \text{ if and only if } \r^{\Sigma_{0}}_{M, >_{\q}}(\q^{*}, \ul^{*}, \ueta^{*}) \neq 0.
\] 
Moreover,
\begin{align*}
\r^{\Sigma_{0}}_{M, >_{\q}}(\q^{*}, \ul^{*}, \ueta^{*}) \hookrightarrow \langle -\zeta_{1}1/2, \cdots, -\zeta_{1}(A_{1} + 1) \rangle \rtimes  \r^{\Sigma_{0}}_{M, >_{\q}}(\q, \ul, \ueta)      
\end{align*}
as the unique irreducible subrepresentation, and
\[
\r^{\Sigma_{0}}_{M, >_{\q}}(\q, \ul, \ueta) = \Jac_{-\zeta_{1}1/2, \cdots, -\zeta_{1}(A_{1}+1)} \r^{\Sigma_{0}}_{M, >_{\q}}(\q^{*}, \ul^{*}, \ueta^{*}).
\]

\begin{proof}
Let us choose $\q_{\gg}$ dominating $\q$ with discrete diagonal restriction, and we require $T_{1} = 0$. Then
it determines $\q^{*}_{\gg}$ which dominates $\q^{*}$. We will assume the proposition for $\q_{\gg}$. Then
\begin{align*}
\r^{\Sigma_{0}}_{M, >_{\q}}(\q^{*}_{\gg}, \ul^{*}, \ueta^{*}) \hookrightarrow \langle -\zeta_{1}1/2, \cdots, -\zeta_{1}(A_{1} + 1) \rangle \rtimes  \r^{\Sigma_{0}}_{M, >_{\q}}(\q_{\gg}, \ul, \ueta).
\end{align*}

Suppose $\r^{\Sigma_{0}}_{M, >_{\q}}(\q^{*}, \ul^{*}, \ueta^{*}) \neq 0$, then after we apply 
\[
\circ_{i>1} \Jac_{(\rho, A_{i} + T_{i}, B_{i} + T_{i}, \zeta_{i}) \mapsto (\rho, A_{i}, B_{i}, \zeta_{i})}
\]
($i$ decreases) and $\Jac_{X^{c}}$ to the full induced representation
\begin{align}
\label{eq: change sign}
\langle -\zeta_{1}1/2, \cdots, -\zeta_{1}(A_{1} + 1) \rangle \rtimes  \r^{\Sigma_{0}}_{M, >_{\q}}(\q_{\gg}, \ul, \ueta)
\end{align} 
we should get something nonzero. Since 
\[
-\zeta_{1}1/2 \text{ and } \zeta_{1}(A_{1} + 1) \notin [\zeta_{i}(A_{i} + T_{i}), \zeta_{i}(B_{i} + 1)]
\]
for $i > 1$,  $\circ_{i>1} \Jac_{(\rho, A_{i} + T_{i}, B_{i} + T_{i}, \zeta_{i}) \mapsto (\rho, A_{i}, B_{i}, \zeta_{i})}$ and $\Jac_{X^{c}}$ can only apply to $\r^{\Sigma_{0}}_{M, >_{\q}}(\q_{\gg}, \ul, \ueta)$. Therefore
\[
\circ_{i>1} \Jac_{(\rho, A_{i} + T_{i}, B_{i} + T_{i}, \zeta_{i}) \mapsto (\rho, A_{i}, B_{i}, \zeta_{i})} \circ \Jac_{X^{c}} \eqref{eq: change sign} = \langle -\zeta_{1}1/2, \cdots, -\zeta_{1}(A_{1} + 1) \rangle \rtimes  \r^{\Sigma_{0}}_{M, >_{\q}}(\q, \ul, \ueta) \neq 0.
\]
This shows $\r^{\Sigma_{0}}_{M, >_{\q}}(\q, \ul, \ueta) \neq 0$.

Suppose $\r^{\Sigma_{0}}_{M, >_{\q}}(\q, \ul, \ueta) \neq 0$, let
\begin{align*}
\mathcal{C}_{X_{i}} & := \begin{pmatrix}
              \zeta_{i} (B_{i} + T_{i}) & \cdots & \zeta_{i}(B_{i} + 1) \\
              \vdots &  & \vdots \\
              \zeta_{i} (A_{i} + T_{i}) & \cdots & \zeta_{i}(A_{i} + 1)
       \end{pmatrix},
\end{align*}
then
\begin{align*}
\r^{\Sigma_{0}}_{M, >_{\q}}(\q^{*}_{\gg}, \ul^{*}, \ueta^{*}) \hookrightarrow \langle -\zeta_{1}1/2, \cdots, -\zeta_{1}(A_{1} + 1) \rangle \times (\times_{i>1} \mathcal{C}_{X_{i}}) \times \mathcal{C}_{X^{c}} \rtimes  \r^{\Sigma_{0}}_{M, >_{\q}}(\q, \ul, \ueta),
\end{align*}
where $i$ increases. For $i > n$, we have $B_{i} > A_{1} + 1$, so $\mathcal{C}_{X_{i}}$ and $<-\zeta_{1}1/2, \cdots, -\zeta_{1}(A_{1} + 1)>$ are interchangeable. For $i <n$, we have $A_{1} \geqslant A_{i}$ and $\zeta_{i} = -\zeta_{1}$, so
\[
[\zeta_{i}(B_{i} + 1), \zeta_{i}(A_{i}+1)] \subseteq [-\zeta_{1}1/2, -\zeta_{1}(A_{1} + 1)].
\]
It follows $\mathcal{C}_{X_{i}}$ and $\langle -\zeta_{1}1/2, \cdots, -\zeta_{1}(A_{1} + 1) \rangle$ are also interchangeable. Therefore,
\begin{align*}
\r^{\Sigma_{0}}_{M, >_{\q}}(\q^{*}_{\gg}, \ul^{*}, \ueta^{*}) \hookrightarrow (\times_{i>1} \mathcal{C}_{X_{i}}) \times \mathcal{C}_{X^{c}} \times \langle -\zeta_{1}1/2, \cdots, -\zeta_{1}(A_{1} + 1) \rangle  \rtimes  \r^{\Sigma_{0}}_{M, >_{\q}}(\q, \ul, \ueta).
\end{align*}
This implies $\r^{\Sigma_{0}}_{M, >_{\q}}(\q^{*}, \ul^{*}, \ueta^{*}) \neq 0$, and
\[
\r^{\Sigma_{0}}_{M, >_{\q}}(\q^{*}, \ul^{*}, \ueta^{*}) \hookrightarrow \langle -\zeta_{1}1/2, \cdots, -\zeta_{1}(A_{1} + 1) \rangle \rtimes  \r^{\Sigma_{0}}_{M, >_{\q}}(\q, \ul, \ueta).      
\]
To see $\r^{\Sigma_{0}}_{M, >_{\q}}(\q^{*}, \ul^{*}, \ueta^{*})$ is the unique irreducible subrepresentation, it suffices to check that
\[
\Jac_{-\zeta_{1}1/2, \cdots, -\zeta_{1}(A_{1}+1)} \Big(\langle -\zeta_{1}1/2, \cdots -\zeta_{1}(A_{1}+1) \rangle \times \r^{\Sigma_{0}}_{M, >_{\q}}(\q, \ul, \ueta)\Big) = \r^{\Sigma_{0}}_{M, >_{\q}}(\q, \ul, \ueta).
\]
As a consequence, we also get
\[
\r^{\Sigma_{0}}_{M, >_{\q}}(\q, \ul, \ueta) = \Jac_{-\zeta_{1}1/2, \cdots, -\zeta_{1}(A_{1}+1)} \r^{\Sigma_{0}}_{M, >_{\q}}(\q^{*}, \ul^{*}, \ueta^{*}).
\]
To complete the proof, we still need to show the proposition for $\q_{\gg}$, and we leave it to the next lemma.

\end{proof}

\begin{lemma}
Proposition~\ref{prop: change sign half-integral} holds for $\q_{\gg}$.
\end{lemma}

\begin{proof}
It is clear that $\r^{\Sigma_{0}}_{M, >_{\q}}(\q_{\gg}, \ul, \ueta) \neq 0$ and $\r^{\Sigma_{0}}_{M, >_{\q}}(\q^{*}_{\gg}, \ul^{*}, \ueta^{*}) \neq 0$ in this case. So we only need to show 
\begin{align*}
\r^{\Sigma_{0}}_{M, >_{\q}}(\q^{*}_{\gg}, \ul^{*}, \ueta^{*}) \hookrightarrow \langle -\zeta_{1}1/2, \cdots, -\zeta_{1}(A_{1} + 1) \rangle \rtimes  \r^{\Sigma_{0}}_{M, >_{\q}}(\q_{\gg}, \ul, \ueta)      
\end{align*}
as the unique irreducible subrepresentation, and
\[
\r^{\Sigma_{0}}_{M, >_{\q}}(\q_{\gg}, \ul, \ueta) = \Jac_{-\zeta_{1}1/2, \cdots, -\zeta_{1}(A_{1}+1)} \r^{\Sigma_{0}}_{M, >_{\q}}(\q^{*}_{\gg}, \ul^{*}, \ueta^{*}).
\]
Let $\q_{\gg, -}$ be obtained from $\q_{\gg}$ by removing $(\rho, A_{1}, B_{1}, \zeta_{1})$. When $l_{1} = 0$ and $\eta_{1} = -1$, the lemma is clear. When $A_{1} = 1/2$, then necessarily $l_{1} = 0$. In this case, if $\eta_{1} = +1$, then 
\begin{align*}
\r^{\Sigma_{0}}_{M, >_{\q}}(\q^{*}_{\gg}, \ul^{*}, \ueta^{*}) &\hookrightarrow \langle -\zeta_{1}1/2, \cdots, \zeta_{1}3/2 \rangle \rtimes  \r^{\Sigma_{0}}_{M, >_{\q}}(\q_{\gg, -}, \ul_{-}, \ueta_{-}) \\
&\hookrightarrow  \rho||^{-\zeta_{1}1/2} \times \rho||^{\zeta_{1}1/2} \times ||^{\zeta_{1}3/2} \rtimes \r^{\Sigma_{0}}_{M, >_{\q}}(\q_{\gg, -}, \ul_{-}, \ueta_{-}) \\
& \cong \rho||^{-\zeta_{1}1/2} \times \rho||^{\zeta_{1}1/2} \times ||^{-\zeta_{1}3/2} \rtimes \r^{\Sigma_{0}}_{M, >_{\q}}(\q_{\gg, -}, \ul_{-}, \ueta_{-}) \\
& \cong \rho||^{-\zeta_{1}1/2} \times ||^{-\zeta_{1}3/2} \times \rho||^{\zeta_{1}1/2} \rtimes \r^{\Sigma_{0}}_{M, >_{\q}}(\q_{\gg, -}, \ul_{-}, \ueta_{-}).
\end{align*}
There exists an irreducible constituent $\sigma$ of $\rho||^{\zeta_{1}1/2} \rtimes \r^{\Sigma_{0}}_{M, >_{\q}}(\q_{\gg, -}, \ul_{-}, \ueta_{-})$ such that 
\begin{align*}
\r^{\Sigma_{0}}_{M, >_{\q}}(\q^{*}_{\gg}, \ul^{*}, \ueta^{*}) \hookrightarrow  \rho||^{-\zeta_{1}1/2} \times ||^{-\zeta_{1}3/2} \rtimes \sigma.
\end{align*}
Since $\Jac_{-\zeta_{1}3/2}\r^{\Sigma_{0}}_{M, >_{\q}}(\q^{*}_{\gg}, \ul^{*}, \ueta^{*}) = 0$, we must have 
\begin{align*}
\r^{\Sigma_{0}}_{M, >_{\q}}(\q^{*}_{\gg}, \ul^{*}, \ueta^{*}) \hookrightarrow \langle-\zeta_{1}1/2, -\zeta_{1}3/2\rangle \rtimes \sigma.
\end{align*}
Suppose $\Jac_{-\zeta_{1}1/2} \sigma \neq 0$, then there exists an irreducible constituent $\sigma'$ of $\Jac_{-\zeta_{1}1/2} \sigma$ such that
\begin{align*}
\r^{\Sigma_{0}}_{M, >_{\q}}(\q^{*}_{\gg}, \ul^{*}, \ueta^{*}) & \hookrightarrow \langle-\zeta_{1}1/2, -\zeta_{1}3/2\rangle \times \rho||^{-\zeta_{1}1/2} \rtimes \sigma' \\
& \cong \rho||^{-\zeta_{1}1/2} \times  \langle-\zeta_{1}1/2, -\zeta_{1}3/2 \rangle \rtimes \sigma'.
\end{align*}
This implies $\Jac_{-\zeta_{1}1/2, -\zeta_{1}1/2} \r^{\Sigma_{0}}_{M, >_{\q}}(\q^{*}_{\gg}, \ul^{*}, \ueta^{*}) \neq 0$, which is impossible. Therefore, we must have $\Jac_{\zeta_{1}1/2} \sigma \neq 0$. In particular, this means 
\(
\sigma = 
\r^{\Sigma_{0}}_{M, >_{\q}}(\q_{\gg}, \ul, \ueta).
\)
So  
\begin{align*}
\r^{\Sigma_{0}}_{M, >_{\q}}(\q^{*}_{\gg}, \ul^{*}, \ueta^{*}) \hookrightarrow  \langle-\zeta_{1}1/2, -\zeta_{1}3/2\rangle \rtimes \r^{\Sigma_{0}}_{M, >_{\q}}(\q_{\gg}, \ul, \ueta).
\end{align*}
To see $\r^{\Sigma_{0}}_{M, >_{\q}}(\q^{*}_{\gg}, \ul^{*}, \ueta^{*})$ is the unique irreducible subrepresentation, it suffices to check that
\[
\Jac_{-\zeta_{1}1/2, -\zeta_{1}3/2} \Big(\langle -\zeta_{1}1/2, -\zeta_{1}3/2 \rangle \rtimes \r^{\Sigma_{0}}_{M, >_{\q}}(\q_{\gg}, \ul, \ueta)\Big) = \r^{\Sigma_{0}}_{M, >_{\q}}(\q_{\gg}, \ul, \ueta).
\]
As a consequence,
\[
\r^{\Sigma_{0}}_{M, >_{\q}}(\q_{\gg}, \ul, \ueta) = \Jac_{-\zeta_{1}1/2, -\zeta_{1}3/2} \r^{\Sigma_{0}}_{M, >_{\q}}(\q^{*}_{\gg}, \ul^{*}, \ueta^{*}).
\]

Next we would like to prove this lemma by induction on $A_{1}$. Let $A_{1} > 1/2$. Suppose $\eta_{1} = +1$, then
\begin{align*}
\r^{\Sigma_{0}}_{M, >_{\q}}(\q^{*}_{\gg}, \ul^{*}, \ueta^{*}) & \hookrightarrow \langle -\zeta_{1}1/2, \cdots, \zeta_{1}(A_{1} + 1) \rangle \rtimes  \r^{\Sigma_{0}}_{M, >_{\q}}\Big(\q_{\gg, -}, \ul_{-}, \ueta_{-}; (\rho, A_{1}, 3/2, l_{1}, -\eta_{1}, -\zeta_{1})\Big) \\
& \hookrightarrow \langle -\zeta_{1}1/2, \cdots, \zeta_{1}(A_{1} + 1) \rangle \times \langle -\zeta_{1}3/2, \cdots, -\zeta_{1}A_{1} \rangle \\
& \rtimes \r^{\Sigma_{0}}_{M, >_{\q}}\Big(\q_{\gg, -}, \ul_{-}, \ueta_{-}; (\rho, A_{1} - 1, 1/2, l_{1}, -\eta_{1}, -\zeta_{1})\Big) \\
& \hookrightarrow \rho||^{-\zeta_{1}1/2} \times \langle -\zeta_{1}3/2, \cdots, -\zeta_{1}A_{1} \rangle \times \langle \zeta_{1}1/2, \cdots, \zeta_{1}A_{1}\rangle \times \rho||^{\zeta_{1}(A_{1}+1)} \\
& \rtimes \r^{\Sigma_{0}}_{M, >_{\q}}\Big(\q_{\gg, -}, \ul_{-}, \ueta_{-}; (\rho, A_{1} - 1, 1/2, l_{1}, -\eta_{1}, -\zeta_{1})\Big) \\
& \cong \rho||^{-\zeta_{1}1/2} \times \langle -\zeta_{1}3/2, \cdots, -\zeta_{1}A_{1}\rangle \times \langle \zeta_{1}1/2, \cdots, \zeta_{1}A_{1}\rangle \times \rho||^{-\zeta_{1}(A_{1}+1)} \\
& \rtimes \r^{\Sigma_{0}}_{M, >_{\q}}\Big(\q_{\gg, -}, \ul_{-}, \ueta_{-}; (\rho, A_{1} - 1, 1/2, l_{1}, -\eta_{1}, -\zeta_{1})\Big) \\
& \cong \rho||^{-\zeta_{1}1/2} \times \langle -\zeta_{1}3/2, \cdots, -\zeta_{1}A_{1}\rangle \times \rho||^{-\zeta_{1}(A_{1}+1)} \times \langle \zeta_{1}1/2, \cdots, \zeta_{1}A_{1} \rangle \\
& \rtimes \r^{\Sigma_{0}}_{M, >_{\q}}\Big(\q_{\gg, -}, \ul_{-}, \ueta_{-}; (\rho, A_{1} - 1, 1/2, l_{1}, -\eta_{1}, -\zeta_{1})\Big) 
\end{align*}
There exists an irreducible constituent $\sigma$ of 
\[
\langle \zeta_{1}1/2, \cdots, \zeta_{1}A_{1}\rangle \rtimes \r^{\Sigma_{0}}_{M, >_{\q}}\Big(\q_{\gg, -}, \ul_{-}, \ueta_{-}; (\rho, A_{1} - 1, 1/2, l_{1}, -\eta_{1}, -\zeta_{1})\Big) 
\] 
such that 
\[
\r^{\Sigma_{0}}_{M, >_{\q}}(\q^{*}_{\gg}, \ul^{*}, \ueta^{*}) \hookrightarrow \rho||^{-\zeta_{1}1/2} \times \langle -\zeta_{1}3/2, \cdots, -\zeta_{1}A_{1} \rangle  \times \rho||^{-\zeta_{1}(A_{1}+1)} \rtimes \sigma
\]
Since $\Jac_{x} \r^{\Sigma_{0}}_{M, >_{\q}}(\q^{*}_{\gg}, \ul^{*}, \ueta^{*}) = 0$ for $x \in [-\zeta_{1}3/2, -\zeta_{1}(A_{1}+1)]$, then 
\begin{align}
\label{eq: change sign half-integral 1}
\r^{\Sigma_{0}}_{M, >_{\q}}(\q^{*}_{\gg}, \ul^{*}, \ueta^{*})  \hookrightarrow \langle -\zeta_{1}1/2, \cdots -\zeta_{1}(A_{1}+1) \rangle \rtimes \sigma.
\end{align}
If we apply $\Jac_{-\zeta_{1}1/2, \cdots -\zeta_{1}(A_{1}+1)}$ to the full induced representation in \eqref{eq: change sign half-integral 1}, then $\Jac_{-\zeta_{1}(A_{1}+1)}$ can only apply to $\langle -\zeta_{1}1/2, \cdots -\zeta_{1}(A_{1}+1) \rangle$. As a consequence, we must have the whole Jacquet functor $\Jac_{-\zeta_{1}1/2, \cdots -\zeta_{1}(A_{1}+1)}$ applied to $\langle -\zeta_{1}1/2, \cdots -\zeta_{1}(A_{1}+1) \rangle$, and hence
\[
\Jac_{-\zeta_{1}1/2, \cdots -\zeta_{1}(A_{1}+1)} \r^{\Sigma_{0}}_{M, >_{\q}}(\q^{*}_{\gg}, \ul^{*}, \ueta^{*}) = \sigma,
\]
which is irreducible. Therefore
\begin{align}
\label{eq: change sign half-integral 2}
\sigma \hookrightarrow \langle \zeta_{1}1/2, \cdots, \zeta_{1}A_{1} \rangle \rtimes \r^{\Sigma_{0}}_{M, >_{\q}}\Big(\q_{\gg, -}, \ul_{-}, \ueta_{-}; (\rho, A_{1} - 1, 1/2, l_{1}, -\eta_{1}, -\zeta_{1})\Big).
\end{align}
By induction, $\r^{\Sigma_{0}}_{M, >_{\q}}(\q_{\gg}, \ul, \ueta)$ is the unique irreducible subrepresentation of the induced representation in \eqref{eq: change sign half-integral 2}, so it has to be equal to $\sigma$. Hence
\[
\r^{\Sigma_{0}}_{M, >_{\q}}(\q^{*}_{\gg}, \ul^{*}, \ueta^{*})  \hookrightarrow \langle -\zeta_{1}1/2, \cdots -\zeta_{1}(A_{1}+1) \rangle \rtimes \r^{\Sigma_{0}}_{M, >_{\q}}(\q_{\gg}, \ul, \ueta).
\]
To see $\r^{\Sigma_{0}}_{M, >_{\q}}(\q^{*}_{\gg}, \ul^{*}, \ueta^{*})$ is the unique irreducible subrepresentation, it suffices to check that
\[
\Jac_{-\zeta_{1}1/2, \cdots -\zeta_{1}(A_{1}+1)} \Big(\langle -\zeta_{1}1/2, \cdots -\zeta_{1}(A_{1}+1) \rangle \rtimes \r^{\Sigma_{0}}_{M, >_{\q}}(\q_{\gg}, \ul, \ueta)\Big) = \r^{\Sigma_{0}}_{M, >_{\q}}(\q_{\gg}, \ul, \ueta).
\]
As a consequence,
\[
\r^{\Sigma_{0}}_{M, >_{\q}}(\q_{\gg}, \ul, \ueta) = \Jac_{-\zeta_{1}1/2, \cdots, -\zeta_{1}(A_{1}+1)} \r^{\Sigma_{0}}_{M, >_{\q}}(\q^{*}_{\gg}, \ul^{*}, \ueta^{*}).
\]

Suppose $\eta_{1} = -1$, we can also assume $l_{1} \neq 0$, then the proof is the same.

\end{proof}




\end{proposition}

\section{General procedure}
\label{sec: procedure}

The three operations (``Pull", ``Expand", and ``Change sign") introduced in the previous section allow us to develop a procedure to find the combinatorial conditions for the nonvanishing of $\r^{\Sigma_{0}}_{M, >_{\q}}(\q, \ul, \ueta)$. 

\subsection{Step one}
\label{subsec: step 1}

We choose an admissible order $>_{\q}$, and we also fix a self-dual unitary irreducible supercuspidal representation $\rho$ of $GL(d_{\rho})$. We index the Jordan blocks in $Jord_{\rho}(\q)$ such that 
\[
(\rho, A_{i}, B_{i}, \zeta_{i}) >_{\q} (\rho, A_{i-1}, B_{i-1}, \zeta_{i-1}).
\]
We choose $n$ such that for $i > n$, 
\[
(\rho, A_{i}, B_{i}, \zeta_{i}) \gg_{2} \cup_{j=1}^{n}\{(\rho, A_{j}, B_{j}, \zeta_{j})\},
\]
and the Jordan blocks for $i>n$ are in ``good shape" (see Remark~\ref{rk: basic general}). Then for $i \leqslant n$, let us choose $(\rho, A, B, \zeta)$ so that $A$ is maximal. We consider the set
\begin{align}
\label{eq: pull set}
\{(\rho, A_{i}, B_{i}, \zeta_{i}) \text{ for } i \leqslant n:  [A_{i}, B_{i}] \subsetneq [A, B] \text{ and } \zeta_{i} = \zeta \}.
\end{align}
If this set is nonempty, we take $(\rho, A', B', \zeta')$ such that $A'$ is maximal within the set. We can rearrange the order $>_{\q}$ for $i \leqslant n$, so that
\[
(\rho, A_{n}, B_{n}, \zeta_{n}) = (\rho, A, B, \zeta) \text{ and } (\rho, A_{n-1}, B_{n-1}, \zeta_{n-1}) = (\rho, A', B', \zeta').
\]
Then we can ``Pull" the pairs $(\rho, A_{n}, B_{n}, \zeta_{n}), (\rho, A_{n-1}, B_{n-1}, \zeta_{n-1})$ using Proposition~\ref{prop: pull}. Suppose the set \eqref{eq: pull set} is empty, but there exists $(\rho, A', B', \zeta')$ such that 
\[
[A', B']  = [A, B] \text{ and } \zeta' = \zeta,
\]
then we can again rearrange the order $>_{\q}$ for $i \leqslant n$, so that
\[
(\rho, A_{n}, B_{n}, \zeta_{n}) = (\rho, A, B, \zeta) \text{ and } (\rho, A_{n-1}, B_{n-1}, \zeta_{n-1}) = (\rho, A', B', \zeta').
\]
And we can ``Pull" the pairs $(\rho, A_{n}, B_{n}, \zeta_{n}), (\rho, A_{n-1}, B_{n-1}, \zeta_{n-1})$ using Proposition~\ref{prop: pull equal length}.

\subsection{Step two}
\label{subsec: step 2}

Following Step one, we suppose the set 
\begin{align}
\label{eq: pull set equal length}
\{(\rho, A_{i}, B_{i}, \zeta_{i}) \text{ for } i \leqslant n:  [A_{i}, B_{i}] \subseteq [A, B] \text{ and } \zeta_{i} = \zeta \} \backslash \{(\rho, A, B, \zeta)\}
\end{align}
is empty. We can still rearrange the order $>_{\q}$ for $i \leqslant n$ such that
\[
(\rho, A_{n}, B_{n}, \zeta_{n}) = (\rho, A, B, \zeta).
\]
Then we ``Expand" $[A_{n}, B_{n}]$, and use Proposition~\ref{prop: expand}.

\subsection{Step three}
\label{subsec: step 3}

Following Step two, let us denote the ``Expansion" of $[A_{n}, B_{n}]$ by $[A^{*}_{n}, B^{*}_{n}]$. The set \eqref{eq: pull set} becomes
\begin{align*}
\label{eq: pull set expand}
\{(\rho, A_{i}, B_{i}, \zeta_{i}) \text{ for } i < n:  [A_{i}, B_{i}] \subsetneq [A^{*}_{n}, B^{*}_{n}] \text{ and } \zeta_{i} = \zeta_{n} \}.
\end{align*}
If this set is nonempty, then we are back to Step one. If this set is empty, then by our definition of ``Expand", it is necessary that $B^{*}_{n} = 1/2$ or $0$, and $\zeta_{i} \neq \zeta_{n}$ for all $i < n$. In this case, we can change the order for $i \leqslant n$ again by switching $(\rho, A^{*}_{n}, B^{*}_{n}, \zeta_{n})$ with $(\rho, A_{i}, B_{i}, \zeta_{i})$ one by one as $i$ goes from $n-1$ to $1$.
Then we can ``Change sign" of $(\rho, A^{*}_{n}, B^{*}_{n}, \zeta_{n})$, and use Proposition~\ref{prop: change sign integral} or Proposition~\ref{prop: change sign half-integral}. After that, we are back to Step one again.

\subsection{Step four}
\label{subsec: step 4}

By the above three steps, we will end up with a collection of parameters $\{\q^{\star}\}$ such that $Jord_{\rho}(\q^{\star})$ is in ``good shape" (cf. Proposition~\ref{prop: pull} and Proposition~\ref{prop: pull equal length}). 
Then we can change $\rho$ and repeat all the previous steps to $\{\q^{\star}\}$.

\appendix

\section{Proof of Proposition~\ref{prop: basic}}
\label{sec: basic}

In this appendix, we will give the proof of Proposition~\ref{prop: basic}, which also includes a proof of Lemma~\ref{lemma: fundamental case}. We will proceed by induction. So let us suppose the proposition holds when $(A_{1} - B_{1}) + (A_{2} - B_{2}) < L$ for some positive integer $L$. Note when $(A_{1} - B_{1}) + (A_{2} - B_{2}) = 0$, this is clear (cf. Theorem~\ref{thm: M/W}). When $(A_{1} - B_{1}) + (A_{2} - B_{2}) = L$, we will first prove the proposition, except for the necessity of condition \eqref{eq: basic condition} in the case that $A_2 = A_1$ and $B_2 = B_1$. This remaining case is actually part of Lemma~\ref{lemma: fundamental case} and will be treated in the end.

\subsection{Proof of Proposition~\ref{prop: basic}}

We first show the necessity of the condition \eqref{eq: basic condition}. So let us suppose $\r^{\Sigma_{0}}_{M, >_{\q}}(\q, \ul, \ueta) \neq 0$, and we take the following two reduction steps.

\begin{itemize}

\item {\bf First reduction}: we assume $A_{2} > A_{1}$ and $l_{2} \neq 0$.

Let us define $\q_{\gg}$ by shifting $(\rho, A_{2}, B_{2}, \zeta_{2})$ to $(\rho, A_{2} + T, B_{2} + T, \zeta_{2})$, such that $\q_{\gg}$ has discrete diagonal restriction and the natural order is the same as $>_{\q}$. Then
\begin{align*}
\r^{\Sigma_{0}}_{M, >_{\q}}(\q_{\gg}, \ul, \ueta) & \hookrightarrow \langle \zeta(B_{2} + T), \cdots, -\zeta(A_{2} + T) \rangle
      \rtimes \r^{\Sigma_{0}}_{M, >_{\q}}\Big(\q_{-}, \ul_{-}, \ueta_{-}; \\
      & (\rho, A_{2} + T - 1, B_{2} + T + 1, l_{2} - 1, \eta_{2}, \zeta), (\rho, A_{1}, B_{1}, l_{1}, \eta_{1}, \zeta)\Big).       
\end{align*}
Note 
\[
\Jac_{(\rho, A_{2} + T, B_{2} + T, \zeta) \mapsto (\rho, A_{2}, B_{2}, \zeta)} \r^{\Sigma_{0}}_{M, >_{\q}}(\q_{\gg}, \ul, \ueta) = \r^{\Sigma_{0}}_{M, >_{\q}}(\q, \ul, \ueta) \neq 0.
\]
So after applying $\Jac_{(\rho, A_{2} + T, B_{2} + T, \zeta) \mapsto (\rho, A_{2}, B_{2}, \zeta)}$ to the full induced representation above, we have
\[
\langle \zeta B_{2}, \cdots, -\zeta A_{2} \rangle \rtimes \r^{\Sigma_{0}}_{M, >_{\q}}\Big(\q_{-}, \ul_{-}, \ueta_{-}; (\rho, A_{2} - 1, B_{2} + 1, l_{2} - 1, \eta_{2}, \zeta), (\rho, A_{1}, B_{1}, l_{1}, \eta_{1}, \zeta)\Big),
\]
which is again nonzero. In particular, 
\[
\r^{\Sigma_{0}}_{M, >_{\q}}\Big(\q_{-}, \ul_{-}, \ueta_{-}; (\rho, A_{2} - 1, B_{2} + 1, l_{2} - 1, \eta_{2}, \zeta), (\rho, A_{1}, B_{1}, l_{1}, \eta_{1}, \zeta)\Big) \neq 0
\]
By our assumption, $A_{2} - 1 \geqslant A_{1}, B_{2} + 1 > B_{1}$ and $l_{2} - 1 \geqslant 0$, so we get by induction assumption
\begin{align*}
\begin{cases}
\eta_{2} = (-1)^{A_{1} - B_{1}}\eta_{1}       & \Rightarrow (A_{2}-1) - (l_{2}-1) \geqslant A_{1} - l_{1}, (B_{2}+1) + (l_{2}-1) \geqslant B_{1} + l_{1},  \\
\eta_{2} \neq (-1)^{A_{1} - B_{1}}\eta_{1}  & \Rightarrow (B_{2}+1) + (l_{2}-1) > A_{1} - l_{1}.   
\end{cases}                                                                                                                                                                                                                                                                                                                                
\end{align*}
This gives the condition \eqref{eq: basic condition}.

\item {\bf Second reduction}: we assume $B_{2} > B_{1}$ and $l_{1} \neq 0$.

We choose $\q_{\gg}$ as in the previous step. Then
\begin{align*}
\r^{\Sigma_{0}}_{M, >_{\q}}(\q_{\gg}, \ul, \ueta) & \hookrightarrow \langle \zeta B_{1}, \cdots, -\zeta A_{1} \rangle
      \rtimes \r^{\Sigma_{0}}_{M, >_{\q}}\Big(\q_{-}, \ul_{-}, \ueta_{-}; \\
      & (\rho, A_{2} + T, B_{2} + T, l_{2}, \eta_{2}, \zeta), (\rho, A_{1} - 1, B_{1} + 1, l_{1} - 1, \eta_{1}, \zeta)\Big).       
\end{align*}
Note 
\[
\Jac_{(\rho, A_{2} + T, B_{2} + T, \zeta) \mapsto (\rho, A_{2}, B_{2}, \zeta)} \r^{\Sigma_{0}}_{M, >_{\q}}(\q_{\gg}, \ul, \ueta) = \r^{\Sigma_{0}}_{M, >_{\q}}(\q, \ul, \ueta) \neq 0.
\]
So after applying $\Jac_{(\rho, A_{2} + T, B_{2} + T, \zeta) \mapsto (\rho, A_{2}, B_{2}, \zeta)}$ to the full induced representation above, we have
\[
\langle \zeta B_{1}, \cdots, -\zeta A_{1} \rangle \rtimes \r^{\Sigma_{0}}_{M, >_{\q}}\Big(\q_{-}, \ul_{-}, \ueta_{-}; (\rho, A_{2}, B_{2}, l_{2}, \eta_{2}, \zeta), (\rho, A_{1} - 1, B_{1} + 1, l_{1} - 1, \eta_{1}, \zeta)\Big),
\]
which is again nonzero. In particular, 
\[
\r^{\Sigma_{0}}_{M, >_{\q}}\Big(\q_{-}, \ul_{-}, \ueta_{-}; (\rho, A_{2}, B_{2}, l_{2}, \eta_{2}, \zeta), (\rho, A_{1} - 1, B_{1} + 1, l_{1} - 1, \eta_{1}, \zeta)\Big) \neq 0
\]
By our assumption, $A_{2} > A_{1} - 1, B_{2} \geqslant B_{1} + 1$ and $l_{1} - 1 \geqslant 0$, so we get by induction assumption
\begin{align*}
\begin{cases}
\eta_{2} = (-1)^{A_{1} - B_{1}}\eta_{1}       & \Rightarrow A_{2} - l_{2} \geqslant (A_{1} - 1) - (l_{1} - 1), B_{2} + l_{2} \geqslant (B_{1} + 1) + (l_{1} - 1),  \\
\eta_{2} \neq (-1)^{A_{1} - B_{1}}\eta_{1}  & \Rightarrow B_{2} + l_{2} > (A_{1} - 1) - (l_{1} - 1).   
\end{cases}                                                                                                                                                                                                                                                                                                                                 
\end{align*}
This again gives the condition \eqref{eq: basic condition}.

\end{itemize}

After these two steps, we are reduced to the following cases:

\begin{itemize}

\item {\bf Case 1}: $A_{2} = A_{1}$ and $B_{2} = B_{1}$. 

This is the remaining case, which will be treated in the end.

\item {\bf Case 2}: $A_{2} = A_{1}, B_{2} > B_{1}$ and $l_{1} = 0$.

In this case, the condition \eqref{eq: basic condition} becomes 
\[
\eta_{2} = (-1)^{A_{1} - B_{1}}\eta_{1} \text{ and } l_{2} = 0.
\]
Note
\begin{align*}
\r^{\Sigma_{0}}_{M, >_{\q}}(\q, \ul, \ueta) & =  \r^{\Sigma_{0}}_{M, >_{\q}}\Big(\q_{-}, \ul_{-}, \ueta_{-}; (\rho, A_{2}, B_{2}, l_{2}, \eta_{2}, \zeta), \\
& (\rho, A_{1}, B_{2}, 0, (-1)^{B_{2} - B_{1}} \eta_{1}, \zeta), (\rho, B_{2} - 1, B_{1}, 0, \eta_{1}, \zeta)\Big).
\end{align*}
Applying the induction assumption to $(\rho, A_{2}, B_{2}, l_{2}, \eta_{2}, \zeta)$ and $(\rho, A_{1}, B_{2}, 0, (-1)^{B_{2} - B_{1}} \eta_{1}, \zeta)$, we get 
\[
\eta_{2} = (-1)^{A_{1} - B_{2}} \cdot (-1)^{B_{2} - B_{1}} \eta_{1} = (-1)^{A_{1} - B_{1}}\eta_{1}, 
\]
and $l_{2} = 0$. This is exactly what we want.

\item {\bf Case 3}: $A_{2} > A_{1}, B_{2} = B_{1}$ and $l_{2} = 0$.

In this case, the condition \eqref{eq: basic condition} becomes 
\[
\eta_{2} = (-1)^{A_{1} - B_{1}}\eta_{1} \text{ and } l_{1} = 0.
\]
Note 
\begin{align*}
\r^{\Sigma_{0}}_{M, >_{\q}}(\q_{\gg}, \ul, \ueta) & =  \r^{\Sigma_{0}}_{M, >_{\q}}\Big(\q_{-}, \ul_{-}, \ueta_{-}; (\rho, A_{2} + T, A_{1} + T + 1, 0, (-1)^{A_{1} - B_{2} + 1}\eta_{2}, \zeta), \\
& (\rho, A_{1} + T, B_{2} + T, 0, \eta_{2}, \zeta), (\rho, A_{1}, B_{1}, l_{1}, \eta_{1}, \zeta)\Big).
\end{align*}
Since 
\[
\Jac_{(\rho, A_{2} + T, B_{2} + T, \zeta) \mapsto (\rho, A_{2}, B_{2}, \zeta)} = \Jac_{(\rho, A_{2} + T, A_{1} + T + 1, \zeta) \mapsto (\rho, A_{2}, A_{1} + 1, \zeta)} \circ \Jac_{(\rho, A_{1} + T, B_{2} + T, \zeta) \mapsto (\rho, A_{1}, B_{2}, \zeta)},
\]
then
\begin{align*}
&\Jac_{(\rho, A_{1} + T, B_{2} + T, \zeta) \mapsto (\rho, A_{1}, B_{2}, \zeta)} \r^{\Sigma_{0}}_{M, >_{\q}}(\q_{\gg}, \ul, \ueta)  = \\
& \r^{\Sigma_{0}}_{M, >_{\q}}\Big(\q_{-}, \ul_{-}, \ueta_{-}; (\rho, A_{2} + T, A_{1} + T + 1, 0, (-1)^{A_{1} - B_{2} + 1}\eta_{2}, \zeta), \\
& (\rho, A_{1}, B_{2}, 0, \eta_{2}, \zeta), (\rho, A_{1}, B_{1}, l_{1}, \eta_{1}, \zeta)\Big) \neq 0.
\end{align*}
Applying the induction assumption to $(\rho, A_{1}, B_{2}, 0, \eta_{2}, \zeta)$ and $(\rho, A_{1}, B_{1}, l_{1}, \eta_{1}, \zeta)$, we get exactly
\[
\eta_{2} = (-1)^{A_{1} - B_{1}}\eta_{1} \text{ and } l_{1} = 0.
\]

\item {\bf Case 4}: $A_{2} > A_{1}, B_{2} > B_{1}$ and $l_{2} = l_{1} = 0$.

If $\eta_{2} = (-1)^{A_{1} - B_{1}}\eta_{1}$, the condition is automatically satisfied. 

If $\eta_{2} \neq (-1)^{A_{1} - B_{1}}\eta_{1}$. We can suppose $B_{2} \leqslant A_{1}$, and let $T = A_{1} - B_{2} + 1$. One observes 
\[
\r^{\Sigma_{0}}_{M, >_{\q}}(\q_{\gg}, \ul, \ueta)  =  \r^{\Sigma_{0}}_{M, >_{\q}}\Big(\q_{-}, \ul_{-}, \ueta_{-};  (\rho, A_{2} + T, B_{1}, 0, \eta_{1}, \zeta)\Big).
\]
So $\Jac_{\zeta(A_{1}+1)} \r^{\Sigma_{0}}_{M, >_{\q}}(\q_{\gg}, \ul, \ueta) = 0$. Therefore,
\[
\Jac_{(\rho, A_{2} +T, B_{2} + T, \zeta) \mapsto (\rho, A_{2}, B_{2}, \zeta)} \r^{\Sigma_{0}}_{M, >_{\q}}(\q_{\gg}, \ul, \ueta) = 0.
\]

\end{itemize}

Next we would like to show the sufficiency of condition~\eqref{eq: basic condition} by computing $\r^{\Sigma_{0}}_{M, >_{\q}}(\q, \ul, \ueta)$ directly. We will take $\q_{\gg}$ to be defined as before.

\begin{itemize}

\item Suppose $l_{1} = l_{2} = 0$. If $\eta_{2} \neq (-1)^{A_{1} - B_{1}}\eta_{1}$, then $B_{2} > A_{1}$ and there is nothing to prove. So let us also assume $\eta_{2} = (-1)^{A_{1} - B_{1}}\eta_{1}$.

  \begin{enumerate}
  
  \item $A_{2} - B_{2} \leqslant A_{1} - B_{1}$.
  
           \begin{align*}
            \r^{\Sigma_{0}}_{M, >_{\q}}(\q_{\gg}, \ul, \ueta) & \hookrightarrow 
               \begin{pmatrix}
              \zeta (B_{2} + T) & \cdots & -\zeta A_{1} \\
              \vdots &  & \vdots \\
              \zeta (A_{2} + T) & \cdots & - \zeta (A_{1} - A_{2} + B_{2})
              \end{pmatrix} \\        
       & \rtimes \r^{\Sigma_{0}}_{M, >_{\q}}\Big(\q_{-}, \ul_{-}, \ueta_{-}; (\rho, A_{1} - A_{2} + B_{2} -1, B_{1}, 0, \eta_{1}, \zeta)\Big).       
           \end{align*}
   It is clear that $\r^{\Sigma_{0}}_{M, >_{\q}}(\q, \ul, \ueta) \neq 0$.

  \item $A_{2} - B_{2} > A_{1} - B_{1}$.
        
        \begin{align*}
         \r^{\Sigma_{0}}_{M, >_{\q}}(\q_{\gg}, \ul, \ueta) & \hookrightarrow 
            \begin{pmatrix}
              \zeta (B_{2} + T) & \cdots & -\zeta A_{1} \\
              \vdots &  & \vdots \\
              \zeta (B_{2} + A_{1} - B_{1} + T) & \cdots & - \zeta B_{1}
             \end{pmatrix} \\
       & \times       
           \begin{pmatrix}
              \zeta (B_{2} + A_{1} - B_{1} + T + 1)  & \cdots & \zeta (B_{2} + A_{1} - B_{1} + 2)  \\
              \vdots &  & \vdots \\
              \zeta (A_{2} + T) & \cdots & \zeta (A_{2} + 1)
           \end{pmatrix} \\
       & \rtimes \r^{\Sigma_{0}}_{M, >_{\q}}\Big(\q_{-}, \ul_{-}, \ueta_{-}; (\rho, A_{2}, B_{2} + A_{1} - B_{1} + 1, 0, - \eta_{1}, \zeta) \Big).       
      \end{align*}        
     It is again clear that $\r^{\Sigma_{0}}_{M, >_{\q}}(\q, \ul, \ueta) \neq 0$.   
        
  \end{enumerate}

\item Suppose $l_{1} \neq 0$ or $l_{2} \neq 0$.

\begin{align*}
\r^{\Sigma_{0}}_{M, >_{\q}}(\q_{\gg}, \ul, \ueta) & \hookrightarrow 
      \begin{pmatrix}
              \zeta (B_{2} + T) & \cdots & -\zeta (A_{2} + T) \\
              \vdots &  & \vdots \\
              \zeta (B_{2} + l_{2} - 1 + T) & \cdots & - \zeta (A_{2} - l_{2} + 1 + T)
       \end{pmatrix} \\
& \times       
       \begin{pmatrix}
              \zeta B_{1} & \cdots & -\zeta A_{1} \\
              \vdots &  & \vdots \\
              \zeta (B_{1} + l_{1} - 1) & \cdots & - \zeta (A_{1} - l_{1} + 1)
       \end{pmatrix} \\
& \rtimes \r^{\Sigma_{0}}_{M, >_{\q}}\Big(\q_{-}, \ul_{-}, \ueta_{-}; (\rho, A_{2} - l_{2} + T, B_{2} + l_{2} + T, 0, \eta_{2}, \zeta), (\rho, A_{1} - l_{1}, B_{1} + l_{1}, 0, \eta_{1}, \zeta)\Big).       
\end{align*}
From our previous discussion, we know 
\[
\r^{\Sigma_{0}}_{M, >_{\q}}\Big(\q_{-}, \ul_{-}, \ueta_{-}; (\rho, A_{2} - l_{2}, B_{2} + l_{2}, 0, \eta_{2}, \zeta), (\rho, A_{1} - l_{1}, B_{1} + l_{1}, 0, \eta_{1}, \zeta)\Big) \neq 0,
\] 
so
\begin{align*}
&\r^{\Sigma_{0}}_{M, >_{\q}}\Big(\q_{-}, \ul_{-}, \ueta_{-}; (\rho, A_{2} - l_{2} + T, B_{2} + l_{2} + T, 0, \eta_{2}, \zeta), (\rho, A_{1} - l_{1}, B_{1} + l_{1}, 0, \eta_{1}, \zeta)\Big) \\
& \hookrightarrow 
      \begin{pmatrix}
              \zeta (B_{2} + l_{2} + T) & \cdots & \zeta (B_{2} + l_{2} + 1) \\
              \vdots &  & \vdots \\
              \zeta (A_{2} - l_{2} + T) & \cdots & \zeta (A_{2} - l_{2} + 1)
       \end{pmatrix} \\
& \rtimes \r^{\Sigma_{0}}_{M, >_{\q}}\Big(\q_{-}, \ul_{-}, \ueta_{-}; (\rho, A_{2} - l_{2}, B_{2} + l_{2}, 0, \eta_{2}, \zeta), (\rho, A_{1} - l_{1}, B_{1} + l_{1}, 0, \eta_{1}, \zeta)\Big).       
\end{align*}
Therefore,
\begin{align*}
\r^{\Sigma_{0}}_{M, >_{\q}}(\q_{\gg}, \ul, \ueta) & \hookrightarrow 
      \begin{pmatrix}
              \zeta (B_{2} + T) & \cdots & \zeta (B_{2} + 1) \\
              \vdots &  & \vdots  \\ 
              \zeta (B_{2} + l_{2} - 1 + T) & \cdots & \zeta(B_{2} + l_{2}) 
       \end{pmatrix} 
 \times
     \underbrace{\begin{pmatrix}
              \zeta B_{2} & \cdots & -\zeta (A_{2} + T) \\
              \vdots &  & \vdots  \\ 
              \zeta (B_{2} + l_{2} - 1) & \cdots & -\zeta (A_{2} - l_{2} + 1 + T)
       \end{pmatrix}}_{I}  \\
& \times       
       \underbrace{\begin{pmatrix}
              \zeta B_{1} & \cdots & -\zeta A_{1} \\
              \vdots &  & \vdots \\
              \zeta (B_{1} + l_{1} - 1) & \cdots & - \zeta (A_{1} - l_{1} + 1)
       \end{pmatrix}}_{II}
 \times       
       \underbrace{\begin{pmatrix}
              \zeta (B_{2} + l_{2} + T) & \cdots & \zeta (B_{2} + l_{2} + 1) \\
              \vdots &  & \vdots \\
              \zeta (A_{2} - l_{2} + T) & \cdots & \zeta (A_{2} - l_{2} + 1)
       \end{pmatrix}}_{III} \\
& \rtimes \r^{\Sigma_{0}}_{M, >_{\q}}\Big(\q_{-}, \ul_{-}, \ueta_{-}; (\rho, A_{2} - l_{2}, B_{2} + l_{2}, 0, \eta_{2}, \zeta), (\rho, A_{1} - l_{1}, B_{1} + l_{1}, 0, \eta_{1}, \zeta)\Big).       
\end{align*}
Since $[\zeta B_{2}, -\zeta(A_{2} + T)] \supseteq [\zeta B_{1}, -\zeta A_{1}]$, $(I)$ and $(II)$ are interchangeable. Also note $B_{2} + l_{2} + 1 > B_{1} + l_{1}$, so we can interchange $(II)$ and $(III)$. It is clear that $(I)$ and $(III)$ are interchangeable too. As a result, 
\begin{align*}
\r^{\Sigma_{0}}_{M, >_{\q}}(\q_{\gg}, \ul, \ueta) & \hookrightarrow 
      \begin{pmatrix}
              \zeta (B_{2} + T) & \cdots & \zeta (B_{2} + 1) \\
              \vdots &  & \vdots  \\ 
              \zeta (B_{2} + l_{2} - 1 + T) & \cdots & \zeta(B_{2} + l_{2}) 
       \end{pmatrix} \\
& \times       
       \underbrace{\begin{pmatrix}
              \zeta (B_{2} + l_{2} + T) & \cdots & \zeta (B_{2} + l_{2} + 1) \\
              \vdots &  & \vdots \\
              \zeta (A_{2} - l_{2} + T) & \cdots & \zeta (A_{2} - l_{2} + 1)
       \end{pmatrix}}_{III}        
 \times
     \underbrace{\begin{pmatrix}
              \zeta B_{1} & \cdots & -\zeta A_{1} \\
              \vdots &  & \vdots \\
              \zeta (B_{1} + l_{1} - 1) & \cdots & - \zeta (A_{1} - l_{1} + 1)
       \end{pmatrix}}_{II} \\
& \times       
       \underbrace{\begin{pmatrix}
              \zeta B_{2} & \cdots & -\zeta (A_{2} + 1) & \cdots & -\zeta (A_{2} + T) \\
              \vdots &  & \vdots   & & \vdots \\
              \zeta (B_{2} + l_{2} - 1) & \cdots & -\zeta (A_{2} - l_{2} + 2) & \cdots & -\zeta (A_{2} - l_{2} + 1 + T)
       \end{pmatrix}}_{I}  \\              
& \rtimes \r^{\Sigma_{0}}_{M, >_{\q}}\Big(\q_{-}, \ul_{-}, \ueta_{-}; (\rho, A_{2} - l_{2}, B_{2} + l_{2}, 0, \eta_{2}, \zeta), (\rho, A_{1} - l_{1}, B_{1} + l_{1}, 0, \eta_{1}, \zeta)\Big).       
\end{align*}
By Proposition~\ref{prop: general irreducibility}, 
\begin{align*}
&       \underbrace{\begin{pmatrix}
              -\zeta (A_{2} + 1) & \cdots & -\zeta (A_{2} + T) \\
              \vdots   & & \vdots \\
              -\zeta (A_{2} - l_{2} + 2) & \cdots & -\zeta (A_{2} - l_{2} + 1 + T)
       \end{pmatrix}}_{IV}  \\              
& \rtimes \r^{\Sigma_{0}}_{M, >_{\q}}\Big(\q_{-}, \ul_{-}, \ueta_{-}; (\rho, A_{2} - l_{2}, B_{2} + l_{2}, 0, \eta_{2}, \zeta), (\rho, A_{1} - l_{1}, B_{1} + l_{1}, 0, \eta_{1}, \zeta)\Big).       
\end{align*}
is irreducible. So we can take the dual of $(IV)$ (see \eqref{eq: dualizing classical}). Moreover, $(IV)^{\vee}$ is interchangeable with 
\begin{align*}
\underbrace{\begin{pmatrix}
             \zeta B_{2} & \cdots & -\zeta A_{2} \\
              \vdots   & & \vdots \\
              \zeta (B_{2} + l_{2} - 1) & \cdots & -\zeta (A_{2} - l_{2} + 1) 
       \end{pmatrix}}_{I_{-}}       
\end{align*}
and $(II)$. Then
\begin{align*}
\r^{\Sigma_{0}}_{M, >_{\q}}(\q_{\gg}, \ul, \ueta) & \hookrightarrow 
      \begin{pmatrix}
              \zeta (B_{2} + T) & \cdots & \zeta (B_{2} + 1) \\
              \vdots &  & \vdots  \\ 
              \zeta (B_{2} + l_{2} - 1 + T) & \cdots & \zeta(B_{2} + l_{2}) 
       \end{pmatrix} \\
& \times       
       \underbrace{\begin{pmatrix}
              \zeta (B_{2} + l_{2} + T) & \cdots & \zeta (B_{2} + l_{2} + 1) \\
              \vdots &  & \vdots \\
              \zeta (A_{2} - l_{2} + T) & \cdots & \zeta (A_{2} - l_{2} + 1)
       \end{pmatrix}}_{III} \\      
& \times       
        \underbrace{\begin{pmatrix}
              \zeta (A_{2} - l_{2} + 1 + T)  & \cdots & \zeta (A_{2} - l_{2} + 2) \\
              \vdots   & & \vdots \\
              \zeta (A_{2} + T)  & \cdots & \zeta (A_{2} + 1)
        \end{pmatrix}}_{(IV)^{\vee}} \\
& \times
     \underbrace{\begin{pmatrix}
              \zeta B_{1} & \cdots & -\zeta A_{1} \\
              \vdots &  & \vdots \\
              \zeta (B_{1} + l_{1} - 1) & \cdots & - \zeta (A_{1} - l_{1} + 1)
       \end{pmatrix}}_{II} 
\times       
       \underbrace{\begin{pmatrix}
              \zeta B_{2} & \cdots & -\zeta A_{2} \\
              \vdots &  & \vdots \\
              \zeta (B_{2} + l_{2} - 1) & \cdots & -\zeta (A_{2} - l_{2} + 1) 
       \end{pmatrix}}_{I_{-}}  \\              
& \rtimes \r^{\Sigma_{0}}_{M, >_{\q}}\Big(\q_{-}, \ul_{-}, \ueta_{-}; (\rho, A_{2} - l_{2}, B_{2} + l_{2}, 0, \eta_{2}, \zeta), (\rho, A_{1} - l_{1}, B_{1} + l_{1}, 0, \eta_{1}, \zeta)\Big).       
\end{align*}
It follows $\r^{\Sigma_{0}}_{M, >_{\q}}(\q, \ul, \ueta) \neq 0$, and
\begin{align*}
\r^{\Sigma_{0}}_{M, >_{\q}}(\q, \ul, \ueta) & \hookrightarrow 
      \underbrace{\begin{pmatrix}
              \zeta B_{1} & \cdots & -\zeta A_{1} \\
              \vdots &  & \vdots \\
              \zeta (B_{1} + l_{1} - 1) & \cdots & - \zeta (A_{1} - l_{1} + 1)
       \end{pmatrix}}_{II} 
\times       
       \underbrace{\begin{pmatrix}
              \zeta B_{2} & \cdots & -\zeta A_{2} \\
              \vdots &  & \vdots \\
              \zeta (B_{2} + l_{2} - 1) & \cdots & -\zeta (A_{2} - l_{2} + 1) 
       \end{pmatrix}}_{I_{-}}  \\              
& \rtimes \r^{\Sigma_{0}}_{M, >_{\q}}\Big(\q_{-}, \ul_{-}, \ueta_{-}; (\rho, A_{2} - l_{2}, B_{2} + l_{2}, 0, \eta_{2}, \zeta), (\rho, A_{1} - l_{1}, B_{1} + l_{1}, 0, \eta_{1}, \zeta)\Big).       
\end{align*}
Finally, one just needs to observe $(II)$ and $(I_{-})$ are interchangeable.  

\end{itemize}

\subsection{The remaining case}

We show the necessity of condition \eqref{eq: basic condition} in the remaining case, i.e., $A_2 = A_1$ and $B_2 = B_1$. Let us define $\q_{\gg}$ by shifting $(\rho, A_{2}, B_{2}, \zeta_{2})$ to $(\rho, A_{2} + T, B_{2} + T, \zeta_{2})$, such that $\q_{\gg}$ has discrete diagonal restriction and it admits the same order $>_{\q}$. Suppose $\r^{\Sigma_{0}}_{M, >_{\q}}(\q, \ul, \ueta) \neq 0$, we first want to show 
\begin{align}
\label{eq: rough condition}
\begin{cases}
\eta_{2} = (-1)^{A_{1} - B_{1}}\eta_{1}       & \Rightarrow |l_{1} - l_{2}| \leqslant 1,  \\
\eta_{2} \neq (-1)^{A_{1} - B_{1}}\eta_{1}  & \Rightarrow l_{1} + l_{2} + 1 > A_{1} - B_{1}.   
\end{cases}   
\end{align}
Let us consider the following situations.

\begin{enumerate}

\item If $l_{1} = l_{2} = 0$, it is clear that one must have $\eta_{2} = (-1)^{A_{1} - B_{1}}\eta_{1}$.

\item If $l_{1} \neq 0$, then

\begin{align*}
\r^{\Sigma_{0}}_{M, >_{\q}}(\q_{\gg}, \ul, \ueta) & \hookrightarrow \langle \zeta B_{1}, \cdots -\zeta A_{1} \rangle
      \rtimes \r^{\Sigma_{0}}_{M, >_{\q}}\Big(\q_{-}, \ul_{-}, \ueta_{-}; \\
      & (\rho, A_{2} + T, B_{2} + T, l_{2}, \eta_{2}, \zeta), (\rho, A_{1} - 1, B_{1} + 1, l_{1} - 1, \eta_{1}, \zeta)\Big).       
\end{align*}
Note 
\[
\Jac_{(\rho, A_{2} + T, B_{2} + T, \zeta) \mapsto (\rho, A_{2}, B_{2}, \zeta)} \r^{\Sigma_{0}}_{M, >_{\q}}(\q_{\gg}, \ul, \ueta) = \r^{\Sigma_{0}}_{M, >_{\q}}(\q, \ul, \ueta) \neq 0.
\]
So after applying $\Jac_{(\rho, A_{2} + T, B_{2} + T, \zeta) \mapsto (\rho, A_{2}, B_{2}, \zeta)}$ to the full induced representation above, we have
\[
\langle \zeta B_{1}, \cdots -\zeta A_{1} \rangle \rtimes \r^{\Sigma_{0}}_{M, >_{\q}}\Big(\q_{-}, \ul_{-}, \ueta_{-}; (\rho, A_{2}, B_{2}, l_{2}, \eta_{2}, \zeta), (\rho, A_{1} - 1, B_{1} + 1, l_{1} - 1, \eta_{1}, \zeta)\Big),
\]
which is again nonzero. In particular, 
\[
\r^{\Sigma_{0}}_{M, >_{\q}}\Big(\q_{-}, \ul_{-}, \ueta_{-}; (\rho, A_{2}, B_{2}, l_{2}, \eta_{2}, \zeta), (\rho, A_{1} - 1, B_{1} + 1, l_{1} - 1, \eta_{1}, \zeta)\Big) \neq 0.
\]
Here we will only need the weak fact that
\[
\r^{\Sigma_{0}}_{M, >_{\q}}\Big(\q_{-}, \ul_{-}, \ueta_{-}; (\rho, A_{2} + 1, B_{2} + 1, l_{2}, \eta_{2}, \zeta), (\rho, A_{1} - 1, B_{1} + 1, l_{1} - 1, \eta_{1}, \zeta)\Big) \neq 0.
\]
By our induction assumption, we can conclude 
\begin{align*}
\begin{cases}
\eta_{2} = (-1)^{(A_{1}-1) - (B_{1}+1)}\eta_{1}       & \Rightarrow (A_{2}+1) - l_{2} \geqslant (A_{1} -1 ) - ( l_{1} - 1), \\ 
& (B_{2} + 1) + l_{2} \geqslant (B_{1}+1) + (l_{1}-1);  \\
\eta_{2} \neq (-1)^{(A_{1}-1) - (B_{1}+1)}\eta_{1}  & \Rightarrow (B_{2} + 1) + l_{2} > (A_{1} -1 ) - ( l_{1} - 1).   
\end{cases}                                                                                                                                                                                                                                                                                                                                
\end{align*}
In the first case, we get $\eta_{2} = (-1)^{A_{1} - B_{1}}\eta_{1}$ and $-1 \leqslant l_{2} - l_{1} \leqslant 1$. In the second case, we have $\eta_{2} \neq (-1)^{A_{1} - B_{1}}\eta_{1}$ and $l_{1} + l_{2} + 1> A_{1} - B_{2} = A_{1} - B_{1}$. 

\item If $l_{2} \neq 0$, then
\begin{align*}
\r^{\Sigma_{0}}_{M, >_{\q}}(\q_{\gg}, \ul, \ueta) & \hookrightarrow \langle \zeta(B_{2} + T), \cdots -\zeta(A_{2} + T) \rangle
      \rtimes \r^{\Sigma_{0}}_{M, >_{\q}}\Big(\q_{-}, \ul_{-}, \ueta_{-}; \\
      & (\rho, A_{2} + T - 1, B_{2} + T + 1, l_{2} - 1, \eta_{2}, \zeta), (\rho, A_{1}, B_{1}, l_{1}, \eta_{1}, \zeta)\Big).       
\end{align*}
Note 
\[
\Jac_{(\rho, A_{2} + T, B_{2} + T, \zeta) \mapsto (\rho, A_{2}, B_{2}, \zeta)} \r^{\Sigma_{0}}_{M, >_{\q}}(\q_{\gg}, \ul, \ueta) = \r^{\Sigma_{0}}_{M, >_{\q}}(\q, \ul, \ueta) \neq 0.
\]
So after applying $\Jac_{(\rho, A_{2} + T, B_{2} + T, \zeta) \mapsto (\rho, A_{2}, B_{2}, \zeta)}$ to the full induced representation above, we have
\[
\langle \zeta B_{2}, \cdots -\zeta A_{2} \rangle \rtimes \r^{\Sigma_{0}}_{M, >_{\q}}\Big(\q_{-}, \ul_{-}, \ueta_{-}; (\rho, A_{2} - 1, B_{2} + 1, l_{2} - 1, \eta_{2}, \zeta), (\rho, A_{1}, B_{1}, l_{1}, \eta_{1}, \zeta)\Big),
\]
which is again nonzero. In particular, 
\[
\r^{\Sigma_{0}}_{M, >_{\q}}\Big(\q_{-}, \ul_{-}, \ueta_{-}; (\rho, A_{2} - 1, B_{2} + 1, l_{2} - 1, \eta_{2}, \zeta), (\rho, A_{1}, B_{1}, l_{1}, \eta_{1}, \zeta)\Big) \neq 0.
\]
Here we will only need the weak fact that
\[
\r^{\Sigma_{0}}_{M, >_{\q}}\Big(\q_{-}, \ul_{-}, \ueta_{-}; (\rho, A_{2}, B_{2} + 2, l_{2} - 1, \eta_{2}, \zeta), (\rho, A_{1}, B_{1}, l_{1}, \eta_{1}, \zeta)\Big) \neq 0.
\]
By our induction assumption, we can conclude 
\begin{align*}
\begin{cases}
\eta_{2} = (-1)^{A_{1} - B_{1}}\eta_{1}       & \Rightarrow A_{2} - (l_{2}-1) \geqslant A_{1} - l_{1}, \\ 
& (B_{2} + 2) + (l_{2} - 1) \geqslant B_{1} + l_{1};  \\
\eta_{2} \neq (-1)^{A_{1} - B_{1}}\eta_{1}  & \Rightarrow (B_{2} + 2) + (l_{2} - 1) > A_{1} - l_{1}.   
\end{cases}                                                                                                                                                                                                                                                                                                                                
\end{align*}
In the first case, we get $\eta_{2} = (-1)^{A_{1} - B_{1}}\eta_{1}$ and $-1 \leqslant l_{2} - l_{1} \leqslant 1$. In the second case, we have $\eta_{2} \neq (-1)^{A_{1} - B_{1}}\eta_{1}$ and $l_{1} + l_{2} + 1> A_{1} - B_{2} = A_{1} - B_{1}$. 

\end{enumerate}

Now we will assume \eqref{eq: rough condition}. If $\eta_{2} = (-1)^{A_{1} - B_{1}}\eta_{1}$, we still need to eliminate the case $|l_{1} - l_{2}| = 1$. If $\eta_{2} \neq (-1)^{A_{1} - B_{1}}\eta_{1}$, we need to eliminate the following cases:
\begin{enumerate}
\item $|l_{1} - l_{2}| = 1$, max $\{l_{1}, l_{2}\} = (A_{1} - B_{1} + 1)/2$.
\item $l_{1} = l_{2} = (A_{1} - B_{1})/2$. 
\end{enumerate}
To simplify the notations, we let $A = A_{1} = A_{2}$ and $B = B_{1} = B_{2}$.

\subsubsection{{\bf Case:} $l_{1} = l_{2} + 1$}



Let us denote $l_{2}$ by $l$. Since $A -l_{1} + 1> B + l_{1} - 1$, then $A - l > B + l$.

  \begin{enumerate}
  
  \item $A - l > B + l +1$.  
  
  \begin{align*}
  \r^{\Sigma_{0}}_{M, >_{\q}}(\q_{\gg}, \ul, \ueta) & \hookrightarrow 
      \begin{pmatrix}
              \zeta B_{1} & \cdots & -\zeta A_{1} \\
              \vdots &  & \vdots \\
              \zeta (B_{1} + l_{1} - 1) & \cdots & - \zeta (A_{1} - l_{1} + 1)
       \end{pmatrix}  \\
  & \times            
      \begin{pmatrix}
              \zeta (B_{2} + T) & \cdots & -\zeta (A_{2} + T) \\
              \vdots &  & \vdots \\
              \zeta (B_{2} + l_{2} - 1 + T) & \cdots & - \zeta (A_{2} - l_{2} + 1 + T)
       \end{pmatrix} \\
  & \times
      \begin{pmatrix}
              \zeta (B_{2} + l_{2} + T) & \cdots & -\zeta (A_{1} - l_{1}) \\
              \vdots &  & \vdots \\
              \zeta (A_{2} - l_{2} -2 + T) & \cdots & - \zeta (B_{1} + l_{1})
       \end{pmatrix}  \\     
  & \rtimes \r^{\Sigma_{0}}_{M, >_{\q}}\Big(\q_{-}, \ul_{-}, \ueta_{-}; (\rho, A_{2} - l_{2} + T, A_{2} - l_{2} - 1 + T, 0, (-1)^{A_{2} - B_{2} - 1}\eta_{2}, \zeta)\Big) \\
  & \hookrightarrow 
   \underbrace{\begin{pmatrix}
              \zeta B & \cdots & -\zeta A \\
              \vdots &  & \vdots \\
              \zeta (B + l) & \cdots & - \zeta (A - l)
       \end{pmatrix}}_{*-1} 
   \times            
     \underbrace{\begin{pmatrix}
              \zeta (B + T) & \cdots & -\zeta A \\
              \vdots &  & \vdots \\
              \zeta (B + l - 1 + T) & \cdots & - \zeta (A - l + 1)
       \end{pmatrix}}_{*-2} \\
  & \times
       \underbrace{\begin{pmatrix}
              -\zeta (A+1) & \cdots & -\zeta (A + T) \\
              \vdots &  & \vdots \\
              -\zeta (A - l + 2 ) & \cdots & - \zeta (A - l + 1 + T)
       \end{pmatrix}}_{I} \\  
  & \times
       \underbrace{\begin{pmatrix}
              \zeta (B + l + T) & \cdots & -\zeta (A - l - 1) \\
              \vdots &  & \vdots \\
              \zeta (A - l -2 + T) & \cdots & - \zeta (B + l + 1)
       \end{pmatrix}}_{II}  \\     
  & \times      
      \underbrace{\begin{pmatrix}
              \zeta (A - l - 1 + T) & \cdots & \zeta (A - l) \\
              \zeta (A - l + T) & \cdots &  \zeta (A - l + 1)
       \end{pmatrix}}_{III}  \\   
  & \rtimes \r^{\Sigma_{0}}_{M, >_{\q}}\Big(\q_{-}, \ul_{-}, \ueta_{-}; (\rho, A - l, A - l - 1, 0, (-1)^{A_{2} - B_{2} - 1}\eta_{2}, \zeta)\Big)        
   \end{align*}  
  We can interchange $(I)$ with $(II)$ and $(III)$. Note 
  \[
     (I) \rtimes \r^{\Sigma_{0}}_{M, >_{\q}}\Big(\q_{-}, \ul_{-}, \ueta_{-}; (\rho, A - l, A - l - 1, 0, (-1)^{A_{2} - B_{2} - 1}\eta_{2}, \zeta)\Big)       
   \]
  is irreducible (see Proposition~\ref{prop: general irreducibility}), so we can also take dual of $(I)$ (see \eqref{eq: dualizing classical}). Moreover, $(*-1)$ and $(*-2)$ are interchangeable. Therefore,
  \begin{align*}
  \r^{\Sigma_{0}}_{M, >_{\q}}(\q_{\gg}, \ul, \ueta) & \hookrightarrow       
   \underbrace{\begin{pmatrix}
              \zeta (B + T) & \cdots & -\zeta A \\
              \vdots &  & \vdots \\
              \zeta (B + l - 1 + T) & \cdots & - \zeta (A - l + 1)
       \end{pmatrix}}_{*-2}   
  \times
   \underbrace{\begin{pmatrix}
              \zeta B & \cdots & -\zeta A \\
              \vdots &  & \vdots \\
              \zeta (B + l) & \cdots & - \zeta (A - l)
       \end{pmatrix}}_{*-1}   \\
  & \times       
      \underbrace{\begin{pmatrix}
              \zeta (B + l + T) & \cdots & -\zeta (A - l - 1) \\
              \vdots &  & \vdots \\
              \zeta (A - l -2 + T) & \cdots & - \zeta (B + l + 1)
       \end{pmatrix}}_{II}  \\     
  & \times      
      \underbrace{\begin{pmatrix}
              \zeta (A - l - 1 + T) & \cdots & \zeta (A - l) \\
              \zeta (A - l + T) & \cdots &  \zeta (A - l + 1)
       \end{pmatrix}}_{III}  \\   
   & \times    
      \underbrace{\begin{pmatrix}
              \zeta (A - l + 1 + T)  & \cdots & \zeta (A - l + 2 ) \\
              \vdots &  & \vdots \\
              \zeta (A + T)  & \cdots & \zeta (A+1)
       \end{pmatrix}}_{(I)^{\vee}} \\  
  & \rtimes \r^{\Sigma_{0}}_{M, >_{\q}}\Big(\q_{-}, \ul_{-}, \ueta_{-}; (\rho, A - l, A - l - 1, 0, (-1)^{A_{2} - B_{2} - 1}\eta_{2}, \zeta)\Big).        
   \end{align*}  
  We can ``combine" $(II)$ with $(III)$, for otherwise $\Jac_{\zeta (A - l + 1 + T)} \r^{\Sigma_{0}}_{M, >_{\q}}(\q_{\gg}, \ul, \ueta) \neq 0$, which is impossible. Here we have used the fact $A - l > B + l +1$, in order to switch $\rho||^{\zeta (A - l + 1 + T)}$ with $(*-2)$. For the same kind of reason, we can ``combine" $(III)$ with $(I)^{\vee}$. Consequently, 
  \begin{align*}
  \r^{\Sigma_{0}}_{M, >_{\q}}(\q_{\gg}, \ul, \ueta) & \hookrightarrow       
   \underbrace{\begin{pmatrix}
              \zeta (B + T) & \cdots & -\zeta A \\
              \vdots &  & \vdots \\
              \zeta (B + l - 1 + T) & \cdots & - \zeta (A - l + 1)
       \end{pmatrix}}_{*-2}   
  \times
   \underbrace{\begin{pmatrix}
              \zeta B & \cdots & -\zeta A \\
              \vdots &  & \vdots \\
              \zeta (B + l) & \cdots & - \zeta (A - l)
       \end{pmatrix}}_{*-1}   \\
  & \times       
      \underbrace{\begin{pmatrix}
              \zeta (B + l + T) & \cdots & \zeta (B + l + 1) & \cdots & -\zeta (A - l - 1) \\
              \vdots &  & \vdots & & \vdots \\
              \zeta (A - l -2 + T) & \cdots & \zeta (A - l - 1) & \cdots & -\zeta (B + l + 1) \\
              \zeta (A - l - 1 + T) & \cdots & \zeta (A - l) & & \\
               \vdots &  & \vdots & &  \\
               \zeta (A + T) & \cdots & \zeta (A + 1) & &
       \end{pmatrix}}_{IV}  \\     
  & \rtimes \r^{\Sigma_{0}}_{M, >_{\q}}\Big(\q_{-}, \ul_{-}, \ueta_{-}; (\rho, A - l, A - l - 1, 0, (-1)^{A_{2} - B_{2} - 1}\eta_{2}, \zeta)\Big).        
   \end{align*}    
  We can further ``combine" $(*-1)$ with $(IV)$, for otherwise $\Jac_{\zeta (A - l)} \r^{\Sigma_{0}}_{M, >_{\q}}(\q_{\gg}, \ul, \ueta) \neq 0$, which is again impossible. Here we have used the fact that 
  \begin{align}
  \label{eq: basic A1}
  \rho||^{-\zeta (A - l)} \rtimes \r^{\Sigma_{0}}_{M, >_{\q}}\Big(\q_{-}, \ul_{-}, \ueta_{-}; (\rho, A - l, A - l - 1, 0, (-1)^{A_{2} - B_{2} - 1}\eta_{2}, \zeta)\Big)
  \end{align}
  is irreducible (we will prove it in end of this case), and $A - l > B + l + 1$. As a result,
  \begin{align*}
  \r^{\Sigma_{0}}_{M, >_{\q}}(\q_{\gg}, \ul, \ueta) & \hookrightarrow       
   \underbrace{\begin{pmatrix}
              \zeta (B + T) & \cdots & -\zeta A \\
              \vdots &  & \vdots \\
              \zeta (B + l - 1 + T) & \cdots & - \zeta (A - l + 1)
       \end{pmatrix}}_{*-2}   \\
  & \times       
      \underbrace{\begin{pmatrix}
              && \zeta B & \cdots & -\zeta A \\
              && \vdots &  & \vdots \\
              && \zeta (B + l) & \cdots & - \zeta (A - l) \\
              \zeta (B + l + T) & \cdots & \zeta (B + l + 1) & \cdots & -\zeta (A - l - 1) \\
              \vdots &  & \vdots & & \vdots \\
              \zeta (A - l -2 + T) & \cdots & \zeta (A - l - 1) & \cdots & -\zeta (B + l + 1) \\
              \zeta (A - l - 1 + T) & \cdots & \zeta (A - l) & & \\
               \vdots &  & \vdots & &  \\
               \zeta (A + T) & \cdots & \zeta (A + 1) & &
       \end{pmatrix}}_{(*-1) + IV}  \\     
  & \rtimes \r^{\Sigma_{0}}_{M, >_{\q}}\Big(\q_{-}, \ul_{-}, \ueta_{-}; (\rho, A - l, A - l - 1, 0, (-1)^{A_{2} - B_{2} - 1}\eta_{2}, \zeta)\Big).        
   \end{align*}    
  Hence
  \begin{align*}
  &\r^{\Sigma_{0}}_{M, >_{\q}}\Big(\q_{-}, \ul_{-}, \ueta_{-}; (\rho, A_{2} + 1, B_{2} + 1, l_{2}, \eta_{2}, \zeta), (\rho, A_{1}, B_{1}, l_{1}, \eta_{1}, \zeta)\Big) \\
  & \hookrightarrow       
   \begin{pmatrix}
              \zeta (B + 1) & \cdots & -\zeta A \\
              \vdots &  & \vdots \\
              \zeta (B + l) & \cdots & - \zeta (A - l + 1)
       \end{pmatrix}  \\
  & \times       
      \begin{pmatrix}
              \zeta B & \cdots & -\zeta A \\
              \vdots &  & \vdots \\              
              \zeta (A - l - 1) & \cdots & -\zeta (B + l + 1) \\            
               \vdots &  &  \\
               \zeta (A + 1) &  & 
       \end{pmatrix}  \\     
  & \rtimes \r^{\Sigma_{0}}_{M, >_{\q}}\Big(\q_{-}, \ul_{-}, \ueta_{-}; (\rho, A - l, A - l - 1, 0, (-1)^{A_{2} - B_{2} - 1}\eta_{2}, \zeta)\Big).        
   \end{align*}      
  If we apply $\Jac_{(\rho, A+1, B+1, \zeta) \mapsto (\rho, A, B, \zeta)}$ to the full induced representation above, we should get zero. This means $\r^{\Sigma_{0}}_{M, >_{\q}}(\q, \ul, \ueta) = 0$.  
  
  To complete the discussion of this case, we still need to show \eqref{eq: basic A1} is irreducible. We will use the criterion of Lemma~\ref{lemma: strategy for irreducibility}. Since 
  \[
  \Jac_{-\zeta(A - l)} \eqref{eq: basic A1} = \r^{\Sigma_{0}}_{M, >_{\q}}\Big(\q_{-}, \ul_{-}, \ueta_{-}; (\rho, A - l, A - l - 1, 0, (-1)^{A_{2} - B_{2} - 1}\eta_{2}, \zeta)\Big),
  \] 
  we see $\eqref{eq: basic A1}$ has a unique subrepresentation $\sigma$, and $\sigma$ is multiplicity free in $s.s. \eqref{eq: basic A1}$. Since 
  \[
  \Jac_{\zeta(A - l)} \eqref{eq: basic A1} = \r^{\Sigma_{0}}_{M, >_{\q}}\Big(\q_{-}, \ul_{-}, \ueta_{-}; (\rho, A - l, A - l - 1, 0, (-1)^{A_{2} - B_{2} - 1}\eta_{2}, \zeta)\Big),
  \]  
  it suffices to show $\Jac_{\zeta(A - l)} \sigma \neq 0$. Note $A - l - 2 > B + l -1$, so 
  \begin{align*}
  \sigma \hookrightarrow \, & \rho||^{-\zeta(A - l)} \times \rho||^{\zeta(A - l - 1)} \rtimes \r^{\Sigma_{0}}_{M, >_{\q}}\Big(\q_{-}, \ul_{-}, \ueta_{-}; (\rho, A - l, A - l, 0, (-1)^{A_{2} - B_{2}}\eta_{2}, \zeta), \\ 
  & (\rho, A - l - 2, A - l - 2, 0, (-1)^{A_{2} - B_{2} - 1}\eta_{2}, \zeta)\Big)  \\
  \cong \, & \rho||^{\zeta(A - l - 1)} \times \rho||^{-\zeta(A - l)} \rtimes \r^{\Sigma_{0}}_{M, >_{\q}}\Big(\q_{-}, \ul_{-}, \ueta_{-}; (\rho, A - l, A - l, 0, (-1)^{A_{2} - B_{2}}\eta_{2}, \zeta), \\ 
  & (\rho, A - l - 2, A - l - 2, 0, (-1)^{A_{2} - B_{2} - 1}\eta_{2}, \zeta)\Big)  \\
  \cong \, & \rho||^{\zeta(A - l - 1)} \times \rho||^{\zeta(A - l)} \rtimes \r^{\Sigma_{0}}_{M, >_{\q}}\Big(\q_{-}, \ul_{-}, \ueta_{-}; (\rho, A - l, A - l, 0, (-1)^{A_{2} - B_{2}}\eta_{2}, \zeta), \\ 
  & (\rho, A - l - 2, A - l - 2, 0, (-1)^{A_{2} - B_{2} - 1}\eta_{2}, \zeta)\Big).     
  \end{align*}
  The last isomorphism does not follow from Lemma~\ref{lemma: basic irreducibility} exactly, but one can prove it using the same argument there together with (\cite{Moeglin:2006}, Proposition 2.7). If 
  \begin{align*}
  \sigma \hookrightarrow & \langle \zeta(A - l), \zeta(A - l - 1) \rangle \rtimes \r^{\Sigma_{0}}_{M, >_{\q}}\Big(\q_{-}, \ul_{-}, \ueta_{-}; (\rho, A - l, A - l, 0, (-1)^{A_{2} - B_{2}}\eta_{2}, \zeta), \\ 
  & (\rho, A - l - 2, A - l - 2, 0, (-1)^{A_{2} - B_{2} - 1}\eta_{2}, \zeta)\Big),   
  \end{align*}  
  then it is clear that $\Jac_{\zeta(A - l)} \sigma \neq 0$. Otherwise, we have
  \begin{align*}
  \sigma \hookrightarrow & \langle \zeta(A - l - 1), \zeta(A - l) \rangle \rtimes \r^{\Sigma_{0}}_{M, >_{\q}}\Big(\q_{-}, \ul_{-}, \ueta_{-}; \\ & (\rho, A - l, A - l, 0, (-1)^{A_{2} - B_{2}}\eta_{2}, \zeta), 
   (\rho, A - l - 2, A - l - 2, 0, (-1)^{A_{2} - B_{2} - 1}\eta_{2}, \zeta)\Big) \\
  \hookrightarrow & \langle \zeta(A - l - 1), \zeta(A - l) \rangle \times \rho||^{\zeta(A - l)} \rtimes \r^{\Sigma_{0}}_{M, >_{\q}}\Big(\q_{-}, \ul_{-}, \ueta_{-}; \\ & (\rho, A - l - 1, A - l -1 , 0, (-1)^{A_{2} - B_{2}}\eta_{2}, \zeta), 
  (\rho, A - l - 2, A - l - 2, 0, (-1)^{A_{2} - B_{2} - 1}\eta_{2}, \zeta)\Big) \\
   \cong \, & \rho||^{\zeta(A - l)} \times \langle \zeta(A - l - 1), \zeta(A - l) \rangle  \rtimes \r^{\Sigma_{0}}_{M, >_{\q}}\Big(\q_{-}, \ul_{-}, \ueta_{-}; \\ & (\rho, A - l - 1, A - l - 2, 0, (-1)^{A_{2} - B_{2} - 1}\eta_{2}, \zeta)\Big).  
  \end{align*} 
  So we again have $\Jac_{\zeta(A - l)} \sigma \neq 0$. This finishes the proof. \\

  \item $A - l = B + l +1$. \\
  
  Following the previous discussion, we find $(II)$ is ``missing", but we can still ``combine" $(III)$ and $(I)^{\vee}$. 
  \begin{align*}
  \r^{\Sigma_{0}}_{M, >_{\q}}(\q_{\gg}, \ul, \ueta) & \hookrightarrow       
   \underbrace{\begin{pmatrix}
              \zeta (B + T) & \cdots & -\zeta A \\
              \vdots &  & \vdots \\
              \zeta (B + l - 1 + T) & \cdots & - \zeta (A - l + 1)
       \end{pmatrix}}_{*-2}   
  \times
   \underbrace{\begin{pmatrix}
              \zeta B & \cdots & -\zeta A \\
              \vdots &  & \vdots \\
              \zeta (B + l) & \cdots & - \zeta (A - l)
       \end{pmatrix}}_{*-1}   \\
  & \times       
      \underbrace{\begin{pmatrix}              
              \zeta (A - l - 1 + T) & \cdots & \zeta (A - l) \\
               \vdots &  & \vdots   \\
               \zeta (A + T) & \cdots & \zeta (A + 1) 
       \end{pmatrix}}_{IV}  \\     
  & \rtimes \r^{\Sigma_{0}}_{M, >_{\q}}\Big(\q_{-}, \ul_{-}, \ueta_{-}; (\rho, A - l, A - l - 1, 0, (-1)^{A_{2} - B_{2} - 1}\eta_{2}, \zeta)\Big).        
   \end{align*}    
  Hence
  \begin{align*}
  &\r^{\Sigma_{0}}_{M, >_{\q}}\Big(\q_{-}, \ul_{-}, \ueta_{-}; (\rho, A_{2} + 1, B_{2} + 1, l_{2}, \eta_{2}, \zeta), (\rho, A_{1}, B_{1}, l_{1}, \eta_{1}, \zeta)\Big) \\
  & \hookrightarrow       
   \begin{pmatrix}
              \zeta (B + 1) & \cdots & -\zeta A \\
              \vdots &  & \vdots \\
              \zeta (B + l) & \cdots & - \zeta (A - l + 1)
       \end{pmatrix}  
   \times       
      \begin{pmatrix}
              \zeta B & \cdots & -\zeta A \\
              \vdots &  & \vdots \\              
              \zeta (B + l) & \cdots & -\zeta (A-l)              
       \end{pmatrix}  \\     
  & \times
       \begin{pmatrix}
              \zeta (A - l) \\
              \vdots \\              
              \zeta (A + 1)              
       \end{pmatrix}
     \rtimes \r^{\Sigma_{0}}_{M, >_{\q}}\Big(\q_{-}, \ul_{-}, \ueta_{-}; (\rho, A - l, A - l - 1, 0, (-1)^{A_{2} - B_{2} - 1}\eta_{2}, \zeta)\Big).        
   \end{align*}      
  We claim the induced representation above has a unique irreducible subrepresentation. It is clear that for any irreducible subrepresentation $\sigma$, one has
  \begin{align*}
  \sigma & \hookrightarrow       
   \langle \zeta B, \cdots -\zeta A \rangle \times \langle \zeta (B + 1), \cdots, - \zeta A \rangle \times \cdots  \times \langle \zeta (B + l -1), \cdots, - \zeta (A - l + 1) \rangle \\
  & \times \langle \zeta (B + l), \cdots, - \zeta (A - l + 1) \rangle \times \langle \zeta (B + l), \cdots, - \zeta (A - l) \rangle  \\   
  & \times
       \begin{pmatrix}
              \zeta (A - l) \\
              \vdots \\              
              \zeta (A + 1)              
       \end{pmatrix}
     \rtimes \r^{\Sigma_{0}}_{M, >_{\q}}\Big(\q_{-}, \ul_{-}, \ueta_{-}; (\rho, A - l, A - l - 1, 0, (-1)^{A_{2} - B_{2} - 1}\eta_{2}, \zeta)\Big).        
   \end{align*}    
  So the corresponding Jacquet module of $\sigma$ under
  \begin{align*}
  \Jac_{X} := & \Jac_{\zeta (A - l), \cdots, \zeta (A + 1)}  \circ \Jac_{\zeta (B + l), \cdots, - \zeta (A - l)} \circ \Jac_{\zeta (B + l), \cdots, - \zeta (A - l + 1)} \circ \\
  & \Jac_{\zeta (B + l -1), \cdots, - \zeta (A - l + 1)} \circ  \cdots \circ \Jac_{\zeta (B + 1), \cdots, - \zeta A} \circ \Jac_{\zeta B, \cdots, -\zeta A}    
  \end{align*}     
  contains the irreducible representation $\r^{\Sigma_{0}}_{M, >_{\q}}\Big(\q_{-}, \ul_{-}, \ueta_{-}; (\rho, A - l, A - l - 1, 0, (-1)^{A_{2} - B_{2} - 1}\eta_{2}, \zeta)\Big)$. On the other hand, we can also apply $\Jac_{X}$ to the full induced representation
  \begin{align*}       
  & \begin{pmatrix}
              \zeta (B + 1) & \cdots & -\zeta A \\
              \vdots &  & \vdots \\
              \zeta (B + l) & \cdots & - \zeta (A - l + 1)
       \end{pmatrix}  
   \times       
      \begin{pmatrix}
              \zeta B & \cdots & -\zeta A \\
              \vdots &  & \vdots \\              
              \zeta (B + l) & \cdots & -\zeta (A-l)              
       \end{pmatrix}  \\     
  & \times
       \begin{pmatrix}
              \zeta (A - l) \\
              \vdots \\              
              \zeta (A + 1)              
       \end{pmatrix}
     \rtimes \r^{\Sigma_{0}}_{M, >_{\q}}\Big(\q_{-}, \ul_{-}, \ueta_{-}; (\rho, A - l, A - l - 1, 0, (-1)^{A_{2} - B_{2} - 1}\eta_{2}, \zeta)\Big),       
   \end{align*}  
   and we get $\r^{\Sigma_{0}}_{M, >_{\q}}\Big(\q_{-}, \ul_{-}, \ueta_{-}; (\rho, A - l, A - l - 1, 0, (-1)^{A_{2} - B_{2} - 1}\eta_{2}, \zeta)\Big)$. This proves our claim. As a result,
   \begin{align*}
  &\r^{\Sigma_{0}}_{M, >_{\q}}\Big(\q_{-}, \ul_{-}, \ueta_{-}; (\rho, A_{2} + 1, B_{2} + 1, l_{2}, \eta_{2}, \zeta), (\rho, A_{1}, B_{1}, l_{1}, \eta_{1}, \zeta)\Big) \\
  & \hookrightarrow       
   \begin{pmatrix}
              \zeta (B + 1) & \cdots & -\zeta A \\
              \vdots &  & \vdots \\
              \zeta (B + l) & \cdots & - \zeta (A - l + 1)
       \end{pmatrix}  
    \times       
      \begin{pmatrix}
              \zeta B & \cdots & -\zeta A \\
              \vdots &  & \vdots \\              
              \zeta (B + l) & \cdots & -\zeta (A-l)  \\
              \zeta (A - l) \\
              \vdots \\              
              \zeta (A + 1)           
       \end{pmatrix}   \\
    &\rtimes \r^{\Sigma_{0}}_{M, >_{\q}}\Big(\q_{-}, \ul_{-}, \ueta_{-}; (\rho, A - l, A - l - 1, 0, (-1)^{A_{2} - B_{2} - 1}\eta_{2}, \zeta)\Big).        
      \end{align*}   
   Therefore, if we apply $\Jac_{(\rho, A+1, B+1, \zeta) \mapsto (\rho, A, B, \zeta)}$ to the full induced representation above, we should get zero. This means $\r^{\Sigma_{0}}_{M, >_{\q}}(\q, \ul, \ueta) = 0$.

    \end{enumerate}

\subsubsection{{\bf Case:} $l_{2} = l_{1} + 1$}


Let us denote $l_{1}$ by $l$. Since $A -l_{2} + 1> B + l_{2} - 1$, then $A - l > B + l$.

    \begin{enumerate}
    
    \item $l \neq 0$ and $A - l > B + l + 1$.     
    
    \begin{align*}
  \r^{\Sigma_{0}}_{M, >_{\q}}(\q_{\gg}, \ul, \ueta) & \hookrightarrow 
      \begin{pmatrix}
              \zeta B_{1} & \cdots & -\zeta A_{1} \\
              \vdots &  & \vdots \\
              \zeta (B_{1} + l_{1} - 1) & \cdots & - \zeta (A_{1} - l_{1} + 1)
       \end{pmatrix}  \\
  & \times            
      \begin{pmatrix}
              \zeta (B_{2} + T) & \cdots & -\zeta (A_{2} + T) \\
              \vdots &  & \vdots \\
              \zeta (B_{2} + l_{2} - 1 + T) & \cdots & - \zeta (A_{2} - l_{2} + 1 + T)
       \end{pmatrix} \\
  & \times
      \begin{pmatrix}
              \zeta (B_{2} + l_{2} + T) & \cdots & -\zeta (A_{1} - l_{1}) \\
              \vdots &  & \vdots \\
              \zeta (A_{2} - l_{2} + T) & \cdots & - \zeta (B_{1} + l_{1} + 2)
       \end{pmatrix}  \\     
  & \rtimes \r^{\Sigma_{0}}_{M, >_{\q}}\Big(\q_{-}, \ul_{-}, \ueta_{-}; (\rho, B_{1} + l_{1} + 1, B_{1} + l_{1}, 0, \eta_{1}, \zeta)\Big) \\
  & \hookrightarrow 
   \underbrace{\begin{pmatrix}
              \zeta B & \cdots & -\zeta A \\
              \vdots &  & \vdots \\
              \zeta (B + l - 1) & \cdots & - \zeta (A - l + 1)
       \end{pmatrix}}_{*-1} 
   \times            
     \underbrace{\begin{pmatrix}
              \zeta (B + T) & \cdots & -\zeta (A + 1) \\
              \vdots &  & \vdots \\
              \zeta (B + l + T) & \cdots & - \zeta (A - l + 1)
       \end{pmatrix}}_{I} \\
  & \times
       \underbrace{\begin{pmatrix}
              -\zeta (A+2) & \cdots & -\zeta (A + T) \\
              \vdots &  & \vdots \\
              -\zeta (A - l + 2 ) & \cdots & - \zeta (A - l + T)
       \end{pmatrix}}_{II} \\  
  & \times
       \underbrace{\begin{pmatrix}
              \zeta (B + l + 1+ T) & \cdots & -\zeta (A - l) \\
              \vdots &  & \vdots \\
              \zeta (A - l -1 + T) & \cdots & - \zeta (B + l + 2)
       \end{pmatrix}}_{III}  \\        
  & \rtimes \r^{\Sigma_{0}}_{M, >_{\q}}\Big(\q_{-}, \ul_{-}, \ueta_{-}; (\rho, B + l + 1, B + l, 0, \eta_{1}, \zeta)\Big). 
   \end{align*}  
   We first interchange $(II)$ and $(III)$, then take dual of $(II)$ (see \eqref{eq: dualizing classical}).    
   \begin{align*}
  \r^{\Sigma_{0}}_{M, >_{\q}}(\q_{\gg}, \ul, \ueta) & \hookrightarrow 
         \underbrace{\begin{pmatrix}
              \zeta B & \cdots & -\zeta A \\
              \vdots &  & \vdots \\
              \zeta (B + l - 1) & \cdots & - \zeta (A - l + 1)
       \end{pmatrix}}_{*-1} 
   \times            
     \underbrace{\begin{pmatrix}
              \zeta (B + T) & \cdots & -\zeta (A + 1) \\
              \vdots &  & \vdots \\
              \zeta (B + l + T) & \cdots & - \zeta (A - l + 1)
       \end{pmatrix}}_{I} \\  
  & \times
       \underbrace{\begin{pmatrix}
              \zeta (B + l + 1+ T) & \cdots & -\zeta (A - l) \\
              \vdots &  & \vdots \\
              \zeta (A - l -1 + T) & \cdots & - \zeta (B + l + 2)
       \end{pmatrix}}_{III}  \\
   & \times
       \underbrace{\begin{pmatrix}
              \zeta (A - l + T)  & \cdots & \zeta (A - l + 2)  \\
              \vdots &  & \vdots \\
              \zeta (A + T)  & \cdots &  \zeta (A+2)
       \end{pmatrix}}_{(II)^{\vee}} \\          
  & \rtimes \r^{\Sigma_{0}}_{M, >_{\q}}\Big(\q_{-}, \ul_{-}, \ueta_{-}; (\rho, B + l + 1, B + l, 0, \eta_{1}, \zeta)\Big). 
   \end{align*}  
  Since $\Jac_{\zeta (B + l + 1+ T)} \r^{\Sigma_{0}}_{M, >_{\q}}(\q_{\gg}, \ul, \ueta) = 0$, we can ``combine" $(I)$ and $(III)$. For the same kind of reason, we can further ``combine" them with $(II)^{\vee}$. So  
    \begin{align*}
  \r^{\Sigma_{0}}_{M, >_{\q}}(\q_{\gg}, \ul, \ueta) & \hookrightarrow 
         \underbrace{\begin{pmatrix}
              \zeta B & \cdots & -\zeta A \\
              \vdots &  & \vdots \\
              \zeta (B + l - 1) & \cdots & - \zeta (A - l + 1)
       \end{pmatrix}}_{*-1}  \\   
  & \times
       \underbrace{\begin{pmatrix}
              \zeta (B + T) & \cdots & \zeta (B + 2) & \cdots & -\zeta (A + 1) \\
              \vdots &  & \vdots & & \vdots \\
              \zeta (A - l -1 + T) & \cdots & \zeta (A - l + 1) & \cdots & - \zeta (B + l + 2) \\
              \zeta (A - l + T)  & \cdots & \zeta (A - l + 2) & &  \\
              \vdots &  & \vdots & & \\
              \zeta (A + T)  & \cdots &  \zeta (A+2) & &
       \end{pmatrix}}_{*-2} \\               
  & \rtimes \r^{\Sigma_{0}}_{M, >_{\q}}\Big(\q_{-}, \ul_{-}, \ueta_{-}; (\rho, B + l + 1, B + l, 0, \eta_{1}, \zeta)\Big) \\
  & \hookrightarrow
       \underbrace{\begin{pmatrix}
              \zeta B & \cdots & -\zeta A \\
              \vdots &  & \vdots \\
              \zeta (B + l - 1) & \cdots & - \zeta (A - l + 1)
       \end{pmatrix}}_{*-1} \\
  & \times 
       \underbrace{\begin{pmatrix}
              \zeta (B + T) & \cdots & \zeta (B + 2) & \cdots & -\zeta A \\
              \vdots &  & \vdots & & \vdots \\
              \zeta (A - l -1 + T) & \cdots & \zeta (A - l + 1) & \cdots & - \zeta (B + l + 1) \\
              \zeta (A - l + T)  & \cdots & \zeta (A - l + 2) & &  \\
              \vdots &  & \vdots & & \\
              \zeta (A + T)  & \cdots &  \zeta (A+2) & &
       \end{pmatrix}}_{(*-2)_{-}}       \\
   & \times                       
       \underbrace{\begin{pmatrix}
               -\zeta (A+1) \\
              \vdots \\
              - \zeta (B + l + 2) 
       \end{pmatrix}}_{IV} 
      \times
        \underbrace{\begin{pmatrix}
               \zeta (B+l) \\
               \zeta (B + l + 1) 
       \end{pmatrix}}_{V} \\
  & \rtimes \r^{\Sigma_{0}}_{M, >_{\q}}\Big(\q_{-}, \ul_{-}, \ueta_{-}; (\rho, B + l, B + l - 1, 0, \eta_{1}, \zeta)\Big)
    \end{align*}  
   We interchange $(*-1)$ and $(*-2)_{-}$, also $(IV)$ and $(V)$. Then we take the dual of $(IV)$.
    \begin{align*}
  \r^{\Sigma_{0}}_{M, >_{\q}}(\q_{\gg}, \ul, \ueta) & \hookrightarrow 
     \underbrace{\begin{pmatrix}
              \zeta (B + T) & \cdots & \zeta (B + 2) & \cdots & -\zeta A \\
              \vdots &  & \vdots & & \vdots \\
              \zeta (A - l -1 + T) & \cdots & \zeta (A - l + 1) & \cdots & - \zeta (B + l + 1) \\
              \zeta (A - l + T)  & \cdots & \zeta (A - l + 2) & &  \\
              \vdots &  & \vdots & & \\
              \zeta (A + T)  & \cdots &  \zeta (A+2) & &
       \end{pmatrix}}_{(*-2)_{-}} \\
   & \times
       \underbrace{\begin{pmatrix}
              \zeta B & \cdots & -\zeta A \\
              \vdots &  & \vdots \\
              \zeta (B + l - 1) & \cdots & - \zeta (A - l + 1)
       \end{pmatrix}}_{*-1}   \\
   & \times
        \underbrace{\begin{pmatrix}
               \zeta (B+l) \\
               \zeta (B + l + 1) \\
       \end{pmatrix}}_{V}
      \times
           \underbrace{\begin{pmatrix}
               \zeta (B + l + 2) \\
              \vdots \\
               \zeta (A + 1) \\
       \end{pmatrix}}_{(IV)^{\vee}}      \\
   & \rtimes \r^{\Sigma_{0}}_{M, >_{\q}}\Big(\q_{-}, \ul_{-}, \ueta_{-}; (\rho, B + l, B + l - 1, 0, \eta_{1}, \zeta)\Big).     
     \end{align*}      
       We can ``combine" $(*-1)$ and $(V)$ for $\Jac_{\zeta (B + l)} \r^{\Sigma_{0}}_{M, >_{\q}}(\q_{\gg}, \ul, \ueta) = 0$. We can also ``combine" $(V)$ and $(IV)^{\vee}$ for $\Jac_{\zeta (B + l + 2)} \r^{\Sigma_{0}}_{M, >_{\q}}(\q_{\gg}, \ul, \ueta) = 0$. Here we have used the fact that $A - l > B + l +1$. So
   \begin{align*}      
   \r^{\Sigma_{0}}_{M, >_{\q}}(\q_{\gg}, \ul, \ueta) & \hookrightarrow 
       \underbrace{\begin{pmatrix}
              \zeta (B + T) & \cdots & \zeta (B + 2) & \cdots & -\zeta A \\
              \vdots &  & \vdots & & \vdots \\
              \zeta (A - l -1 + T) & \cdots & \zeta (A - l + 1) & \cdots & - \zeta (B + l + 1) \\
              \zeta (A - l + T)  & \cdots & \zeta (A - l + 2) & &  \\
              \vdots &  & \vdots & & \\
              \zeta (A + T)  & \cdots &  \zeta (A+2) & &
       \end{pmatrix}}_{(*-2)_{-}} \\
   & \times
       \underbrace{\begin{pmatrix}
              \zeta B & \cdots & -\zeta A \\
              \vdots &  & \vdots \\
              \zeta (B + l - 1) & \cdots & - \zeta (A - l + 1) \\
              \zeta (B+l) & & \\              
              \vdots \\
               \zeta (A + 1)         
       \end{pmatrix}}_{(*-1)_{+}}   \\
    &  \rtimes \r^{\Sigma_{0}}_{M, >_{\q}}\Big(\q_{-}, \ul_{-}, \ueta_{-}; (\rho, B + l, B + l - 1, 0, \eta_{1}, \zeta)\Big).          
    \end{align*}    
    Then
    \begin{align*}
  &\r^{\Sigma_{0}}_{M, >_{\q}}\Big(\q_{-}, \ul_{-}, \ueta_{-}; (\rho, A_{2} + 1, B_{2} + 1, l_{2}, \eta_{2}, \zeta), (\rho, A_{1}, B_{1}, l_{1}, \eta_{1}, \zeta)\Big) \\
  & \hookrightarrow       
       \begin{pmatrix}
              \zeta (B + 1) & \cdots & -\zeta A \\
              \vdots &  & \vdots \\
              \zeta (A - l) & \cdots & - \zeta (B + l + 1) \\
       \end{pmatrix}
    \times
       \begin{pmatrix}
              \zeta B & \cdots & -\zeta A \\
              \vdots &  & \vdots \\
              \zeta (B + l - 1) & \cdots & - \zeta (A - l + 1) \\
              \zeta (B+l) & & \\              
              \vdots \\
               \zeta (A + 1)         
       \end{pmatrix}   \\
    &  \rtimes \r^{\Sigma_{0}}_{M, >_{\q}}\Big(\q_{-}, \ul_{-}, \ueta_{-}; (\rho, B + l, B + l - 1, 0, \eta_{1}, \zeta)\Big).          
    \end{align*}     
    Therefore, if we apply $\Jac_{(\rho, A+1, B+1, \zeta) \mapsto (\rho, A, B, \zeta)}$ to the full induced representation above, we should get zero. This means $\r^{\Sigma_{0}}_{M, >_{\q}}(\q, \ul, \ueta) = 0$.     \\

    \item $l \neq 0$ and $A - l = B + l + 1$. \\
    
    It follows from the previous discussion that
    \begin{align*}
  \r^{\Sigma_{0}}_{M, >_{\q}}(\q_{\gg}, \ul, \ueta) & \hookrightarrow 
         \underbrace{\begin{pmatrix}
              \zeta B & \cdots & -\zeta A \\
              \vdots &  & \vdots \\
              \zeta (B + l - 1) & \cdots & - \zeta (A - l + 1)
       \end{pmatrix}}_{*-1}  \\   
  & \times
       \underbrace{\begin{pmatrix}
              \zeta (B + T) & \cdots & \zeta (B + 2) & \cdots & -\zeta (A + 1) \\
              \vdots &  & \vdots & & \vdots \\
              \zeta (A - l -1 + T) & \cdots & \zeta (A - l + 1) & \cdots & - \zeta (B + l + 2) \\
              \zeta (A - l + T)  & \cdots & \zeta (A - l + 2) & &  \\
              \vdots &  & \vdots & & \\
              \zeta (A + T)  & \cdots &  \zeta (A+2) & &
       \end{pmatrix}}_{*-2} \\               
  & \rtimes \r^{\Sigma_{0}}_{M, >_{\q}}\Big(\q_{-}, \ul_{-}, \ueta_{-}; (\rho, B + l + 1, B + l, 0, \eta_{1}, \zeta)\Big)
  \end{align*}
    Since $(*-1)$ and $(*-2)$ are interchangeable, then we have
    \begin{align*}
  &\r^{\Sigma_{0}}_{M, >_{\q}}\Big(\q_{-}, \ul_{-}, \ueta_{-}; (\rho, A_{2} + 1, B_{2} + 1, l_{2}, \eta_{2}, \zeta), (\rho, A_{1}, B_{1}, l_{1}, \eta_{1}, \zeta)\Big) \\
  & \hookrightarrow       
            \begin{pmatrix}
              \zeta (B + 1) & \cdots & -\zeta (A + 1) \\
              \vdots &  & \vdots \\
              \zeta (A - l) & \cdots & - \zeta (B + l + 2) \\
       \end{pmatrix}
    \times
       \begin{pmatrix}
              \zeta B & \cdots & -\zeta A \\
              \vdots &  & \vdots \\
              \zeta (B + l - 1) & \cdots & - \zeta (A - l + 1) 
              \end{pmatrix}   \\
    & \rtimes \r^{\Sigma_{0}}_{M, >_{\q}}\Big(\q_{-}, \ul_{-}, \ueta_{-}; (\rho, B + l + 1, B + l, 0, \eta_{1}, \zeta)\Big).          
    \end{align*}      
    It follows
    \begin{align*}
    &\r^{\Sigma_{0}}_{M, >_{\q}}\Big(\q_{-}, \ul_{-}, \ueta_{-}; (\rho, A_{2} + 1, B_{2} + 1, l_{2}, \eta_{2}, \zeta), (\rho, A_{1}, B_{1}, l_{1}, \eta_{1}, \zeta)\Big) \\
    & \hookrightarrow
           \langle \zeta B, \cdots -\zeta A \rangle \times \langle \zeta (B + 1), \cdots, - \zeta (A + 1) \rangle \times \cdots  \times \langle \zeta (B + l -1), \cdots, - \zeta (A - l + 1) \rangle \\
  & \times \langle \zeta (A - l - 1), \cdots, - \zeta (B + l + 3) \rangle \times \langle \zeta (A - l), \cdots, - \zeta (B + l + 2) \rangle  \\   
  & \rtimes \r^{\Sigma_{0}}_{M, >_{\q}}\Big(\q_{-}, \ul_{-}, \ueta_{-}; (\rho, B + l + 1, B + l, 0, \eta_{1}, \zeta)\Big)  
     \end{align*}
    Therefore
    \begin{align}
    \label{eq: basic A2}
    \Jac_{X} \, \r^{\Sigma_{0}}_{M, >_{\q}}\Big(\q_{-}, \ul_{-}, \ueta_{-}; (\rho, A_{2} + 1, B_{2} + 1, l_{2}, \eta_{2}, \zeta), (\rho, A_{1}, B_{1}, l_{1}, \eta_{1}, \zeta)\Big) \neq 0,
    \end{align}
    where   
  \begin{align*}
  \Jac_{X}  := & \Jac_{\zeta (A - l), \cdots, - \zeta (B + l + 2)} \circ \Jac_{\zeta (A - l - 1), \cdots, - \zeta (B + l + 3)} \circ \\
  & \Jac_{\zeta (B + l -1), \cdots, - \zeta (A - l + 1)} \circ \cdots \circ  \Jac_{\zeta (B + 1), \cdots, - \zeta (A + 1)} \circ \Jac_{\zeta B, \cdots, -\zeta A}.  
  \end{align*}     
  On the other hand, we can rewrite
  \begin{align*}
  &\r^{\Sigma_{0}}_{M, >_{\q}}\Big(\q_{-}, \ul_{-}, \ueta_{-}; (\rho, A_{2} + 1, B_{2} + 1, l_{2}, \eta_{2}, \zeta), (\rho, A_{1}, B_{1}, l_{1}, \eta_{1}, \zeta)\Big) \\
  & \hookrightarrow       
            \begin{pmatrix}
              \zeta (B + 1) & \cdots & -\zeta A \\
              \vdots &  & \vdots \\
              \zeta (A - l) & \cdots & - \zeta (B + l + 1) \\
       \end{pmatrix}
      \times
        \underbrace{\begin{pmatrix}
               -\zeta (A+1) \\
               \vdots \\
              - \zeta (B + l + 2) \\
       \end{pmatrix}}_{IV}  \\
    & \times
       \begin{pmatrix}
              \zeta B & \cdots & -\zeta A \\
              \vdots &  & \vdots \\
              \zeta (B + l - 1) & \cdots & - \zeta (A - l + 1) 
              \end{pmatrix}   
     \rtimes \r^{\Sigma_{0}}_{M, >_{\q}}\Big(\q_{-}, \ul_{-}, \ueta_{-}; (\rho, B + l + 1, B + l, 0, \eta_{1}, \zeta)\Big) \\
    & \cong
          \begin{pmatrix}
              \zeta (B + 1) & \cdots & -\zeta A \\
              \vdots &  & \vdots \\
              \zeta (A - l) & \cdots & - \zeta (B + l + 1) \\
       \end{pmatrix}
      \times
       \begin{pmatrix}
              \zeta B & \cdots & -\zeta A \\
              \vdots &  & \vdots \\
              \zeta (B + l - 1) & \cdots & - \zeta (A - l + 1) 
              \end{pmatrix}   \\
     & \times
        \underbrace{\begin{pmatrix}
               -\zeta (A+1) \\
               \vdots \\
              - \zeta (B + l + 2) \\
       \end{pmatrix}}_{IV}         
         \rtimes \r^{\Sigma_{0}}_{M, >_{\q}}\Big(\q_{-}, \ul_{-}, \ueta_{-}; (\rho, B + l + 1, B + l, 0, \eta_{1}, \zeta)\Big)  
    \end{align*}   
     So there exits an irreducible constituent $\sigma$ of 
     \[
     (IV)  \rtimes \r^{\Sigma_{0}}_{M, >_{\q}}\Big(\q_{-}, \ul_{-}, \ueta_{-}; (\rho, B + l + 1, B + l, 0, \eta_{1}, \zeta)\Big),
     \]   
     such that
     \begin{align*}
  &\r^{\Sigma_{0}}_{M, >_{\q}}\Big(\q_{-}, \ul_{-}, \ueta_{-}; (\rho, A_{2} + 1, B_{2} + 1, l_{2}, \eta_{2}, \zeta), (\rho, A_{1}, B_{1}, l_{1}, \eta_{1}, \zeta)\Big) \\
  & \hookrightarrow                   
          \begin{pmatrix}
              \zeta (B + 1) & \cdots & -\zeta A \\
              \vdots &  & \vdots \\
              \zeta (A - l) & \cdots & - \zeta (B + l + 1) \\
       \end{pmatrix}
      \times
       \begin{pmatrix}
              \zeta B & \cdots & -\zeta A \\
              \vdots &  & \vdots \\
              \zeta (B + l - 1) & \cdots & - \zeta (A - l + 1) 
              \end{pmatrix}  \rtimes \sigma
    \end{align*}   
   We claim $\Jac_{\zeta(B + l + 2)} \sigma = 0$. Otherwise, we necessarily have
   \begin{align*}
   \Jac_{\zeta(B + l + 2)} \sigma = 
      \begin{pmatrix}
               -\zeta (A+1) \\
               \vdots \\
              - \zeta (B + l + 3) \\
       \end{pmatrix}         
         \rtimes \r^{\Sigma_{0}}_{M, >_{\q}}\Big(\q_{-}, \ul_{-}, \ueta_{-}; (\rho, B + l + 1, B + l, 0, \eta_{1}, \zeta)\Big)
    \end{align*}
    which is irreducible by Proposition~\ref{prop: general irreducibility}. Then
    \begin{align*}
    \sigma 
    & \hookrightarrow \rho||^{\zeta(B + l + 2)} 
      \times
      \begin{pmatrix}
               -\zeta (A+1) \\
               \vdots \\
              - \zeta (B + l + 3) \\
       \end{pmatrix}         
         \rtimes \r^{\Sigma_{0}}_{M, >_{\q}}\Big(\q_{-}, \ul_{-}, \ueta_{-}; (\rho, B + l + 1, B + l, 0, \eta_{1}, \zeta)\Big) \\
    & \cong
       \begin{pmatrix}
               -\zeta (A+1) \\
               \vdots \\
              - \zeta (B + l + 3) \\
       \end{pmatrix}  
       \times \rho||^{\zeta(B + l + 2)}     
       \rtimes \r^{\Sigma_{0}}_{M, >_{\q}}\Big(\q_{-}, \ul_{-}, \ueta_{-}; (\rho, B + l + 1, B + l, 0, \eta_{1}, \zeta)\Big).         
           \end{align*}
     It is necessary that
       \begin{align*}
    \sigma 
    & \hookrightarrow 
       \begin{pmatrix}
               -\zeta (A+1) \\
               \vdots \\
              - \zeta (B + l + 3) \\
       \end{pmatrix}  
       \rtimes \r^{\Sigma_{0}}_{M, >_{\q}}\Big(\q_{-}, \ul_{-}, \ueta_{-}; (\rho, B + l + 2, B + l + 2, 0, -\eta_{1}, \zeta), \\
     & (\rho, B + l, B + l, 0, \eta_{1}, \zeta)\Big).         
           \end{align*}   
     In particular
     \[
     \Jac_{-\zeta(B + l + 2)} \r^{\Sigma_{0}}_{M, >_{\q}}\Big(\q_{-}, \ul_{-}, \ueta_{-}; (\rho, B + l + 2, B + l + 2, 0, -\eta_{1}, \zeta), (\rho, B + l, B + l, 0, \eta_{1}, \zeta)\Big) = 0.
     \]         
     Now we have
      \begin{align*}
  &\r^{\Sigma_{0}}_{M, >_{\q}}\Big(\q_{-}, \ul_{-}, \ueta_{-}; (\rho, A_{2} + 1, B_{2} + 1, l_{2}, \eta_{2}, \zeta), (\rho, A_{1}, B_{1}, l_{1}, \eta_{1}, \zeta)\Big) \\
  & \hookrightarrow                   
          \begin{pmatrix}
              \zeta (B + 1) & \cdots & -\zeta A \\
              \vdots &  & \vdots \\
              \zeta (A - l) & \cdots & - \zeta (B + l + 1) \\
       \end{pmatrix}
      \times
       \begin{pmatrix}
              \zeta B & \cdots & -\zeta A \\
              \vdots &  & \vdots \\
              \zeta (B + l - 1) & \cdots & - \zeta (A - l + 1) 
              \end{pmatrix}  \\
    & \times 
      \begin{pmatrix}
               -\zeta (A+1) \\
               \vdots \\
              - \zeta (B + l + 3) \\
       \end{pmatrix}  
       \rtimes \r^{\Sigma_{0}}_{M, >_{\q}}\Big(\q_{-}, \ul_{-}, \ueta_{-}; (\rho, B + l + 2, B + l + 2, 0, -\eta_{1}, \zeta), \\
     & (\rho, B + l, B + l, 0, \eta_{1}, \zeta)\Big).         
    \end{align*} 
    If we apply 
    \begin{align*}
  \Jac_{X'}  := & \Jac_{\zeta (A - l - 1), \cdots, - \zeta (B + l + 3)} \circ  \Jac_{\zeta (B + l -1), \cdots, - \zeta (A - l + 1)} \circ \cdots \circ \\
  & \Jac_{\zeta (B + 1), \cdots, - \zeta (A + 1)} \circ \Jac_{\zeta B, \cdots, -\zeta A}.  
    \end{align*}    
     to the full induced representation above, we will get
     \begin{align*}                   
     & \langle \zeta (A - l), \cdots, - \zeta (B + l + 1) \rangle     
       \rtimes \r^{\Sigma_{0}}_{M, >_{\q}}\Big(\q_{-}, \ul_{-}, \ueta_{-}; (\rho, B + l + 2, B + l + 2, 0, -\eta_{1}, \zeta), \\
     & (\rho, B + l, B + l, 0, \eta_{1}, \zeta)\Big).         
    \end{align*}
    Here we have used the fact that $A - l = B + l + 1$. Note
    \[
    \Jac_{X} =  \Jac_{\zeta (A - l), \cdots, - \zeta (B + l + 2)} \circ \Jac_{X'}
    \]
    and
    \begin{align*}                   
     & \Jac_{\zeta (A - l), \cdots, - \zeta (B + l + 2)}  \langle \zeta (A - l), \cdots, - \zeta (B + l + 1) \rangle   
       \rtimes \r^{\Sigma_{0}}_{M, >_{\q}}\Big(\q_{-}, \ul_{-}, \ueta_{-}; \\
     & (\rho, B + l + 2, B + l + 2, 0, -\eta_{1}, \zeta), 
      (\rho, B + l, B + l, 0, \eta_{1}, \zeta)\Big) = 0.         
    \end{align*}    
    This contradicts to \eqref{eq: basic A2}. So we have shown our claim.
    
For $\r^{\Sigma_{0}}_{M, >_{\q}}(\q, \ul, \ueta)$ being nonzero, there necessarily exits $x \in [B + 1, A + 1]$ such that $\Jac_{\zeta x, \cdots, \zeta (A+1)} \sigma \neq 0$. By our claim, $\Jac_{\zeta x} \sigma \neq 0$ implies $x = B + l$. It follows 
\[
\sigma \hookrightarrow \langle \zeta (B + l), \cdots, \zeta(A + 1) \rangle \rtimes \r^{\Sigma_{0}}_{M, >_{\q}}\Big(\q_{-}, \ul_{-}, \ueta_{-}; (\rho, B + l, B + l - 1, 0, \eta_{1}, \zeta)\Big).
\]
Then
   \begin{align*}
  & \r^{\Sigma_{0}}_{M, >_{\q}}\Big(\q_{-}, \ul_{-}, \ueta_{-}; (\rho, A_{2} + 1, B_{2} + 1, l_{2}, \eta_{2}, \zeta), (\rho, A_{1}, B_{1}, l_{1}, \eta_{1}, \zeta)\Big) \\
  &   \begin{pmatrix}
              \zeta (B + 1) & \cdots & -\zeta A \\
              \vdots &  & \vdots \\
              \zeta (A - l) & \cdots & - \zeta (B + l + 1) \\
       \end{pmatrix}
      \times
       \underbrace{\begin{pmatrix}
              \zeta B & \cdots & -\zeta A \\
              \vdots &  & \vdots \\
              \zeta (B + l - 1) & \cdots & - \zeta (A - l + 1) 
              \end{pmatrix}}_{*-1}   \\
     & \times
        \underbrace{\begin{pmatrix}
               \zeta (B + l) \\
               \vdots \\
               \zeta (A + 1) \\
       \end{pmatrix}}_{VI}         
         \rtimes \r^{\Sigma_{0}}_{M, >_{\q}}\Big(\q_{-}, \ul_{-}, \ueta_{-}; (\rho, B + l, B + l - 1, 0, \eta_{1}, \zeta)\Big)  
    \end{align*}   
Since 
\begin{align*}
\Jac_{\zeta(B + l), \zeta(B+1)} \r^{\Sigma_{0}}_{M, >_{\q}}\Big(\q_{-}, \ul_{-}, \ueta_{-}; (\rho, A_{2} + 1, B_{2} + 1, l_{2}, \eta_{2}, \zeta), (\rho, A_{1}, B_{1}, l_{1}, \eta_{1}, \zeta)\Big) = 0,
\end{align*}    
we can ``combine" $(*-1)$ and $(VI)$. So
\begin{align*}
  & \r^{\Sigma_{0}}_{M, >_{\q}}\Big(\q_{-}, \ul_{-}, \ueta_{-}; (\rho, A_{2} + 1, B_{2} + 1, l_{2}, \eta_{2}, \zeta), (\rho, A_{1}, B_{1}, l_{1}, \eta_{1}, \zeta)\Big) \\
  & \hookrightarrow 
     \begin{pmatrix}
              \zeta (B + 1) & \cdots & -\zeta A \\
              \vdots &  & \vdots \\
              \zeta (A - l) & \cdots & - \zeta (B + l + 1) \\
       \end{pmatrix}
    \times
       \begin{pmatrix}
              \zeta B & \cdots & -\zeta A \\
              \vdots &  & \vdots \\
              \zeta (B + l - 1) & \cdots & - \zeta (A - l + 1) \\
              \zeta (B+l) & & \\              
              \vdots \\
               \zeta (A + 1)         
       \end{pmatrix}   \\
    &  \rtimes \r^{\Sigma_{0}}_{M, >_{\q}}\Big(\q_{-}, \ul_{-}, \ueta_{-}; (\rho, B + l, B + l - 1, 0, \eta_{1}, \zeta)\Big)
\end{align*}
If we apply $\Jac_{(\rho, A+1, B+1, \zeta) \mapsto (\rho, A, B, \zeta)}$ to the full induced representation above, we should get zero. This means $\r^{\Sigma_{0}}_{M, >_{\q}}(\q, \ul, \ueta) = 0$. \\

    \item $l = 0$ and $A > B + 1$. \\
    
    From the previous discussion in (1), we have
    \begin{align*}
  \r^{\Sigma_{0}}_{M, >_{\q}}(\q_{\gg}, \ul, \ueta) & \hookrightarrow 
      \underbrace{\begin{pmatrix}
              \zeta (B + T) & \cdots & \zeta (B + 2) & \cdots & -\zeta (A + 1) \\
              \vdots &  & \vdots & & \vdots \\
              \zeta (A -1 + T) & \cdots & \zeta (A + 1) & \cdots & - \zeta (B + 2) \\             
              \zeta (A + T)  & \cdots &  \zeta (A+2) & &
       \end{pmatrix}}_{(*-2)} \\                  
   & \rtimes \r^{\Sigma_{0}}_{M, >_{\q}}\Big(\q_{-}, \ul_{-}, \ueta_{-}; (\rho, B + 1, B, 0, \eta_{1}, \zeta)\Big) \\
   & \hookrightarrow
   \underbrace{\begin{pmatrix}
              \zeta (B + T) & \cdots & \zeta (B + 2) & \cdots & -\zeta A \\
              \vdots &  & \vdots & & \vdots \\
              \zeta (A -1 + T) & \cdots & \zeta (A + 1) & \cdots & - \zeta (B + 1) \\             
              \zeta (A + T)  & \cdots &  \zeta (A+2) & &
       \end{pmatrix}}_{(*-2)_{-}} \\
   & \times
       \underbrace{\begin{pmatrix}
               -\zeta (A+1) \\
              \vdots \\
              - \zeta (B + 2) 
       \end{pmatrix}}_{IV}   \rtimes \r^{\Sigma_{0}}_{M, >_{\q}}\Big(\q_{-}, \ul_{-}, \ueta_{-}; (\rho, B + 1, B, 0, \eta_{1}, \zeta)\Big) 
    \end{align*}  
   There exists an irreducible constituent $\sigma$ of 
   \[
   (IV) \rtimes \r^{\Sigma_{0}}_{M, >_{\q}}\Big(\q_{-}, \ul_{-}, \ueta_{-}; (\rho, B + 1, B, 0, \eta_{1}, \zeta)\Big)
   \]
   such that
  \begin{align*}
  \r^{\Sigma_{0}}_{M, >_{\q}}(\q_{\gg}, \ul, \ueta) & \hookrightarrow 
    \underbrace{\begin{pmatrix}
              \zeta (B + T) & \cdots & \zeta (B + 2) & \cdots & -\zeta A \\
              \vdots &  & \vdots & & \vdots \\
              \zeta (A -1 + T) & \cdots & \zeta (A + 1) & \cdots & - \zeta (B + 1) \\             
              \zeta (A + T)  & \cdots &  \zeta (A+2) & &
       \end{pmatrix}}_{(*-2)_{-}} \rtimes \sigma.
    \end{align*}   
    We claim $\Jac_{x} \sigma = 0$ for $x \in [\zeta (B + 1), \zeta (A + 1)]$. This is clear when $x \neq \zeta (B + 2)$. If $\Jac_{\zeta (B + 2)} \sigma \neq 0$, then $\Jac_{\zeta (B + 2)} \r^{\Sigma_{0}}_{M, >_{\q}}(\q_{\gg}, \ul, \ueta) \neq 0$, and we get a contradiction. Here we have used the fact that $A > B + 1$. Note
     \begin{align*}
  &\r^{\Sigma_{0}}_{M, >_{\q}}\Big(\q_{-}, \ul_{-}, \ueta_{-}; (\rho, A_{2} + 1, B_{2} + 1, l_{2}, \eta_{2}, \zeta), (\rho, A_{1}, B_{1}, l_{1}, \eta_{1}, \zeta)\Big) \\
  & \hookrightarrow       
   \begin{pmatrix}
              \zeta (B + 1) & \cdots & -\zeta A \\
              \vdots &  & \vdots \\
              \zeta A & \cdots & - \zeta (B + 1)
       \end{pmatrix}   \rtimes  \sigma.        
    \end{align*}     
    If we apply $\Jac_{(\rho, A+1, B+1, \zeta) \mapsto (\rho, A, B, \zeta)}$ to the full induced representation above, we should get zero. 
    This means $\r^{\Sigma_{0}}_{M, >_{\q}}(\q, \ul, \ueta) = 0$. \\

    \item $l = 0$ and $A = B + 1$. \\
    
    We can further simplify from the previous case (3) that 
    \begin{align*}
  \r^{\Sigma_{0}}_{M, >_{\q}}(\q_{\gg}, \ul, \ueta) & \hookrightarrow 
      \underbrace{\begin{pmatrix}
              \zeta (B + T) & \cdots & \zeta (B + 2) & \cdots & -\zeta (A + 1) \\           
              \zeta (A + T)  & \cdots &  \zeta ( A + 2 ) & &
       \end{pmatrix}}_{(*-2)} \\                  
   & \rtimes \r^{\Sigma_{0}}_{M, >_{\q}}\Big(\q_{-}, \ul_{-}, \ueta_{-}; (\rho, B + 1, B, 0, \eta_{1}, \zeta)\Big) \\
   & \hookrightarrow
      \underbrace{\begin{pmatrix}
              \zeta (B + T) & \cdots & \zeta (B + 2) & \cdots & -\zeta (B+1) \\
              \zeta (B + 1 + T)  & \cdots &  \zeta (B+3) & &
       \end{pmatrix}}_{(*-2)_{-}} \\
   & \times \rho||^{- \zeta (B + 2)}  \rtimes \r^{\Sigma_{0}}_{M, >_{\q}}\Big(\q_{-}, \ul_{-}, \ueta_{-}; (\rho, B + 1, B, 0, \eta_{1}, \zeta)\Big). 
    \end{align*}  
   So 
    \begin{align*}
  &\r^{\Sigma_{0}}_{M, >_{\q}}\Big(\q_{-}, \ul_{-}, \ueta_{-}; (\rho, A_{2} + 1, B_{2} + 1, l_{2}, \eta_{2}, \zeta), (\rho, A_{1}, B_{1}, l_{1}, \eta_{1}, \zeta)\Big) \\
  & \hookrightarrow       
   \langle \zeta (B + 1), \cdots, - \zeta (B+1) \rangle \times \rho||^{- \zeta (B + 2)} \\
  & \rtimes \r^{\Sigma_{0}}_{M, >_{\q}}\Big(\q_{-}, \ul_{-}, \ueta_{-}; (\rho, B + 1, B, 0, \eta_{1}, \zeta)\Big).        
    \end{align*}     
   There exists an irreducible constituent $\sigma$ of 
   \[
   \rho||^{- \zeta (B + 2)} \rtimes \r^{\Sigma_{0}}_{M, >_{\q}}\Big(\q_{-}, \ul_{-}, \ueta_{-}; (\rho, B + 1, B, 0, \eta_{1}, \zeta)\Big)
   \] 
   such that
   \begin{align*}
  &\r^{\Sigma_{0}}_{M, >_{\q}}\Big(\q_{-}, \ul_{-}, \ueta_{-}; (\rho, A_{2} + 1, B_{2} + 1, l_{2}, \eta_{2}, \zeta), (\rho, A_{1}, B_{1}, l_{1}, \eta_{1}, \zeta)\Big) \\
  & \hookrightarrow       
   \langle \zeta (B + 1), \cdots, - \zeta (B+1) \rangle \rtimes \sigma.     
    \end{align*}   
   We claim $\Jac_{x} \sigma = 0$ for $x \in [\zeta (B + 1), \zeta (B + 2)]$. It is clear $\Jac_{\zeta (B + 1)} \sigma = 0$. Suppose $\Jac_{\zeta (B + 2)} \sigma \neq 0$, then
   \[
   \sigma \hookrightarrow \rho||^{\zeta (B + 2)|} \rtimes \r^{\Sigma_{0}}_{M, >_{\q}}\Big(\q_{-}, \ul_{-}, \ueta_{-}; (\rho, B + 1, B, 0, \eta_{1}, \zeta)\Big)
   \]
   as the unique irreducible subrepresentation. So 
   \[
   \sigma =  \r^{\Sigma_{0}}_{M, >_{\q}}\Big(\q_{-}, \ul_{-}, \ueta_{-}; (\rho, B+2, B+2, -\eta_{1}, \zeta), (\rho, B, B, \eta_{1}, \zeta)\Big).
   \]
   This implies $\Jac_{- \zeta (B + 2)} \sigma = 0$. In particular, 
   \[
   \Jac_{\zeta (B + 1), \cdots, - \zeta (B + 2)} \Big( \langle \zeta (B + 1), \cdots, - \zeta (B+1) \rangle \rtimes \sigma \Big) = 0.
   \]
   On the other hand, 
   \[
   \Jac_{\zeta (B + 1), \cdots, - \zeta (B + 2)} \r^{\Sigma_{0}}_{M, >_{\q}}\Big(\q_{-}, \ul_{-}, \ueta_{-}; (\rho, A_{2} + 1, B_{2} + 1, l_{2}, \eta_{2}, \zeta), (\rho, A_{1}, B_{1}, l_{1}, \eta_{1}, \zeta)\Big) \neq 0.
   \]
   So we get a contradiction. As a result, 
   \[
   \Jac_{\zeta (B+1), \zeta (B+2)} \Big( \langle \zeta (B + 1), \cdots, - \zeta (B+1) \rangle \rtimes \sigma \Big) = 0,
   \]
   and hence $\r^{\Sigma_{0}}_{M, >_{\q}}(\q, \ul, \ueta) = 0$.

    \end{enumerate}

\subsubsection{{\bf Case:} $\eta_{2} \neq (-1)^{A_{1} - B_{1}}\eta_{1}$ and $l_{1} = l_{2} = (A_{1} - B_{1})/2$}

Let $l = l_{1} = l_{2} \neq 0$, then $A - l = B + l$.

   \begin{align*}
  \r^{\Sigma_{0}}_{M, >_{\q}}(\q_{\gg}, \ul, \ueta) & \hookrightarrow 
      \begin{pmatrix}
              \zeta B_{1} & \cdots & -\zeta A_{1} \\
              \vdots &  & \vdots \\
              \zeta (B_{1} + l_{1} - 1) & \cdots & - \zeta (A_{1} - l_{1} + 1)
       \end{pmatrix}  \\
  & \times            
      \begin{pmatrix}
              \zeta (B_{2} + T) & \cdots & -\zeta (A_{2} + T) \\
              \vdots &  & \vdots \\
              \zeta (B_{2} + l_{2} - 1 + T) & \cdots & - \zeta (A_{2} - l_{2} + 1 + T)
       \end{pmatrix} \\    
  & \times \langle \zeta (B_{2} + l_{2} + T), \cdots, \zeta (B_{2} + l_{2} + 2) \rangle \\
  & \rtimes \r^{\Sigma_{0}}_{M, >_{\q}}\Big(\q_{-}, \ul_{-}, \ueta_{-}; (\rho, B_{2} + l_{2} + 1, B_{2} + l_{2} + 1, \eta_{2}, \zeta), (\rho, B_{1} + l_{1}, B_{1} + l_{1}, \eta_{1}, \zeta)\Big) \\
  & \hookrightarrow 
   \underbrace{\begin{pmatrix}
              \zeta B & \cdots & -\zeta A \\
              \vdots &  & \vdots \\
              \zeta (B + l - 1) & \cdots & - \zeta (A - l + 1)
       \end{pmatrix}}_{*-1} 
   \times            
     \underbrace{\begin{pmatrix}
              \zeta (B + T) & \cdots & -\zeta A \\
              \vdots &  & \vdots \\
              \zeta (B + l - 1 + T) & \cdots & - \zeta (A - l + 1)
       \end{pmatrix}}_{*-2} \\
  & \times
       \underbrace{\begin{pmatrix}
              -\zeta (A+1) & \cdots & -\zeta (A + T) \\
              \vdots &  & \vdots \\
              -\zeta (A - l + 2 ) & \cdots & - \zeta (A - l + 1 + T)
       \end{pmatrix}}_{I} \\      
  & \times  \underbrace{\langle \zeta (B + l + T), \cdots \zeta (B + l + 2) \rangle}_{II}
     \times 
       \underbrace{\begin{pmatrix}
              \zeta (B + l)  \\
              \zeta (B + l + 1 ) 
       \end{pmatrix}}_{III} \\
  & \rtimes \r^{\Sigma_{0}}_{M, >_{\q}}\Big(\q_{-}, \ul_{-}, \ueta_{-}; (\rho, B + l, B + l, - \eta_{1}, \zeta), (\rho, B + l - 1, B + l - 1, \eta_{1}, \zeta)\Big).
   \end{align*}  
  We interchange $(I)$ with $(II)$ and $(III)$, then we take dual of $(I)$.
  \begin{align*}
  \r^{\Sigma_{0}}_{M, >_{\q}}(\q_{\gg}, \ul, \ueta) & \hookrightarrow 
      \underbrace{\begin{pmatrix}
              \zeta B & \cdots & -\zeta A \\
              \vdots &  & \vdots \\
              \zeta (B + l - 1) & \cdots & - \zeta (A - l + 1)
       \end{pmatrix}}_{*-1} 
   \times            
     \underbrace{\begin{pmatrix}
              \zeta (B + T) & \cdots & -\zeta A \\
              \vdots &  & \vdots \\
              \zeta (B + l - 1 + T) & \cdots & - \zeta (A - l + 1)
       \end{pmatrix}}_{*-2} \\      
  & \times  \underbrace{\langle \zeta (B + l + T), \cdots \zeta (B + l + 2) \rangle }_{II}
     \times 
       \underbrace{\begin{pmatrix}
              \zeta (B + l)  \\
              \zeta (B + l + 1 ) 
       \end{pmatrix}}_{III} \\
  & \times
       \underbrace{\begin{pmatrix}
              \zeta (A - l + 1 + T)  & \cdots & \zeta (A - l + 2 ) \\
              \vdots &  & \vdots \\
               \zeta (A + T)  & \cdots &  \zeta (A+1)
       \end{pmatrix}}_{(I)^{\vee}} \\
   & \rtimes \r^{\Sigma_{0}}_{M, >_{\q}}\Big(\q_{-}, \ul_{-}, \ueta_{-}; (\rho, B + l, B + l, - \eta_{1}, \zeta), (\rho, B + l - 1, B + l - 1, \eta_{1}, \zeta)\Big) \\
   & \hookrightarrow \underbrace{\begin{pmatrix}
              \zeta B & \cdots & -\zeta A \\
              \vdots &  & \vdots \\
              \zeta (B + l - 1) & \cdots & - \zeta (A - l + 1)
       \end{pmatrix}}_{*-1} 
   \times            
     \underbrace{\begin{pmatrix}
              \zeta (B + T) & \cdots & -\zeta A \\
              \vdots &  & \vdots \\
              \zeta (B + l - 1 + T) & \cdots & - \zeta (A - l + 1)
       \end{pmatrix}}_{*-2} \\      
  & \times  \underbrace{\langle \zeta (B + l + T), \cdots \zeta (B + l + 2) \rangle }_{II}
     \times 
       \underbrace{\begin{pmatrix}
              \zeta (B + l)  \\
              \zeta (B + l + 1 ) 
       \end{pmatrix}}_{III} \\
  & \times
       \underbrace{\begin{pmatrix}
              \zeta (A - l + 1 + T)  & \cdots & \zeta (A - l + 3) \\
              \vdots &  & \vdots \\
               \zeta (A + T)  & \cdots &  \zeta (A+2)
       \end{pmatrix}}_{(I)^{\vee}_{-}} 
     \times
       \underbrace{\begin{pmatrix}
              \zeta (A - l + 2) \\
              \vdots  \\
               \zeta (A + 1)
       \end{pmatrix}}_{IV} \\
   & \rtimes \r^{\Sigma_{0}}_{M, >_{\q}}\Big(\q_{-}, \ul_{-}, \ueta_{-}; (\rho, B + l, B + l, - \eta_{1}, \zeta), (\rho, B + l - 1, B + l - 1, \eta_{1}, \zeta)\Big).
   \end{align*}  
   Note $(*-1)$ is interchangeable with $(*-2)$, $(II)$ and $(I)^{\vee}_{-}$. And $(I)^{\vee}_{-}$ is also interchangeable with $(III)$. So
   \begin{align*}
  \r^{\Sigma_{0}}_{M, >_{\q}}(\q_{\gg}, \ul, \ueta) & \hookrightarrow          
     \underbrace{\begin{pmatrix}
              \zeta (B + T) & \cdots & -\zeta A \\
              \vdots &  & \vdots \\
              \zeta (B + l - 1 + T) & \cdots & - \zeta (A - l + 1)
       \end{pmatrix}}_{*-2}      
   \times  \underbrace{\langle \zeta (B + l + T), \cdots \zeta (B + l + 2) \rangle }_{II} \\
  & \times
       \underbrace{\begin{pmatrix}
              \zeta (A - l + 1 + T)  & \cdots & \zeta (A - l + 3) \\
              \vdots &  & \vdots \\
               \zeta (A + T)  & \cdots &  \zeta (A+2)
       \end{pmatrix}}_{(I)^{\vee}_{-}} 
   \times 
     \underbrace{\begin{pmatrix}
              \zeta B & \cdots & -\zeta A \\
              \vdots &  & \vdots \\
              \zeta (B + l - 1) & \cdots & - \zeta (A - l + 1)
       \end{pmatrix}}_{*-1} \\
  & \times        
       \underbrace{\begin{pmatrix}
              \zeta (B + l)  \\
              \zeta (B + l + 1 ) 
       \end{pmatrix}}_{III} 
   \times
       \underbrace{\begin{pmatrix}
              \zeta (A - l + 2) \\
              \vdots  \\
               \zeta (A + 1)
       \end{pmatrix}}_{IV} \\
   & \rtimes \r^{\Sigma_{0}}_{M, >_{\q}}\Big(\q_{-}, \ul_{-}, \ueta_{-}; (\rho, B + l, B + l, - \eta_{1}, \zeta), (\rho, B + l - 1, B + l - 1, \eta_{1}, \zeta)\Big).
   \end{align*}  
  Since $\Jac_{\zeta (B + l + T)} \r^{\Sigma_{0}}_{M, >_{\q}}(\q_{\gg}, \ul, \ueta) = \Jac_{\zeta (A - l + 1 + T)} \r^{\Sigma_{0}}_{M, >_{\q}}(\q_{\gg}, \ul, \ueta) = 0$, we can ``combine" $(*-2)$, $(II)$ and $(I)^{\vee}_{-}$.
  \begin{align*}
  \r^{\Sigma_{0}}_{M, >_{\q}}(\q_{\gg}, \ul, \ueta) & \hookrightarrow          
     \underbrace{\begin{pmatrix}
              \zeta (B + T) & \cdots & \zeta (B + 2) & \cdots & -\zeta A \\
              \vdots &  & \vdots \\
              \zeta (B + l - 1 + T) & \cdots &  \zeta (B + l + 1) & \cdots & - \zeta (A - l + 1) \\
              \zeta (B + l + T) & \cdots & \zeta (B + l + 2) \\
              \vdots &  & \vdots \\
               \zeta (A + T)  & \cdots &  \zeta (A+2)
       \end{pmatrix}}_{(*-2)_{+}} \\      
  & \times 
     \underbrace{\begin{pmatrix}
              \zeta B & \cdots & -\zeta A \\
              \vdots &  & \vdots \\
              \zeta (B + l - 1) & \cdots & - \zeta (A - l + 1)
       \end{pmatrix}}_{*-1} \\
  & \times        
       \underbrace{\begin{pmatrix}
              \zeta (B + l)  \\
              \zeta (B + l + 1 ) 
       \end{pmatrix}}_{III} 
   \times
       \underbrace{\begin{pmatrix}
              \zeta (A - l + 2) \\
              \vdots  \\
               \zeta (A + 1)
       \end{pmatrix}}_{IV} \\
   & \rtimes \r^{\Sigma_{0}}_{M, >_{\q}}\Big(\q_{-}, \ul_{-}, \ueta_{-}; (\rho, B + l, B + l, - \eta_{1}, \zeta), (\rho, B + l - 1, B + l - 1, \eta_{1}, \zeta)\Big).
   \end{align*}  
   Since $\Jac_{\zeta (B + l)} \r^{\Sigma_{0}}_{M, >_{\q}}(\q_{\gg}, \ul, \ueta) = \Jac_{\zeta (A - l + 2)} \r^{\Sigma_{0}}_{M, >_{\q}}(\q_{\gg}, \ul, \ueta) = 0$, we can ``combine" $(*-1)$, $(III)$ and $(IV)$.
   \begin{align*}
  \r^{\Sigma_{0}}_{M, >_{\q}}(\q_{\gg}, \ul, \ueta) & \hookrightarrow          
     \underbrace{\begin{pmatrix}
              \zeta (B + T) & \cdots & \zeta (B + 2) & \cdots & -\zeta A \\
              \vdots &  & \vdots \\
              \zeta (B + l - 1 + T) & \cdots &  \zeta (B + l + 1) & \cdots & - \zeta (A - l + 1) \\
              \zeta (B + l + T) & \cdots & \zeta (B + l + 2) \\
              \vdots &  & \vdots \\
               \zeta (A + T)  & \cdots &  \zeta (A+2)
       \end{pmatrix}}_{(*-2)_{+}} \\      
  & \times 
     \underbrace{\begin{pmatrix}
              \zeta B & \cdots & -\zeta A \\
              \vdots &  & \vdots \\
              \zeta (B + l - 1) & \cdots & - \zeta (A - l + 1) \\
              \zeta (B + l)  & & \\
              \vdots & & \\
              \zeta (A + 1) & &
       \end{pmatrix}}_{(*-1)_{+}} \\
   & \rtimes \r^{\Sigma_{0}}_{M, >_{\q}}\Big(\q_{-}, \ul_{-}, \ueta_{-}; (\rho, B + l, B + l, - \eta_{1}, \zeta), (\rho, B + l - 1, B + l - 1, \eta_{1}, \zeta)\Big).
   \end{align*}  
   Then
   \begin{align*}
  &\r^{\Sigma_{0}}_{M, >_{\q}}\Big(\q_{-}, \ul_{-}, \ueta_{-}; (\rho, A_{2} + 1, B_{2} + 1, l_{2}, \eta_{2}, \zeta), (\rho, A_{1}, B_{1}, l_{1}, \eta_{1}, \zeta)\Big) \\
  & \hookrightarrow       
   \begin{pmatrix}
              \zeta (B + 1) & \cdots & -\zeta A \\
              \vdots &  & \vdots \\
              \zeta (B + l) & \cdots & - \zeta (A - l + 1)
       \end{pmatrix}  
    \times       
      \begin{pmatrix}
              \zeta B & \cdots & -\zeta A \\
              \vdots &  & \vdots \\              
              \zeta (B + l - 1) & \cdots & -\zeta (A- l + 1)  \\
              \zeta (B + l) \\
              \vdots \\              
              \zeta (A + 1)           
       \end{pmatrix}   \\
   & \rtimes \r^{\Sigma_{0}}_{M, >_{\q}}\Big(\q_{-}, \ul_{-}, \ueta_{-}; (\rho, B + l, B + l, - \eta_{1}, \zeta), (\rho, B + l - 1, B + l - 1, \eta_{1}, \zeta)\Big).
      \end{align*}   
   Therefore, if we apply $\Jac_{(\rho, A+1, B+1, \zeta) \mapsto (\rho, A, B, \zeta)}$ to the full induced representation above, we should get zero. This means $\r^{\Sigma_{0}}_{M, >_{\q}}(\q, \ul, \ueta) = 0$.   
   

\section{Example}
\label{sec: example}

In this appendix, we want to demonstrate how our procedure (cf. Section~\ref{sec: procedure}) works in a simple example. We fix $\rho$ and choose $\q \in \cQ{G}$, such that
\[
Jord(\q) = \{(\rho, A_{3}, B_{3}, \zeta_{3}), (\rho, A_{2}, B_{2}, \zeta_{2}), (\rho, A_{1}, B_{1}, \zeta_{1})\}.
\] 
We also  assume 

\begin{itemize}

\item $A_{i}, B_{i} \in \mathbb{Z}$ for $i = 1, 2, 3$; 

\item $A_{3} \geqslant A_{2} \geqslant A_{1}$ and $B_{3} \geqslant B_{2} \geqslant B_{1}$;

\item $\zeta_{3} = \zeta_{1} = +1$ and $\zeta_{2} = -1$.

\end{itemize}
We put an order $>_{\q}$ on $Jord(\q)$ such that
\[
(\rho, A_{3}, B_{3}, \zeta_{3}) >_{\q} (\rho, A_{2}, B_{2}, \zeta_{2}) >_{\q} (\rho, A_{1}, B_{1}, \zeta_{1}).
\]
We would like to find all $(\ul, \ueta)$ such that $\r^{\Sigma_{0}}_{>_{\q}}(\q, \ul, \ueta) \neq 0$.

\subsection{Results} 

First of all, we have 
\[
0 \leqslant l_{i} \leqslant [(A_{i} - B_{i} + 1)/2],
\] 
and 
\begin{align*}
\prod_{i = 1}^{3} \eta_{i}^{A_{i} - B_{i} + 1} (-1)^{[(A_{i}-B_{i}+1)/2] + l_{i}} = 1
\end{align*}
(cf. \eqref{eq: quasisplit}). Next, we formulate the necessary and sufficient conditions as linear constraints on $\{l_{i}\}$ for each selection of signs $\{\eta_{i}\}$.

\begin{enumerate}

\item If $\eta_{3} = (-1)^{(A_{1} - B_{1}) + (A_{2} - B_{2})}\eta_{1}$ and $\eta_{2} = (-1)^{A_{1} - B_{1}}\eta_{1}$, then
\[
\begin{cases}
l_{3} + l_{1} > A_{1} - B_{3} \\

-B_{2} \leqslant l_{2} - l_{1} \leqslant A_{2} - (A_{1} - B_{1}) \\

(A_{1} - B_{1}) - B_{3} + 1 \leqslant l_{3} - l_{2} + 2l_{1} \leqslant A_{3} + (A_{1} - B_{1}) - (A_{2} - B_{2}) + 1

\end{cases}
\]

\item If $\eta_{3} = (-1)^{(A_{1} - B_{1}) + (A_{2} - B_{2})}\eta_{1}$ and $\eta_{2} \neq (-1)^{A_{1} - B_{1}}\eta_{1}$, then
\[
\begin{cases}
l_{3} + l_{1} > A_{1} - B_{3} \\

l_{1} + l_{2} > (A_{1} - B_{1}) - B_{2} \\

l_{3} + l_{2} + 2l_{1} >  (A_{1} - B_{1}) + (A_{2} - B_{2}) - B_{3} + 1

\end{cases}
\]
 
\item If $\eta_{3} \neq (-1)^{(A_{1} - B_{1}) + (A_{2} - B_{2})}\eta_{1}$, $\eta_{2} = (-1)^{A_{1} - B_{1}}\eta_{1}$, and
\[
l_{3} - l_{1} < (A_{3} - B_{3})/2 - (A_{1} - B_{1}) + l_{1} 
\]
then
\[
\begin{cases}
-(B_{3} - B_{1}) \leqslant l_{3} - l_{1} \leqslant  A_{3} - A_{1} \\

-B_{2} \leqslant l_{2} - l_{1} \leqslant A_{2} - (A_{1} - B_{1}) \\

l_{3} + l_{2} - 2l_{1} >  (A_{2} - B_{2}) - (A_{1} - B_{1}) - B_{3} - 1

\end{cases}
\]

\item If $\eta_{3} \neq (-1)^{(A_{1} - B_{1}) + (A_{2} - B_{2})}\eta_{1}$, $\eta_{2} = (-1)^{A_{1} - B_{1}}\eta_{1}$, and
\[
l_{3} - l_{1} \geqslant (A_{3} - B_{3})/2 - (A_{1} - B_{1}) + l_{1} 
\]
then
\[
\begin{cases}
-(B_{3} - B_{1}) \leqslant l_{3} - l_{1} \leqslant  A_{3} - A_{1} \\

-B_{2} \leqslant l_{2} - l_{1} \leqslant A_{2} - (A_{1} - B_{1}) \\

(A_{1} - B_{1}) -A_{3} \leqslant -l_{3} - l_{2} + 2l_{1} \leqslant (A_{1} - B_{1}) - (A_{2} - B_{2}) + B_{3}

\end{cases}
\]

\item If $\eta_{3} \neq (-1)^{(A_{1} - B_{1}) + (A_{2} - B_{2})}\eta_{1}$, $\eta_{2} \neq (-1)^{A_{1} - B_{1}}\eta_{1}$, and
\[
l_{3} - l_{1} < (A_{3} - B_{3})/2 - (A_{1} - B_{1}) + l_{1} 
\]
then
\[
\begin{cases}
-(B_{3} - B_{1}) \leqslant l_{3} - l_{1} \leqslant  A_{3} - A_{1} \\

l_{1} + l_{2} > (A_{1} - B_{1}) - B_{2} \\

-(A_{1} - B_{1}) - B_{3} - 1 \leqslant l_{3} - l_{2} - 2l_{1} \leqslant A_{3} - (A_{1} - B_{1}) - (A_{2} - B_{2}) - 1

\end{cases}
\]

\item If $\eta_{3} \neq (-1)^{(A_{1} - B_{1}) + (A_{2} - B_{2})}\eta_{1}$, $\eta_{2} \neq (-1)^{A_{1} - B_{1}}\eta_{1}$, and
\[
l_{3} - l_{1} \geqslant (A_{3} - B_{3})/2 - (A_{1} - B_{1}) + l_{1} 
\]
then
\[
\begin{cases}
-(B_{3} - B_{1}) \leqslant l_{3} - l_{1} \leqslant  A_{3} - A_{1} \\

l_{1} + l_{2} > (A_{1} - B_{1}) - B_{2} \\

-l_{3} + l_{2} + 2l_{1} > (A_{1} - B_{1}) + (A_{2} - B_{2}) - A_{3}

\end{cases}
\]

\end{enumerate}

\begin{example}
\label{eg: three intervals}
Let $[A_{3}, B_{3}] = [40, 10], [A_{2}, B_{2}] = [37, 7]$ and $[A_{1}, B_{1}] = [8, 4]$, i.e.,
\[
\q = \rho \otimes \nu_{51} \otimes \nu_{31} \, \+ \,  \rho \otimes \nu_{31} \otimes \nu_{45} \, \+ \, \rho \otimes \nu_{13} \otimes \nu_{5}.
\]
First, we have  
\(
0 \leqslant l_{1} \leqslant 2,
0 \leqslant l_{2} \leqslant 15,
0 \leqslant l_{3} \leqslant 15, 
\)
and $(-1)^{l_{1} + l_{2} + l_{3}} \eta_{1} \eta_{2} \eta_{3} = 1$. Note $B_{3} > A_{1}$, so the conditions from \eqref{eq: case 1} are always satisfied. Also note $B_{2} > A_{1} - B_{1}$, then the conditions from \eqref{eq: case 2} are always satisfied too. Therefore, we can simplify the nonvanishing conditions as follows:

\begin{enumerate}

\item If $\eta_{3} = \eta_{1}$ and $\eta_{2} = \eta_{1}$, then
\[
-5 \leqslant l_{3} - l_{2} + 2l_{1} \leqslant 15
\]

\item If $\eta_{3} = \eta_{1}$ and $\eta_{2} \neq \eta_{1}$, then
\[
l_{3} + l_{2} + 2l_{1} >  25
\]
 
\item If $\eta_{3} \neq \eta_{1}$, $\eta_{2} = \eta_{1}$, and
\[
l_{3} - l_{1} < 11 + l_{1} 
\]
then
\[
l_{3} + l_{2} - 2l_{1} >  15
\]

\item If $\eta_{3} \neq \eta_{1}$, $\eta_{2} = \eta_{1}$, and
\[
l_{3} - l_{1} \geqslant 11 + l_{1} 
\]
then
\[
-36 \leqslant -l_{3} - l_{2} + 2l_{1} \leqslant -16
\]

\item If $\eta_{3} \neq \eta_{1}$, $\eta_{2} \neq \eta_{1}$, and
\[
l_{3} - l_{1} < 11 + l_{1} 
\]
then
\[
-15 \leqslant l_{3} - l_{2} - 2l_{1} \leqslant 5
\]

\item If $\eta_{3} \neq \eta_{1}$, $\eta_{2} \neq \eta_{1}$, and
\[
l_{3} - l_{1} \geqslant 11 + l_{1} 
\]
then
\[
-l_{3} + l_{2} + 2l_{1} > -6
\]

\end{enumerate}
To find the size of $\Pkt{\q}^{\Sigma_{0}}$ is equivalent to counting integral points in certain polytopes for each of the above six cases. By running a simple computer program, we can get $|\Pkt{\q}^{\Sigma_{0}}| = 1651$.

\end{example}

\subsection{Deduction}

First, we ``Expand" $[A_{3}, B_{3}]$ to $[A^{*}_{3}, B^{*}_{3}]$ such that $B^{*}_{3} = B_{1}$, and we denote the resulting parameter by $\q^{*}$.
Then $\r^{\Sigma_{0}}_{>_{\q}}(\q, \ul, \ueta) = \r^{\Sigma_{0}}_{>_{\q}}(\q^{*}, \ul^{*}, \ueta^{*})$, where
\[
l_{1}^{*} = l_{1}, \quad l_{2}^{*} = l_{2}, \quad l_{3}^{*} = l_{3} + (B_{3} - B_{1}),
\]
and
\[
\eta_{1}^{*} = \eta_{1}, \quad \eta_{2}^{*} = \eta_{2}, \quad \eta_{3}^{*} = \eta_{3}.
\]
Next, we change the order $>_{\q}$ to $>'_{\q}$:
\[
(\rho, A_{3}, B_{3}, \zeta_{3}) >'_{\q} (\rho, A_{1}, B_{1}, \zeta_{1}) >'_{\q} (\rho, A_{2}, B_{2}, \zeta_{2}).
\]
So $\r^{\Sigma_{0}}_{>_{\q}}(\q^{*}, \ul^{*}, \ueta^{*}) = \r^{\Sigma_{0}}_{>'_{\q}}(\q^{*}, \ul^{'*}, \ueta^{'*})$, where
\[
l_{1}^{'*} = l_{1}^{*}, \quad l_{2}^{'*} = l_{2}^{*}, \quad l_{3}^{'*} = l_{3}^{*},
\]
and
\[
\eta_{1}^{'*} = (-1)^{A_{2} - B_{2} + 1}\eta_{1}^{*}, \quad \eta_{2}^{'*} = (-1)^{A_{1} - B_{1} + 1}\eta_{2}^{*}, \quad \eta_{3}^{'*} = \eta_{3}^{*}.
\]
Then we can ``Pull" $[A^{*}_{3}, B^{*}_{3}]$, $[A_{1}, B_{1}]$. As a consequence, we have $\r^{\Sigma_{0}}_{>'_{\q}}(\q^{*}, \ul^{'*}, \ueta^{'*}) \neq 0$ if and only if all the three conditions below are satisfied.
\begin{enumerate}

\item $\r^{\Sigma_{0}}_{>'_{\q}}(\q^{*}_{1}, \ul^{'*}, \ueta^{'*}) \neq 0$, \\

\item $\r^{\Sigma_{0}}_{>'_{\q}}(\q^{*}_{2}, \ul^{'*}, \ueta^{'*}) \neq 0$, \\

\item $\r^{\Sigma_{0}}_{>''_{\q}}(\q^{*}_{3}, \ul^{''*}, \ueta^{''*}) \neq 0$.

\end{enumerate}
From each of these cases, we will get some explicit conditions on $(\ul, \ueta)$.

{\bf Case (1):} $\q^{*}_{1}$ is obtained from $\q^{*}$ by shifting both $[A^{*}_{3}, B^{*}_{3}]$ and $[A_{1}, B_{1}]$ away by $T_{1}$, so that 
\(
(\rho, A_{1} + T_{1}, B_{1} + T_{1}, \zeta_{1}) \gg (\rho, A_{2}, B_{2}, \zeta_{2}).
\)
Since $Jord(\q^{*}_{1})$ is in ``good shape", we have $\r^{\Sigma_{0}}_{>'_{\q}}(\q^{*}_{1}, \ul^{'*}, \ueta^{'*}) \neq 0$ if and only if
\begin{align}
\label{eq: case 1}
\begin{cases}
\eta_{3}^{'*} = (-1)^{A_{1} - B_{1}}\eta_{1}^{'*}   & \Rightarrow 0 \leqslant l^{'*}_{3} - l^{'*}_{1} \leqslant (A^{*}_{3} - B^{*}_{3}) - (A_{1} - B_{1}),  \\
\eta_{3}^{'*} \neq (-1)^{A_{1} - B_{1}}\eta_{1}^{'*}   & \Rightarrow l^{'*}_{3} + l^{'*}_{1} > A_{1} - B_{1}.   
\end{cases} 
\end{align}
Now we want to translate these conditions to that on $(\ul, \ueta)$. Note
\[
\eta_{3}^{'*} = \eta_{3}^{*} = \eta_{3}, \quad \quad \quad  \eta_{1}^{'*} = (-1)^{A_{2} - B_{2} + 1}\eta_{1}^{*} = (-1)^{A_{2} - B_{2} + 1}\eta_{1},
\]
and
\[
l^{'*}_{3} = l^{*}_{3} = l_{3} + (B_{3} - B_{1}), \quad \quad \quad l^{'*}_{1} = l^{*}_{1} = l_{1}.
\]
So we will get the following conditions from \eqref{eq: case 1}. 

\begin{itemize}

\item

If $\eta_{3} \neq (-1)^{(A_{1} - B_{1}) + (A_{2} - B_{2})} \eta_{1}$, then
\[
0 \leqslant l_{3} + (B_{3} - B_{1}) - l_{1} \leqslant (A_{3} + (B_{3} - B_{1}) - B_{1}) - (A_{1} - B_{1}), 
\]
which implies
\[
-(B_{3} - B_{1}) \leqslant l_{3} - l_{1} \leqslant A_{3} - A_{1}.
\]

\item
If $\eta_{3} = (-1)^{(A_{1} - B_{1}) + (A_{2} - B_{2})} \eta_{1}$, then
\[
l_{3} + (B_{3} - B_{1}) + l_{1} > A_{1} - B_{1},
\]
which implies
\[
l_{3} + l_{1} > A_{1} - B_{3}.
\]
\end{itemize}

{\bf Case (2):} $\q^{*}_{2}$ is obtained from $\q^{*}$ by shifting $[A^{*}_{3}, B^{*}_{3}]$ away by $T_{2}$, so that 
\(
(\rho, A^{*}_{3} + T_{2}, B^{*}_{3} + T_{2}, \zeta_{3}) \gg \{ (\rho, A_{2}, B_{2}, \zeta_{2}), (\rho, A_{1}, B_{1}, \zeta_{1}) \}.
\)
Note $\r^{\Sigma_{0}}_{>'_{\q}}(\q^{*}_{2}, \ul^{'*}, \ueta^{'*}) = \r^{\Sigma_{0}}_{>_{\q}}(\q^{*}_{2}, \ul^{*}, \ueta^{*})$. So we can ``Expand" $[A_{2}, B_{2}]$ to $[A^{*}_{2}, B^{*}_{2}]$ such that $B^{*}_{2} = 0$, and we denote the resulting parameter by $\q^{**}_{2}$. Then $\r^{\Sigma_{0}}_{>_{\q}}(\q^{*}_{2}, \ul^{*}, \ueta^{*}) = \r^{\Sigma_{0}}_{>_{\q}}(\q^{**}_{2}, \ul^{**}, \ueta^{**})$, where
\[
l_{1}^{**} = l_{1}^{*}, \quad l_{2}^{**} = l_{2}^{*} + B_{2}, \quad l_{3}^{**} = l_{3}^{*},
\]
and
\[
\eta_{1}^{**} = \eta_{1}^{*}, \quad \eta_{2}^{**} = \eta_{2}^{*}, \quad \eta_{3}^{**} = \eta_{3}^{*}.
\]
Finally, we can change the order again 
\[
\r^{\Sigma_{0}}_{>_{\q}}(\q^{**}_{2}, \ul^{**}, \ueta^{**}) = \r^{\Sigma_{0}}_{>'_{\q}}(\q^{**}_{2}, \ul^{'**}, \ueta^{'**}), 
\]
where 
\[
l_{1}^{'**} = l_{1}^{**}, \quad l_{2}^{'**} = l_{2}^{**}, \quad l_{3}^{'**} = l_{3}^{**},
\]
and
\[
\eta_{1}^{'**} = (-1)^{A_{2} - B_{2} + 1}\eta_{1}^{**}, \quad \eta_{2}^{'**} = (-1)^{A_{1} - B_{1} + 1}\eta_{2}^{**}, \quad \eta_{3}^{'**} = \eta_{3}^{**}.
\]
Then
\[
\r^{\Sigma_{0}}_{>'_{\q}}(\q^{**}_{2}, \ul^{'**}, \ueta^{'**}) = \r^{\Sigma_{0}}_{>'_{\q}}\Big(\q^{**}_{2-}, \ul^{'**}_{-}, \ueta^{'**}_{-}; (\rho, A^{*}_{2}, B^{*}_{2}, \eta^{'**}_{2}, -\zeta_{2})\Big) \neq 0,
\]
if and only if
\begin{align}
\label{eq: case 2}
\begin{cases}
\eta^{'**}_{1} = (-1)^{A_{2} - B_{2}}\eta^{'**}_{2}       & \Rightarrow 0 \leqslant l^{'**}_{2} - l^{'**}_{1} \leqslant (A^{*}_{2} - B^{*}_{2}) - (A_{1} - B_{1}),  \\
\eta^{'**}_{1} \neq (-1)^{A_{2} - B_{2}}\eta^{'**}_{2}  & \Rightarrow l^{'**}_{2} + l^{'**}_{1} > A_{1} - B_{1}.   
\end{cases} 
\end{align}
 
Now we want to translate these conditions to that on $(\ul, \ueta)$. Note
\begin{align*}
\eta^{'**}_{1} = (-1)^{A_{2} - B_{2} + 1}\eta_{1}^{**} = (-1)^{A_{2} - B_{2} + 1}\eta_{1}^{*} = (-1)^{A_{2} - B_{2} + 1}\eta_{1} \\
\eta^{'**}_{2} = (-1)^{A_{1} - B_{1} + 1}\eta_{2}^{**} = (-1)^{A_{1} - B_{1} + 1}\eta_{2}^{*} = (-1)^{A_{1} - B_{1} + 1}\eta_{2}
\end{align*}
and
\[
l^{'**}_{2} = l_{2}^{**} = l_{2}^{*} + B_{2} = l_{2} + B_{2}, \quad \quad l^{'**}_{1} = l_{1}^{**} = l_{1}^{*} = l_{1}.
\]
So we will get the following conditions from \eqref{eq: case 2}.

\begin{itemize}

\item If $\eta_{2} = (-1)^{A_{1} - B_{1}}\eta_{1}$, then
\[
0 \leqslant l_{2} + B_{2} - l_{1} \leqslant (A_{2} + B_{2} - 0) - (A_{1} - B_{1}),
\]
which implies
\[
-B_{2} \leqslant l_{2} - l_{1} \leqslant A_{2} - (A_{1} - B_{1}).
\]

\item If $\eta_{2} \neq (-1)^{A_{1} - B_{1}}\eta_{1}$, then 
\[
l_{2} + B_{2} + l_{1} > A_{1} - B_{1},
\]
which implies
\[
l_{2} + l_{1} > (A_{1} - B_{1}) - B_{2}.
\]

\end{itemize}

{\bf Case (3):} $\q^{*}_{3}$ is obtained from $\q^{*}$ by shifting $[A_{1}, B_{1}]$ away by $T_{3}$, so that 
\(
(\rho, A_{1} + T_{3}, B_{1} + T_{3}, \zeta_{1}) \gg \{ (\rho, A^{*}_{3}, B^{*}_{3}, \zeta_{3}), (\rho, A_{2}, B_{2}, \zeta_{2}) \}.
\)
The order $>''_{\q}$ is given by
\[
(\rho, A_{1}, B_{1}, \zeta_{1}) >''_{\q} (\rho, A_{3}, B_{3}, \zeta_{3}) >''_{\q} (\rho, A_{2}, B_{2}, \zeta_{2}).
\]
And $\r^{\Sigma_{0}}_{>'_{\q}}(\q^{*}, \ul^{'*}, \ueta^{'*}) = \r^{\Sigma_{0}}_{>''_{\q}}(\q^{*}, \ul^{''*}, \ueta^{''*})$, where $(\ul^{''*}, \ueta^{''*}) = S^{+}(\ul^{'*}, \ueta^{'*})$. In particular,
\[
\eta^{''*}_{2} = \eta^{'*}_{2}, \quad \quad l^{''*}_{2} = l^{'*}_{2}.
\]

         
         
               
Then we can ``Expand" $[A^{*}_{3}, B^{*}_{3}]$ to $[A^{**}_{3}, B^{**}_{3}]$ such that $B^{**}_{3} = 0$, and we denote the resulting parameter by $\q^{**}_{3}$. 
It follows $\r^{\Sigma_{0}}_{>''_{\q}}(\q^{*}_{3}, \ul^{''*}, \ueta^{''*}) = \r^{\Sigma_{0}}_{>''_{\q}}(\q^{*}_{3}, \ul^{''**}, \ueta^{''**})$, where
\[
l_{1}^{''**} = l_{1}^{''*}, \quad l_{2}^{''**} = l_{2}^{''*}, \quad l_{3}^{''**} = l_{3}^{''*} + B_{1},
\]
and
\[
\eta_{1}^{''**} = \eta_{1}^{''*}, \quad \eta_{2}^{''**} = \eta_{2}^{''*}, \quad \eta_{3}^{''**} = \eta_{3}^{''*}.
\]
We change the order $>''_{\q}$ to $>'''_{\q}$:
\[
(\rho, A_{1}, B_{1}, \zeta_{1}) >'''_{\q} (\rho, A_{2}, B_{2}, \zeta_{2}) >'''_{\q} (\rho, A_{3}, B_{3}, \zeta_{3}),
\]
then
\[
\r^{\Sigma_{0}}_{>''_{\q}}(\q^{**}_{2}, \ul^{''**}, \ueta^{''**}) = \r^{\Sigma_{0}}_{>'''_{\q}}(\q^{**}_{2}, \ul^{'''**}, \ueta^{'''**}), 
\]
where 
\[
l_{1}^{'''**} = l_{1}^{''**}, \quad l_{2}^{'''**} = l_{2}^{''**}, \quad l_{3}^{'''**} = l_{3}^{''**},
\]
and
\[
\eta_{1}^{'''**} = \eta_{1}^{''**}, \quad \eta_{2}^{'''**} = (-1)^{A_{3} - B_{3} + 1}\eta_{2}^{''**}, \quad \eta_{3}^{'''**} = (-1)^{A_{2} - B_{2} + 1}\eta_{3}^{''**}.
\]
Then
\[
\r^{\Sigma_{0}}_{>'''_{\q}}(\q^{**}_{3}, \ul^{'''**}, \ueta^{'''**}) = \r^{\Sigma_{0}}_{>'''_{\q}}\Big(\q^{**}_{3 -}, \ul^{'''**}_{-}, \ueta^{'''**}_{-}; (\rho, A^{**}_{3}, B^{**}_{3}, \eta^{'''**}_{3}, -\zeta_{3})\Big) \neq 0,
\]
if and only if
\begin{align}
\label{eq: case 3}
\begin{cases}
\eta^{'''**}_{2} = (-1)^{A_{3} - B_{3}}\eta^{'''**}_{3}       & \Rightarrow 0 \leqslant l^{'''**}_{3} - l^{'''**}_{2} \leqslant (A^{**}_{3} - B^{**}_{3}) - (A_{2} - B_{2}),  \\
\eta^{'''**}_{2} \neq (-1)^{A_{3} - B_{3}}\eta^{'''**}_{3}   & \Rightarrow l^{'''**}_{3} + l^{'''**}_{2} > A_{2} - B_{2}.   
\end{cases} 
\end{align}

Now we want to translate these conditions to that on $(\ul, \ueta)$. Note
\begin{align*}
\eta^{'''**}_{2} & = (-1)^{A_{3} - B_{3} + 1}\eta_{2}^{''**} = (-1)^{A_{3} - B_{3} + 1}\eta_{2}^{''*} =  (-1)^{A_{3} - B_{3} + 1}\eta_{2}^{'*} \\
& = (-1)^{A_{3} - B_{3} + 1}(-1)^{A_{1} - B_{1} + 1}\eta_{2}^{*}
= (-1)^{(A_{3} - B_{3}) + (A_{1} - B_{1})}\eta_{2} \\
\eta^{'''**}_{3} & = (-1)^{A_{2} - B_{2} + 1}\eta_{3}^{''**}
\end{align*}
and
\begin{align*}
l^{'''**}_{2} & = l_{2}^{''**} = l_{2}^{''*} = l_{2}^{'*} = l_{2}^{*} = l_{2} \\ 
l^{'''**}_{3} & = l_{3}^{''**} = l_{3}^{''*} + B_{1}
\end{align*}
To proceed further, we need to use the formula for $(\ul^{''*}, \ueta^{''*}) = S^{+}(\ul^{'*}, \ueta^{'*})$.
\begin{itemize}

\item If $\eta^{'*}_{3} \neq (-1)^{A_{1} - B_{1}} \eta^{'*}_{1}$, then $\eta^{''*}_{1} = (-1)^{A_{3} - B_{3}} \eta^{''*}_{3}$ and 
         \[
         \begin{cases}
         l^{'*}_{1} = l^{''*}_{1} \\
          l^{'*}_{3} - l^{''*}_{3} = (A_{1} - B_{1} - 2l^{'*}_{1}) + 1 \\
         \eta^{''*}_{1} = (-1)^{A_{3} - B_{3}} \eta^{'*}_{1} 
         \end{cases}
         \]
It follows    
\begin{align*}
& \eta^{'*}_{3} \neq (-1)^{A_{1} - B_{1}} \eta^{'*}_{1} \\
\Rightarrow \quad &  \eta^{*}_{3} \neq (-1)^{A_{1} - B_{1}} (-1)^{A_{2} - B_{2} + 1}\eta^{*}_{1} \\
\Rightarrow \quad &  \eta_{3} \neq (-1)^{(A_{1} - B_{1}) + (A_{2} - B_{2}) + 1}\eta_{1} \\
\Rightarrow \quad &  \eta_{3} = (-1)^{(A_{1} - B_{1}) + (A_{2} - B_{2})}\eta_{1} \\
\end{align*}  
We also have
\begin{align*}
& (-1)^{A_{3} - B_{3}} \eta^{''*}_{3} = \eta^{''*}_{1} = (-1)^{A_{3} - B_{3}} \eta^{'*}_{1} \\
\Rightarrow \quad & \eta^{''*}_{3} = \eta^{'*}_{1} \\
\Rightarrow \quad & \eta^{''*}_{3} = (-1)^{A_{2} - B_{2} + 1}\eta^{*}_{1} \\
\Rightarrow \quad & \eta^{''*}_{3} = (-1)^{A_{2} - B_{2} + 1}\eta_{1} \\
\end{align*}
and
\begin{align*}
l_{3}^{''*} & = l^{'*}_{3} - (A_{1} - B_{1} - 2l^{'*}_{1}) - 1 \\
& = l^{*}_{3} - (A_{1} - B_{1} - 2l^{*}_{1}) - 1 \\
& = l_{3} + (B_{3} - B_{1}) - (A_{1} - B_{1} - 2l_{1}) - 1 \\
& = l_{3} + B_{3} - A_{1} + 2l_{1} - 1
\end{align*}         
So we will get the following conditions from \eqref{eq: case 3}.         
\begin{itemize}
\item If $\eta_{2} = (-1)^{A_{1} - B_{1}}\eta_{1}$, then
\[
0 \leqslant (l_{3} + B_{3} - A_{1} + 2l_{1} - 1) + B_{1} - l_{2} \leqslant (A_{3} + B_{3} - 0) - (A_{2} - B_{2})
\]
which implies
\[
(A_{1} - B_{1}) - B_{3} + 1 \leqslant l_{3} - l_{2} + 2l_{1} \leqslant A_{3} + (A_{1} - B_{1}) - (A_{2} - B_{2}) + 1
\]

\item If $\eta_{2} \neq (-1)^{A_{1} - B_{1}}\eta_{1}$, then
\[
(l_{3} + B_{3} - A_{1} + 2l_{1} - 1) + B_{1} + l_{2} > A_{2} - B_{2}
\]
which implies
\[
l_{3} + l_{2} + 2l_{1} >  (A_{1} - B_{1}) + (A_{2} - B_{2}) - B_{3} + 1
\]
\end{itemize}

\item If $\eta^{'*}_{3} = (-1)^{A_{1} - B_{1}} \eta^{'*}_{1}$ and 
         \[
         l^{'*}_{3} - l^{'*}_{1} < (A^{*}_{3} - B^{*}_{3})/2 - (A_{1} - B_{1}) + l^{'*}_{1},
         \] 
         then $\eta^{''*}_{1} \neq (-1)^{A_{3} - B_{3}}\eta^{''*}_{3}$ and 
         \[
         \begin{cases}
         l^{'*}_{1} = l^{''*}_{1} \\
         l^{''*}_{3} - l^{'*}_{3} = (A_{1} - B_{1} - 2l^{'*}_{1}) + 1 \\
         \eta^{''*}_{1} = (-1)^{A_{3} - B_{3}} \eta^{'*}_{1} 
         \end{cases}
         \]         
It follows
\begin{align*}
& \eta^{'*}_{3} = (-1)^{A_{1} - B_{1}} \eta^{'*}_{1} \\
\Rightarrow \quad &  \eta^{*}_{3} = (-1)^{A_{1} - B_{1}} (-1)^{A_{2} - B_{2} + 1}\eta^{*}_{1} \\
\Rightarrow \quad &  \eta_{3} = (-1)^{(A_{1} - B_{1}) + (A_{2} - B_{2}) + 1}\eta_{1} \\
\Rightarrow \quad &  \eta_{3} \neq (-1)^{(A_{1} - B_{1}) + (A_{2} - B_{2})}\eta_{1} 
\end{align*}         
and
\begin{align*}
& l^{'*}_{3} - l^{'*}_{1} < (A^{*}_{3} - B^{*}_{3})/2 - (A_{1} - B_{1}) + l^{'*}_{1} \\
\Rightarrow \quad & l^{*}_{3} - l^{*}_{1} < (A_{3} - B_{3})/2 + (B_{3} - B_{1}) - (A_{1} - B_{1}) + l^{*}_{1} \\
\Rightarrow \quad & l_{3} + (B_{3} - B_{1}) - l_{1} < (A_{3} - B_{3})/2 + (B_{3} - B_{1}) - (A_{1} - B_{1}) + l_{1} \\
\Rightarrow \quad & l_{3} - l_{1} < (A_{3} - B_{3})/2 - (A_{1} - B_{1}) + l_{1} 
\end{align*}
We also have
\begin{align*}
& (-1)^{A_{3} - B_{3}} \eta^{''*}_{3} \neq \eta^{''*}_{1} = (-1)^{A_{3} - B_{3}} \eta^{'*}_{1} \\
\Rightarrow \quad & \eta^{''*}_{3} = - \eta^{'*}_{1} \\
\Rightarrow \quad & \eta^{''*}_{3} = - (-1)^{A_{2} - B_{2} + 1}\eta^{*}_{1} \\
\Rightarrow \quad & \eta^{''*}_{3} = (-1)^{A_{2} - B_{2}}\eta_{1} \\
\end{align*}
and
\begin{align*}
l_{3}^{''*} & = l^{'*}_{3} + (A_{1} - B_{1} - 2l^{'*}_{1}) + 1 \\
& = l^{*}_{3} + (A_{1} - B_{1} - 2l^{*}_{1}) + 1 \\
& = l_{3} + (B_{3} - B_{1}) + (A_{1} - B_{1} - 2l_{1}) + 1 \\
& = l_{3} -2l_{1} + A_{1} + B_{3} - 2B_{1} + 1
\end{align*}    
So we will get the following conditions from \eqref{eq: case 3}.         
\begin{itemize}
\item If $\eta_{2} \neq (-1)^{A_{1} - B_{1}}\eta_{1}$, then
\[
0 \leqslant (l_{3} -2l_{1} + A_{1} + B_{3} - 2B_{1} + 1) + B_{1} - l_{2} \leqslant (A_{3} + B_{3} - 0) - (A_{2} - B_{2})
\]
which implies
\[
-(A_{1} - B_{1}) - B_{3} - 1 \leqslant l_{3} - l_{2} - 2l_{1} \leqslant A_{3} - (A_{1} - B_{1}) - (A_{2} - B_{2}) - 1
\]

\item If $\eta_{2} = (-1)^{A_{1} - B_{1}}\eta_{1}$, then
\[
(l_{3} -2l_{1} + A_{1} + B_{3} - 2B_{1} + 1) + B_{1} + l_{2} > A_{2} - B_{2}
\]
which implies
\[
l_{3} + l_{2} - 2l_{1} >  (A_{2} - B_{2}) - (A_{1} - B_{1}) - B_{3} - 1
\]
\end{itemize}

\item If $\eta^{'*}_{3} = (-1)^{A_{1} - B_{1}} \eta^{'*}_{1}$ and 
         \[
         l^{'*}_{3} - l^{'*}_{1} \geq (A^{*}_{3} - B^{*}_{3})/2 - (A_{1} - B_{1}) + l^{'*}_{1},
         \] 
         then $\eta^{''*}_{1} = (-1)^{A_{3} - B_{3}}\eta^{''*}_{3}$ and 
         \[
         \begin{cases}
         l^{'*}_{1} = l^{''*}_{1} \\
         (l^{''*}_{3} - l^{''*}_{1}) + (l^{'*}_{3} - l^{'*}_{1}) = (A^{*}_{3} - B^{*}_{3}) - (A_{1} - B_{1}) \\
         \eta^{''*}_{1} = (-1)^{A_{3} - B_{3}} \eta^{'*}_{1} 
         \end{cases}
         \]         
It follows
\begin{align*}
& \eta^{'*}_{3} = (-1)^{A_{1} - B_{1}} \eta^{'*}_{1} \\
\Rightarrow \quad &  \eta^{*}_{3} = (-1)^{A_{1} - B_{1}} (-1)^{A_{2} - B_{2} + 1}\eta^{*}_{1} \\
\Rightarrow \quad &  \eta_{3} = (-1)^{(A_{1} - B_{1}) + (A_{2} - B_{2}) + 1}\eta_{1} \\
\Rightarrow \quad &  \eta_{3} \neq (-1)^{(A_{1} - B_{1}) + (A_{2} - B_{2})}\eta_{1} 
\end{align*}         
and
\begin{align*}
& l^{'*}_{3} - l^{'*}_{1} \geqslant (A^{*}_{3} - B^{*}_{3})/2 - (A_{1} - B_{1}) + l^{'*}_{1} \\
\Rightarrow \quad & l^{*}_{3} - l^{*}_{1} \geqslant (A_{3} - B_{3})/2 + (B_{3} - B_{1}) - (A_{1} - B_{1}) + l^{*}_{1} \\
\Rightarrow \quad & l_{3} + (B_{3} - B_{1}) - l_{1} \geqslant (A_{3} - B_{3})/2 + (B_{3} - B_{1}) - (A_{1} - B_{1}) + l_{1} \\
\Rightarrow \quad & l_{3} - l_{1} \geqslant (A_{3} - B_{3})/2 - (A_{1} - B_{1}) + l_{1} 
\end{align*}
We also have
\begin{align*}
& (-1)^{A_{3} - B_{3}} \eta^{''*}_{3} = \eta^{''*}_{1} = (-1)^{A_{3} - B_{3}} \eta^{'*}_{1} \\
\Rightarrow \quad & \eta^{''*}_{3} = \eta^{'*}_{1} \\
\Rightarrow \quad & \eta^{''*}_{3} = (-1)^{A_{2} - B_{2} + 1}\eta^{*}_{1} \\
\Rightarrow \quad & \eta^{''*}_{3} = (-1)^{A_{2} - B_{2} + 1}\eta_{1} \\
\end{align*}
and
\begin{align*}
l_{3}^{''*} & =  l^{''*}_{1} - (l^{'*}_{3} - l^{'*}_{1}) + (A^{*}_{3} - B^{*}_{3}) - (A_{1} - B_{1}) \\
& = 2l^{'*}_{1} - l^{'*}_{3} + (A_{3} - B_{3}) + 2(B_{3} - B_{1}) - (A_{1} - B_{1}) \\
& = 2l^{*}_{1} - l^{*}_{3} + (A_{3} - B_{3}) + 2(B_{3} - B_{1}) - (A_{1} - B_{1}) \\
& = 2l_{1} - l_{3} - (B_{3} - B_{1}) + (A_{3} - B_{3}) + 2(B_{3} - B_{1}) - (A_{1} - B_{1}) \\
& = 2l_{1} - l_{3} + (A_{3} - B_{3}) + (B_{3} - B_{1}) - (A_{1} - B_{1}) \\
& = 2l_{1} - l_{3} + A_{3} - A_{1}  
\end{align*}    
So we will get the following conditions from \eqref{eq: case 3}.         
\begin{itemize}
\item If $\eta_{2} = (-1)^{A_{1} - B_{1}}\eta_{1}$, then
\[
0 \leqslant (2l_{1} - l_{3} + A_{3} - A_{1}) + B_{1} - l_{2} \leqslant (A_{3} + B_{3} - 0) - (A_{2} - B_{2})
\]
which implies
\[
(A_{1} - B_{1}) -A_{3} \leqslant -l_{3} - l_{2} + 2l_{1} \leqslant (A_{1} - B_{1}) - (A_{2} - B_{2}) + B_{3}
\]

\item If $\eta_{2} \neq (-1)^{A_{1} - B_{1}}\eta_{1}$, then
\[
(2l_{1} - l_{3} + A_{3} - A_{1}) + B_{1} + l_{2} > A_{2} - B_{2}
\]
which implies
\[
-l_{3} + l_{2} + 2l_{1} > (A_{1} - B_{1}) + (A_{2} - B_{2}) - A_{3}
\]
\end{itemize}

\end{itemize}

\bibliographystyle{amsalpha}

\bibliography{reps}

\end{document}

%% file: Macros_reps.tex
\newtheorem{theorem}{Theorem}[section]
\newtheorem{lemma}[theorem]{Lemma}
\newtheorem{proposition}[theorem]{Proposition}
\newtheorem{corollary}[theorem]{Corollary}

\newtheorem{question}[theorem]{Question}
\newtheorem*{theorem*}{Theorem}
\theoremstyle{remark}
\newtheorem{remark}[theorem]{Remark}

\newtheorem{example}[theorem]{Example}

\numberwithin{equation}{section}

\newcommand{\Q}[1]{\Psi(#1)}

\newcommand{\cQ}[1]{\bar{\Psi}(#1)}

\newcommand{\q}{\psi}

\newcommand{\Pkt}[1]{\Pi_{#1}}

\newcommand{\cPkt}[1]{\bar{\Pi}_{#1}}

\renewcommand{\r}{\pi}

\renewcommand{\S}[1]{\mathcal{S}_{#1}}

\newcommand{\+}{\oplus}

\renewcommand{\L}[1]{{}^L#1}

\newcommand{\D}[1]{\widehat{#1}}

\newcommand{\Gal}[1]{\Gamma_{#1}}

\renewcommand{\Im}{\text{Im}\,}

\newcommand{\Ind}{\text{Ind}}

\newcommand{\Cent}{\text{Cent}}

\newcommand{\e}{\varepsilon}

\newcommand{\Rep}{\text{Rep}}

\newcommand{\Jac}{\text{Jac}}

\newcommand{\ul}{\underline{l}}

\newcommand{\ueta}{\underline{\eta}}

